\documentclass[12pt]{amsart}
\usepackage[utf8]{inputenc}
\usepackage{format}

\title{Unipotent Representations of Complex Groups and Extended Sommers Duality}

\begin{document}

\begin{abstract}
Let $G$ be a complex reductive algebraic group. In \cite{LMBM}, we have defined a finite set of irreducible admissible representations of $G$ called \emph{unipotent representations}, generalizing the \emph{special unipotent representations} of Arthur (\cite{Arthur1983}) and Barbasch-Vogan (\cite{BarbaschVogan1985}). These representations are defined in terms of filtered quantizations of symplectic singularities and are expected to form the building blocks of the unitary dual of $G$. In this paper, we provide a description of these representations in terms of the Langlands dual group $G^{\vee}$. To this end, we construct a duality map $D$ from the set of pairs $(\OO^{\vee},\bar{C})$ consisting of a nilpotent orbit $\OO^{\vee} \subset \fg^{\vee}$ and a conjugacy class $\bar{C}$ in Lusztig's canonical quotient $\bar{A}(\OO^{\vee})$ to the set of finite covers of nilpotent orbits in $\fg^*$. 
\end{abstract}

\maketitle

\section{Introduction}

Let $G$ be a complex reductive algebraic group with Lie algebra $\fg$. A \emph{nilpotent cover} is a finite connected $G$-equivariant cover of a nilpotent co-adjoint orbit $\OO \subset \fg^*$. Write $\Cov(G)$ for the (finite) set of isomorphism classes of nilpotent covers for $G$. In \cite{LMBM}, we attach to each nilpotent cover $\widetilde{\OO}$ a finite set $\mathrm{Unip}_{\widetilde{\OO}}(G)$ of irreducible unitary representations of $G$ called \emph{unipotent representations}. Such representations possess a variety of special properties and are conjectured to form the building blocks of the unitary dual. They include, as a proper subset, all special unipotent representations in the sense of Arthur (\cite{Arthur1983}) and Barbasch-Vogan (\cite{BarbaschVogan1985}). The purpose of this article is to give a description of the sets $\mathrm{Unip}_{\widetilde{\OO}}(G)$ in terms of the Langlands dual group $G^{\vee}$. 

The dual-group parameters which appear in our description are called \emph{Lusztig-Achar data}. A Lusztig-Achar datum for $G^{\vee}$ is a pair $(\OO^{\vee},\bar{C})$ consisting of a nilpotent orbit $\OO^{\vee} \subset \fg^{\vee}$ and a conjugacy class $\bar{C}$ in Lusztig's canonical quotient $\bar{A}(\OO^{\vee})$ of the $G^{\vee}$-equivariant fundamental group of $\OO^{\vee}$, see Section \ref{subsec:Abar}. Such objects have previously appeared in the work of Achar \cite{Achar2003}. Write $\LA(G^{\vee})$ for the (finite) set of Lusztig-Achar data for $G^{\vee}$. We will also consider the subset $\LA^*(G^{\vee}) \subset \LA(G^{\vee})$ of so-called \emph{special Lusztig-Achar data}, defined in \cite[Section 3]{Achar2003}. This subset includes all pairs of the form $(\OO^{\vee},1)$ (among others).

Choose a Cartan subalgebra $\fh \subset \fg$ and let $W$ denote the Weyl group. Let $\fh_{\RR}^*$ denote the real form of the dual space $\fh^*$ spanned by the roots in $G$. To each Lusztig-Achar datum $(\OO^{\vee},\bar{C})$, we attach a $W$-invariant subset $S(\OO^{\vee},\bar{C}) \subset \fh_{\RR}^*$. If $(\OO^{\vee},\bar{C})$ is special we show in Theorem \ref{thm:inflchars} that there is a unique minimal-length $W$-orbit $\gamma(\OO^{\vee},\bar{C}) \subset S(\OO^{\vee},\bar{C})$.  This $W$-orbit determines an infinitesimal character for $U(\fg)$ by the Harish-Chandra isomorphism (which we also denote by $\gamma(\OO^{\vee},\bar{C})$), and hence a maximal ideal $I(\OO^{\vee},\bar{C})$ in $U(\fg)$. We consider the finite set
\begin{align*}
\mathrm{Unip}_{(\OO^{\vee},\bar{C})}(G) = \{&\text{irreducible $G$-equivariant Harish-Chandra $U(\fg)$-bimodules}\\
&\text{which are annihilated on both sides by } I(
\OO^{\vee},\bar{C})\}.\end{align*}
The notation suggests a relationship between $\mathrm{Unip}_{(\OO^{\vee},\bar{C})}(G)$ and $\mathrm{Unip}_{\widetilde{\OO}}(G)$. To make this relationship precise, we define in Section \ref{sec:mainresults} a natural duality map
$$D: \LA^*(G^{\vee}) \to \Cov(G)$$
This map generalizes the duality maps of Barbasch-Vogan-Lusztig-Spaltenstein (\cite{BarbaschVogan1985}), Sommers (\cite{Sommers2001}), and Mason-Brown-Matvieievskyi-Losev (\cite{LMBM}), and is equivalent (under some nontrivial identifications) to the duality map of Achar (\cite{Achar2003}). It enjoys various nice properties. For example, it is injective, maps distinguished Lusztig-Achar data to birationally rigid covers, and intertwines saturation of Lusztig-Achar data with birational induction of nilpotent covers. Our first main result is the following.

\begin{theorem}[See Theorem \ref{thm:inflchars} below]\label{theorem:main1intro}
Let $(\OO^{\vee},\bar{C}) \in \LA^*(G^{\vee})$ and let $\widetilde{\OO} = D(\OO^{\vee},\bar{C})$. Then there is an equality
$$\mathrm{Unip}_{(\OO^{\vee},\bar{C})}(G) = \mathrm{Unip}_{\widetilde{\OO}}(G).$$
\end{theorem}

Our second main result gives a parameterization of the set $\mathrm{Unip}_{(\OO^{\vee},\bar{C})}(G^{\vee})$ in terms of the Langlands dual group. For each special Lusztig-Achar datum $(\OO^{\vee},\bar{C}) \in \LA^*(G^{\vee})$, define
$$\gamma = \gamma(\OO^{\vee},\bar{C}), \quad s = \exp(2\pi i \gamma), \quad R^{\vee} = Z_{G^{\vee}}(s)^{\circ}, \quad L^{\vee} = Z_{G^{\vee}}(\gamma).$$
Note that $R^{\vee}$ is a pseudo-Levi subgroup of $G^{\vee}$ and $L^{\vee}$ is a Levi subgroup of $R^{\vee}$. Hence, we can consider the Richardson orbit $\OO_{R^{\vee}}$ for $R^{\vee}$ corresponding to $L^{\vee}$ (this is the nilpotent orbit for $R^{\vee}$ obtained by Lusztig-Spaltenstein induction from the $0$ orbit for $L^{\vee}$).

\begin{theorem}[See Corollary \ref{cor:HC} below]\label{theorem:main2intro}
Assume $G$ is adjoint and let $(\OO^{\vee},\bar{C}) \in \mathrm{LA}^*(G^{\vee})$. Then there is a natural bijection
$$\mathrm{Unip}_{(\OO^{\vee},\bar{C})}(G) \simeq \{\text{irreducible representations of } \bar{A}(\OO_{R^{\vee}})\}.$$
\end{theorem}
\begin{rmk}
In fact, the representations in $\mathrm{Unip}_{(\OO^{\vee},\bar{C})}(G)$ are the irreducible objects in the monoidal category $\HC^G(U(\fg)/I(\OO^{\vee},\bar{C}))$ of $G$-equivariant Harish-Chandra $U(\fg)$-modules annihilated on both sides by $I(\OO^{\vee},\bar{C})$, see Section \ref{sec:unipotentbimodules}. The bijection in Theorem \ref{theorem:main2intro} comes from a monoidal equivalence of categories
$$\HC^G(U(\fg)/I(\OO^{\vee},\bar{C})) \simeq \bar{A}(\OO_{R^{\vee}})\modd$$
\end{rmk}

We note that the main definitions in this paper (the duality map $D$, the the set $S(\OO^{\vee},\bar{C}) \subset \fh_{\RR}^*$, and the pair $(R^{\vee},\OO_{R^{\vee}})$) are formulated in a canonical (i.e. case-free) manner, although many of our proofs require case-by-case anaylsis.

To deduce Theorem \ref{theorem:main2intro}, we were forced to prove several new results on conical symplectic singularities and nilpotent covers, which may be of independent interest (see Section \ref{sec:symplecticsingularities}, in particular Sections \ref{sec:Wcovers}-\ref{sec:maximal}). The main such result is Theorem \ref{thm:maximaltomaximal} --- it states that birational induction of nilpotent covers takes maximal covers to maximal covers (in the sense of \cite[Section 6.5]{LMBM}, see also Lemma \ref{lem:coverfacts}).

Now fix a nilpotent adjoint orbit $\OO^{\vee} \subset \fg^{\vee}$. Choose an $\mathfrak{sl}(2)$-triple $(e^{\vee},f^{\vee},h^{\vee})$ with $e^{\vee} \in \OO^{\vee}$. The corresponding homomorphism $\psi_{\OO^{\vee}}: SL(2) \to G^{\vee}$ is a unipotent Arthur parameter for the complex group $G$ and the corresponding Arthur packet coincides with the set $\mathrm{Unip}_{(\OO^{\vee},1)}(G)$. So in fact Theorem \ref{theorem:main2intro} provides, as a special case, a parameterization of the elements of the unipotent Arthur packets for a complex reductive group (such a parameterization was previously obtained in \cite{BarbaschVogan1985} and \cite{Wong2023} only in the special case when $\OO^{\vee}$ is special).

We conclude with a few remarks regarding symplectic duality, which appears to be connected, in somewhat mysterious ways, to the main results of this paper. In \cite[Section 9.3]{LMBM} it was conjectured that the nilpotent Slodowy slice $S^\vee$ to the nilpotent orbit $\OO^\vee\subset \fg^\vee$ is symplectically dual to the affinization of the nilpotent cover $D(\OO^\vee, 1)$. In ongoing work, Finkelberg, Hanany, and Nakajima produce for each nilpotent orbit $\OO^{\vee}$ in $\mathfrak{so}(2n)$ or $\mathfrak{sp}(2n)$ an ortho-symplectic quiver gauge theory $Q$ with Higgs branch $S^\vee$ and Coulomb branch isomorphic to the affinization of a certain cover of the special nilpotent orbit $d(\OO^\vee)$ with Galois group $\bar{A}(\OO^\vee)$. It is reasonable to conjecture that this nilpotent cover is isomorphic to $D(\OO^{\vee},1)$. The results of this paper offer strong evidence for this. Indeed, it follows from Theorem \ref{thm:Gamma} that $D(\OO^{\vee},1)$ is a Galois cover of $d(\OO^{\vee})$ with Galois group $\bar{A}(\OO^\vee)$. 


To each conjugacy class $\bar{C}$ in $\Bar{A}(\OO^\vee)$ we can associate a certain finite group $\Pi$ depending on $\bar{C}$ which acts on both $S^\vee$ and $\Spec(\CC[D(\OO^{\vee},1)])$ by graded Poisson automorphisms. Based on the observations in \cite{KP_special}, we expect that the variety of fixed points $\Spec(\CC[D(\OO^{\vee},1)])^\Pi$ is identified with $\Spec(\CC[D(\OO^{\vee},\bar{C})])$. We believe that the assignment 
$$(S^\vee, \Pi) \mapsto \Spec(\CC[D(\OO^{\vee},\Bar{C})])$$
should be regarded as a special case of a (still highly conjectural) \emph{equivariant} version of symplectic duality. This topic will be explored in a future paper. 



\subsection{Structure of paper}

In Section \ref{sec:prelim}, we recall some preliminaries related to nilpotent orbits and Lie theory, including (birational) induction, Lusztig's canonical quotient, primitive ideals in the unversal enveloping algebra, and Sommers duality. In Section \ref{sec:symplecticsingularities}, we recall some preliminaries related to symplectic singularities and unipotent ideals. This section also includes several new results. In Section \ref{sec:mainresults} we state our main results on unipotent representations. The proofs of these results appear in Sections \ref{sec:combinatoricsclassical}-\ref{sec:tables}. The classical cases are proved in Sections \ref{sec:combinatoricsclassical} and \ref{sec:proofsclassical}; the exceptional cases are handled in Sections \ref{sec:proofsexceptional} and \ref{sec:tables}.

\subsection{Acknowledgments}

We would like to thank Jeffrey Adams, Dan Barbasch, Ivan Losev, Hiraku Nakajima, Alexander Premet, Eric Sommers, and David Vogan for helpful discussions. The third author is grateful to Chen Jiang and Yoshinori Namikawa for numerous discussions on birational geometry, and to Binyong Sun for his hospitality and inspiring discussions during his visits to the Institute for Advanced Study at Zhejiang University. 

The work of S.Yu was partially supported by China NSFC grants (Grants No. 12001453 and 12131018) and Fundamental Research Funds for the Central Universities (Grants No. 20720200067 and 20720200071).

\tableofcontents

\section{Preliminaries}\label{sec:prelim}

Let $G$ be a complex connected reductive algebraic group with Lie algebra $\fg$. A \emph{nilpotent orbit} for $G$ is a co-adjoint orbit $\OO \subset \fg^*$ which is stable under scaling. Let
$$\Orb(G) := \{\text{nilpotent orbits $\OO \subset \fg^*$}\}$$
It is well-known that $\Orb(G)$ is finite and independent of isogeny.  Let $\cN \subset \fg^*$ denote the union of all nilpotent orbits in $\fg^*$.

For each $\OO \in \mathsf{Orb}(G)$, choose an element $e \in \OO$, and let $A(\OO)$ denote the (finite) component group of the centralizer of $e$ in $G$. Note that $A(\OO)$ is independent (up to conjugacy) of the choice of $e$ in $\OO$. 

A \emph{conjugacy datum} for $G$ is a pair $(\OO,C)$ consisting of a nilpotent orbit $\OO \in \mathsf{Orb}(G)$ and a conjugacy class $C$ in $A(\OO)$. Let
$$\Conj(G) := \{\text{conjugacy data $(\OO,C)$ for $G$}\}$$
Since $\Orb(G)$ is finite and $A(\OO)$ is finite (for each $\OO \in \Orb(G)$), $\Conj(G)$ is finite. 

A \emph{pseudo-Levi subgroup} of $G$ is the identity component of the centralizer of a semisimple element $s \in G$. A \emph{McNinch-Sommers datum} for $G$ is a triple $(M,tZ^{\circ},\OO_M)$ consisting of a pseudo-Levi subgroup $M \subset G$, a coset $tZ^{\circ}$ in the component group $Z/Z^{\circ}$ of $Z = Z(M)$, and a nilpotent orbit $\OO_M \in \Orb(M)$ such that $M=Z_G(tZ^{\circ})^{\circ}$. Note that $G$ acts by conjugation on the set of McNinch-Sommers data. Let
$$\MS(G) := \{\text{McNinch-Sommers data $(M,tZ^{\circ},\OO_M)$ for $G$}\}/G.$$
\subsection{Saturation}\label{subsec:saturation}

Suppose $L \subset G$ is a Levi subgroup. If $\OO_L \in \Orb(L)$, then $G\cdot \OO_L \in \Orb(G)$. This defines a map
\begin{equation}\label{eq:satorb}\mathrm{Sat}^G_L: \Orb(L) \to \Orb(G), \qquad \mathrm{Sat}^G_L\OO_L = G \cdot \OO_L.\end{equation}
Now choose $e \in \OO_L$. The inclusion $Z_L(e) \subset Z_G(e)$ induces a group homomorphism
\begin{equation}\label{eq:iota}\iota: A(\OO_L) \to A(\mathrm{Sat}^G_L\OO_L).\end{equation}
(in fact, $\iota$ is injective, but we will not use this fact). Thus we get a map
\begin{equation}\label{eq:satconj}\mathrm{Sat}^G_L: \Conj(L) \to \Conj(G), \qquad \mathrm{Sat}^G_L(\OO_L,C_L) = (\mathrm{Sat}^G_L\OO_L, \iota(C_L)).\end{equation}
Finally, suppose $(M,tZ^{\circ},\OO_M) \in \MS(L)$. Since $Z_G(tZ^{\circ}) \subseteq Z_G(Z^{\circ}) \subseteq Z_G(Z(L)^{\circ})=L$, we have $Z_G(tZ^{\circ}) = Z_L(tZ^{\circ})^{\circ}=M$. So there is a tautological map
$$\Sat^G_L: \MS(L) \to \MS(G), \qquad \Sat^G_L(M,tZ^{\circ},\OO_M) = (M,tZ^{\circ},\OO_M)$$
A nilpotent orbit (resp. conjugacy datum, resp. McNinch-Sommers datum) is \emph{distinguished} if cannot be obtained by saturation from a proper Levi subgroup. Let $\fz(\fa)$ denote the center of a Lie algebra $\fa$.

\begin{lemma}\label{lem:distinguishedMS}
Let $(M,tZ^{\circ},\OO_M) \in \MS(G)$. Then the following are equivalent:
\begin{itemize}
    \item[(i)] $(M,tZ^{\circ},\OO_M)$ is distinguished.
    \item[(ii)] $M$ is not contained in a proper Levi subgroup of $G$.
    \item[(iii)] $\fz(\fg)=\fz(\fm)$.
    \item[(iv)] $M$ is of maximal semisimple rank.
\end{itemize}
\end{lemma}

\begin{proof}
Clearly (ii) implies (i). To see that (i) implies (ii), suppose that $M$ is contained in a proper Levi subgroup $L$ of $G$. Then $Z_L(tZ^{\circ})^{\circ} = Z_G(tZ^{\circ})^{\circ} \cap L = M \cap L = M$. So $(M,tZ^{\circ},\OO_M) \in \MS(L)$, i.e. $(M,tZ^{\circ},\OO_M) \in \MS(G)$ is not distinguished.

The equivalence of (ii) and (iii) is an immediate consequence of the following well-known facts: if $L$ is a Levi subgroup of $G$, then $L=Z_G(\fz(\fl))$ and $L \neq G$ if and only if $\fz(\fl)\neq \fz(\fg)$. (iii) and (iv) are equivalent by definition.
\end{proof}

We say that a McNinch-Sommers datum $(M,tZ^{\circ},\OO_M)$ is \emph{large} if $\OO_M$ is distinguished. Denote the set of conjugacy classes of large McNinch-Sommers data by $\MS^{large}(G)$. It is clear that largeness is preserved under saturation. Let $(M,tZ^{\circ},\OO_M) \in \MS(G)$. Choose $e \in \OO_M$. Then $Z^{\circ} \subset Z_G(e)^{\circ}$ and $t \in Z_G(e)$. Thus, we get a map
$$\pi: \MS(G) \to \Conj(G), \qquad \pi(M,tZ^{\circ},\OO_M) = (G \cdot e, tZ_G(e)^{\circ}).$$
We will sometimes write $\pi^G$ instead of $\pi$ when we wish to emphasize the dependence on $G$.
\begin{lemma}\label{lem:alphabeta}
The following are true
\begin{itemize}
    \item[(i)] The restriction of $\pi$ to $\MS^{large}(G)$ is a bijection onto $\Conj(G)$.
    \item[(ii)] If $L \subset G$ is a Levi subgroup, then the following diagram commutes
    \begin{center}
    \begin{tikzcd}
    \MS(L) \ar[d,twoheadrightarrow,"\pi^L"] \ar[r,"\mathrm{Sat}^G_L"] & \MS(G) \ar[d,twoheadrightarrow,"\pi^G"]\\
    \Conj(L) \ar[r,"\mathrm{Sat}^G_L"]& \Conj(G)
    \end{tikzcd}
    \end{center}
    \item[(iii)] If $(\OO,C) \in \Conj(G)$ is distinguished, then $\pi^{-1}(\OO,C)$ consists of large, distinguished McNinch-Sommers data.
\end{itemize}
\end{lemma}

\begin{proof}
(i) is the content of \cite[Theorem 1]{SommersMcNinch}. (ii) is immediate from the definitions. For (iii), suppose $(\OO,C) \in \Conj(G)$ is distinguished. It follows from (ii) that $\pi^{-1}(\OO,C)$ consists of distinguished McNinch-Sommers data. Suppose $(M,tZ(M)^{\circ},\OO_M) \in \pi^{-1}(\OO,C)$ is not large, i.e. $\OO_M$ is not distinguished. Then there is a proper Levi subgroup $L \subset M$ and a nilpotent orbit $\OO_L \in \Orb(L)$ such that $\OO_M=\Sat^G_L \OO_L$. A Levi subgroup of a pseudo-Levi subgroup of $G$ is a pseudo-Levi subgroup of $G$. So $L$ is a pseudo-Levi subgroup of $G$. Since $L \subset M$, there are inclusions $Z(M) \subset Z(L)$ and $Z(M)^{\circ} \subset Z(L)^{\circ}$. Hence, $tZ(M)^{\circ}$ determines a coset $tZ(L)^{\circ}$ in $Z(L)/Z(L)^{\circ}$. Since $tZ(M)^{\circ} \subseteq tZ(L)^{\circ}$, we have
$$Z_G(tZ(L)^{\circ})^{\circ} \subseteq  Z_G(tZ(M)^{\circ})^{\circ} = M$$
Also
$$Z_G(tZ(L)^{\circ})^{\circ} \subseteq Z_G(Z(L)^{\circ})$$
Thus,
$$Z_G(tZ(L)^{\circ})^{\circ} \subseteqq Z_M(Z(L)^{\circ}) = L$$
On the other hand, clearly $L \subseteq Z_G(tZ(L)^{\circ})^{\circ}$. So in fact, $L=Z_G(tZ(L)^{\circ})^{\circ}$, i.e. $(L,tZ(L)^{\circ},\OO_L) \in \MS(G)$. By the definition of $\pi$, we have $\pi(L,tZ(L)^{\circ},\OO_L)$, so $(L,tZ(L)^{\circ},\OO_L)$ must be distinguished by the argument above. And yet, since $L \subset M$ is a proper Levi subgroup, we have $\fz(\fg) \subseteq \fz(\fm) \subsetneq \fz(\fl)$, so $(L,tZ(L)^{\circ},\OO_L)$ is not distinguished by Lemma \ref{lem:distinguishedMS}. This is a contradiction.
\end{proof}

Write 
\begin{align*}
    \Orb_0(G) &:= \{(L,\OO_L) \mid \OO_L \in \Orb(L) \text{ distinguished}\}/G\\
    \Conj_0(G) &:= \{(L,(\OO_L,C_L)) \mid (\OO_L,C_L) \in \Conj(L) \text{ distinguished}\}/G\\
    \MS_0(G) &:= \{(L,(M,tZ^{\circ},\OO_M)) \mid (M,tZ^{\circ},\OO_M) \in \MS(L)\text{ distinguished}\}/G\\
    \MS^{large}_0(G) &= \{(L,(M,tZ^{\circ},\OO_M)) \mid (M,tZ^{\circ},\OO_M) \in \MS^{large}(L)\text{ distinguished}\}/G\\
\end{align*}
where $L$ runs over all Levi subgroups of $G$. Saturation gives rise to surjective maps
$$\Orb_0(G) \to \Orb(G), \quad \Conj_0(G) \to \Conj(G), \quad \MS_0(G) \to \MS(G), \quad \MS^{large}_0(G) \to \MS^{large}(G)$$
At the risk of abusing notation, we denote all four maps by `$\Sat$'.

\begin{prop}\label{prop:BalaCarter}
The maps 
$$\Orb_0(G) \to \Orb(G), \quad \Conj_0(G) \to \Conj(G), \quad \MS_0(G) \to \MS(G), \quad \MS^{large}_0(G) \to \MS^{large}(G)$$
are bijections.
\end{prop}

\begin{proof}
The assertion for $\Orb(G)$ is the classical Bala-Carter theorem, see \cite[Theorem 8.2.12]{CM} for a proof. By Lemma \ref{lem:distinguishedMS}, a McNinch-Sommers datum $(M,tZ^{\circ},\OO_M) \in \MS(G)$ is distinguished if and only if $\fz(\fg)=\fz(\fm)$. On the other hand, there is a unique Levi subgroup $L \subset G$ containing $M$ such that $\fz(\fm)=\fz(\fl)$, namely $L:=Z_G(\fz(\fm))$. This proves this assertion for both $\MS(G)$ and $\MS^{large}(G)$. The assertion for $\Conj(G)$ follows from the assertion for $\MS^{large}(G)$ and Lemma \ref{lem:alphabeta}.
\end{proof}

\subsection{Nilpotent covers}\label{sec:nilpotentcovers}

A \emph{nilpotent cover} for $G$ is a finite \'{e}tale $G$-equivariant cover of a nilpotent co-adjoint $G$-orbit. A \emph{morphism of nilpotent covers} is a finite \'{e}tale map $\widetilde{\OO} \to \widehat{\OO}$ which intertwines the covering maps $\widetilde{\OO} \to \OO$ and $\widehat{\OO} \to \OO$ (any such map is automatically $G$-equivariant, see \cref{ex:nilpotentcover}). Let $\Aut(\widetilde{\OO},\OO)$ denote the set of invertible endomorphisms of $\widetilde{\OO} \to \OO$. If we fix a morphism of covers $\widetilde{\OO} \to \widehat{\OO}$, we can similarly define $\Aut(\widetilde{\OO},\widehat{\OO})$. In this case, we call $\Aut(\widetilde{\OO},\widehat{\OO})$ the \emph{Galois group} of the cover $\widetilde{\OO} \to \widehat{\OO}$. We say that the moprhism $\widetilde{\OO} \to \widehat{\OO}$ is \emph{Galois} if it induces an isomorphism $\widehat{\OO} \simeq \widetilde{\OO}/\Aut(\widetilde{\OO},\widehat{\OO})$.
Let
$$\Cov(G) := \{\text{isomorphism classes of nilpotent covers $\widetilde{\OO}$ for $G$}\}$$
If $\OO \in \Orb(G)$, then (isomorphism classes of) nilpotent covers of $\OO$ are in one-to-one correspondence with (conjugacy classes of) subgroups of $A(\OO)$. In particular, $\Cov(G)$ is finite. Occasionally, we will also consider the category $\mathcal{C}ov(G)$ consisting of nilpotent covers for $G$, equipped with morphisms of nilpotent covers.

Following \cite{LMBM}, we will define an equivalence relation $\sim$ on $\Cov(G)$. Suppose $\widetilde{\OO} \to \widehat{\OO}$ is a morphism of covers. Then there is an induced map of affine varieties $\Spec(\CC[\widetilde{\OO}]) \to \Spec(\CC[\widehat{\OO}])$. We say this map is {\it almost \'{e}tale} if it is \'{e}tale over all $G$-orbits in $\Spec(\CC[\widehat{\OO}])$ of codimension $2$ (it is automatically \'{e}tale over the open $G$-orbit $\widehat{\OO}$). We write $\widetilde{\OO} \geqslant \widehat{\OO}$ if there exists a morphism $\widetilde{\OO} \to \widehat{\OO}$ such that the induced map $\Spec(\CC[\widetilde{\OO}]) \to \Spec(\CC[\widehat{\OO}])$ has this property. This defines a partial order on $\Cov(G)$. 

\begin{definition}[Definition 6.5.1, \cite{LMBM}] \label{defn:cover_equivalence}
Let $\sim$ be the equivalence relation on $\Cov(G)$ defined by taking the symmetric closure of $\geqslant$. For $\widetilde{\OO} \in \Cov(G)$, let $[\widetilde{\OO}]$ denote the equivalence class of $\widetilde{\OO}$.
\end{definition}

We will need some basic facts about equivalence classes in $\Cov(G)$.

\begin{lemma}[Lemma 6.5.3, \cite{LMBM}]\label{lem:coverfacts}
Let $[\widetilde{\OO}] \subset \Cov(G)$ be an equivalence class. Then the following are true:
\begin{itemize}
    \item[(i)] $[\widetilde{\OO}]$ contains a unique maximal cover $\widetilde{\OO}_{max}$.
    \item[(ii)] $\widetilde{\OO}_{max}$ is Galois over every cover in $[\widetilde{\OO}]$.
\end{itemize}
\end{lemma}

\subsection{Induction}\label{subsec:induction}

Suppose $L \subset G$ is a Levi subgroup and let $\mathbb{O}_L \in \Orb(L)$. Fix a parabolic subgroup $P \subset G$ with Levi decomposition $P = LU$. The annihilator of $\fp$ in $\fg^*$ is a $P$-stable subspace $\fp^{\perp} \subset \fg^*$. Form the $G$-equivariant fiber bundle $G \times^P (\overline{\mathbb{O}}_L \times \fp^{\perp})$ over the partial flag variety $G/P$. There is a proper $G$-equivariant map
$$\mu: G \times^P (\overline{\mathbb{O}}_L \times \fp^{\perp}) \to \mathfrak{g}^* \qquad \mu(g,\xi) = \Ad^*(g)\xi$$
The image of $\mu$ is a closed irreducible $G$-invariant subset of $\cN$, and hence the closure in $\fg^*$ of a nilpotent $G$-orbit, denoted $\mathrm{Ind}^G_L\mathbb{O}_L  \in \Orb(G)$. It is a standard fact that $\Ind^G_L \OO_L$ is independent of the choice of parabolic $P$. The correspondence
$$\mathrm{Ind}^G_L: \Orb(L) \to \Orb(G)$$
is called \emph{Lusztig-Spaltenstein} induction (\cite{LS}). 

Now let $\widetilde{\OO}_L \in \Cov(L)$ and form the affine variety $\Spec(\CC[\widetilde{\OO}_L])$. There is an $L$-action on $\Spec(\CC[\widetilde{\OO}_L])$ (induced from the $L$-action on $\widetilde{\mathbb{O}}_L$) and a finite surjective $L$-equivariant map $\Spec(\CC[\widetilde{\OO}_L]) \to \overline{\mathbb{O}}_L$. Let $\widetilde{\mu}$ denote the composition
$$G \times^P (\Spec(\CC[\widetilde{\OO}_L]) \times \fp^{\perp}) \to G \times^P (\overline{\mathbb{O}}_L \times \fp^{\perp}) \overset{\mu}{\to} \fg^*.$$
The image of $\widetilde{\mu}$ is the closure of $\Ind^G_L\OO_L$ and the preimage of $\Ind^G_L\OO_L$ is a finite \'{e}tale $G$-equivariant cover, denoted $\Bind^G_L \widetilde{\OO}_L \in \Cov(G)$. Again, $\Bind^G_L\widetilde{\OO}_L$ is independent of the choice of parabolic $P$ (see \cite[Proposition 2.4.1(i)]{LMBM}). The correspondence
$$\Bind^G_L: \Cov(L) \to \Cov(G)$$
is called \emph{birational induction}. We will need several basic facts about birational induction.

\begin{prop}\label{prop:propsofbind}
Birational induction has the following properties:

\begin{enumerate}[label=(\roman*)]
    \item Suppose $\widetilde{\OO}_L, \widehat{\OO}_L \in \mathcal{C}ov(L)$ and let $\widetilde{\OO} = \Bind^G_L \widetilde{\OO}_L$, $\widehat{\OO} = \Bind^G_L \widehat{\OO}_L$. Then any $L$-equivariant morphism $p_L: \widehat{\OO}_L \to \widetilde{\OO}_L$ induces a canonically defined finite $G$-equivariant morphism $p = \Bind_L^G (p_L):  \widehat{\OO} \to \widetilde{\OO}$ with $\deg p = \deg p_L$. Therefore birational induction induces a well-defined functor $\mathcal{B}ind_L^G: \mathcal{C}ov(L) \to \mathcal{C}ov(G)$, $\widetilde{\OO}_L \mapsto \Bind_L^G \widetilde{\OO}_L$. 
    \item Suppose the morphism $p_L: \widehat{\OO}_L \to \widetilde{\OO}_L$ in (i) is a finite Galois $L$-equivariant covering. Then the induced covering $p = \Bind_L^G(p_L): \widehat{\OO} \to \widetilde{\OO}$ is a finite Galois $G$-equivariant covering. Moreover, the induced group homomorphism
$$ \Aut(\widehat{\OO}_L, \widetilde{\OO}_L) \xrightarrow{\sim} \Aut(\widehat{\OO},\widetilde{\OO})$$
is an isomorphism.
    \item Suppose $M \subset G$ is a Levi subgroup containing $L$. Then there is a canonical natural equivalence of functors
    $$\mathcal{B}ind^G_L \simeq \mathcal{B}ind^G_M \circ \mathcal{B}ind^M_L.$$
    \item For any $\widetilde{\OO}_L \in \Cov(L)$, $\deg(\widetilde{\OO}_L \to \OO_L)$ divides $\deg(\Bind^G_L \widetilde{\OO}_L \to \Ind^G_L \OO_L)$.
\end{enumerate}
\end{prop}

\begin{proof}
 Let $\widetilde{Y}=G \times^P (\Spec(\CC[\widetilde{\OO}_L]) \times \fp^{\perp})$ and $\widehat{Y} = G \times^P (\Spec(\CC[\widehat{\OO}_L]) \times \fp^{\perp})$. Extend the morphism $p_L$ to $p_L: \Spec(\CC[\widehat{\OO}_L])\to \Spec(\CC[\widetilde{\OO}_L])$. $p_L$ induces a natural finite morphism $\widehat{Y}\to \widetilde{Y}$ which restricts to the morphism $p=\Bind_L^G (p_L):  \widehat{\OO} \to \widetilde{\OO}$ as claimed in (i). The equality $\deg p = \deg p_L$ is clear by construction. The independence of $p=\Bind_L^G (p_L)$ on the parabolic $P$ follows from the argument of \cite[Lemma 4.1]{Losev4}.
 
 For (ii), we argue as follows. Every automorphism $\gamma$ of $\widehat{Y} \to \widetilde{Y}$ restricts to an automorphism $\widetilde{\gamma}|_{\widetilde{\OO}} \in \Aut (\widehat{\OO}, \widetilde{\OO})$. Thus, we obtain an injective homomorphism $\Aut(\widehat{\OO}_L, \widetilde{\OO}_L)\to \Aut(\widehat{\OO}, \widetilde{\OO})$. We wish to show that this homomorphism is surjective. Note that
$$|\Aut(\widehat{\OO}, \widetilde{\OO})|\leq \deg p = \deg p_L=|\Aut(\widehat{\OO}_L, \widetilde{\OO}_L)|$$
and equality holds if and only if the cover is Galois. The statement follows.

(iii) and (iv) are (ii) and (iv) of \cite[Proposition 2.4.1]{LMBM}.
\end{proof}

A nilpotent orbit (resp. nilpotent cover) is \emph{rigid} (resp. \emph{birationally rigid}) if it cannot be obtained by induction (resp. birational induction) from a proper Levi subgroup. Note that if $\OO \in \Orb(G)$ is rigid and $\widetilde{\OO}$ covers $\OO$, then $\widetilde{\OO}$ is birationally rigid. Let
$$\Cov_0(G) := \{(L,\widetilde{\OO}_L) \mid \widetilde{\OO}_L \in \Cov(L) \text{ is birationally rigid}\}/G$$
where $L$ runs over all Levi subgroups of $G$.

\begin{prop}[Proposition 2.4.1(iii), \cite{LMBM}]\label{prop:bindinjective}
The map
$$\Bind: \Cov_0(G) \to \Cov(G), \qquad \Bind(L,\widetilde{\OO}_L) = \Bind^G_L \widetilde{\OO}_L$$
is a bijection. 
\end{prop}

We will see in Section \ref{subsec:unipotent} that birational induction preserves the equivalence relation on nilpotent covers and takes maximal covers to maximal covers. 

We conclude this subsection by describing a large class of orbits which are birationally induced from $\{0\}$. Suppose $\OO \in \Orb(G)$. Using an $\Ad(\fg)$-invariant identification $\fg \simeq \fg^*$, we can regard $\OO$ as a nilpotent $G$-orbit in $\fg$. Choose an element $e \in \OO$ and an $\mathfrak{sl}(2)$-triple $(e,f,h)$. The operator $\ad(h)$ defines a $\ZZ$-grading on $\fg$
$$\fg = \bigoplus_{i \in \ZZ} \fg_i, \qquad \fg_i := \{\xi \in \fg \mid \ad(h)(\xi) = i\xi\}.$$
We say that $\OO$ is \emph{even} if $\fg_i=0$ for every odd integer $i$. In any case, we can define a parabolic subalgebra
\begin{equation}\label{eq:JMlevi}
 \fp_{\OO} = \mathfrak{l}_{\OO} \oplus \mathfrak{n}_{\OO}, \qquad \fl_{\OO} := \fg_0, \qquad \fn_{\OO} := \bigoplus_{i \geq 1} \fg_i.\end{equation}
We call $\fp_{\OO}$ (resp. $\fl_{\OO}$) the \emph{Jacobson-Morozov} parabolic (resp. Levi) associated to $\OO$. Both $\fp_{\OO}$ and $\fl_{\OO}$ are well-defined up to conjugation by $G$. The following result is well-known. The proof is contained in \cite{Kostant1959}, see also \cite[Thm 3.3.1]{CM}.

\begin{prop}\label{prop:even}
Suppose $\OO$ is an even nilpotent $G$-orbit. Then
$$\OO = \mathrm{Bind}^G_{L_{\OO}} \{0\}.$$
\end{prop}

\subsection{Birational induction and equivariant fundamental groups}\label{subsec:pi1_induction}

Fix the notation of Section \ref{subsec:induction}, e.g. $L$, $P$, $\widetilde{\OO}_L$, $\widetilde{\OO}$, $\widetilde{\mu}: G \times^P (\Spec(\CC[\widetilde{\OO}_L]) \times \fp^{\perp}) \to \overline{\OO}$ and so on. Let $U \subset P$ denote the unipotent radical. In this section, we will construct a surjective homomorphism $\phi^G_L(\widetilde{\OO}_L): \pi_1^G(\widetilde{\OO}) \to \pi_1^L(\widetilde{\OO}_L)$ between the equivariant fundamental groups of $\widetilde{\OO}$ and $\widetilde{\OO}_L$. 

Let $\widetilde{Z}^0 := G \times^P (\widetilde{\mathbb{O}}_L \times \fp^{\perp})$, a fiber bundle over the partial flag variety $G/P$. The inclusion $i: \widetilde{\mathbb{O}}=\widetilde{\mu}^{-1}(\OO)\subset \widetilde{Z}^0$ induces a group homomorphism
\begin{equation}\label{eqn:istar}
i_*: \pi_1(\widetilde{\mathbb{O}}) \to \pi_1(\widetilde{Z}^0).\end{equation}
Since $i: \widetilde{\mathbb{O}} \hookrightarrow \widetilde{Z}^0$ is an open embedding of smooth complex manifolds, the complement $\widetilde{Z}^0 - i(\widetilde{\mathbb{O}})$ is of real codimension $\geq 2$. Hence, the homomorphism (\ref{eqn:istar}) is surjective.

Note that $\widetilde{Z}^0$ is a vector bundle over the homogeneous space $G \times^P \widetilde{\mathbb{O}}_L$ with fiber $\fp^{\perp}$. So there is a natural isomorphism
$$\pi_1(\widetilde{Z}^0) \simeq \pi_1(G \times^P \widetilde{\mathbb{O}}_L).$$
If we fix a base point $x \in \widetilde{\mathbb{O}}$, we get a fibration $G \to \widetilde{\mathbb{O}}$ with fiber $G_x$. This fibration gives rise to an exact sequence of homotopy groups
$$\pi_1(G) \to \pi_1(\widetilde{\mathbb{O}}) \to \pi_0(G_x) \to 1.$$
The final (nontrivial) term is isomorphic to $\pi_1^G(\widetilde{\mathbb{O}})$.

Similarly, if we fix a base point $(1,y) \in G \times^P \widetilde{\mathbb{O}}_L$, we get a fibration $G \to G \times^P \widetilde{\mathbb{O}}_L$ with fiber $P_y$. This fibration gives rise to an exact sequence of homotopy groups
$$\pi_1(G) \to \pi_1(G \times^P \widetilde{\mathbb{O}}_L) \to \pi_0(P_y) \to 1.$$
Since $U\subset P_y$, we have $P_y= L_y \ltimes U$. It follows that $\pi_0(P_y) \simeq \pi_0(L_y) \simeq \pi_1^L(\widetilde{\mathbb{O}}_L)$. 

Now choose $x \in \widetilde{\OO}$ and $y \in \widetilde{\OO}_L$ such that $x=(1, \bar{x})$, where $\bar{x} \in \widetilde{\OO}_L \times \fp^{\perp}$, and $\bar{x}$ is mapped to $y$ under the projection $\widetilde{\OO}_L \times \fp^{\perp} \twoheadrightarrow\widetilde{\OO}_L$. Then we have group isomorphisms
    \[ \pi_0(L_{\bar{x}}) \simeq \pi_0(P_{\bar{x}}) \simeq \pi_0(G_{x}) = \pi^G_1(\widetilde{\OO})\]
induced by the inclusions $L \subset P \subset G$. There is a commutative diagram of groups
\begin{center}
    \begin{tikzcd}
      \pi_1(G) \ar[r] \ar[d,equals] & \pi_1(\widetilde{\mathbb{O}}) \ar[r] \ar[d,twoheadrightarrow,"i_*"] & \pi_1^G(\widetilde{\mathbb{O}}) \ar[r] & 1\\
      \pi_1(G) \ar[r] & \pi_1(\widetilde{Z}^0) \ar[r] & \pi_1^L(\widetilde{\mathbb{O}}_L) \ar[r] & 1
    \end{tikzcd}
\end{center}
where the fundamental groups $\pi_1(G)$, $\pi_1(\widetilde{\OO})$, $\pi_1(\widetilde{Z}^0)$, and $\pi_1^L(\widetilde{\OO}_L)$ are defined with respect to the base points $1 \in G$, $x \in \widetilde{\OO}$, $x \in \widetilde{Z}^0$, and $y \in \widetilde{\OO}_L$, respectively. We define 
\[\phi_L^G(\widetilde{\OO}_L): \pi_1^G(\widetilde{\OO}) \to \pi_1^L(\widetilde{\OO}_L)\]
to be the unique homomorphism which makes the above diagram commute. Sometimes we simply write $\phi$ for $\phi_L^G(\widetilde{\OO}_L)$ when there is no ambiguity.

The following lemma about properties of $\phi$ is standard and the proof is left to the reader.

\begin{lemma}\label{lem:phi}
    With the notations above, the following are true:
    \begin{itemize}
     \item[(i)]
        Let $\widetilde{\OO}_L, \widehat{\OO}_L \in \mathcal{C}ov(L)$ and let $\widetilde{\OO} = \Bind^G_L \widetilde{\OO}_L$, $\widehat{\OO} = \Bind^G_L \widehat{\OO}_L$. Let $p_L: \widehat{\OO}_L \to \widetilde{\OO}_L$ be an $L$-equivariant covering map, inducing a $G$-equivariant covering map $p=\Bind_L^G (p_L): \widetilde{\OO} \to \widehat{\OO}$ as in \cref{prop:propsofbind}. Then we have the following commutative diagram
        \begin{center}
        \begin{tikzcd}
            \pi_1^G(\widetilde{\OO}) \ar[r, hookrightarrow, "p_*"] \ar[d,twoheadrightarrow,"\phi"] & \pi_1^G(\widehat{\OO}) \ar[d, twoheadrightarrow,"\phi"] \\
            \pi_1^L(\widetilde{\OO}_L) \ar[r, hookrightarrow, "(p_L)_*"] & \pi_1^L(\widehat{\OO}_L) 
        \end{tikzcd}
        \end{center}
        Moreover, we have $\phi^{-1}(\pi_1^G(\widetilde{\OO}_L)) = \pi_1^G(\widetilde{\OO})$.

	\item[(ii)]
          Under the group isomorphisms $\pi_0(L_{\bar{x}}) \simeq \pi^G_1(\widetilde{\OO})$ and $\pi_0(L_y) \simeq \pi^L_1 (\widetilde{\OO}_L)$, the map $\phi: \pi^G_1(\widehat{\OO}) \to \pi^L_1 (\widetilde{\OO}_L)$ is identified with the map $\pi_0(L_{\bar{x}}) \to \pi_0(L_{y})$ induced by the natural inclusion $L_{\bar{x}} \subset L_{y}$. 
	\item[(iii)]
	   Suppose $M \subset G$ is a Levi subgroup containing $L$. Let $\widetilde{\OO}_L \in \Cov(L)$ and $\widetilde{\OO}_M= \Bind_L^M \widetilde{\OO}_L$ and $\widetilde{\OO} = \Bind_M^G \widetilde{\OO}_M = \Bind_L^G \widetilde{\OO}_L$. Then the map $\phi_L^G(\OO_L): \pi_1^G(\widetilde{\OO}) \to \pi_1^L(\widetilde{\OO}_L)$ is equal to the composition of $\phi_L^M(\widetilde{\OO}_L): \pi_1^M(\widetilde{\OO}_M) \to \pi_1^L(\widetilde{\OO}_L)$ and $\phi_M^G(\widetilde{\OO}_M): \pi_1^G(\widetilde{\OO}) \to \pi_1^M(\widetilde{\OO}_M)$.
	\end{itemize}
\end{lemma}

\subsection{Primitive ideals}

Let $U(\fg)$ be the universal enveloping algebra of $\fg$. Recall that $U(\fg)$ has a natural natural filtration and there is a $G$-equivariant isomorphism $\gr U(\fg) \simeq S(\fg)$. Now let $I \subset U(\fg)$ be a two-sided ideal. Then $\gr(I)$ corresponds under the isomorphism $\gr U(\fg) \simeq S(\fg)$ to a $G$-invariant ideal in $S(\fg)$. The \emph{associated variety} of $I$ is defined to be the vanishing locus $V(I) \subset \fg^*$ of this ideal.

A two-sided ideal $I \subset U(\fg)$ is said to be \emph{primitive} if it is the annihilator of a simple left $U(\fg)$-module. If $I$ is primitive, then $V(I)$ is irreducible (see \cite{Joseph1985}) and hence the closure of a (unique) nilpotent orbit.

By Quillen's lemma, the intersection of $I$ with the center $\mathfrak{Z}(\fg)$ of $U(\fg)$ is the kernel of an algebra homomorphism $\mathfrak{Z}(\fg) \to \CC$, called the \emph{infinitesimal character} of $I$. Such homomorphisms are identified via the Harish-Chandra isomorphism with $W$-orbits on $\fh^*$. For each $\gamma \in \fh^*/W$, there is a \emph{unique} maximal (primitive) ideal in $U(\fg)$ with infinitesimal character $\gamma$, which we will denote by $J(\gamma) \subset U(\fg)$.

\subsection{BVLS duality}

Let $G^{\vee}$ be the Langlands dual group of $G$ and let $\fg^{\vee}$ be its Lie algebra. If we fix a Cartan subalgebra $\fh \subset \fg$, we get a Cartan subalgebra $\fh^{\vee} \subset \fg^{\vee}$ and a canonical identification $\fh^{\vee} \simeq \fh^*$. To each nilpotent $G^{\vee}$-orbit $\OO^{\vee} \in \Orb(G^{\vee})$ we can associate a maximal ideal in $U(\fg)$ in the following manner. First, using an $G^{\vee}$-invariant isomorphism $(\fg^{\vee})^* \simeq \fg^{\vee}$, we can identify $\OO^{\vee}$ with a nilpotent $G^{\vee}$-orbit in $\fg^{\vee}$ (still denoted $\OO^{\vee}$). Next, we choose an $\mathfrak{sl}(2)$-triple $(e^{\vee},f^{\vee},h^{\vee})$ with $e^{\vee} \in \OO^{\vee}$ and $h^{\vee} \in \fh^{\vee} \simeq \fh^*$. The element $\frac{1}{2}h^{\vee}$ is well-defined up to the $W$-action on $\fh^*$, and hence determines an infinitesimal character for $U(\fg)$, which we will denote by $\gamma_{\OO^{\vee}}$.

\begin{definition}\label{def:specialunipotent}
Let $\OO^{\vee} \in \Orb(G^{\vee})$. Then the \emph{special unipotent ideal} attached to $\OO^{\vee}$ is the unique maximal ideal $J(\gamma_{\OO^{\vee}}) \subset U(\fg)$ with infinitesimal character $\gamma_{\OO^{\vee}}$.
\end{definition}

For each nilpotent orbit $\OO^{\vee} \in \Orb(G^{\vee})$, let $d(\OO^{\vee})$ denote the (unique) open $G$-orbit in $V(J(\gamma_{\OO^{\vee}}))$. This defines a map
$$d: \Orb(G^{\vee}) \to \Orb(G)$$
called \emph{Barbasch-Vogan-Lusztig-Spaltenstein (BVLS) duality}. We will sometimes write $d^G$ instead of $d$, when we wish to emphasize the dependence on $G$. 

\begin{prop}[Proposition A2, \cite{BarbaschVogan1985}]\label{prop:propsofd}
The map $d$ enjoys the following properties:
\begin{itemize}
    \item[(i)] $d$ is order-reversing (with respect to the clsoure orderings on $\Orb(G)$ and $\Orb(G^{\vee})$).
    \item[(ii)] $d^3=d$.
    \item[(iii)] If $L \subset G$ is a Levi subgroup, then
    $$d^G \circ \Sat^{G^{\vee}}_{L^{\vee}} = \Ind^G_L \circ d^L$$
\end{itemize}
\end{prop}

An orbit $\OO \in \Orb(G)$ is said to be \emph{special} if it lies in the image of $d^G$. We write $\Orb^*(G)$ for the set of special nilpotent orbits.

\subsection{Truncated induction}\label{subsec:truncated}

Let $M^{\vee} \subset G^{\vee}$ be a pseudo-Levi subgroup and let $M$ be the Langlands dual group of $M^\vee$. In \cite[Chapter 13.3]{Lusztig1984}, Lusztig defines a map called \emph{truncated induction}
$$j^G_M: \Orb^*(M) \to \Orb(G)$$
This map generalizes Lusztig-Spaltenstein induction in the following sense: if $M$ is a Levi subgroup of $G$, then $j^G_M$ coincides with the restriction of $\Ind^G_M$ to $\Orb^*(M)$. 

The definition of $j^G_M$ is not particularly relevant for the purposes of this paper, so we will be brief in recalling it. Choose a Cartan subalgebra $\fh \subset \fm$ and let $W$ (resp. $W_M$) denote the Weyl group of $G$ (resp. $M$). There is a surjective map 
\begin{equation}\label{eq:springer}\{(\OO,\psi) \mid \OO \in \Orb(G), \ \psi \in \widehat{A(\OO)}\}
 \twoheadrightarrow \widehat{W} \sqcup \{0\}, \qquad (\OO,\psi) \mapsto E_{(\OO,\psi)}.\end{equation}
called the \emph{Springer correspondence} (see \cite{Springerconstruction}). We say that an irreducible representation $\psi$ of $A(\OO)$ is of \emph{Springer type} if $E_{(\OO,\psi)} \neq 0$ (the trivial representation of $A(\OO)$ is always of Springer type). For each irreducible representation $E$ of $W$, Lusztig defines a nonnegative integer $b_E$ called the \emph{b-value} or \emph{fake degree} of $E$. If $\OO \in \Orb^*(M)$, it is shown in \cite[Chapter 13.3]{Lusztig1984} that $\Ind^W_{W_M} E_{\OO,1}$ contains a unique irreducible subrepresentation $E'$ such that $b_{E'} = b_{E_{(\OO,1)}}$ and $E' = E_{(\OO',1)}$ for some $\OO' \in \Orb(G)$. Then $j^G_M$ is defined by $j^G_M \OO=\OO'$.

\subsection{Sommers duality}

Consider the map
$$\underline{d}_S: \MS(G^{\vee}) \to \Orb(G), \qquad \underline{d}_S(M^{\vee},tZ^{\circ},\OO_M) = j^G_M d(\OO_{M^{\vee}})$$
\emph{Sommers duality} is defined to be the composition
$$d_S: \Conj(G^{\vee}) \overset{\pi^{-1}
}{\to} \MS^{large}(G^{\vee}) \overset{\underline{d}_S}{\to} \Orb(G)$$
where $\pi:\MS^{large}(G^{\vee}) \to \Conj(G^{\vee})$ is the bijection of Lemma \ref{lem:alphabeta}(i). We will sometimes write $d_S^G$ instead of $d_S$ when we wish to emphasize the dependence on $G$.

\begin{prop}[Section 6, \cite{Sommers2001}]\label{prop:propsofds}
The map $d_S: \Conj(G^{\vee}) \to \Orb(G)$ has the following properties:
\begin{itemize}
    \item[(i)] $d_S$ is surjective.
    \item[(ii)] If $L \subset G$ is a Levi subgroup, then
    $$d_S^G \circ \Sat^{G^{\vee}}_{L^{\vee}} = \Ind^G_L \circ d_S^L$$
    \item[(iii)] $d_S(\OO^{\vee},1) = d(\OO^{\vee})$, for every $\OO^{\vee} \in \Orb(G^{\vee})$.
\end{itemize}
\end{prop}

\subsection{Lusztig's canonical quotient}\label{subsec:Abar}

For each $\OO \in \Orb(G)$, there is a canonically defined quotient group $\bar{A}(\OO)$ of $A(\OO)$, called \emph{Lusztig's canonical quotient}. In \cite{Lusztig1984}, this quotient is defined in the case when $\OO$ is \emph{special}, but in fact Lusztig's definition is valid for arbitrary $\OO$. An alternative description of $\bar{A}(\OO)$ is provided in \cite[Section 5]{Sommers2001}, which we briefly recall here. 

For each conjugacy datum $(\OO,C) \in \Conj(G)$, Sommers defines a nonnegative integer $\tilde{b}_{(\OO,C)}$ called the $\tilde{b}$\emph{-value} of $(\OO,C)$ as follows. If $(M, tZ^{\circ},\OO_M) \in \MS(G)$ is such that $\pi(M, tZ^{\circ},\OO_M) = (\OO, C)$, then
$$\tilde{b}_{(\OO,C)} := \frac{1}{2}(\dim M - \dim(d^M(\OO_M))- \rank G ),$$
Equivalently, $\tilde{b}_{(\OO,C)}$ is the dimension of the Springer fiber attached to any element in $d^M(\OO_M)$. Sommers shows that $\tilde{b}_{(\OO,C)}$ only depends on $(\OO,C)$ (i.e. is independent of the lift $(M, tZ^{\circ},\OO_M)$) \cite[Proposition 1]{Sommers2001}. In fact, $\tilde{b}_{(\OO,C)}$ coincides with the dimension of the Springer fibers of the orbit $d_S(\OO,C)$ in $\fg^\vee$ (see \cite[Section 6]{Sommers2001}). The $\tilde{b}$-values satisfy $\tilde{b}_{\OO,C} \geq \tilde{b}_{\OO,1}$ for any conjugacy class $C$ of $A(\OO)$ (\cite[Proposition 3]{Sommers2001}).

Let $N'$ be the union of all conjugacy classes $C$ in $A(\OO)$ such that $\tilde{b}_{\OO,C} = \tilde{b}_{\OO,1}$. A case-by-case calculation shows that $N'$ is a subgroup of $A(\OO)$ and coincides with a subgroup $N$ which Lusztig defines using Springer theory and generic degrees of irreducible Weyl group representations, so that $\bar{A}(\OO) = A(\OO)/ N = A(\OO)/ N'$ (\cite[Theorem 6]{Sommers2001}).

\subsection{Lusztig-Achar data}\label{subsec:Achar}
A \emph{Lusztig-Achar datum} for $G$ is a pair $(\OO, \bar{C})$ consisting of a nilpotent orbit $\OO \in \Orb(G)$ and a conjugacy class $\bar{C}$ in $\bar{A}(\OO)$. Let
$$\LA(G) := \{\text{Lusztig-Achar data $(\OO,\bar{C})$ for $G$}\}$$
Since the groups $\bar{A}(\OO)$ are independent of isogeny, so is the set $\LA(G)$. Of course, there is a natural projection $\Conj(G)\to \LA(G) $. The following is a consequence of \cite[Theorems 5.1,6.1]{Achar2003}.
\begin{lemma}\label{lem:AtoAbar}
Let $\OO \in \Orb(G)$ and let $C_1$ and $C_2$ be conjugacy classes in $A(\OO)$. Then the following are equivalent
\begin{itemize}
    \item[(i)] $d_S(\OO,C_1) = d_S(\OO,C_2)$.
    \item[(ii)] $C_1$ and $C_2$ have the same image in $\bar{A}(\OO)$. 
\end{itemize}
\end{lemma}

\begin{cor}
The map $d_S: \Conj(G^{\vee}) \to \Orb(G)$ factors through the projection $\Conj(G^{\vee}) \to \LA(G^{\vee})$. 
\end{cor}

We wish to define a notion of `saturation' for Lusztig-Achar data analogous to (\ref{eq:satconj}). So let $L \subset G$ be a Levi subgroup and let $\OO_L \in \Orb(L)$. 

\begin{prop}\label{prop:iotaAbar}
The homomorphism $\iota: A(\OO_L) \to A(\Sat^G_L\OO_L)$ (cf. (\ref{eq:iota})) descends to a homomorphism $\bar{\iota}: \bar{A}(\OO_L) \to \bar{A}(\Sat^G_L\OO_L)$. 
\end{prop}

\begin{proof}
  Let $\OO = \Sat^G_L\OO_L$. It follows from the definition of $\Tilde{b}$-value and Proposition \ref{prop:propsofds}(ii) that 
  $$\tilde{b}(\Sat^G_L(\OO_L, C_L)) = \Tilde{b}(\OO_L,C_L)$$ 
  for any conjugacy class $C_L$ in $A(\OO_L)$. Therefore $\iota(H'_{\OO_L}) \subset H'_\OO$ and the proposition holds. 
\end{proof}

Using Proposition \ref{prop:iotaAbar}, we can define a map
$$\mathrm{Sat}^G_L: \LA(L) \to \LA(G), \qquad \mathrm{Sat}^G_L(\OO_L, \bar{C}_L) = (\mathrm{Sat}^G_L\OO_L, \bar{\iota}(\bar{C}_L))$$
A Lusztig-Achar datum $(\OO,\bar{C})$ is \emph{distinguished} if it cannot be obtained by saturation from a proper Levi subgroup. Let
$$\LA_0(G) := \{(L,(\OO_L,\bar{C}_L)) \mid (\OO_L,\bar{C}_L) \text{ is distinguished}\}/G$$
where, as usual, $L$ runs over all Levi subgroups of $G$. 

\begin{lemma}\label{lem:K(O)}
Suppose $Z(G)$ is connected. Then for each $\OO \in \Orb(G)$, there is a subgroup $K(\OO) \subset A(\OO)$ such that
    \begin{itemize}
        \item[(i)] The quotient map $A(\OO) \to \bar{A}(\OO)$ restricts to an isomorphism $K(\OO) \xrightarrow{\sim} \bar{A}(\OO)$.
        \item[(ii)] The natural map from conjugacy classes in $K(\OO)$ to conjugacy classes in $A(\OO)$ is injective.
        \item[(iii)] If $L \subset G$ is a Levi subgroup and $\OO_L \in \Orb(L)$ such that $\OO=\mathrm{Sat}^G_L \OO_L$, then the homomorphism $\iota$ defined in (\ref{eq:iota}) maps $K(\OO_L)$ to $K(\OO)$.
    \end{itemize}
\end{lemma}

\begin{proof}
For $G$ a classical group, subgroups with these properties were exhibited in \cite[Section 5]{Sommers2001} (see also \cref{subsec:Sat_LA} for the case of distinguished Lusztig-Achar data). For $G$ an adjoint exceptional group, there are three possibilities:
\begin{enumerate}
    \item $A(\OO) \simeq \bar{A}(\OO)$.
    \item $A(\OO) \simeq S_2$, $\bar{A}(\OO) \simeq 1$.
    \item $G$ is of type $E_8$, $\OO$ is the distinguished orbit $E_8(b_6)$, $A(\OO) \simeq S_3$, and $\bar{A}(\OO) \simeq S_2$.
\end{enumerate}
In the first case, we take $K(\OO)=A(\OO)$. In the second case, we take $K(\OO)=\{1\}$. In the third case, we take $K(\OO)$ to be any order $2$ subgroup of $A(\OO)$. It is an easy exercise to check that conditions (i), (ii), and (iii) are satisfied for these choices of $K(\OO)$.
\end{proof}

\begin{prop}\label{prop:uniquedistinguishedAbar}
The map
$$\mathrm{Sat}: \LA_0(G) \to \LA(G), \qquad \mathrm{Sat}(L,(\OO_L,\bar{C}_L)) = \Sat^G_L(\OO_L,\bar{C}_L)$$
is a bijection.
\end{prop}

\begin{proof}
Since $\LA(G)$ is independent of isogeny, we can assume in the proof that $Z(G)$ is connected. Let $\bar{C}$ be a conjugacy class in $\bar{A}(\OO)$ and let $C'$ be the preimage of $\bar{C}$ under the isomorphism $K(\OO) \xrightarrow{\sim} \bar{A}(\OO)$ of Lemma \ref{lem:K(O)}(i). By Lemma \ref{lem:K(O)}(ii), there is a unique conjugacy class $C$ in $A(\OO)$ such that $C \cap K(\OO) = C'$. Taking $(\OO,\bar{C})$ to $(\OO,C)$ defines an injective map $\eta: \LA(G) \to \Conj(G)$, right-inverse to the projection $\Conj(G) \to \LA(G)$. 

Now suppose $L \subset G$ is a Levi subgroup. Then $Z(L)$ is connected (see e.g. \cite[Lemma 9]{NairPrasad}) and the following diagrams commute
\begin{equation} \label{diag:Conj_LA_Sat}
    \begin{tikzcd}
    \Conj(L) \ar[r,"\Sat^G_L"] \ar[d,twoheadrightarrow]& \Conj(G) \ar[d,twoheadrightarrow, ]\\
    \LA(L) \ar[r,"\Sat^G_L"]& \LA(G) 
    \end{tikzcd}
    \begin{tikzcd}
    \LA(L) \ar[r,"\Sat^G_L"] \ar[d,hookrightarrow, "\eta^L"]& \LA(G) \ar[d,hookrightarrow, "\eta^G"]\\
    \Conj(L) \ar[r,"\Sat^G_L"]& \Conj(G) 
    \end{tikzcd}
\end{equation}
(the first diagram commutes by definition, and the second by Lemma \ref{lem:K(O)}(iii)). 

Suppose $(L_1,(\OO_{L_1},\bar{C}_{L_1})), (L_2,(\OO_{L_2},\bar{C}_{L_2})) \in \LA_0(G)$ such that $\Sat(\OO_{L_1},\bar{C}_{L_1}))=\Sat(\OO_{L_2},\bar{C}_{L_2})) = (\OO,\bar{C})$. Let $(\OO_{L_1},C_{L_1}) = \eta(\OO_{L_1},\bar{C}_{L_1})$, $(\OO_{L_2},C_{L_2}) = \eta(\OO_{L_2},\bar{C}_{L_2})$, and $(\OO,C) = \eta(\OO,\bar{C})$. By the first commutative diagram the conjugacy data $(\OO_{L_1},C_{L_1})$ and $(\OO_{L_2},C_{L_2})$ are distinguished. By the second commutative diagram $\Sat(L_1,(\OO_{L_1},C_{L_1})) =\Sat(L_2,(\OO_{L_2},C_{L_2}) = (\OO,C)$. So by Proposition \ref{prop:BalaCarter}, $(L_1,(\OO_{L_1},C_{L_1}))$ and $(L_2,(\OO_{L_2},C_{L_2}))$ are conjugate. Hence, $(L_1,(\OO_{L_1},\bar{C}_{L_1}))$ and $(L_2,(\OO_{L_2},\bar{C}_{L_2}))$ are conjugate. This completes the proof. 
\end{proof}

We conclude by describing a useful parameterization of $\LA(G)$. A \emph{Sommers datum} for $G$ is a pair $(M,\OO_M)$ consisting of a pseudo-Levi subgroup $M \subset G$ and a distinguished nilpotent orbit $\OO_M \in \Orb(M)$. Note that $G$ acts by conjugation on the set of Sommers data. Let
$$\Som(G) := \{\text{Sommers data $(M,\OO_M)$ for $G$}\}/G$$
If $L \subset G$ is a Levi subgroup, there is a tautological map
$$\Sat^G_L: \Som(L) \to \Som(G), \qquad \Sat^G_L(M,\OO_M) = (M,\OO_M).$$
which is clearly compatible with $\Sat^G_L: \MS(L) \to \MS(G)$ under the forgetful maps $\MS(G) \to \Som(G)$ and $\MS(L) \to \Som(L)$. 

We define an equivalence relation $\sim$ on $\Som(G)$ as follows: $(M_1,\OO_{M_1}) \sim (M_2, \OO_{M_2})$ if and only if $\Sat^G_{M_1}\OO_{M_1} = \Sat^G_{M_2}\OO_{M_2}$ and $j^{G^{\vee}}_{M_1^{\vee}} d(\OO_{M_1^{\vee}}) =j^{G^{\vee}}_{M_2^{\vee}} d(\OO_{M_2^{\vee}})$.

\begin{lemma}\label{Lem:Sommers_data}
Let $G_{ad}$ denote the adjoint form of $G$. Then the following are true:
\begin{itemize}
    \item[(i)] The forgetful map
$$\MS^{large}(G_{ad}) \to \Som(G_{ad}) = \Som(G)$$
is a bijection. 
\item[(ii)] The fibers of the surjective map
$$\Som(G) \xrightarrow{\sim} \MS^{large}(G_{ad}) \xrightarrow{\sim} \Conj(G_{ad}) \twoheadrightarrow \LA(G_{ad}) = \LA(G)$$
(here, the first map is the inverse of the forgetful map in (i), the second map is $\pi^{G_{ad}}$, and the third map is the natural quotient map) are exactly the equivalence classes in $\Som(G)$. In particular, there is a natural bijection
$$\Som(G)/\sim \, \xrightarrow{\sim} \LA(G).$$
\item[(iii)] Suppose $(\OO,\bar{C}) \in \LA(G)$ corresponds under the bijection in (ii) to an equivalence class $\{(M_1,\OO_{M_1}),...,(M_k,\OO_{M_k})\} \subset \Som(G)$. Then $(\OO,\bar{C})$ is distinguished if and only if one (equivalently, all) of the pseudo-Levi subgroups $M_1,...,M_k$ is of maximal semisimple rank. In this case, the elements in the equivalence class in $\Som(G)$ correspond bijectively via the bijection $\Som(G) \xrightarrow{\sim} \Conj(G_{ad})$ in (ii) to the elements in the preimage of $(\OO,\bar{C})$ under the projection $\Conj(G_{ad}) \twoheadrightarrow \LA(G)$.
\item[(iv)] Suppose $(\OO,\bar{C}) \in \LA(G)$ is distinguished and $(M,tZ^{\circ},\OO_M) \in \MS(G)$ maps to $(\OO,\bar{C})$ under $\MS(G) \overset{\pi}{\to} \Conj(G) \to \LA(G)$. Then $(M,\OO_M)$ belongs to the equivalence class in $\Som(G)$ corresponding to $(\OO,\bar{C})$ under the bijection in (ii).
\end{itemize} 
\end{lemma}

\begin{proof}
(i) is \cite[Proposition 35]{SommersMcNinch}. (ii) is immediate from Lemma \ref{lem:AtoAbar}. (iii) and (iv) follow from part (i) and (iii) of Lemma \ref{lem:alphabeta} and Lemma \ref{lem:distinguishedMS}. Indeed, the first commtative diagram in \eqref{diag:Conj_LA_Sat} implies that any lift $(\OO,C)$ of a distinguished pair $(\OO, \bar{C})$ in $\Conj(G_{ad})$ is also distinguished. Then Lemma \ref{lem:alphabeta} (iii) implies that the preimage $\pi^{-1}(\OO,C)$ of $(\OO,C)$ under the map $\pi: \MS(G_{ad}) \to \Conj(G_{ad})$ is included in $\MS^{large}(G_{ad})$. Therefore by Lemma \ref{lem:alphabeta} (i), $\pi^{-1}(\OO,C)$ consists of only one element, which corresponds to a unique element $(M,\OO_M) \in \Som(G)$, where $M$ is of maximal semisimple rank by Lemma \ref{lem:distinguishedMS}. 
\end{proof}

\subsection{Special Lusztig-Achar data}\label{subsec:specialAchar}

\begin{definition}\label{defn:specialAchar}
  Let $(\OO,\bar{C}) \in \LA(G)$ and let $\OO^{\vee} = d_S^{G^{\vee}}(\OO,\bar{C})$. Following Achar (\cite{Achar2003}) we say that $(\OO,\bar{C})$ is \emph{special} if there is a conjugacy class $\bar{C}' \subset \bar{A}(\OO^{\vee})$ such that
  $$d_S^G(\OO^{\vee},\bar{C}') = \OO.$$
\end{definition}

Let
$$\LA^*(G) := \{\text{special Lusztig-Achar data $(\OO,\bar{C})$ for $G$}\}$$
and
$$\LA^*_0(G) := \{(L,(\OO_L,\bar{C}_L)) \mid (\OO_L,\bar{C}_L) \text{ is special and distinguished}\}/G$$
where, as usual, $L$ runs over all Levi subgroups of $G$. We will need several basic facts about special Achar data.

\begin{prop}\label{prop:specialtospecial}
The following are true:
\begin{itemize}
    \item[(i)] For any orbit $\OO \in \Orb(G)$, the Achar datum $(\OO,1)$ is special.
    \item[(ii)] The bijection of Proposition \ref{prop:uniquedistinguishedAbar} induces an injection
    $$\Sat^{-1}: \LA^*(G) \hookrightarrow \LA^*_0(G)$$
\end{itemize}
\end{prop}

\begin{proof}
(i) is \cite[Proposition 2.7]{Achar2003}. We proceed to proving (ii). In classical types all distinguished Achar data are special, see \cite[Section 5.2]{Achar2003}. So (ii) is immediate. In exceptional types, a list of (special) Achar data is given in \cite[Section 6]{Achar2003}. There are only two non-special distinguished Achar data in exceptional types, the (unique) nontrivial conjugacy class for the nilpotent orbit $A_4+A_1$ in $E_7$ and the (unique) nontrivial conjugacy class for the nilpotent orbit $E_6(a_1)+A_1$ in $E_8$. Saturation to $E_8$ takes the former to the (unique) nontrivial conjugacy class for the nilpotent orbit $A_4+A_1$ in $E_8$, which is non-special by \cite[Section 6]{Achar2003}. This completes the proof of (ii).
\end{proof}

\section{Symplectic singularities and unipotent ideals}\label{sec:symplecticsingularities}

In this section, we will review some preliminary facts about conical symplectic singularities and their Poisson deformations and filtered quantizations. We will also establish several new facts about conical symplectic singularities and nilpotent covers. The main new result is Theorem \ref{thm:maximaltomaximal}. It states that birational induction takes the maximal cover in an equivalence class to the maximal cover in an equivalence class. This result is essential for our parameterization of unipotent representations (Theorem \ref{thm:Gamma}).

\subsection{Poisson deformations}\label{sec:Poissondef}

In this subsection, we will recall the definitions of (graded, formal) Poisson deformations of (graded) Poisson varieties.

A \emph{Poisson scheme} is a scheme $X$ equipped with a Poisson bracket on its structure sheaf $\cO_X$. If $S$ is an affine scheme, a \emph{Poisson $S$-scheme} is a scheme $X$ equipped with a morphism $f:X \to S$ and a $f^{-1}\cO_S$-linear Poisson bracket on $\cO_X$. These definitions have obvious formal counterparts (a \emph{formal Poisson scheme} is a formal scheme with a Poisson bracket on its structure sheaf, and so on). 

\begin{definition}\label{def:Poissondeformations}
Let $X$ be a Poisson variety and let $S$ be a (possibly formal) affine scheme with a distinguished closed point $0$. Then a \emph{(formal) Poisson deformation of $X$ over $S$} is a flat (formal) Poisson $S$-scheme $\sX_S$ equipped with a Poisson isomorphism $\iota:\sX_S\times_S \{0\} \xrightarrow{\sim} X$. If $S$ is a formal scheme complete at $0\in S$, we say that $\sX_S$ is a formal Poisson deformation of $X$.

An isomorphism between Poisson deformations $(\sX_S, \iota)$ and $(\sX'_S, \iota')$ over $S$ is an isomorphism of Poisson schemes $\sX_S \to \sX'_S$ over $S$ such that the induced isomorphism $\sX_S\times_S \{0\} \xrightarrow{\sim} \sX'_S\times_S \{0\}$ intertwines $\iota$ and $\iota'$.
\end{definition}

Let $X$ be a Poisson variety. Let $\Art$ denote the category of local Artinian $\CC$-algebras with residue field $\CC$. For $R \in \Art$, define $\PD_X(R)$ to be the set of isomorphism classes of Poisson deformations over $\Spec(R)$. This defines a functor $\PD_X : \Art \to \Set$. Let $\CC\{\epsilon\} := \CC[\epsilon] / (\epsilon^2)$ denote the ring of dual numbers over $\CC$. Then $\PD_X(\CC\{\epsilon\})$ is the tangent space of $\PD_X$.

Next, we recall the notion of a graded Poisson deformation. Let $X$ be a graded normal Poisson variety. By this, we mean a normal variety equipped with two additional structures:
\begin{itemize}
    \item A rational $\CC^{\times}$-action. Let $\cO_X$ denote the structure sheaf of $X$ with respect to the conical topology on $X$ (i.e. the topology with open subsets equal to $\CC^{\times}$-invariant Zariski open subsets);
    \item A Poisson bracket $\{\cdot, \cdot\}: \cO_X \otimes \cO_X \to \cO_X$ of degree $-d$.
\end{itemize}

\begin{definition}
Let $X$ be a graded normal Poisson variety and let $S$ be an affine scheme with a rational $\CC^{\times}$-action contracting onto a distinguished closed point $0 \in S$. Then a \emph{graded Poisson deformation of $X$ over $S$} is a Poisson deformation $(\mathcal{X}_S,\iota)$ of $X$ over $S$ (cf. Definition \ref{def:Poissondeformations}) equipped with a rational $\CC^{\times}$-action on $\mathcal{X}_S$ such that
\begin{itemize}
    \item[(i)] The action rescales the Poisson bracket by $t \mapsto t^{-d}$.
    \item[(ii)] The action is compatible with the action on $S$.
    \item[(iii)] $\iota$ is $\CC^{\times}$-equivariant.
\end{itemize}
An isomorphism between graded Poisson deformations $(\mathcal{X}_S,\iota)$ and $(\mathcal{X}_S',\iota')$ over $S$ is an isomorphism of Poisson deformations which is $\CC^{\times}$-equivariant.
\end{definition}

\subsection{Filtered quantizations}\label{sec:filteredquant}

In this subsection, we will recall the definitions of filtered quantizations of graded Poisson algebras and graded Poisson varieties.

Let $A$ be a \emph{graded Poisson algebra} of degree $-d \in \ZZ_{<0}$. By this, we mean a finitely-generated commutative associative algebra equipped with two additional structures: an algebra grading 
$$A = \bigoplus_{i=-\infty}^{\infty} A_i;$$
and a Poisson bracket $\{\cdot, \cdot\}$ of degree $-d$
$$\{A_i, A_j\}\subset A_{i+j-d}, \qquad i,j \in \ZZ.$$

\begin{definition}\label{def:filteredquant_algebra}
A \emph{filtered quantization} of $A$ is a pair $(\cA,\theta)$ consisting of
\begin{itemize}
    \item[(i)] an associative algebra $\cA$ equipped with a complete and separated filtration by subspaces
    $$\cA = \bigcup_{i=-\infty}^{\infty} \cA_{\leq i}, \qquad ... \subseteq \cA_{\leq -1} \subseteq \cA_{\leq 0} \subseteq \cA_{\leq 1} \subseteq ...$$
    such that
    $$[\cA_{\leq i}, \cA_{\leq j}] \subseteq \cA_{\leq i+j-d} \qquad i,j \in \ZZ,$$
    and
    \item[(ii)] an isomorphism of graded Poisson algebras
    $$\theta: \gr(\cA) \xrightarrow{\sim} A,$$
    where the Poisson bracket on $\gr(\cA)$ is defined by
    $$\{a+\cA_{\leq i-1}, b+\cA_{\leq j-1}\}=[a,b]+\cA_{\leq i+j-d-1}, \qquad a \in \cA_{\leq i}, \ b \in \cA_{\leq j}.$$
\end{itemize}
An isomorphism of filtered quantizations $(\cA_1, \theta_1) \xrightarrow{\sim} (\cA_2, \theta_2)$ is an isomorphism of filtered algebras $\phi: \cA_1 \xrightarrow{\sim} \cA_2$ such that $\theta_1 = \theta_2 \circ \gr(\phi)$. Denote the set of isomorphism classes of quantizations of $A$ by $\mathrm{Quant}(A)$.
\end{definition}

Often, the isomorphism $\theta$ is clear from the context, and will be omitted from the notation. However, the reader should keep in mind that a filtered quantization $(\cA,\theta)$ is \emph{not} determined up to isomorphism by $\cA$ alone.

Now let $X$ be a graded normal Poisson variety.

\begin{definition}\label{def:filteredquant_variety}
A \emph{filtered quantization} of $X$ is a pair $(\calD,\theta)$ consisting of
\begin{itemize}
    \item[(i)] a sheaf $\mathcal{D}$ of associative algebras in the conical topology on $X$, equipped with a complete and separated filtration by subsheaves of vector spaces
    $$\calD = \bigcup_{i=-\infty}^{\infty} \calD_{\leq i}, \qquad ... \subseteq \calD_{\leq -1} \subseteq \calD_{\leq 0} \subseteq \calD_{\leq 1} \subseteq ...$$
    such that
    $$[\calD_{\leq i}, \calD_{\leq j}] \subseteq \calD_{\leq i+j-d} \qquad i,j \in \ZZ,$$
    and
    \item[(ii)] an isomorphism of sheaves of graded Poisson algebras
    $$\theta: \gr(\calD) \xrightarrow{\sim} \cO_X,$$
    where the Poisson bracket on $\gr(\calD)$ is defined by
    $$\{a+\calD_{\leq i-1}, b+\calD_{\leq j-1}\}=[a,b]+\calD_{\leq i+j-d-1}, \qquad a \in \calD_{\leq i}, \ b \in \calD_{\leq j}.$$
\end{itemize}
An isomorphism of filtered quantizations $(\calD_1, \theta_1) \xrightarrow{\sim} (\calD_2, \theta_2)$ is an isomorphism of sheaves of filtered algebras $\phi: \calD_1 \xrightarrow{\sim} \calD_2$ such that $\theta_1 = \theta_2 \circ \gr(\phi)$. Denote the set of isomorphism classes of quantizations of $X$ by $\mathrm{Quant}(X)$.
\end{definition}

We conclude this subsection by recalling the definition of a \emph{Hamiltonian quantization} of a graded Poisson algebra with Hamiltonian $G$-action. Let $A$ be a graded Poisson algebra of degree $-d$. Suppose $G$ is an algebraic group which acts rationally on $A$ by graded Poisson automorphisms. Write $\mathrm{Der}(A)$ for the Lie algebra of derivations of $A$. The $G$-action on $A$ induces by differentiation a Lie algebra homomorphism
$$\fg \to \mathrm{Der}(A), \qquad \xi \mapsto \xi_A,$$
We say that $A$ (or the $G$-action on $A$) is \emph{Hamiltonian} if there is a $G$-equivariant map $\varphi: \mathfrak{g} \to A_d$ (called a \emph{classical co-moment map}) such that
$$\{\varphi(\xi), a\} = \xi_A(a), \qquad \xi \in \fg, \quad a \in A.$$ 
A filtered quantization $(\cA,\theta)$ is $G$-\emph{equivariant} if $G$ acts rationally on $\cA$ by filtered algebra automorphisms and the isomorphism $\theta: \gr(\cA) \xrightarrow{\sim} A$ is $G$-equivariant. In this case, we get a Lie algebra homomorphism
$$\fg \to \mathrm{Der}(\cA), \qquad \xi \mapsto \xi_{\cA}.$$

\begin{definition}\label{def:hamiltonian}
Suppose $A$ is a graded Poisson algebra equipped with a Hamiltonian $G$-action. A \emph{Hamiltonian} quantization of $A$ is a triple $(\cA,\theta,\Phi)$ consisting of
 \begin{itemize}
     \item[(i)] a $G$-equivariant filtered quantization $(\cA,\theta)$ of $A$, and
     \item[(ii)] a $G$-equivariant map $\Phi: \fg \to \cA_{\leq d}$ (called a \emph{quantum co-moment map}) such that
     $$[\Phi(\xi),a] = \xi_{\cA}(a), \qquad \xi \in \fg, \quad a \in \cA.$$
 \end{itemize}
An isomorphism $(\mathcal{A}_1,\theta_1,\Phi_1) \xrightarrow{\sim} (\mathcal{A}_2,\theta_2,\Phi_2)$ of Hamiltonian quantizations of $A$ is a $G$-equivariant isomorphism of filtered algebras $\phi: \cA_1 \to \cA_2$ such that $\theta_1 = \theta_2 \circ \gr(\phi)$ and $\Phi_2 = \phi \circ \Phi_1$. Denote the set of isomorphism classes of Hamiltonian quantizations of $A$ by $\mathrm{Quant}^G(A)$.
\end{definition}

\subsection{Symplectic singularities and $\QQ$-factorial terminalizations} \label{subsec:symplectic_singularties}

Let $X$ be a normal Poisson variety.

\begin{definition}[\cite{Beauville2000}, Definition 1.1]
We say that $X$ has \emph{symplectic singularities} if
\begin{itemize}
    \item[(i)] The regular locus $X^{reg} \subset X$ is symplectic; denote the symplectic form by $\omega$.
    \item[(ii)] There is a resolution of singularities $\rho: Y \to X$ such that $\rho^* \omega$ extends to a regular (not necessarily symplectic) $2$-form on $Y$.
\end{itemize}
\end{definition}

Now let $X$ be a graded normal Poisson variety.

\begin{definition}
We say that $X$ is a \emph{conical symplectic singularity} if $X$ has symplectic singularities and the $\CC^{\times}$-action on $X$ is contracting onto a point.
\end{definition}

We note that every conical symplectic singularity is automatically affine.

\begin{example}\label{ex:nilpotentcover}
Let $G$ be a complex connected reductive algebraic group and let $\widetilde{\OO} \to \OO$ be a finite \'{e}tale $G$-equivariant cover of a nilpotent co-adjoint $G$-orbit $\OO \subset \fg^*$. There is a natural $\CC^{\times}$-action on $\widetilde{\OO}$ defined in the following manner (see \cite[Section 1]{BrylinskiKostant1994} for details and proofs). Consider the `doubled' dilation action of $\CC^{\times}$ on $\fg^*$, i.e.
$$z \cdot \xi = z^2\xi, \qquad z \in \CC^{\times}, \ \xi \in \fg^*.$$
This action preserves $\OO$, since $\OO$ is a nilpotent orbit, and lifts to a unique $\CC^{\times}$-action on $\widetilde{\OO}$, which commutes with the $G$-action on $\widetilde{\OO}$. It follows that any morphism of covers $\widetilde{\OO} \to \widehat{\OO}$ is automatically $\cm$-equivariant. The $\cm$-action on $\widetilde{\OO}$ induces a non-negative grading on the ring of regular functions $\CC[\widetilde{\OO}]$.

There is also a natural $G$-equivariant symplectic form on $\widetilde{\OO}$, obtained by pullback from the Kirillov-Kostant-Souriau form on $\OO$. This form induces a Poisson bracket on $\CC[\widetilde{\OO}]$, and this bracket is of degree $-2$ with respect to the grading defined above (see again \cite[Section 1]{BrylinskiKostant1994}). It is known that $X=\Spec(\CC[\widetilde{\OO}])$ is a conical symplectic singularity (see \cite[Lemma 2.5]{LosevHC}).

The natural map $\widetilde{\OO} \to X$ is an open embedding, and the $G$-action on $\widetilde{\OO}$ extends to a $G$-action on $X$. The $G$-equivariant covering map $\widetilde{\OO} \to \OO$ extends to a finite $G$-equivariant surjection $X \to \overline{\OO}$. The $G$-action on $\widetilde{\OO}$ (resp. $X$) is Hamiltonian; the moment map is the composition $\widetilde{\OO} \to \OO \subset \fg^*$ (resp. $X \to \overline{\OO} \subset \fg^*$). Note that any morphism of covers $\widetilde{\OO} \to \widehat{\OO}$ is symplectic and intertwines the moment maps, hence is automatically $G$-(hence $G \times \cm$-)equivariant, since $G$ is connected.
\end{example}

Recall that a normal variety $Y$ is $\QQ$\emph{-factorial} if every Weil divisor has a nonzero integer multiple which is Cartier. 

\begin{prop}[Proposition 2.1, \cite{LosevSRA}]
Let $X$ be a Poisson variety with symplectic singularities. Then there is a birational projective morphism $\rho: Y \to X$ such that 
\begin{itemize}
    \item[(i)] $Y$ is an irreducible Poisson variety.
    \item[(ii)] $Y$ is $\QQ$-factorial.
    \item[(iii)] The singular locus of $Y$ is of codimension $\geq 4$. 
\end{itemize}
\end{prop}

The morphism $\rho: Y \to X$ in the proposition above (or the variety $Y$, if the morphism is understood) is called a $\QQ$\emph{-factorial terminalization} of $X$. If $X$ is a conical symplectic singularity, then there is a $\CC^{\times}$-action on $Y$ such that $\rho$ is $\CC^{\times}$-equivariant, see \cite[A.7]{Namikawa3}. By \cite{Namikawa_bir}, every conical symplectic singularity has only finitely-many isomorphism classes of $\QQ$-factorial terminalizations.

\subsection{The Namikawa space and Weyl group}\label{sec:Namikawa}

Let $X$ be a conical symplectic singularity. Associated to $X$ are two important invariants: a finite-dimensional complex vector space $\fP = \fP^X$ called the \emph{Namikawa space} and a finite Coxeter group $W=W^X$ called the \emph{Namikawa Weyl group}. In this section, we will recall several equivalent definitions of these objects. 

\begin{definition}
Let $X$ be a conical symplectic singularity and let $\rho: Y \to X$ be a $\QQ$-factorial terminalization. The \emph{Namikawa space} associated to $X$ is the finite-dimensional complex vector space
$$\fP = \fP^X := H^2(Y^{reg},\CC)$$
\end{definition}

We will see in a moment that $\fP$ depends only on $X$ (and not on the choice of $\rho : Y \to X$). Since $X$ is a symplectic singularity, it contains finitely-many symplectic leaves (\cite[Theorem 2.3]{Kaledin2006}). Let $\fL_k$, $k = 1, 2, \ldots, t$, denote the symplectic leaves of codimension 2. For each such leaf $\fL_k \subset X$, the formal slice to $\fL_k \subset X$ is identified with the formal neighborhood at $0$ in a Kleinian singularity $\Sigma_k = \CC^2/\Gamma_k$. Under the McKay correspondence, $\Gamma_k$ corresponds to a complex simple Lie algebra $\fg_k$ of type A, D, or E. Fix a Cartan subalgebra $\fh_k \subset \fg_k$. Write $\Lambda_k \subset \fh_k^*$ for the weight lattice and $W_k$ for the Weyl group. If we choose a point $x \in \fL_k$, there is a natural identification $H^2(\widetilde{\mu}^{-1}(x),\ZZ) \simeq \Lambda_k$, and $\pi_1(\fL_k)$ acts on $\Lambda_k$ by diagram automorphisms. The \emph{partial Namikawa space} corresponding to $\fL_k$ is the subspace
$$\fP_{k} = \fP_k^X := (\fh_{k}^*)^{\pi_1(\fL_k)}.$$
Also define $\fP_{0} = \fP_0^X := H^2(X^{reg},\CC)$.

\begin{prop}[\cite{Losev4}, Lem 2.8]\label{prop:partialdecomp}
There is a linear isomorphism
\begin{equation}\label{eq:partialdecomp}\fP \xrightarrow{\sim} \bigoplus_{k=0}^t \fP_k.\end{equation}
\end{prop}

For each codimension 2 leaf $\fL_k \subset X$, consider the subgroup of monodromy invariants $W_k^{\pi_1(\fL_k)} \subset W_k$. Note that there is a natural action of $W_k^{\pi_1(\fL_k)}$ on $\fP_{k} = (\fh_{k}^*)^{\pi_1(\fL_k)}$. 

\begin{definition}
The \emph{Namikawa Weyl group} associated to $X$ is the finite group 
$$W = W^X := \prod_{k=1}^t W_k^{\pi_1(\fL_k)}.$$
\end{definition}

Note that $W$ acts on $\fP$ via the isomorphism (\ref{eq:partialdecomp}) (the action on $\fP_{0}$ is trivial).

\subsection{Filtered quantizations of conical symplectic singularities}

In this section, we will recall the classification of filtered quantizations of conical symplectic singularities. Let $X$ be a conical symplectic singularity and let $\rho: Y \to X$ be a graded $\QQ$-factorial terminalization of $X$. For any graded smooth symplectic variety $V$, there is a (non-commutative) \emph{period map}
$$\mathrm{Per}: \mathrm{Quant}(V) \to H^2(V,\CC),$$
see \cite[Sec 4]{BK}, \cite[Sec 2.3]{Losev_isofquant}.

\begin{prop}[\cite{Losev4}, Prop 3.1(1)]\label{prop:P=Quant(Y)}
The maps 
$$\mathrm{Quant}(Y) \overset{|_{Y^{\mathrm{reg}}}}{\to} \mathrm{Quant}(Y^{\mathrm{reg}}) \overset{\mathrm{Per}}{\to} H^2(Y^{\mathrm{reg}},\CC) = \fP^X$$
are bijections.
\end{prop}

For $\lambda \in \fP^X$, let  $\mathcal{D}_{\lambda}$ denote the corresponding filtered quantization of $Y$ and let $\cA_{\lambda} := \Gamma(Y,\mathcal{D}_{\lambda})$. 

\begin{theorem}[Prop 3.3, Thm 3.4, \cite{Losev4}]\label{thm:quantssofsymplectic}
The following are true:

\begin{itemize}
\item[(i)] For every $\lambda \in \fP^X$, the algebra $\cA_\lambda$ is a filtered quantization of $\CC[X]$.
    \item[(ii)] Every filtered quantization of $\CC[X]$ is isomorphic to $\cA_{\lambda}$ for some $\lambda \in \fP^X$.
    \item[(iii)] For every $\lambda, \lambda' \in \fP^X$, we have $\cA_{\lambda} \simeq \cA_{\lambda'}$ if and only if $\lambda' \in W^X \cdot \lambda$.
\end{itemize}
Hence, the map $\lambda \mapsto \cA_{\lambda}$ induces a bijection
$$\fP^X/W^X \simeq \mathrm{Quant}(\CC[X]), \qquad W^X \cdot \lambda \mapsto \cA_{\lambda}.$$
\end{theorem}

We conclude by recalling an equivariant version of Theorem \ref{thm:quantssofsymplectic} from \cite{LMBM}. Let $G$ be a connected reductive algebraic group and suppose $A:=\CC[X]$ admits a Hamiltonian $G$-action, see Section \ref{sec:filteredquant}. Define the \emph{extended Namikawa space}
$$\overline{\fP}^X := \fP^X \oplus \fz(\mathfrak{g})^*$$
This space should be viewed as the equivariant counterpart of $\fP^X$. Let $W^X$ act on $\overline{\fP}^X$ via the decomposition above (the $W^X$-action on the second factor is defined to be trivial). 

\begin{prop}[Lem 4.11.2, \cite{LMBM}]\label{prop:Hamiltonian}
Let $G$ be a connected reductive algebraic group and suppose $A:=\CC[X]$ admits a Hamiltonian $G$-action. Then the following are true:
\begin{itemize}
    \item[(i)] There is a unique classical co-moment map $\varphi: \fg \to A_d$.
    \item[(ii)] Every filtered quantization $\cA \in \mathrm{Quant}(A)$ has a unique $G$-equivariant structure.
    \item[(iii)] For every $\cA \in \mathrm{Quant}(A)$ and $\chi \in \fz(\fg)^*$, there is a unique quantum co-moment map $\Phi_{\chi}: \fg \to \cA_{\leq d}$ such that $\Phi|_{\fz(\fg)} = \chi$.
\end{itemize}
In particular, there is a canonical bijection
$$\overline{\fP}^X/W^X \xrightarrow{\sim} \mathrm{Quant}^G(A) \qquad W^X \cdot (\lambda, \chi) \mapsto (\cA_{\lambda}, \Phi_{\chi}).$$
\end{prop}

\begin{definition}[Def 5.0.1, \cite{LMBM}]\label{def:canonical}
The \emph{canonical quantization} of $\CC[X]$ is the Hamiltonian quantization corresponding to the parameter $0 \in \overline{\fP}^X$. 
\end{definition}

\subsection{$\QQ$-factorial terminalizations of nilpotent covers}\label{sec:Qfactorialnilp}

Let $\widetilde{\OO} \in \Cov(G)$ and let $X = \Spec(\CC[\widetilde{\OO}])$. Then $X$ is a conical symplectic singularity (cf. Example \ref{ex:nilpotentcover}). In this section, we will give explicit (Lie-theoretic) descriptions of the $\QQ$-factorial terminalizations of $X$, the Namikawa space $\fP(\widetilde{\OO}):=\fP^X$, and the partial Namikawa spaces $\fP_k(\widetilde{\OO}) := \fP_k^X$.

Choose $(L,\widetilde{\OO}_L) \in \Cov_0(G)$ such that $\Bind(L,\widetilde{\OO}_L) = \widetilde{\OO}$. By Proposition \ref{prop:bindinjective}, $(L,\widetilde{\OO}_L)$ is unique (up to conjugation by $G$). Let $X_L = \Spec(\CC[\widetilde{\OO}_L])$. Choose a parabolic subgroup $P\subset G$ with Levi decomposition $P=LN$. Form the variety $Y=G \times^P(X_L \times \fp^{\perp})$ and consider the proper $G$-equivariant map $\widetilde{\mu}:Y \to \overline{\OO}$ defined in Section \ref{subsec:induction}. Then $\widetilde{\mu}$ admits a Stein factorization $Y \xrightarrow{\rho} X \to \overline{\OO}$. The first map is projective birational and the second is finite. 

\begin{theorem}[{\cite[Cor 4.3]{Mitya2020}, \cite[Lemma 7.2.4]{LMBM}}] \label{thm:terminalization_cover}
The following are true.
  \begin{enumerate} [label=(\roman*)]
      \item 
      The morphism $\rho: Y \to X$ is a $\QQ$-factorial terminalization of $X$.
      \item 
      Any $\QQ$-factorial terminalization $Y$ of $X$ is of this form.
  \end{enumerate}   
\end{theorem}

If $\fa$ is a Lie algebra, we write $\fX(\fa)$ for the vector space of Lie algebra homomorphisms $\fa \to \CC$. Let $\pi: Y \to G/P$ be the natural projection map which makes $Y$ into a fiber bundle over $G/P$. Let $\underline{\pi}$ denote the restriction of $\pi$ to $Y^{reg}$. Note that $H^2(G/P,\CC)$ is identified with $\fX(\fl \cap [\fg,\fg])$. Consider the composite map 
\begin{equation}
    \eta: \fX(\fl \cap [\fg,\fg]) \simeq H^2(G/P,\CC) \xrightarrow{\underline{\pi}^*} H^2(Y^{reg}, \CC) = \fP(\widetilde{\OO}),
\end{equation}
where $\underline{\pi}^*: H^2(G/P,\CC) \to H^2(Y^{reg}, \CC) = \fP(\widetilde{\OO})$ is the pullback map on cohomology induced by $\underline{\pi}$.

\begin{prop}[Proposition 7.2.2, \cite{LMBM}]\label{prop:Namikawacenter}
The map
$\eta: \fX(\fl \cap [\fg,\fg]) \xrightarrow{\sim} \fP(\widetilde{\OO})$
is an isomorphism.
\end{prop}

We can also describe the spaces $H^2(\widetilde{\OO},\CC)$ and $\fP_k(\widetilde{\OO}) := \fP_k^X$ in terms of Lie-theoretic data. We begin by describing $H^2(\widetilde{\OO},\CC)$. Assume for simplicity that $G$ is simply connected and semisimple. Let $R$ denote the reductive part of the stabilizer of $e \in \OO$ and let $\mathfrak{r}$ be its Lie algebra. We note that $\fr$ does not depend on the choice of $e$ in $\OO$, and the adjoint action of $R$ on $\fz(\fr)$ factors through $R/R^{\circ} \simeq \pi_1(\widetilde{\OO})$. 

\begin{lemma}[Lemma 7.2.7, \cite{LMBM}]\label{lem:computeH2}
There is a natural identification
$$H^2(\widetilde{\OO},\CC) \simeq \fz(\mathfrak{r})^{\pi_1(\widetilde{\OO})}$$    
\end{lemma}

\begin{rmk}\label{rmk:H2}
If $\widetilde{\OO} = \widehat{\OO}$ is the \emph{universal} cover of $\OO$, then $H^2(\widehat{\OO},\CC) \simeq \fz(\mathfrak{r})$ by Lemma \ref{lem:computeH2}. In particular,  $H^2(\widehat{\OO},\CC)=0$ if and only if $\mathfrak{r}$ is semisimple. On the other hand, if $\widetilde{\OO} = \OO$, then $H^2(\widetilde{\OO},\CC) \simeq \fz(\mathfrak{r})^{\pi_1(\OO)}$ was computed in every case by Biswas and Chatterjee in \cite{BISWAS}. 
\end{rmk}

We proceed to describing the partial Namikawa spaces $\fP_k(\widetilde{\OO})$ for $k \geq 1$. Assume that $H^2(\widetilde{\OO},\CC)=0$. Suppose $Q \subset G$ is a parabolic subgroup with Levi factor $M$ and $\widetilde{\OO}_M \in \Cov(M)$ satisfies $\widetilde{\OO}=\mathrm{Bind}^G_M \widetilde{\OO}_M$. The triple $(Q,M,\widetilde{\OO}_M)$ gives rise to a projective birational morphism 
\begin{equation}\label{eq:partialresolution}
\rho;: G \times^Q (\Spec(\CC[\widetilde{\OO}_{M}]) \times \fq^{\perp}) \to \Spec(\CC[\widetilde{\OO}])\end{equation}
\begin{prop}[Prop 7.5.6, \cite{LMBM}]\label{prop:adaptedresolutiondatum}
For each codimension 2 leaf $\fL_k \subset \widetilde{X}$, there is a unique pair $(M_k,\widetilde{\OO}_{M_k})$ consisting of a Levi subgroup $M_k \subset G$ and a nilpotent cover $\widetilde{\OO}_{M_k} \in \Cov(M_k)$ such that
\begin{itemize}
    \item[(i)] $L \subset M_k$.
    \item[(ii)] $\widetilde{\OO} = \mathrm{Bind}^G_{M_k} \widetilde{\OO}_{M_k}$.
    \item[(iii)] For any parabolic $Q \subset G$ with Levi factor $M_k$, the morphism (\ref{eq:partialresolution}) resolves $\Sigma_k$ and preserves $\Sigma_j$ for $j\neq k$.
\end{itemize}
\end{prop}

The pair $(M_k,\widetilde{\OO}_{M_k})$ appearing in Proposition \ref{prop:adaptedresolutiondatum} is called the $\fL_k$-\emph{adapted resolution datum.} 

\begin{prop}[Proposition 3.6.4, \cite{MBMat}]
Let $\fL_k \subset \Spec(\CC[\widetilde{\OO}])$ be a codimension 2 leaf and let $(M_k, \widetilde{\OO}_{M_k})$ be the $\fL_k$-adapted resolution datum. Then the isomorphism $\fP \simeq \fX(\fl \cap [\fg,\fg])$ of Proposition \ref{prop:Namikawacenter} restricts to an isomorphism $\fP_k \simeq \fX(\fm_k \cap [\fg,\fg])$.
\end{prop}

Comparing Propositions \ref{prop:Namikawacenter} and \ref{prop:partialdecomp}, we arrive at the following (purely geometric) criterion for birational rigidity.

\begin{prop}[Proposition 3.7.1, \cite{MBMat}]\label{prop:criterionbirigidcover}
Let $\widetilde{\OO} \in \Cov(G)$. Then $\widetilde{\OO}$ is birationally rigid if and only if the following conditions hold:
\begin{itemize}
    \item[(i)]  $H^2(\widetilde{\OO},\CC)=0$.
    \item[(ii)] $\Spec(\CC[\widetilde{\OO}])$ has no codimension 2 leaves.
\end{itemize}
\end{prop}

We conclude by sketching a simple criterion for birational induction. For any $\OO \in \Orb(G)$, let $\mathcal{P}_{rig}(\OO)$ denote the set of $G$-conjugacy classes of pairs $(L,\OO_L)$ consisting of a Levi subgroup $L \subset G$ and a rigid nilpotent orbit $\OO_L \in \Orb(L)$ such that $\OO=\Ind^G_L \OO_L$. This set can be computed in classical types using the results \cite[Section 7]{CM} and in exceptional types using the tables in \cite[Section 4]{deGraafElashvili}. Consider the integer
\begin{equation}\label{eq:m(O)}
m(\OO) = \max \{\dim \fz(\fl) \mid (L,\OO_L) \in \mathcal{P}_{rig}(\OO)\}.\end{equation}

\begin{lemma}\label{lem:bindcriterion}
Suppose $m(\OO) = \dim \fP$ and there is a unique pair $(L,\OO_L) \in \mathcal{P}_{rig}(\OO)$ such that $\dim(\fz(\fl)) = m(\OO)$. Then  $\OO=\Bind^G_L \OO_L$. 
\end{lemma}

\begin{proof}
Suppose $\OO$ is birationally induced from a birationally rigid orbit $\OO_M \in \Orb(M)$ for a Levi subgroup $M \subset G$. By Proposition \ref{prop:Namikawacenter}, $\dim(\fz(\fm)) = m(\OO)$. Suppose $\OO_M$ is not rigid. Then there is a proper Levi subgroup $K \subset M$ and a rigid orbit $\OO_{K} \in \Orb(K)$ such that $\OO=\Ind^G_K \OO_K$. So $(K,\OO_K) \in \mathcal{P}_{rig}(\OO)$ and $\dim (\fz(\mathfrak{k})) > \dim (\fz(\fm)) = m(\OO)$, a contradiction. Hence, $(M,\OO_M) \in \mathcal{P}_{rig}(\OO)$ and so $G$-conjugate to $(L,\OO_L)$. 
\end{proof}

\subsection{Universal Poisson deformations of conical symplectic singularities}\label{subsec:universal_deform_symplectic}

Let $X$ be an affine symplectic singularity. Let $X^\diamond \subset X$ denote the complement in $X$ to the union of all symplectic leaves of codimension $\geq 4$. Then 
$$X^\diamond = X^{reg} \sqcup \bigsqcup_{k=1}^t \mathfrak{L}_k,$$
where $\mathcal{L}_k$ are the codimension 2 leaves in $X$ as in Section \ref{sec:Namikawa}. As observed by Namikawa \cite[p. 52]{Namikawa2011}, the variety $X^\diamond$ admits a unique $\QQ$-factorial terminalization $\pi: \widetilde{X}^\diamond \to X^\diamond$ (up to isomorphism), and $\widetilde{X}^\diamond$ is smooth. Now we take a $\QQ$-factorial terminalization $\rho: Y \to X$ of $X$. Set $Y^\diamond = \rho^{-1}(X^\diamond)$. Then $Y^\diamond$ lies in the regular locus $Y^{reg}$ of $Y$ and the restriction 
$\rho|_{Y^\diamond}: Y^\diamond \to X^\diamond$ is also a $\QQ$-factorial terminalization of $X^\diamond$. Hence there is a unique isomorphism between $Y^\diamond$ and $\widetilde{X}^\diamond$ as varieties over $X^\diamond$. By (2.4) in the proof of \cite[Proposition 2.14]{Losev4}, we have $\operatorname{codim}_{Y} Y \backslash Y^\diamond \geqslant 2$ and so $\operatorname{codim}_{Y^{reg}} Y^{reg} \backslash Y^\diamond \geqslant 2$. This implies that the restriction map $\fP = H^2(Y^{reg}, \CC) \to H^2(Y^\diamond, \CC) \simeq H^2(\widetilde{X}^\diamond, \CC)$ is an isomorphism by a standard argument involving a long exact sequence of cohomology groups (see e.g. the proof of \cite[Theorem 5.1(i)]{Namikawa}). 

By the proof of \cite[Theorems 5.1, 5.2]{Namikawa}, the inclusions $X^\diamond \subset X$ and $Y^\diamond \subset Y$ induce natural transformations $\PD_{X} \to \PD_{X^\diamond}$ and $\PD_{Y} \to \PD_{Y^\diamond}$. The birational maps $\pi$, $\rho$ and $\rho|_{Y^\diamond}$ induce natural transformations 
  $$\pi_*: \PD_{\widetilde{X}^\diamond} \to \PD_{X^\diamond}, \quad \rho_*: \PD_Y \to \PD_X \quad \text{and} \quad (\rho|_{X^\diamond})_*: \PD_{Y^\diamond} \to \PD_{X^\diamond}.$$
These natural transformations fit into the following commutative diagram:

\begin{equation}\label{eq:PD}
  \begin{tikzcd}
    \PD_Y \ar[r] \ar[d, twoheadrightarrow, "\rho_*"] & \PD_{Y^\diamond} \ar[r, "\sim"] \ar[d,twoheadrightarrow, "(\rho|_{X^\diamond})_*"] & \PD_{\widetilde{X}^\diamond} \ar[d, twoheadrightarrow, "\pi_*"] \\
    \PD_X \ar[r]  & \PD_{X^\diamond} \ar[r, equal]  & \PD_{X^\diamond} 
  \end{tikzcd}
\end{equation}

\begin{theorem}[Theorems 5.1, 5.2, \cite{Namikawa}]\label{thm:formal_deformation}
    The following are true:
    \begin{enumerate}[label=(\roman*)]
        \item 
        The horizontal transformations $\PD_{X} \to \PD_{X^\diamond}$ and $\PD_{Y} \to \PD_{Y^\diamond}$ in \eqref{eq:PD} are isomorphisms.
        \item
        We have $\PD_Y(\CC\{\epsilon\}) \simeq \PD_{Y^\diamond}(\CC\{\epsilon\}) \simeq \PD_{\widetilde{X}^\diamond} (\CC\{\epsilon\}) \simeq \fP$.
        \item
        All deformation functors in \eqref{eq:PD} are unobstructed and pro-representable. $\PD_Y \simeq \PD_{Y^\diamond}  \simeq \PD_{\widetilde{X}^\diamond}$ are pro-represented by $\CC[\fP]^\wedge = \CC[\fP^\wedge]$, the completion of $\CC[\fP]$ at the maximal ideal corresponding to $0 \in \fP$.
    \end{enumerate}
\end{theorem}

By Theorem \ref{thm:formal_deformation}, each Poisson variety appearing in \eqref{eq:PD} admits a universal formal Poisson deformation. Now assume that $X$ is a conical symplectic singularity. Then by \cite[Lemma 20]{Namikawa3}, the $\cm$-action on $X$ induces $\cm$-actions on the universal formal Poisson deformations of the Poisson varieties in \eqref{eq:PD}. Using these $\CC^{\times}$-actions, Namikawa shows that the universal formal Poisson deformations of $X$ and $Y$ can be algebraized to graded Poisson deformations (\cite[Lemma 22]{Namikawa3}). And Losev shows that these graded Poisson deformations are \emph{universal} in the following sense (\cite[Section 2.2-2.4]{Losev4}).

\begin{definition}\label{defn:universal_graded_deform}
    Let $X$ be a graded Poisson variety. A \emph{universal graded Poisson deformation of $X$} is a graded Poisson deformation $\sX_{univ}$ of $X$ over some conical affine scheme $S$ satisfying the following universal property: let $\sX_B$ be any graded Poisson deformation of $X$ over a conical affine scheme $B$. Then there is a unique $\cm$-equivariant morphism $B \to S$ and a (not necessarily unique) isomorphism of graded Poisson deformations $\sX_B \xrightarrow{\sim} \sX_{univ} \times_S B$ over $B$.
\end{definition}

The following results are proved in \cite{Namikawa, Namikawa2, Namikawa3} and \cite{Losev4}. We use the formulation in \cite[Section 4.7]{LMBM}.

\begin{theorem} \label{thm:universal_deform}
    Let $X$ be a conical symplectic singularity with Poisson bracket of degree $d>0$. Then there exist universal graded Poisson deformations $\sX_{univ} \to \fP/W$ and $\sY_{univ} \to \fP$ of $X$ and $Y$, respectively. There exists a surjective projective Poisson morphism $\tilde{\rho}: \sY_{univ} \to \sX_{univ}$ such that:
    \begin{enumerate}[label=(\roman*)]		
	\item 
        The following diagram commutes and all maps are $\cm$-equivariant:
		\begin{center}
			\begin{tikzcd}
				\sY_{univ} \ar[r, twoheadrightarrow] \ar[d, twoheadrightarrow, "\tilde{\rho}"] & \fP \ar[d, twoheadrightarrow] \\
				\sX_{univ} \ar[r, twoheadrightarrow] & \fP/W
			\end{tikzcd}
		\end{center}
        Here, $\fP \to \fP/W$ is the quotient map and $\cm$ acts linearly on $\fP$ by $t.v \mapsto t^d v$, i.e., all vectors in $\fP$ are eigenvectors of weight $-d$. The induced sequence of morphisms $\sY_{univ} \to \sX_{univ} \times_{\fP/W} \fP \to \sX_{univ}$ is the Stein factorization of $\tilde{\rho}$.
        \item
        The algebra $\CC[\sY_{univ}]$ carries an action of $W$ by graded Poisson algebra automorphisms, which makes the map $\CC[\fP] \to \CC[\sY_{univ}]$ $W$-equivariant. The induced action on $\CC[Y]$ is trivial.
        \item 
        The map $\tilde{\rho}$ induces an isomorphism $\CC[\sX_{univ}] \xrightarrow{\sim} \CC[\sY_{univ}]^W$.
\end{enumerate}
\end{theorem}

Set $\sX_\fP := \sX_{univ} \times_{\fP/W} \fP$. Then $\sX_\fP$ is a Poisson scheme over $\fP$ and (i) of Theorem \ref{thm:universal_deform} says that $\tilde{\rho}$ induces a morphism $\sY_{univ} \to \sX_\fP$ of Poisson schemes over $\fP$.

For any $\lambda \in \fP$, let $X_\lambda$ (resp. $Y_\lambda$) denote the fiber of the universal deformation $\sX_{univ}$ (resp. $\sY_{univ}$) over $W\lambda$ (resp. $\lambda$). By Theorem \ref{thm:universal_deform}, the morphism $\tilde{\rho}: \sY_{univ} \to \sX_{univ}$ restricts to a projective birational surjective morphism $Y_\lambda \to X_\lambda$ with connected fibers, which induces an isomorphism $\CC[X_\lambda] \xrightarrow{\sim} \CC[Y_\lambda]$ of Poisson algebras. Let $\fP^{reg} \subset \fP$ denote the subset consisting of all $\lambda \in \fP$ for which the morphism $Y_\lambda \to X_\lambda$ is an isomorphism (or equivalently, $Y_\lambda$ is affine) and let $\fP^{sing} := \fP \backslash \fP^{reg}$.

\begin{theorem}[{\cite[Main Theorem, (i)]{Namikawa_bir}}] \label{thm:fP^reg}
    The subset $\fP^{sing} \subset \fP$ is a $W$-stable union of finitely many rational hyperplanes in $\fP$, including the walls corresponding to the $W$-action. In particular, $\fP^{reg}$ is a Zariski-dense open subset of $\fP$.
\end{theorem}

\subsection{Universal Poisson deformations of nilpotent covers}\label{subsec:universal_deform_nilpotent}

Continue with the notation of Section \ref{subsec:universal_deform_symplectic}. In this subsection, we will give an explicit description of $\sY_{univ}$ in the case of nilpotent covers. Let $\widetilde{\OO} \in \Cov(G)$. Choose $(L,\widetilde{\OO}_L) \in \Cov_0(G)$ such that $\Bind(L,\widetilde{\OO}_L) = \widetilde{\OO}$. Let $X = \Spec(\CC[\widetilde{\OO}])$ and $X_L = \Spec(\CC[\widetilde{\OO}_L])$. Recall that by Theorem \ref{thm:terminalization_cover}, any $\QQ$-factorial terminalization $Y$ of $X$ is of the form $Y=G \times^P(X_L \times \fp^{\perp})$, where $P\subset G$ is a parabolic subgroup with Levi decomposition $P=LN$. 

Before we continue let us introduce the general notion of an {\it induced variety}. Let $Z$ be a Poisson variety equipped with a Hamiltonian $L$-action and a moment map $\mu_L: Z \to \fl^*$.  Let $Z \times_{\fl^*} \fn^{\perp}$ be the pullback of $\mu_L: Z \to \fl^*$ along the projection $\fn^{\perp}=(\fg/\fn)^* \twoheadrightarrow (\fp/\fn)^* = \fl^*$. Let $P$ act on $Z$ via the quotient morphism $P \twoheadrightarrow P/N = L$ and on $\fn^{\perp}$ in the natural way. These actions induce a $P$-action on $Z \times_{\fl^*} \fn^{\perp}$. We define
$$ \Ind_P^G Z := G \times^P \left( Z \times_{\fl^*} \fn^{\perp} \right).$$
Note that $\Ind^G_P Z \simeq G \times^P (Z \times \fp^{\perp})$. In particular, if $Z=X_L$, then $\Ind^G_P Z$ coincides with the variety $Y =G \times^P (X_L \times \fp^{\perp})$. Finally, we note that $\Ind^G_P Z$ is naturally a graded Poisson variety with a Hamiltonian $G$-action, see \cite[Section 7.3]{LMBM} for details. 

Set
  $$ \fz := \fX(\fl \cap [\fg, \fg]), \quad \fz^\circ = \{ \chi \in \fz \,|\, G_\chi = L \}, \quad \text{and} \quad X_L^e := X_L \times \fz.$$
(To make sense of the stabilizer $G_{\chi}$ in the definition of $\fz^{\circ}$, we use the isomorphism $\fg* \simeq \fg$ induced by the Killing form to identify $\fz$ with $\fz(\fl \cap [\fg,\fg])$). The Poisson structure on $X_L$ induces a Poisson structure on $X_L^e$. The $L$-action on $X_L$ and the natural (trivial) $L$-action on $\fz$ induces a Hamiltonian $L$-action on $X_L^e$, such that the moment map $\mu_L^e: X_L^e \to \fl^*$ is the composition of the projection map $X_L^e \twoheadrightarrow X_L$ with the moment map $X_L \to \fl^*$. Hence we can define
$$\sY_\fz := \Ind_P^G X_L^e = G \times^P \left( \fz \times X_L \times \fp^\perp \right).$$
Then we have a natural $L$-equivariant projection map $X_L^e \twoheadrightarrow \fz$, which induces a projection map $\sY_\fz \twoheadrightarrow G \times^P \fz \twoheadrightarrow \fz$ so that $\sY_\fz$ is a graded Poisson scheme over $\fz$. Note that we have a Poisson isomorphism
$$ \sY_\fz \times_{\fz} \{0\} \simeq Y = G \times^P \left(X_L \times \fp^\perp \right)$$
and hence $\sY_\fz$ is a graded Poisson deformation of the $\QQ$-factorial terminalization $Y$ of $X$. 
Set $\sX_\fz := \CC[\sY_\fz]$. This is a graded affine Poisson scheme over $\fz$.
\begin{prop}[{\cite[Proposition 7.2.2]{LMBM}}] \label{prop:universal_deform_cover}
The following are true:
\begin{enumerate} [label=(\roman*)]
    \item 
       The universal graded deformation $\sY_{univ}$ of $Y$ in Theorem \ref{thm:universal_deform} can be identified with $\sY_\fz$ over $\fz$ as graded Poisson deformations of $Y$ under the isomorphism $\eta: \fz \xrightarrow{\sim} \fP$ of Proposition \ref{prop:Namikawacenter}.
    \item
       The isomorphism $\sY_\fz \simeq \sY_{univ}$ in (i) induces a graded $G$-equivariant isomorphism $\sX_\fz \simeq \sX_{univ} \times_{\fP/W} \fP$ of Poisson schemes under the isomorphism $\eta: \fz \xrightarrow{\sim} \fP$ (cf. Theorem \ref{thm:universal_deform}, (i)). 
    \item 
       Under the isomorphism $\eta: \fz \xrightarrow{\sim} \fP$, the open subset $\fP^{reg} \subset \fP$ in Theorem \ref{thm:fP^reg} corresponds to the open subset $\fz^\circ \subset \fz$.      
\end{enumerate}
\end{prop}
\begin{proof}
    (i) is \cite[Proposition 7.2.2, (iii)]{LMBM}; (ii) follows from Theorem \ref{thm:universal_deform}; (iii) is \cite[Lemma 7.2.4, (i)]{LMBM}.
\end{proof}

We can describe the base change of $\sX_\fz$ and $\sY_\fz$ along the open inclusion $\fz^\circ \hookrightarrow \fz$ explicitly as follows. Define open subschemes 
$$(X_L^e)^\circ := X_L \times \fz^\circ \subset X_L^e, \quad \sX_\fz^\circ := \sX_\fz \times_{\fz} \fz^\circ \subset \sX_\fz \quad \text{and} \quad
\sY_\fz^\circ := \sY_\fz \times_{\fz} \fz^\circ \subset \sY_\fz$$ 
Note that $(X_L^e)^\circ$ is stable under the $L$-action on $X_L^e$. 
Then 
$$\sY_\fz^\circ = G \times^P \left( (X_L^e)^\circ \times_{\fl^*} \fn^{\perp} \right) = G \times^P \left( \fz^\circ \times X_L \times \fp^{\perp} \right). $$ 
There is a canonical $G \times \CC^\times$-equivariant isomorphism
\begin{equation}\label{eq:isom_regular1}
	G \times^L (X_L^e)^\circ = G \times^P \left( N \times (X_L^e)^\circ \right) \xrightarrow{\sim} \sY_\fz^\circ = G \times^P \left( (X_L^e)^\circ \times_{\fl^*} \fn^{\perp} \right) 
\end{equation}
induced by the isomorphism
\begin{align}
	N \times (X_L^e)^\circ & \xrightarrow{\hspace*{0.25em} \sim \hspace*{0.25em}}  (X_L^e)^\circ \times_{\fl^*} \fn^{\perp}   \\
	(n, x^e) \quad & \longmapsto \, n.(x^e, \mu_L^e(x^e)) \nonumber
\end{align}
for any $n \in N$ and $x^e \in (X_L^e)^\circ$. Here we regard $\mu_L^e(x^e)$ as a vector in $\fn^{\perp}$ using the $L$-equivariant section $\fl^* \hookrightarrow \fn^{\perp}$ of the natural projection $\fn^{\perp} \twoheadrightarrow \fl^*$ induced by the triangular decomposition $\fg = \fn \oplus \fl \oplus \fn^-$. This means that, even though $\sY_\fz$ depends on the choice of the parabolic $P$, over $\fz^\circ$ there are canonical isomophisms between various $\sY_\fz^\circ$ for different $P$ which preserve the Hamiltonian actions. Moreover, by Proposition \ref{prop:universal_deform_cover}, the map $\sY_\fz^\circ \to \sX_\fz^\circ$ is an isomorphism.

\subsection{Graded Poisson automorphisms of conical symplectic singularities} \label{subsec:Poisson_auto}

Let $X$ be an affine symplectic singularity and define the open subvariety $X^\diamond \subset X$ and its unique $\QQ$-factorial terminalization $\widetilde{X}^\diamond$ as in Section \ref{subsec:universal_deform_symplectic}. Let $\theta$ be a Poisson automorphism of $X$. Note that $\theta$ preserves $X^\diamond$, hence $\theta$ lifts uniquely to a Poisson automorphism $\tilde{\theta}$ of $\widetilde{X}^\diamond$, so that the following diagram commutes
\begin{center}
  \begin{tikzcd}
    \widetilde{X}^\diamond \ar[r, "\tilde{\theta}"] \ar[d, twoheadrightarrow, "\pi"] & \widetilde{X}^\diamond  \ar[d,twoheadrightarrow, "\pi"]  \\
    X^\diamond \ar[r, "\theta"]  & X^\diamond 
  \end{tikzcd}
\end{center}
Correspondingly, we have a commutative diagram of functors:
\begin{center}
  \begin{tikzcd}
    \PD_{\widetilde{X}^\diamond} \ar[r, "\tilde{\theta}_*"] \ar[d, twoheadrightarrow, "\pi_*"] & \PD_{\widetilde{X}^\diamond}  \ar[d,twoheadrightarrow, "\pi_*"]  \\
    \PD_{X^\diamond} \ar[r, "\tilde{\theta}_*"]  & \PD_{X^\diamond}
  \end{tikzcd}
\end{center}
By the (formal) universality of $\PD_{X^\diamond}$ and  $\PD_{\widetilde{X}^\diamond}$, this in turn corresponds to a commutative diagram
\begin{equation}\label{diag:fP_complete}
  \begin{tikzcd}
    \fP^\wedge \ar[r, "\tilde{\theta}_*"] \ar[d, twoheadrightarrow, "\pi_*"] & \fP^\wedge \ar[d,twoheadrightarrow, "\pi_*"]  \\
    (\fP/W)^\wedge \ar[r, "\tilde{\theta}_*"]  &  (\fP/W)^\wedge
  \end{tikzcd}
\end{equation}
of completions of affine schemes at their distinguished points.

\begin{prop}\label{prop:auto_fP}
    Assume that $X$ is a conical symplectic singularity and $\theta$ is a graded (i.e., $\cm$-equivariant) Poisson automorphism of $X$. Then the commutative diagram \eqref{diag:fP_complete} can be obtained by completion from a commutative diagram of linear spaces
    \begin{equation}\label{diag:fP}
      \begin{tikzcd}
        \fP \ar[r] \ar[d, twoheadrightarrow] & \fP \ar[d,twoheadrightarrow]  \\
        \fP/W \ar[r]  & \fP/W
      \end{tikzcd}
    \end{equation}
    where the top horizontal map is linear and the vertical ones are quotient maps. The map $\fP \to \fP$ is given by the pullback map $(\tilde{\theta}^{-1})^*:  H^2(\widetilde{X}^\diamond, \CC) \to H^2(\widetilde{X}^\diamond, \CC)$ induced by $\tilde{\theta}^{-1}$.
\end{prop}
\begin{proof}
    $X^\diamond$ is $\cm$-stable and the $\cm$-action lifts to a $\cm$-action on $\widetilde{X}^\diamond$. These $\cm$-actions induce those on $\fP^\wedge$ and $(\fP/W)^\wedge$ and their completions, so that the morphisms in \eqref{diag:fP_complete} are $\cm$-equivariant. By taking $\cm$-finite part of the coordinate algebra, we can algebraize the diagram \eqref{diag:fP_complete} and get \eqref{diag:fP}, whose morphisms are all $\cm$-equivariant. Let $\gamma: \fP \to \fP$ denote the top horizontal map in \eqref{diag:fP}.

    Since $\fP$ is a vector space, the Zariski tangent space $T_0 \fP$ of $\fP$ at $0$ is canonically identified with $\fP$ itself. Under this identification, it is easy to see that the differential $d\gamma: T_0 \fP \to T_0 \fP$ of $\gamma$ coincides with $\gamma$ itself since $\gamma$ is $\cm$-equivariant. Indeed, by Theorem part (i) of \ref{thm:universal_deform}, $\cm$ acts linearly on $\fP$ by $t.v \mapsto t^d v$. Therefore $\gamma$ being $\cm$-equivariant just means that $\gamma$ commutes with the scalar mulitiplication on $\fP$, and hence 
      $$d\gamma(v) = \frac{d}{dt} \gamma(tv) = \frac{d}{dt} [t\gamma(v)] = \gamma(v), \quad \forall \, v \in T_0 \fP \simeq \fP.$$
    In particular, $\gamma$ is a linear automorphism of $\fP$. On the other hand, $T_0 \fP = \PD_{\widetilde{X}^\diamond}(\CC\{\epsilon\}) \simeq H^2(\widetilde{X}^\diamond, \CC)$ and $d\gamma = (\tilde{\theta}^{-1})^*$. 
\end{proof}

For any graded Poisson varity $X$, let $\Aut(X)$ denote the group of graded Poisson automorphisms of $X$. Note that if $X = \Spec \CC[\widetilde{\OO}]$ for $\widetilde{\OO} \in \Cov(G)$, then $\Aut(X) = \Aut(\widetilde{\OO})$. For a general conical symplectic singularity $X$, Proposition \ref{prop:auto_fP} yields an action of $\Aut(X)$ on $\fP = H^2(\widetilde{X}^\diamond, \CC)$. 
    
\begin{prop}\label{prop:aut_Namikawa_space}
  Let $X$ be a conical symplectic singularity. Then the action of $\Aut(X)$ on $\fP=\fP^X$ enjoys the following properties:
  \begin{enumerate} [label=(\roman*)]
    \item
      It normalizes $W=W^X \subset GL(\fP)$, so that the quotient map $\fP \twoheadrightarrow \fP / W$ is $\Aut(X)$-equivariant.
    \item
      It lifts to an $\Aut(X)$-action on $X_\fP$ by graded Poisson automorphisms, which normalizes the $W$-action. The induced action on $X_0  \simeq X$ is just the defining action of $\Aut(X)$ on $X$.
  \end{enumerate}
\end{prop}

\begin{proof}

    (i) is an immediate consequence of Proposition \ref{prop:auto_fP}. For (ii), note that by the universal property we have a natural $\Aut(X)$-action on the universal family $\sX_{univ}$ over $\fP/W$ by graded Poisson automorphisms, and this action is compatible with the $\Aut(X)$-action on $\fP/W$. Then by (i), the fiber product $X_\fP = \sX_{univ} \times_{\fP / W} \fP$ admits an action of $\Aut(X)$. 
    Now $W$ can be identified with the group of graded Poisson automorphisms of $X_\fP$ whose restrictions to the central fiber $X_{\fP} \times_{\fP} \{0\} \simeq X$ is identity map (this will appear in \cite[Section 7]{LMBM}, see also \cite{IvanNotes}). Clearly the $\Aut(X)$-action on $X_\fP$ normalizes $W$.
\end{proof}

We conclude this section with an easy lemma.
\begin{lemma}\label{lem:Gaois_almostetale_Namikawa_space}
  Let $X_1$ and $X_2$ be two affine symplectic singularities and define open subvarieties $X_i^\diamond \subset X_i$, $i=1, 2$, as in Section \ref{subsec:universal_deform_symplectic}. Suppose $p: X_1 \to X_2$ is a finite almost \'{e}tale (Poisson) morphism, such that its restriction $p|_{X_1^\diamond}: X_1^\diamond \to X_2^\diamond$ is a Galois covering with Galois group $\Gamma$. Then $p$ induces a canonical injective linear map $p^*: \fP^{X_2} \hookrightarrow \fP^{X_1}$ whose image is the the space $(\fP^{X_1})^\Gamma$ of $\Gamma$-fixed vectors of $\fP^{X_1}$.
\end{lemma}

\begin{proof}
    Let $\widetilde{X}_2^\diamond \to X_2^\diamond$ be the unique (smooth) $\QQ$-factorial terminalization of $X_2^\diamond$ and set $\widetilde{X}_1^\diamond = \widetilde{X}_2^\diamond \times_{X_2^\diamond} X_1^\diamond \to \widetilde{X}_2^\diamond$ to be the base change of $\widetilde{X}_2^\diamond \to X_2^\diamond$ along the morphism $p|_{X_1^\diamond}: X_1^\diamond \to X_2^\diamond$. Then $\widetilde{X}_1^\diamond \to \widetilde{X}_2^\diamond$ is also a Galois covering with Galois group $\Gamma$ and the canonical projection $\widetilde{X}_1^\diamond \to X_1^\diamond$ is a (smooth) $\QQ$-factorial terminalization of $X_1^\diamond$. Therefore the pullback map $p^*: H^2(\widetilde{X}_2^\diamond, \CC) \to H^2(\widetilde{X}_1^\diamond, \CC)$ satisfies the claim.
\end{proof}

\subsection{The extended Namikawa Weyl groups of nilpotent covers}\label{sec:Wcovers}

Fix all of the notation of \cref{subsec:universal_deform_nilpotent}, e.g. $\widetilde{\OO}$, $X$, $(L,\widetilde{\OO}_L)$, $X_L$, $\fz$, and so on. In this section, we will give an explicit description of the Namikawa Weyl group $W(\widetilde{\OO}) := W^X$ in terms of $(L,\widetilde{\OO}_L)$. This will require some additional notation. Let $N_G(L)$ denote the normalizer of $L$ in $G$. Let $\Ad^*$ denote the co-adjoint action of $N_G(L)$ on $\fl^*$ and let $\mu_L: X_L \to \fl^*$ denote the moment map. Consider the group
$$N_G(L, \widetilde{\OO}_L) := \{ (n, \zeta) \in N_G(L) \times \Aut(X_L) \,|\, \Ad(n) \circ \mu_L = \mu_L \circ \zeta \}.$$
We can regard $L$ as a normal subgroup of $N_G(L,\widetilde{\OO}_L)$ via the natural embedding
$$L \hookrightarrow N_G(L, \widetilde{\OO}_L), \qquad l \mapsto (l, l).$$
\begin{definition} 
The {extended Namikawa-Weyl group} associated to $\widetilde{\OO}$ is the finite group
      $$ \widetilde{W}(\widetilde{\OO}) := N_G(L, \widetilde{\OO}_L)/L. $$
\end{definition}

We note that $\widetilde{W}(\widetilde{\OO})$ acts on $\fz = \fX(\fl \cap [\fg,\fg])$ via the natural map $\widetilde{W}(\widetilde{\OO}) \to N_G(L)/L$. 

Next, we will describe an action of $\widetilde{W}(\widetilde{\OO})$ on $\sX_\fz$. For a general graded Poisson variety $X$ with a Hamiltonian $G$-action that commutes with the $\cm$-action, let $\Aut^G(X) \subset \Aut(X)$ denote the subgroup consisting of $G \times \cm$-equivariant Poisson automorphisms of $X$. Note that if $X = \Spec(\CC[\widetilde{\OO}])$ for some $\widetilde{\OO} \in \Cov(G)$, we have  $\Aut^G(X)=\Aut^G(\widetilde{\OO})=\Aut(\widetilde{\OO}, \OO)$ by \cref{ex:nilpotentcover}.

\begin{theorem}\label{thm:action_extended_Namikawa-Weyl}
With notations as in Section \ref{subsec:universal_deform_nilpotent}, the following are true:
\begin{enumerate}[label=(\roman*)]
   \item 
     There is a natural action of $\widetilde{W}(\widetilde{\OO})$ on $\sX_\fz$ by $G \times \cm$-equivariant Poisson automorphisms preserving the moment map and lifting the action on $\fz$. Moreover, the action preserves $\sX_\fz^\circ$ and its restriction on $\sX_\fz^\circ$ induces an action on $G \times^L (X_L^e)^\circ $ via the isomorphism \eqref{eq:isom_regular1}, which descends from the $N_G(L, \widetilde{\OO}_L)$-action on $G \times^L (X_L^e)^\circ $ given explicitly by
	  $$ [g, x, \beta] \mapsto [g n^{-1}, \zeta(x), n.\beta],$$
	  for any $g \in G$, $(x, \beta) \in (X_L^e)^\circ = X_L \times \fz^\circ$ and $(n, \zeta) \in N_G(L, \widetilde{\OO}_L)$. The action of $\widetilde{W}(\widetilde{\OO})$ on $\sX_\fz$ is uniquely characterized by this property. 
   \item  
     Let $\varphi: X_\chi \to X_{\chi'}$, $\chi, \chi' \in \fz^\circ$, be a Hamiltonian isomorphism, i.e., a $G$-equivariant Poisson isomorphism intertwining the moment maps to $\fg^*$. Then $\varphi$ is induced by a unique element of $\widetilde{W}(\widetilde{\OO})$.
   \item
     There is a short exact sequence of groups,
      \begin{equation} \label{exsq:extend_Namikawa-Weyl}
         1 \to W(\widetilde{\OO}) \to \widetilde{W}(\widetilde{\OO}) \to \Aut^G(X) \to 1,
      \end{equation}
    where the surjective homomorphism $\widetilde{W}(\widetilde{\OO}) \twoheadrightarrow \Aut^G(X)$ is given by the restriction of the $\widetilde{W}(\widetilde{\OO})$-action on $\sX_\fz$ described in (i) to the fiber $\sX_\fz \times_\fz \{ 0\} \simeq X$. 
\end{enumerate}
\end{theorem}
\begin{proof}
The last statement of (i) follows from the fact that $\sX_\fz$ is reduced and separated. The rest will appear in \cite[Section 7]{LMBM}, see also \cite{IvanNotes}.
\end{proof}
%

In view of \cref{prop:aut_Namikawa_space}, we can form the semi-direct product
$$\widetilde{W}^X_G := W^X \rtimes \Aut^G(X),$$ 
This is a finite group which naturally acts on $\fP$, $\sX_{univ}$ and $\sX_\fP$. Note that the definitions of $\widetilde{W}^X_G$ and its actions are independent of choices. The following will appear in \cite[Section 7]{LMBM}, see also \cite{IvanNotes}.

\begin{prop}\label{prop:split_Weyl_O}
A choice of a parabolic subgroup $P = LN \subset G$ determines an isomorphism $\kappa_P: \widetilde{W}(\widetilde{\OO}) \xrightarrow{\sim} \widetilde{W}^X_G$, such that
 \begin{enumerate}
     \item 
     $\kappa_P$ is compatible with the short exact sequence \eqref{exsq:extend_Namikawa-Weyl}, i.e., it restricts to an isomorphism from $W(\widetilde{\OO}) \subset \widetilde{W}(\widetilde{\OO})$ onto $W^X \subset \widetilde{W}^X_G$ and the induced map between the quotients is the identity map of $\Aut^G(X)$. In particular, \eqref{exsq:extend_Namikawa-Weyl} splits.
     \item
     The isomorphism $\eta: \fz \xrightarrow{\sim} \fP$ intertwines the action of $\widetilde{W}(\widetilde{\OO})$ on $\fz$ and the action of $\widetilde{W}^X_G$ on $\fP$ under the identification $\kappa_P: \widetilde{W}(\widetilde{\OO}) \xrightarrow{\sim} \widetilde{W}^X_G$.
 \end{enumerate}
\end{prop}

\subsection{Extended Namikawa Weyl group vs parabolic induction}

Fix all of the notation of \cref{subsec:universal_deform_nilpotent}, e.g. $\widetilde{\OO}$, $X$, $(L,\widetilde{\OO}_L)$, $X_L$, $\fz$, and so on. Suppose $M$ is a Levi subgroup of $G$ containing $L$. Set $\widetilde{\OO}_M = \Bind_L^M \widetilde{\OO}_L$ and $X_M = \Spec(\CC[\widetilde{\OO}_M])$. We want to compare the Namikawa spaces and (extended) Namikawa groups of $\widetilde{\OO}$ and $\widetilde{\OO}_M$. To simplify the statements, we assume in this section that $G$ is semisimple.

Choose parabolic subgroups $P = LN \subset G$ and $Q=MU \subset G$, such that $L \subset M$ and $P \subset Q$. Let $P_M = P \cap M = L N_M$, a parabolic subgroup of $M$ with Levi factor $L$ and unipotent radical $N_M=N \cap M$. Note that $\fn = \fn_M \oplus \fu$. We have an inclusion $\mathfrak{X}(\fm) \hookrightarrow \mathfrak{X}(\fl)$ and a projection $\mathfrak{X}(\fl) \twoheadrightarrow \mathfrak{X}(\fl \cap [\fm, \fm])$ induced by restriction of characters. We write 
\[
\fz_\fm^\fg := \mathfrak{X}(\fm), \quad \fz_\fl^\fm := \mathfrak{X}(\fl \cap [\fm, \fm]).
\] Note that we have natural isomorphisms $\fz \simeq \fz(\fl)^*$, $\fz_\fm^\fg \simeq \fz(\fm)^*$ and $\fz_\fl^\fm \simeq \fz(\fl \cap [\fm, \fm])^*$. Moreover, we have a direct sum decomposition $\fz(\fl) = \fz(\fm) \oplus \fz(\fl \cap [\fm, \fm])$. The latter induces a direct sum decomposition $\fz = \fz_\fm^\fg \oplus \fz_\fl^\fg$. Set $\sY_\fz := \Ind_{P}^G X_L^e$,  $\sY_{\fz_\fm^\fg}^M:= \Ind_{P_M}^M (X_L \times \fz_\fm^\fg)$, $\sY_\fz^M := \Ind_{P_M}^M X_L^e$ and $Y_M = \Ind_{P_M}^M X_L$. As before, set $\sX_\fz := \CC[\sY_\fz]$ and similarly $\sX_{\fz_\fm^\fg}^M := \CC[\sY_{\fz_\fm^\fg}^M]$, $\sX_{\fz}^M := \CC[\sY_{\fz}^M]$. Then $\sY_\fz^M = \sY_{\fz_\fm^\fg}^M \times \fz_\fm^\fg$ and $\sX_\fz^M = X_{\fz_\fm^\fg}^M \times \fz_\fm^\fg$. 

There is a canonical $G \times \CC^\times$-equivariant isomorphism
\begin{equation} \label{eq:isom_ind_family}
	\sY_\fz \xrightarrow{\sim} \Ind_Q^G \sY_\fz^M 
\end{equation}
or
\begin{align}
	G \times^P \left( X_L^e \times_{\fl^*} \fn^{\perp} \right)
	& \xrightarrow{\: \sim \:}   
	G \times^Q \left(  M \times^{P_M} \left( X_L^e \times_{\fl^*} (\fm/\fn_M)^* \right) \times_{\fm^*}  (\fg/\fu)^* \right) \\
	[g, x^e, \alpha] \qquad & \longmapsto \qquad [g, 1_M, x^e, \pi_\fm(\alpha), \iota_\fu(\alpha)] \nonumber
\end{align}
for any $g \in G$, $x^e \in X_L^e$ and $\alpha \in \fn^{\perp}$, where $1_M$ denotes the identity element of $M$. Here we are using the canonical isomorphism 
$$\fn^{\perp} \xrightarrow{\sim} (\fm/\fn_M)^* \times_{\fm^*}  (\fg/\fu)^*, \quad \alpha \mapsto (\pi_\fm(\alpha), \iota_\fu(\alpha)), $$
of $L$-representations, which is given by the pullback diagram
\begin{center}
  \begin{tikzcd}
    \fn^{\perp} \ar[r, hookrightarrow, "\iota_{\fu}"] \ar[d, twoheadrightarrow, "\pi_{\fm}"] & (\fg/\fu)^* \ar[d, twoheadrightarrow] \\
    (\fm / \fn_M)^* \ar[r, hookrightarrow] & \fm^* 
   \end{tikzcd}
\end{center}
where $\iota_{\fu}$ is the dual map of the projection $\fg/\fu \twoheadrightarrow \fg/\fn$ and $\pi_{\fm}$ is the dual map of the inclusion $\fm/\fn_M \hookrightarrow \fg/\fn$. It can be checked that \eqref{eq:isom_ind_family} is well-defined and intertwines the maps to $\fg^* \times \fz$.

Recall that $\sY_\fz \times_{\fz} \{ 0 \}$ is isomorphic to $Y = G \times^P \left( X_L \times (\fg / \fp)^* \right)$. Therefore restricting the isomorphism \eqref{eq:isom_ind_family} to the fibers over $0 \in \fz$ gives an isomorphism 
\begin{equation}
  Y = G \times^P (X_L \times \fp^\perp) \xrightarrow{\sim} \Ind_P^G Y_M = G \times^Q (Y_M \times \fq^\perp). 
\end{equation}
(cf. \cite[Lemma 7.4.1]{LMBM}). Thus we can view $Y^{reg}$ as a fiber bundle over $G/Q$ with fiber $Y_M^{reg} \times \fq^\perp$, so that there are pullback maps
\begin{equation} \label{eq:pullback_Y}
  H^2(G/Q, \CC) \to H^2(Y^{reg}, \CC), \qquad \fP(\widetilde{\OO}) = H^2(Y^{reg}, \CC) \to \fP(\widetilde{\OO}_M) = H^2(Y^{reg}_M, \CC).
\end{equation}
On the other hand, we have identifications
\begin{equation} \label{eq:isom_Ind_M}
    \mathfrak{X}(\fm) \xrightarrow{\sim} H^2(G/Q, \CC), \quad \eta: \mathfrak{X}(\fl) \xrightarrow{\sim} H^2(Y^{reg}, \CC), \quad \eta_M: \mathfrak{X}(\fl \cap [\fm, \fm]) \xrightarrow{\sim} H^2(Y^{reg}_M, \CC),
\end{equation}
where the latter two are from Proposition \ref{prop:Namikawacenter}. The following result is Proposition 7.4.2 of \cite{LMBM}. 
%
\begin{prop}\label{prop:pullback}
    Under the isomorphisms in \eqref{eq:isom_Ind_M}, the pullback maps in \eqref{eq:pullback_Y} correspond to the inclusion $\mathfrak{X}(\fm) \hookrightarrow \mathfrak{X}(\fl)$ and the projection $\mathfrak{X}(\fl) \twoheadrightarrow \mathfrak{X}(\fl \cap [\fm, \fm])$, respectively.
\end{prop}

Next we compare the extended Namikawa groups $\widetilde{W}(\widetilde{\OO})$ and $\widetilde{W}(\widetilde{\OO}_M)$. Note that the natural inclusion $N_M(L) \subset N_G(L)$ induces an inclusion $N_M(L, \widetilde{\OO}_L) \subset N_G(L, \widetilde{\OO}_L)$ and hence an injective group homomorphism $\iota_{\widetilde{W}} : \widetilde{W}(\widetilde{\OO}_M) \hookrightarrow \widetilde{W}(\widetilde{\OO})$. The restriction of the standard action of $\widetilde{W}(\widetilde{\OO})$ on $\fz = \fz_\fm^\fg \oplus \fz_\fl^\fg$ to $\widetilde{W}(\widetilde{\OO}_M)$ is the product of the standard action of $\widetilde{W}(\widetilde{\OO}_M)$ on $\fz_\fl^\fg$ and the trivial one on $\fz_\fm^\fg$.

Note that there is also a natural inclusion $\iota_{\Aut}: \Aut^M(X_M) \hookrightarrow \Aut^G(X)$ of groups induced by the natural action of $\Aut^M(X_M) \simeq \Aut^M(\widetilde{\OO}_M)$ on $\Bind_M^G \widetilde{\OO}_M = \Bind_M^G \Bind_L^M \widetilde{\OO}_L = \Bind_L^G \widetilde{\OO}_L = \widetilde{\OO}$ (cf. Proposition \ref{prop:propsofbind}). We will also need the following result, which is immediate from the definition of the $\Aut(X)$-action on the Namikawa space $\fP$.
\begin{lemma}\label{lem:fP_projection_equiv}
   The pullback map $H^2(Y^{reg},\CC) \to H^2(Y_M^{reg},\CC)$ in \eqref{eq:pullback_Y} is $\Aut^M(X_M)$-equivariant, where $\Aut^M(X_M)$ acts on $H^2(Y^{reg},\CC)$ via the inclusion $\iota_{\Aut}: \Aut^M(X_M)=\Aut^M(\widetilde{\OO}_M) \hookrightarrow \Aut^G(\widetilde{\OO})$.
\end{lemma}

Let $(\fz_\fm^\fg)^\circ \subset \fz_\fm^\fg$ and $(\fz_\fl^\fm)^\circ \subset \fz_\fl^\fm$ be the open subsets defined in the same way as $\fz^\circ$ (cf. \cref{subsec:universal_deform_nilpotent}), but for $\fm \subset \fg$ and $\fl \subset \fm$ respectively. Set 
\[\fz^\bullet := \fz^\circ \cap [(\fz_\fm^\fg)^\circ \times (\fz_\fl^\fm)^\circ] \subset \fz.\] 
This is the complement in $\fz$ to a finite union of hyperplanes, and hence a dense open subset of $\fz$. Moreover, it is stable under the action of $\widetilde{W}(\widetilde{\OO}_M)$. 

Set $(X^e_L)^\bullet := X_L \times \fz^\bullet$ , $(\sY_\fz^{M})^\bullet  := \sY_\fz^M \times_{\fz} \fz^\bullet$ and so on. Then 
$$ (\sY_\fz^M)^\bullet  =  M \times^{P_M} \left( (X_L^e)^\bullet \times_{\fl^*} (\fm/\fn_M)^* \right).$$ 
Similarly to the isomorphism \eqref{eq:isom_regular1}, we have an $M$-equivariant isomorphism $M \times^L (X_L^e)^\bullet \xrightarrow{\sim} (\sY_\fz^M)^\bullet$, which induces a $G$-equivariant isomorphism
\begin{equation}\label{eq:isom_regular2}
	G \times^M \left( M \times^L (X_L^e)^\bullet \right) \xrightarrow{\sim} G \times^M (\sY_\fz^{M})^\bullet.
\end{equation}
Again a similar construction gives a $G$-equivariant isomorphism
\begin{equation}\label{eq:isom_regular3}
	G \times^M (\sY_\fz^{M})^\bullet \xrightarrow{\sim} \Ind_Q^G (\sY_\fz^{M})^\bullet.
\end{equation}
Composing the isomorphisms  \eqref{eq:isom_regular2} and \eqref{eq:isom_regular3} gives an isomorphism 
\begin{equation}\label{eq:isom_regular4}
	G \times^M \left( M \times^L (X_L^e)^\bullet \right) \xrightarrow{\sim} \Ind_Q^G (\sY_\fz^{M})^\bullet.
\end{equation}

Now we consider the following diagram of $G$-equivariant isomorphisms of varieties over $\fg^* \times \fz^\bullet$. 

\begin{equation}  \label{diag:NamikawaWeyl_parabolic}
  \begin{tikzcd}
    G \times^L (X_L^e)^\bullet \ar[r, "\sim"]  \ar[d, "\sim" labl1]  & \sY_\fz^\bullet \ar[r, "\sim"] \ar[d, "\sim" labl1] & \sX_\fz^\bullet  \\
    G \times^M \left( M \times^L (X_L^e)^\bullet \right) \ar[r, "\sim"] & \Ind_Q^G (\sY_\fz^{M})^\bullet  \ar[r, "\sim"]  & \Ind_Q^G (\sX_\fz^M)^\bullet \ar[u, "\sim" labl2] 
   \end{tikzcd}
\end{equation}
The leftmost vertical isomorphism is the tautological one, the middle vertical isomorphism is the base change of the isomorphism $\sY_\fz \xrightarrow{\sim} \Ind_Q^G \sY_\fz^M$ in \eqref{eq:isom_ind_family} to $\fg^*  \times \fz^\bullet$, and the rightmost vertical isomorphism is given by Stein factorization. The top left isomorphism is \eqref{eq:isom_regular1} and the top right isomorphism is the base change of the natural morphism $\sY_\fz \to \sX_\fz$ to $\fg^*  \times \fz^\bullet$. The bottom left isomorphism is \eqref{eq:isom_regular4}, and the bottom right isomorphism is induced by the isomorphism $(\sY_\fz^{M})^\bullet \xrightarrow{\sim} (\sX_\fz^M)^\bullet$. Therefore the right square sub-diagram is the base change to $\fg^* \times \fz^\bullet$ of the corresponding diagram without the decoration $\bullet$.
	
Each variety in \eqref{diag:NamikawaWeyl_parabolic} carries a natural action of $\widetilde{W}(\widetilde{\OO}_M)$ which is the restriction of an action on the corresponding variety without the decoration $\bullet$ and is compatible with the $\widetilde{W}(\widetilde{\OO}_M)$-action on $\fz^\bullet$ via the projection to $\fz^\bullet$. On the varieties in the top row, $\widetilde{W}(\widetilde{\OO}_M)$ acts via the natural inclusion $\iota_{\widetilde{W}} : \widetilde{W}(\widetilde{\OO}_M) \hookrightarrow \widetilde{W}(\widetilde{\OO})$ and the $\widetilde{W}(\widetilde{\OO})$-action in Theorem \ref{thm:action_extended_Namikawa-Weyl}. On the varieties in the bottom row, the $\widetilde{W}(\widetilde{\OO}_M)$-actions are induced by the actions on
$$ M \times^L (X_L^e)^\bullet  \simeq (\sY_\fz^{M})^\bullet  \simeq (\sX_\fz^M)^\bullet $$
by applying Theorem \ref{thm:action_extended_Namikawa-Weyl} again to $M$, $L$ and $\widetilde{\OO}_M$. Note that $\widetilde{W}(\widetilde{\OO}_M)$ acts trivially on the second factor of $\sX_\fz^M = \sX_{\fz_\fm^\fg}^M \times \fz_\fm^\fg$.

The discussion above and Theorem \ref{thm:action_extended_Namikawa-Weyl}(i) proves the following lemma.

\begin{lemma} \label{lem:NamikawaWeyl_parabolic}
  The diagram \eqref{diag:NamikawaWeyl_parabolic} commutes and all the isomorphisms there are $\widetilde{W}(\widetilde{\OO}_M)$-equivariant.
\end{lemma}

Now we are ready to examine the relationship between $\widetilde{W}(\widetilde{\OO}_M)$ and $ \widetilde{W}(\widetilde{\OO})$.

\begin{prop}\label{prop:compare_extended_Weyl}
The homomorphisms $\iota_{\widetilde{W}} : \widetilde{W}(\widetilde{\OO}_M) \hookrightarrow \widetilde{W}(\widetilde{\OO})$ and $\iota_{\Aut}: \Aut^M(\widetilde{\OO}_M) \hookrightarrow \Aut^G(\widetilde{\OO})$ fit into the following commutative diagram 
	\begin{center}
		\begin{tikzcd}
			1 \ar[r] & W(\widetilde{\OO}_M) \ar[r] \ar[d, hookrightarrow] & \widetilde{W}(\widetilde{\OO}_M) \ar[r] \ar[d, hookrightarrow, "\iota_{\widetilde{W}}"] & \Aut^M(\widetilde{\OO}_M) \ar[r] \ar[d, hookrightarrow, "\iota_{\Aut}"] & 1\\
			1 \ar[r]& W(\widetilde{\OO}) \ar[r]  & \widetilde{W}(\widetilde{\OO}) \ar[r]& \Aut^G(\widetilde{\OO}) \ar[r] & 1
		\end{tikzcd}
	\end{center}
where the top and the bottom rows are short exact sequences from \eqref{exsq:extend_Namikawa-Weyl}
\end{prop}

\begin{proof}
 Note that $\Ind_Q^G (\sX_\fz^M)$ and $\sX_\fz$ are reduced and separated. Therefore by Lemma \ref{lem:NamikawaWeyl_parabolic}, the Stein factorization map $\Ind_Q^G (\sX_\fz^M) \twoheadrightarrow \sX_\fz$ is $\widetilde{W}(\widetilde{\OO}_M)$-equivariant. The commutativity of the diagram then follows immediately. 
\end{proof}


\subsection{Unipotent ideals}\label{subsec:unipotent}

Let $G$ be a connected reductive algebraic group and let $\widetilde{\OO} \in \Cov(G)$. Recall the canonical quantization $(\cA_0,\Phi_0)$ of $\CC[\widetilde{\OO}]$, cf. Definition \ref{def:canonical}. 

\begin{definition}[Definition 6.0.1, \cite{LMBM}]\label{def:unipotentideal}
The \emph{unipotent ideal} attached to $\widetilde{\OO}$ is the primitive ideal
$$I(\widetilde{\OO}) := \ker{(\Phi_0: U(\fg) \to \cA_0)} \subset U(\fg)$$
We write $\gamma(\widetilde{\OO}) \in \fh^*/W$ for its infinitesimal character.
\end{definition}

We will need several basic facts about unipotent ideals and their infinitesimal characters. Recall the equivalence relation $\sim$ on nilpotent covers defined in Section \ref{sec:nilpotentcovers}.

\begin{prop}[Propositions 6.5.4, 8.1.1, \cite{LMBM}]\label{prop:unipotentfacts}
The following are true:
\begin{itemize}
    \item[(i)] For every $\widetilde{\OO} \in \Cov(G)$, $I(\widetilde{\OO}) \subset U(\fg)$ is a completely prime maximal ideal with associated variety $\overline{\OO}$.
    \item[(ii)] For $\widetilde{\OO},\widehat{\OO} \in \Cov(G)$, we have 
    $$I(\widetilde{\OO}) = I(\widehat{\OO}) \iff \widetilde{\OO} \sim \widehat{\OO}.$$
    \item[(iii)] Suppose $L \subset G$ is a Levi subgroup. Let $\widetilde{\OO}_L \in \Cov(L)$ and $\widetilde{\OO} = \Bind^G_L \widetilde{\OO}_L$. Then 
    $$\gamma(\widetilde{\OO}_L) = \gamma(\widetilde{\OO})$$
    in $\fh^*/W$. 
\end{itemize}
\end{prop}

\begin{proof}
The maximality in (i) is \cite[Theorem 5.0.1]{MBMat}. The rest of (i) is \cite[Proposition 6.1.2]{LMBM}. (ii) is \cite[Proposition 6.5.4]{LMBM}. (iii) is \cite[Proposition 8.1.1]{LMBM}.
\end{proof}

\begin{cor}\label{cor:bindequivalence}
$\Bind^G_L$ descends to a map on equivalence classes
    $$\Bind^G_L: \Cov(L)/\sim \, \to \Cov(G)/\sim$$
\end{cor}

\begin{proof}
Suppose $\widetilde{\OO}_L \sim \widehat{\OO}_L$ in $\Cov(L)$ and let $\widetilde{\OO} = \Bind^G_L \widetilde{\OO}_L$, $\widehat{\OO} = \Bind^G_L \widehat{\OO}_L$. By (ii) and (iii) of Proposition \ref{prop:unipotentfacts}, we have
$$\gamma(\widetilde{\OO}) = \gamma(\widehat{\OO})$$
By (i) of Proposition \ref{prop:unipotentfacts}, $I(\widetilde{\OO})$ and $I(\widehat{\OO})$ are maximal ideals in $U(\fg)$, and hence uniquely determined by their infinitesimal characters. So in fact
$$I(\widetilde{\OO}) = I(\widehat{\OO})$$
and therefore $\widetilde{\OO} \sim \widehat{\OO}$ by Proposition \ref{prop:unipotentfacts}(ii).

\end{proof}

For the next proposition, choose a $W$-invariant inner product on $\fh^*$ and write $\lVert\cdot\rVert$ for the associated norm.

\begin{prop}\label{prop:inequalityinflchar}
Suppose $\widehat{\OO} \to \widetilde{\OO}$ is a finite $G$-equivariant Galois cover. Then
\begin{itemize}
    \item[(i)] $\lVert\gamma(\widehat{\OO})\rVert \geq \lVert\gamma(\widetilde{\OO})\rVert$.
    \item[(ii)] There is equality in (i) if and only if the covering is almost \'{e}tale (in which case, $\gamma(\widehat{\OO})$ and $\gamma(\widetilde{\OO})$ are conjugate under $W$).
\end{itemize}
\end{prop}

\begin{proof}
Let $\cA_0(\widetilde{\OO})$ denote the canonical quantization of $\CC[\widetilde{\OO}]$ and let $\Gamma = \Aut(\widehat{\OO},\widetilde{\OO})$. Then $\Gamma$ acts on $\cA_0(\widehat{\OO})$ by filtered algebra automorphisms, and there is an isomorphism of filtered quantizations $\cA_0(\widehat{\OO})^{\Gamma} \simeq \cA_{\epsilon}(\widetilde{\OO})$ for some element $\epsilon \in \fP(\widetilde{\OO})$, see \cite[Proposition 5.3.1]{LMBM}. Choose $(L,\widetilde{\OO}) \in \Cov_0(G)$ such that $\widetilde{\OO} = \Bind^G_L \widetilde{\OO}_L$. Let $\delta$ denote the image of $\epsilon$ under the isomorphism $\fP(\widetilde{\OO}) \simeq \fX(\fl \cap [\fg,\fg])$. Then by \cite[Proposition 8.1.3]{LMBM}
$$\gamma(\widetilde{\OO}) = \gamma(\widetilde{\OO}_L), \qquad \gamma(\widehat{\OO}) = \delta + \gamma(\widetilde{\OO}_L),$$
where $\gamma(\widetilde{\OO}_L) \in (\fh/\fz(\fl))^*$. Since $\fh/(\fz(\fl))^*$ and $\fz(\fl)^*$ are orthogonal subspaces of $\fh^*$, we have
$$\lVert\gamma(\widehat{\OO})\rVert^2 = \lVert\gamma(\widetilde{\OO}_L) + \delta\rVert^2 = \lVert\gamma(\widetilde{\OO}_L)\rVert^2 + \lVert\delta\rVert^2 \geq \lVert\gamma(\widetilde{\OO})\rVert^2.$$
This proves (i). 

If the covering $\widehat{\OO} \to \widetilde{\OO}$ is almost \'{e}tale, then $\gamma(\widetilde{\OO})=\gamma(\widehat{\OO})$ by Proposition \ref{prop:unipotentfacts}(ii). Otherwise, it is clear from the proof of \cite[Proposition 5.3.1]{LMBM} that $\delta \neq 0$. And therefore, the inequality in (i) is strict. This proves (ii).
\end{proof}

\begin{cor}\label{cor:bind_almostetale}
Let $\widetilde{\OO}_L,\widehat{\OO}_L \in \Cov(L)$ and suppose $p: \widehat{\OO}_L \to \widetilde{\OO}_L$ is a Galois cover. Then $p$ is almost \'{e}tale if and only if $\Bind^G_L \widehat{\OO}_L \sim \Bind^G_L \widetilde{\OO}_L$.
\end{cor}

\begin{proof}
Let $\widetilde{\OO} = \Bind^G_L \widetilde{\OO}_L$ and $\widehat{\OO} = \Bind^G_L \widehat{\OO}_L$. If $p$ is almost \'{e}tale, then $\widehat{\OO}_L \sim \widetilde{\OO}_L$, and so $\widehat{\OO} \sim \widehat{\OO}$ by Corollary \ref{cor:bindequivalence} Conversely, suppose $p$ is not almost \'{e}tale. Then by Proposition \ref{prop:inequalityinflchar}, there is a strict inequality
$$\lVert\gamma(\widehat{\OO}_L)\rVert > \lVert\gamma(\widetilde{\OO}_L)\rVert.$$
Since the norm is $W$-invariant, this implies that the elements $\gamma(\widehat{\OO}_L)$ and $\gamma(\widetilde{\OO}_L)$ are not conjugate under $W$. But by Proposition \ref{prop:unipotentfacts}(iii), $\gamma(\widehat{\OO}) = \gamma(\widehat{\OO}_L)$ and $\gamma(\widetilde{\OO}) = \gamma(\widetilde{\OO}_L)$. So by Proposition \ref{prop:unipotentfacts}(ii), $\widehat{\OO}$ and $\widetilde{\OO}$ are in different equivalence classes.
\end{proof}

To conclude this section, we note that some of the unipotent ideals of Definition \ref{def:unipotentideal} (particularly in classical types) have previously appeared in the work of various experts, see e.g \cite{McGovern1994} and \cite{Barbasch_dualpairs}. The original contribution in \cite{LMBM} is to provide a \emph{uniform} definition of unipotent ideals.

\subsection{Unipotent bimodules}\label{sec:unipotentbimodules}

Let $G$ be a connected reductive algebraic group. A  \emph{$G$-equivariant Harish-Chandra $U(\fg)$-bimodule} is a finitely-generated $U(\fg)$-bimodule $X$ such that the adjoint action of $\fg$ on $X$ integrates to a rational action of $G$. Let $\HC^G(U(\fg))$ denote the category of $G$-equivariant Harish-Chandra $U(\fg)$-bimodules (with $U(\fg)$-bimodule homomorphisms). If $I \subset U(\fg)$ is a two-sided ideal, let $\HC^G(U(\fg)/I)$ denote the full subcategory of $\HC^G(U(\fg))$ consisting of bimodules $X \in \HC^G(U(\fg))$ such that $IX=XI=0$. This is a monoidal category under $\otimes_{U(\fg)}$. 

\begin{definition}[Definition 6.0.2, \cite{LMBM}]
Let $\widetilde{\OO} \in \Cov(G)$. A \emph{unipotent bimodule attached to $\widetilde{\OO}$} is an irreducible object in $\HC^G(U(\fg)/I(\widetilde{\OO}))$.
\end{definition}

We conclude this subsection by recalling a description of the category $\HC^G(U(\fg)/I(\widetilde{\OO}))$ given in \cite[Section 6]{LMBM}.

Recall from Lemma \ref{lem:coverfacts} that for any nilpotent cover $\widetilde{\OO} \in \Cov(G)$, there is a unique maximal element $\widetilde{\OO}_{max}$ in the equivalence class of $\widetilde{\OO}$. Define
\begin{equation}\label{eq:defofGamma}\Gamma(\widetilde{\OO}) := \Aut(\widetilde{\OO}_{max},\OO)\end{equation}
Note that $\Gamma(\widetilde{\OO})$ is a finite group. 

\begin{theorem}[Theorem 6.6.2, \cite{LMBM}]\label{thm:classificationHC}
Let $\widetilde{\OO} \in \Cov(G)$. Then there is an equivalence of monoidal categories
$$\HC^G(U(\fg)/I(\widetilde{\OO})) \simeq \Gamma(\widetilde{\OO})\modd$$
\end{theorem}

\subsection{Birational induction preserves maximal covers}\label{sec:maximal}

In this section, we will show that birational induction takes maximal covers to maximal covers. First we compare the preimages of covers in the same equivalence class under the map $\Bind: \Cov_0(G) \to \Cov(G)$. 

\begin{prop}\label{prop:compare_birigid_data}
	Suppose $\widetilde{\OO}, \widehat{\OO} \in \Cov(G)$ and $p: \widetilde{\OO} \to \widehat{\OO}$ is a finite Galois $G$-equivariant almost \'{e}tale covering map (cf. Section \ref{sec:nilpotentcovers}). Suppose $\widetilde{\OO} =  \Bind^G_L \widetilde{\OO}_L$, where $L$ is a Levi subgroup of $G$ and $\widetilde{\OO}_L \in \Cov(L)$ is birationally rigid. Then there exist a Levi subgroup $M \subset G$ containing $L$ and a birationally rigid cover $\widehat{\OO}_M \in \Cov(M)$, such that
    \begin{enumerate}
        \item[(i)]
	   $\Bind_M^G \widehat{\OO}_M \simeq \widehat{\OO}$.
	\item[(ii)]
	   There is a finite Galois almost \'{e}tale $M$-equivariant covering map 
            $$p_M: \widetilde{\OO}_M := \Bind_L^M \widetilde{\OO}_L \to \widehat{\OO}_M,$$ 
          such that the induced covering map $\Bind_M^G (p_M): \widetilde{\OO} = \Bind_M^G \widetilde{\OO}_M \to \Bind_M^G \widehat{\OO}_M$ corresponds to $p$ under the isomorphism in (i).
	\end{enumerate}	
\end{prop}

\begin{proof}
    Without loss of generality, we may assume $G$ is semisimple. Set $\Gamma = \Aut(\widetilde{\OO},\widehat{\OO})$. Then $\Gamma$ is a subgroup of $\Aut(\widetilde{\OO}, \OO)$ and hence, by Proposition \ref{prop:split_Weyl_O}(i), can be regarded as a subgroup of $\widetilde{W}(\widetilde{\OO})$ whose intersection with $W(\widetilde{\OO})$ is trivial. Therefore $\Gamma$ acts on 
      $$\fP(\widetilde{\OO}) \simeq  \mathfrak{X}(\fl) = (\fl / [\fl, \fl])^* = \fz(\fl)^* \simeq \fz(\fl),$$ 
    where the last identification is by the Killing form. Let  $M := Z_G (\fz(\fl)^\Gamma)$ be the centralizer of $\fz(\fl)^\Gamma$ in $G$. Then $M$ is a Levi subgroup containing $L$ with Lie algebra $\fm$ satisfying $\fz(\fm) = \fz(\fl)^\Gamma$. Note that $N_G(L)$ acts on $\fz(\fl)$ by conjugation. The subgroup of $N_G(L)$ consisting of elements that fix every vector in $\fz(\fm) \subset \fz(\fl)$ is exactly $N_M(L)$. Therefore $\Gamma$ in fact lies in $\widetilde{W}(\widetilde{\OO}_M) \subset \widetilde{W}(\widetilde{\OO})$. Furthermore, since $\Gamma$ has trivial intersection with $W(\widetilde{\OO})$, it also has trivial intersection with $W(\widetilde{\OO}_M)$ by Proposition \ref{prop:compare_extended_Weyl}. Hence $\Gamma$ is mapped isomorphically onto its image in $\Aut^M(\widetilde{\OO}_M)$, denoted as $\Gamma_M$, under the map $\widetilde{W}(\widetilde{\OO}_M) \twoheadrightarrow \Aut^M(\widetilde{\OO}_M)$ and $\Gamma_M$ in turn maps isomorphically onto $\Gamma$ under the map $\iota_{\Aut} : \Aut^M(\widetilde{\OO}_M) \hookrightarrow \Aut^G(\widetilde{\OO})$. Therefore $\Gamma \simeq \Gamma_M$ acts on $\widetilde{\OO}_M$ freely and $M$-equivariantly, and intertwines the moment map $\widetilde{\OO}_M \to \fm^*$. Thus we can form the quotient $\widehat{\OO}_M := \widetilde{\OO}_M / \Gamma_M$, which is an $M$-equivariant nilpotent cover of $\OO_M$; let $p_M: \widetilde{\OO}_M \to \widehat{\OO}_M$ denote the quotient map. Then $\widehat{\OO}_M$ is an $M$-equivariant nilpotent cover and $p_M$ is a finite Galois $M$-equivariant covering map, such that $\Aut(\widetilde{\OO}_M,\widehat{\OO}_M) = \Gamma_M$. By \cref{prop:propsofbind} (ii), the induced covering map $\Bind(p_M) : \widetilde{\OO} \to \Bind_M^G \widehat{\OO}_M$ is Galois and the composite map $\Aut(\widetilde{\OO}_M,\widehat{\OO}_M) \xrightarrow{\sim} \Aut(\widetilde{\OO},\Bind_M^G \widehat{\OO}_M) \hookrightarrow \Aut^G(\widetilde{\OO})$ is nothing else but the isomorphism $\Gamma_M \xrightarrow{\sim} \Gamma$. Hence $\Bind_M^G \widehat{\OO}_M \simeq \widetilde{\OO} / \Gamma  \simeq \widehat{\OO}$. The claim that $p_M$ is almost \'{e}tale follows from Corollary \ref{cor:bind_almostetale}. 
    
    It remains to show that $\widehat{\OO}_M$ is birationally rigid. By Proposition \ref{prop:Namikawacenter}, we have $\fP(\widetilde{\OO}_M) \simeq \fz(\fl \cap [\fm, \fm])^*$. By Proposition \ref{prop:pullback}, the map $\fP(\widetilde{\OO}) \to \fP(\widetilde{\OO}_M)$ in \eqref{eq:pullback_Y} corresponds to the restriction map $r: \fz(\fl)^* \twoheadrightarrow (\fz(\fl) \cap [\fm, \fm])^*$, which is $\Gamma$-equivariant by Lemma \ref{lem:fP_projection_equiv}. By the construction of $M$, the kernel of $r$ is $\fz(\fm)^* = (\fz(\fl)^\Gamma)^* = (\fz(\fl)^*)^\Gamma = \fP(\widetilde{\OO})^\Gamma$, hence the action of $\Gamma_M = \Gamma$ on $\fP(\widetilde{\OO}_M)$ only fixes $0$. Now by Lemma \ref{lem:Gaois_almostetale_Namikawa_space}, $\fP(\widehat{\OO}_M) = \fP(\widetilde{\OO}_M)^\Gamma = 0$, therefore Proposition \ref{prop:Namikawacenter} implies that $\widehat{\OO}_M$ is birationally rigid.    
    
\end{proof}

\begin{theorem}\label{thm:maximaltomaximal}
  If $\widehat{\OO}_M \in \Cov(M)$ is maximal in its equivalence class, then $\Bind^G_M \widehat{\OO}_M \in \Cov(G)$ is maximal in its equivalence class. 
\end{theorem}

\begin{proof}
  Set $\widehat{\OO} = \Bind^G_M \widehat{\OO}_M$. We first claim that the statement can be reduced to the case when $\widehat{\OO}_M$ is birationally rigid. Indeed, we can always choose $(L, \widehat{\OO}_{L}) \in \Cov_0(M)$ so that $\Bind_{L}^M \widehat{\OO}_{L} = \widehat{\OO}_M$ and hence $\Bind_{L}^G \widehat{\OO}_{L} = \widehat{\OO}$. Then $\widehat{\OO}_{L}$ must be maximal in its equivalence class. Otherwise, suppose $\widehat{\OO}_{L}^{max}$ is the maximal cover in the equivalence class $[\widehat{\OO}_{L}]$  and there is a nontrivial finite Galois $L$-equivariant covering map $\widehat{\OO}_{L}^{max} \to \widehat{\OO}_{L}$. By Proposition \ref{prop:propsofbind}, (ii), the induced covering map $\Bind_{L}^M \widehat{\OO}_{L}^{max} \to \widehat{\OO}_M$ is nontrivial. But by Proposition \ref{cor:bindequivalence}, $\Bind_{L}^M \widehat{\OO}_{L}^{max}$ and $\widehat{\OO}_M$ are in the same equivalence class. This contradicts with the maximality of $\widehat{\OO}_M$.
  
  Thus we can assume that $\widehat{\OO}_M$ is birationally rigid. Let $\widetilde{\OO}$ be the maximal cover in the equivalence class $[\widehat{\OO}]$ and $p: \widetilde{\OO} \to \widehat{\OO}$ be a Galois covering map. By applying Proposition \ref{prop:compare_birigid_data} to $p$, we know that there exists a Levi subgroup $M'$, a birationally rigid cover $\widehat{\OO}_{M'} \in \Cov(M')$, an $\widetilde{\OO}_{M'} \in \Cov(M')$, such that $\Bind_{M'}^G \widehat{\OO}_{M'} \simeq \widehat{\OO}$ and $\Bind_{M'}^G \widetilde{\OO}_{M'} \simeq \widetilde{\OO}$, together with an $M'$-equivariant almost \'{e}tale covering map $p_{M'}: \widetilde{\OO}_{M'} \to \widehat{\OO}_{M'}$ which induces $p$. By Proposition \ref{prop:bindinjective}, we can assume $M'=M$ and $\widehat{\OO}_{M'}=\widehat{\OO}_M$ possibly after $G$-conjugation. But $\widehat{\OO}_M$ is maximal, therefore $p_{M'}$ is an isomorphism and so is $p$.
\end{proof}

\section{Main results}\label{sec:mainresults}

Let $G$ be a complex connected reductive algebraic group with Langlands dual group $G^{\vee}$. 

\begin{definition}\label{def:Lusztigcover}
For any $\OO \in \Orb(G)$, the \emph{Lusztig cover} of $\OO$ is the finite $G$-equivariant cover $\widetilde{\OO} \to \OO$ corresponding to the kernel of the map $A(\OO) \twoheadrightarrow \bar{A}(\OO)$.
\end{definition}

\begin{rmk}\label{rmk:Lusztigcoverisogeny}
The Lusztig cover is independent of isogeny in the following sense. Suppose $G' \to G$ is a covering group and let $\OO \in \Orb(G)$. Write $\widetilde{\OO}_{Lus}$ (resp. $\widetilde{\OO}_{Lus}'$) for the Lusztig cover of $\OO$ associated to $G$ (resp. $G'$). Since $\bar{A}(\OO)$ is independent of isogeny, the action of $G'$ on $\widetilde{\OO}'_{Lus}$ descends to an action of $G$, and $\widetilde{\OO}'_{Lus} \simeq \widetilde{\OO}_{Lus}$ as $G$-equivariant covers of $\OO$. 
\end{rmk}

\begin{prop}\label{prop:distinguishedbirigid}
Suppose $(\OO^{\vee},\bar{C}) \in \LA(G^{\vee})$ is a special distinguished Lusztig-Achar datum and let $\OO=d_S(\OO^{\vee},\bar{C})$. Then
\begin{itemize}
    \item[(i)] $\widetilde{\OO}_{Lus}$ is birationally rigid. 
\end{itemize}
Furthermore
\begin{itemize}
    \item[(ii)] The restriction of $d_S$ to the set of special distinguished Lusztig-Achar data is injective.
\end{itemize}
\end{prop}

\begin{proof}
If $G$ is a simple classical group, then (i) and (ii) are proved in Section \ref{subsec:proofsclassical1}. If $G$ is a simple adjoint group of exceptional type, then (i) and (ii) are proved in Section \ref{subsec:proofsexceptional1}. To reduce to these cases, it suffices to prove the following:
\begin{itemize}
    \item[(a)] If the assertions hold for $G_1$ and $G_2$, they also hold for the product $G_1 \times G_2$.
    \item[(b)] Suppose $G' \to G$ is a covering group. Then the assertions hold for $G$ if and only if they hold for $G'$.
\end{itemize}

For assertion (ii), both (a) and (b) are obvious. Indeed, the sets $\LA^*(G^{\vee})$, $\Orb(G)$ and the map $d_S: \LA^*(G^{\vee}) \to \Orb(G)$ are independent of isogeny and $d_S^G=d_S^{G_1} \times d_S^{G_2}$. For (i), (a) is clear (for essentially the same reasons). For (b), we argue as follows. By Remark \ref{rmk:Lusztigcoverisogeny}, there is an isomorphism of algebraic varieties $\widetilde{\OO}_{Lus} \simeq \widetilde{\OO}_{Lus}'$. Now (b) follows at once from Proposition \ref{prop:criterionbirigidcover}.
\end{proof}

\begin{rmk}\label{rmk:nonspecial}
We note that Proposition \ref{prop:distinguishedbirigid} is not true if we drop the assumption that $(\OO^{\vee},\bar{C}) \in \LA(G^{\vee})$ is special. For example, let $G$ be the simple adjoint group of type $E_7$ and let $\OO^{\vee}=A_4+A_1$. Then $\bar{A}(\OO^{\vee})=S_2$. Let $\bar{C}$ denote the nontrivial conjugacy class in $\bar{A}(\OO^{\vee})$. Then $(\OO^{\vee},\bar{C})$ is distinguished, but not special. Note that $\OO=d_S(\OO^{\vee},\bar{C})=A_3+A_2+A_1$ and $A(\OO)=1$. So $\widetilde{\OO}_{univ}=\OO$. But $\OO$ is not birationally rigid, see \cite[Proposition 3.8.3]{MBMat} (notably, if we replace $G$ with its simply connected form, then $A(\OO)=\ZZ_2$ and hence $\widetilde{\OO}_{univ}$ is a 2-fold cover of $\OO$. By \cite[Proposition 3.9.5]{MBMat}, this cover is birationally rigid).
\end{rmk}

We are now prepared to define our duality map $D: \LA^*(G^{\vee}) \to \Cov(G)$. First, consider the map
$$D_0: \LA^*_0(G^{\vee}) \to \Cov_0(G), \qquad D_0(L^{\vee},(\OO_{L^{\vee}},\bar{C}_{L^{\vee}})) = (L,d_S^L(\OO_{L^{\vee}},\bar{C}_{L^{\vee}})_{Lus})$$
This is well-defined by Proposition \ref{prop:distinguishedbirigid}(ii). We define $D$ to be the composition
$$D: \LA^*(G^{\vee}) \overset{\Sat^{-1}}{\to} \LA_0^*(G^{\vee}) \overset{D_0}{\to} \Cov_0(G) \overset{\Bind}{\to} \Cov(G)$$
where $\Sat^{-1}$ is the map of Proposition \ref{prop:specialtospecial}(ii). We will sometimes write $D^G$ to indicate the dependence on $G$.

\begin{prop}\label{prop:propsofD}
The map 
$$D: \LA^*(G^{\vee}) \to \Cov(G)$$
has the following properties:
\begin{itemize}
    \item[(i)] $D$ is injective.
    \item[(ii)] If $L \subset G$ is a Levi subgroup, then
    $$D^G \circ \Sat^{G^{\vee}}_{L^{\vee}} = \Bind^G_L \circ D^L $$
    \item[(iii)] $(\OO^{\vee},\bar{C})$ is distinguished if and only if $D(\OO^{\vee},\bar{C})$ is birationally rigid.

    \item[(iv)] $D$ is independent of isogeny in the following sense: if $\widetilde{G} \to G$ is a covering group, then the following diagram commutes 
    \begin{center}
        \begin{tikzcd}
            \LA^*(\widetilde{G}^{\vee}) \ar[r,"D^{\widetilde{G}}"] & \Cov(\widetilde{G}) \\
            \LA^*(G^{\vee}) \ar[r,"D^G"] \ar[u,equals]& \Cov(G) \ar[u]
        \end{tikzcd}
    \end{center}
\end{itemize}
\end{prop}

\begin{proof}
For (i), we note that $D$ is the composition of three injective maps: $\Sat^{-1}: \LA^*(G^{\vee}) \to \LA^*_0(G^{\vee})$ is injective by Proposition \ref{prop:specialtospecial}(ii), $D_0$ is injective by Proposition \ref{prop:distinguishedbirigid}(ii), and $\Bind: \Cov_0(G) \to \Cov(G)$ is injective by Proposition \ref{prop:bindinjective}. Hence, $D$ is injective, proving (i). 

For (ii), let $L \subset G$ be a Levi subgroup of $G$ and let $(\OO_{L^{\vee}},\bar{C}_{L^{\vee}}) \in \LA^*(L^{\vee})$. Suppose $\Sat(K^{\vee},(\OO_{K^{\vee}},\bar{C}_{K^{\vee}})) = (\OO_{L^{\vee}},\bar{C}_{L^{\vee}})$ for $(K^{\vee},(\OO_{K^{\vee}},\bar{C}_{K^{\vee}})) \in \LA^*_0(L^{\vee})$ and let $\widetilde{\OO}_K = d_S^K(\OO_{K^{\vee}},\bar{C}_{K^{\vee}})_{univ}$. Then
$$D^G(\Sat^{G^{\vee}}_{L^{\vee}}(\OO_{L^{\vee}},\bar{C}_{L^\vee}) = \Bind^G_K\widetilde{\OO}_L$$
whereas
$$\Bind^G_L(D^L(\OO_{L^{\vee}},\bar{C}_{L^\vee})) = \Bind^G_L \Bind^L_K \widetilde{\OO}_K.$$
Now (ii) is immediate from the transitivity of birational induction, see Proposition \ref{prop:propsofbind}(ii).

(iii) follows from (ii) together with Proposition \ref{prop:distinguishedbirigid}(i). (iv) is a consequence of Remark \ref{rmk:Lusztigcoverisogeny}.
\end{proof}

\begin{rmk}
We note that our map $D$ is not in general surjective (nor is its composition with $\Cov(G) \to \Cov(G)/\sim$). Indeed, let $G$ be the (unique) simple group of type $F_4$, and let $\OO = A_2$. There is a $G$-equivariant double cover $\widetilde{\OO}$  of $\OO$, which is birationally rigid by \cite[Proposition 3.9.5]{MBMat}. Since $A(\OO)\simeq \ZZ_2$ and $\OO$ is birationally induced, it is clear that $\widetilde{\OO}$ is the unique element of its equivalence class. If $\widetilde{\OO}$ were to belong to the image of $D$, then by Proposition \ref{prop:propsofD}(iii) it would have to be the case that $\widetilde{\OO} = D(\OO^{\vee},\bar{C})$ for some distinguished special Lustig-Achar datum $(\OO^{\vee},\bar{C})$. This would imply that $\OO = d_S(\OO^{\vee},\bar{C})$. Examining Table \ref{table:F4}, we see that no such distinguished special $(\OO^{\vee},\bar{C})$ exists. 
\end{rmk}

\begin{rmk}
We note that our duality map $D$ generalizes the duality maps of Barbasch-Vogan-Lusztig-Spaltenstein (denoted $d$), Sommers (denoted $d_S$), Losev-Mason-Brown-Matvieievskyi (denoted $\tilde{d}$), and Achar (denoted $d_A$) in the following sense: if $(\OO^{\vee},\bar{C}) \in \LA^*(G^{\vee})$, then
\begin{itemize}
    \item[(i)] $D(\OO^{\vee},\bar{C})$ is a cover of $d_S(\OO^{\vee},\bar{C})$.
    \item[(ii)] If $\bar{C}=1$, then $D(\OO^{\vee},\bar{C})$ belongs to the equivalence class $\tilde{d}(\OO^{\vee})$.
    \item[(iii)]  If $\bar{C}=1$, then $D(\OO^{\vee},\bar{C})$ is a cover of $d(\OO^{\vee})$.
    \item[(iv)] Let $d_A(\OO^{\vee},\bar{C}) = (\OO,\bar{C}')$. By \cite{Lusztig1997} and \cite[Section 7.1]{Achar2003}, $\bar{A}(\OO)$ is a Coxeter group and admits a Coxeter presentation unique up to conjugacy. We can then associate a parabolic subgroup $H_{\bar{C}'} \subset \bar{A}(\OO)$ to the conjugacy class $\bar{C}'$ up to conjugacy. Let $H \subset A(\OO)$ be preimage of $H_{\bar{C}'}$ under the quotient map $A(\OO) \twoheadrightarrow \bar{A}(\OO)$. Then $H$ determines a $G$-equivariant cover $\widetilde{\OO}$ of $\OO$. One can check that $\widetilde{\OO}$ is equivalent to $D(\OO^{\vee},\bar{C})$ in the sense of Definition \ref{defn:cover_equivalence}. This will be proved in a future paper.
\end{itemize}
\end{rmk}

For $\gamma \in \fh^*$, define
$$s = \exp(2\pi i \gamma), \qquad L^{\vee}_{\gamma,0} = Z_{G^{\vee}}(\gamma), \qquad L^{\vee}_{\gamma} = Z_{G^{\vee}}(s)^{\circ}.$$
Note that $L^{\vee}_{\gamma}$ is a pseudo-Levi subgroup of $G^{\vee}$ and $L^{\vee}_{\gamma,0}$ is a Levi subgroup of $L^{\vee}_{\gamma}$. Consider the McNinch-Sommers datum
\begin{equation}\label{eq:MS_map}
\MS(\gamma) := (L^{\vee}_{\gamma}, sZ^{\circ},\Ind^{L^{\vee}_{\gamma}}_{L^{\vee}_{\gamma,0}}\{0\}) \in \MS(G^{\vee})
\end{equation}
This defines a map
$$\MS: \fh^* \to \MS(G^{\vee})$$
Composing with the projection $\MS(G^{\vee}) \overset{\pi}{\to} \Conj(G^{\vee}) \to \LA(G^{\vee})$, we get a further map
$$\LA: \fh^* \to \LA(G^{\vee}).$$
We will sometimes write $\MS^{G^{\vee}}: \fh^* \to \MS(G^{\vee})$ and $\LA^{G^{\vee}}: \fh^* \to \LA(G^{\vee})$ to indicate the dependence on $G^{\vee}$.

Let $\fh_{\RR}^* \subset \fh^*$ denote the real form of $\fh^*$ spanned by the roots in $G$. To any Lusztig-Achar datum $(\OO^{\vee},\bar{C}) \in \LA(G^{\vee})$, we attach a $W$-invariant subset $S(\OO^{\vee},\bar{C}) \subset \fh_{\RR}^*$ as follows. First, choose a Levi subgroup $L^{\vee} \supset H^{\vee}$ and a distinguished Lusztig-Achar datum $(\OO_{L^{\vee}},\bar{C}_{L^{\vee}}) \in \LA(L^{\vee})$ such that $(\OO^{\vee},\bar{C}) = \Sat^{G^{\vee}}_{L^{\vee}} (\OO_{L^{\vee}},\bar{C}_{L^{\vee}})$. By Proposition \ref{prop:uniquedistinguishedAbar}, the pair $(\OO_{L^{\vee}},\bar{C}_{L^{\vee}})$ is unique up to conjugation by $L^\vee$. Define
$$S(\OO^{\vee},\bar{C}) := W \cdot \left((\LA^{L^{\vee}})^{-1}(\OO_{L^{\vee}},\bar{C}_{L^{\vee}}) \cap \fh_{\RR}^*\right)$$ 
This is a $W$-invariant subset of $\fh_{\RR}^*$, independent of the choice of $(L^{\vee},(\OO_{L^{\vee}},\bar{C}_{L^{\vee}}))$. Since it is $W$-invariant, we can (and often will) regard $S(\OO^{\vee},\bar{C})$ as a subset of $\fh^*/W$.

Now choose a $W$-invariant non-degenerate symmetric form on $\fh^*$ and write $\lVert \cdot \rVert$ for the associated norm. 

\begin{theorem}\label{thm:inflchars}
Let $(\OO^{\vee},\bar{C}) \in \LA^*(G^{\vee})$. Then there is a unique minimal-length $W$-orbit
$$\gamma(\OO^{\vee},\bar{C}) \in S(\OO^{\vee},\bar{C})$$
Furthermore,
$$\gamma(\OO^{\vee},\bar{C}) = \gamma(D(\OO^{\vee},\bar{C})).$$
\end{theorem}

\begin{proof}
Choose $L^{\vee} \subset G^{\vee}$ and a distinguished Lusztig-Achar datum $(\OO_{L^{\vee}},\bar{C}_{L^{\vee}})$ such that $(\OO^{\vee},\bar{C}) = \Sat^{G^{\vee}}_{L^{\vee}} (\OO_{L^{\vee}},\bar{C}_{L^{\vee}})$. Then by definition $S(\OO^{\vee},\bar{C}) = W \cdot S(\OO_{L^{\vee}},\bar{C}_{L^{\vee}})$. So if $\gamma_L$ is a minimal-length $W_L$-orbit in $S(\OO_{L^{\vee}},\bar{C}_{L^{\vee}})$, then $W\cdot \gamma_L$ is a minimal-length $W$-orbit in $S(\OO^{\vee},\bar{C})$. On the other hand, Proposition \ref{prop:propsofD}(iv) implies
$$D^G(\OO^{\vee},\bar{C}) = \Bind^G_L(D^L(\OO_{L^{\vee}},\bar{C}_{L^{\vee}}))$$
So by Proposition \ref{prop:unipotentfacts}(iii)
$$\gamma(D^G(\OO^{\vee},\bar{C})) = \gamma(D^L(\OO_{L^{\vee}},\bar{C}_{L^{\vee}}))$$
as $W$-orbits in $\fh^*$. So if $\gamma_L = \gamma(D^L(\OO_{L^{\vee}},\bar{C}_{L^{\vee}}))$, then $\gamma(\OO^{\vee},\bar{C}) = \gamma(D^G(\OO^{\vee},\bar{C}))$. Thus, we can reduce to the case when $(\OO^{\vee},\bar{C})$ is distinguished. Arguing as in the proof of Proposition \ref{prop:distinguishedbirigid}, we can further reduce to the case when $G$ is simple and adjoint. Now the classical cases are handled in Section \ref{subsec:proofsclassical2}. The exceptional cases are handled in Section \ref{subsec:proofsexceptional1}.
\end{proof}

\begin{rmk}
    The statement of \cref{thm:inflchars} is inspired by the discussion in \cite[Section 11]{Barbasch1989}. In \emph{loc. cit.}, Barbasch considers some infinitesimal characters in classical types which are defined by a minimality property similar to the one defining $\gamma(\OO^{\vee},\bar{C})$. The novel observation in \cref{thm:inflchars} is that $\gamma(\OO^{\vee},\bar{C})$ coincides with the infinitesimal character associated to the canonical quantization of the dual cover $D(\OO^{\vee},\bar{C})$. 

    We also remark that our minimality result is similar in spirit to some conjectures posed by Sommers and Gunnells in \cite[Section 5]{SommersGunnells}. 
\end{rmk}

\begin{rmk}
Let $(\OO^{\vee},\bar{C}) \in \LA(G^{\vee})$. Even if $(\OO^{\vee},\bar{C})$ fails to be special, the set $S(\OO^{\vee},\bar{C})$ may still contain a unique minimal-length $W$-orbit. However, this $W$-orbit will typically \emph{not} correspond to the infinitesimal character of a unipotent ideal. For example, take $G$ and $(\OO^{\vee},\bar{C}) \in \LA(G^{\vee})$ as in Remark \ref{rmk:nonspecial}. In this case, the set $S(\OO^{\vee},\bar{C})$ contains a unique minimal-length $W$-orbit, namely the $W$-orbit of the weight $\gamma=\rho/4+\varpi_8/4$ (here $\rho$ is the half-sum of the positive roots and $\varpi_8$ is the fundamental weight corresponding to the extremal node on the longest leg of the Dynkin diagram, see Remark \ref{rmk:algorithmgamma} for a description of the method used to compute this minimum). We note that $\gamma$ is not the infinitesimal character of a unipotent ideal. Moreover, an atlas computation shows that the spherical irreducible Harish-Chandra bimodule with left and right infinitesimal character $\gamma$ is not unitary. 
\end{rmk}

Now let $\widetilde{\OO} = D(\OO^{\vee},\bar{C})$ and consider the category 
$\HC^G(U(\fg)/I(\widetilde{\OO})$ of unipotent bimodules. Recall from Theorem \ref{thm:classificationHC} that this category is equivalent to finite-dimensional representations of a certain finite group $\Gamma(\widetilde{\OO})$, see (\ref{eq:defofGamma}). Our next task is to describe this finite group in terms of $(\OO^{\vee},\bar{C})$. We will need the following proposition.

\begin{prop}\label{prop:twoORvs}
Let $(\OO^{\vee},\bar{C}) \in \LA^*(G^{\vee})$. Choose $(L^{\vee},(\OO_{L^{\vee}},\bar{C}_{L^{\vee}})) \in \LA^*_0(G^{\vee})$ such that $(\OO^{\vee},\bar{C}) = \Sat(L^{\vee},(\OO_{L^{\vee}},\bar{C}_{L^{\vee}}))$ and let
$$\gamma = \gamma^{L^{\vee}}(\OO^{\vee},\bar{C}), \quad (R_0^{\vee},sZ(R_0^{\vee})^{\circ},\OO_{R_0}^{\vee}) = \MS^{L^{\vee}}(\gamma), \quad (R^{\vee},sZ(R^{\vee})^{\circ},\OO_{R^\vee}) = \MS^{G^{\vee}}(\gamma).$$
Then
$$\OO_{R^{\vee}} = \Sat^{R^{\vee}}_{R_0^{\vee}} \OO_{R_0^{\vee}}.$$
\end{prop}

\begin{proof}
Arguing as in the proof of Proposition \ref{prop:distinguishedbirigid}, we can reduce to the case when $G$ is a simple group of adjoint type. For $G$ classical, see Section \ref{subsec:twoORvsclassical}. For $G$ exceptional, see Section \ref{subsec:proofsexceptional2}.
\end{proof}

Now consider the composition
\begin{equation}\label{eq:defofL}\mathbb{L}: \LA^*(G^{\vee}) \overset{\gamma}{\to} \fh^*/W \overset{\MS}{\to} \MS(G^{\vee})\end{equation}
The following is immediate from Proposition \ref{prop:twoORvs} and Theorem \ref{thm:inflchars}.

\begin{lemma}\label{lem:lifting}
The map $\mathbb{L}: \LA^*(G^{\vee}) \to \MS(G^{\vee})$ is right-inverse to $\pi: \MS(G^{\vee}) \to \LA^*(G^{\vee})$.
\end{lemma}


\begin{rmk}
We note that \cref{prop:twoORvs} is false for non-special pairs $(\OO^\vee, \Bar{C})$. Indeed, let $G^{\vee}=SO(17)$, let $\OO^{\vee}=\OO_{[5,4^2,3,1]}$, and let $\bar{C}$ be the unique non-trivial conjugacy class in $\bar{A}(\OO^\vee)\simeq \ZZ_2$. Then $\gamma(\OO^\vee, \Bar{C})=(\frac{5}{2}, \frac{3}{2}, \frac{3}{2}, \frac{3}{2}, \frac{1}{2}, \frac{1}{2}, \frac{1}{2}, \frac{1}{2})$. It follows that  $R^{\vee} = SO(16)$, and $\OO_{R^\vee}=\Ind_{GL(1)\times GL(3)\times GL(4)}^{SO(16)} \{0\}=\OO_{[5,5,3,3]}$. On the other hand, $R_0^\vee = GL(4) \times SO(8)$, and $\OO_{R_0^\vee}=\OO_{[4]}\times \OO_{[5,3]}$. Thus, $\Sat^{R^{\vee}}_{R_0^{\vee}} \OO_{R_0^{\vee}}=\OO_{[5,4,4,3]}$. We remark that in this case $\pi(R^\vee, sZ(R^\vee)^{\circ}, \Sat^{R^{\vee}}_{R_0^{\vee}} \OO_{R_0^{\vee}})\neq (\OO^\vee, \Bar{C})$, so Lemma \ref{lem:lifting} is false as well.
\end{rmk}

\begin{theorem}\label{thm:Gamma}
Assume $G$ is adjoint. Let $(\OO^{\vee},\bar{C}) \in \LA^*(G^{\vee})$. Let 
 $\mathbb{L}(\OO^{\vee},\bar{C}) = (R^{\vee},sZ^{\circ},\OO_{R^{\vee}})$ and $\widetilde{\OO} = D(\OO^{\vee},\bar{C})$. Then there is a group isomorphism
$$\bar{A}(\OO_{R^{\vee}}) \simeq \Gamma(\widetilde{\OO}).$$
\end{theorem}

\begin{proof}
Arguing as in the proof of Proposition \ref{prop:distinguishedbirigid}, we can reduce to the case when $G$ is a simple group of adjoint type. For $G$ classical, see Section \ref{sec:proofsclassical3}. For $G$ exceptional, see Section \ref{subsec:proofsexceptional2}. 
\end{proof}

\begin{rmk}\label{rmk:even}
We note that there is a case-free proof of Theorem \ref{thm:Gamma} in the following special case: $\OO^{\vee}$ is even and $\bar{C}=1$. Under these assumptions, $(R^{\vee},\OO_{R^{\vee}}) = (G^{\vee},\OO^{\vee})$ and $I(\widetilde{\OO}) = J(\gamma_{\OO^{\vee}})$, cf. \cite[Proposition 9.2.1]{LMBM}. So it suffices to show that $\bar{A}(\OO^{\vee}) \simeq \Gamma(D(\OO^{\vee},\bar{C}))$. Let $\OO=d(\OO^{\vee})$. By \cite[Proposition 7.4 and Theorem 6.1]{LosevHC}, there is a monoidal equivalence
$$\HC^G(U(\fg)/J(\gamma_{\OO^{\vee}})) \simeq \bar{A}(\OO)\modd$$
And by \cite[Theorem 6.6.2]{LMBM}, there is a monoidal equivalence
$$\HC^G(U(\fg)/J(\gamma_{\OO^{\vee}})) \simeq \Gamma(D(\OO^{\vee},\bar{C}))\modd$$
So by the Tannakian formalism, there is a group isomorphism $\bar{A}(\OO) \simeq \Gamma(D(\OO^{\vee},\bar{C}))$. Since $\OO^{\vee}$ is even, and hence special, there is a further isomorphism $\bar{A}(\OO^{\vee}) \simeq \bar{A}(\OO)$, see \cite{Lusztig1984}. Thus, $\bar{A}(\OO^{\vee}) \simeq \Gamma(D(\OO^{\vee},\bar{C}))$, as asserted.
\end{rmk}

Combining Theorem \ref{thm:Gamma} with Theorem \ref{thm:classificationHC}, we arrive at the following result.

\begin{cor}\label{cor:HC}
Assume $G$ is adjoint. Let $(\OO^{\vee},\bar{C}) \in \LA^*(G^{\vee})$. Let  $\mathbb{L}(\OO^{\vee},\bar{C}) = (R^{\vee},sZ^{\circ},\OO_{R^{\vee}})$ and $\widetilde{\OO} = D(\OO^{\vee},\bar{C})$. Then there is a monoidal quivalence of categories
$$\HC^G(U(\fg)/I(\widetilde{\OO})) \simeq \bar{A}(\OO_{R^{\vee}})\modd $$
\end{cor}

\begin{rmk}
Corollary \ref{cor:HC} implies that the unipotent representations attached to $\widetilde{\OO}$ are in one-to-one correspondence with irreducible representations of $\bar{A}(\OO_{R^{\vee}})$. In the special case when $\OO^{\vee}$ is even and $\bar{C}=1$, this (weaker) statement was proved in \cite[Theorem III]{BarbaschVogan1985}. Their result was later extended to the case when $\OO^{\vee}$ is special in \cite[Appendix A]{Wong2023}.
\end{rmk}

\section{Combinatorics in classical types}\label{sec:combinatoricsclassical}

A \emph{partition} of $n \in \ZZ_{\geq 0}$ is a non-increasing sequence of positive integers $\lambda = [\lambda_1, \lambda_2, ..., \lambda_k]$ such that $n=\sum_i \lambda_i$. Write $\#\lambda$ for $k$ and $|\lambda|$ for $n$. For a positive integer $x$, we write $m_{\lambda}(x)$ for the multiplicity of $x$ in $\lambda$ and $\mathrm{ht}_{\lambda}(x)$ for the \emph{height} of $x$ in $\lambda$, i.e. $\height_{\lambda}(x) = \sum_{y \geq x} m_{\lambda}(y)$. Note that $\height_{\lambda}(x)$ makes sense even if $x$ is not a part of $\lambda$. For notational convenience, we will often abbreviate partitions by writing multiplicities as exponents, i.e. $[4^2,3,1^3]$ denotes the partition $[4,4,3,1,1,1]$ of $14$.

If $\lambda \in \mathcal{P}(m)$ and $\mu \in \mathcal{P}(n)$, we write $\lambda \cup \mu \in \mathcal{P}(m+n)$ for the partition obtained by adding multiplicities and $\lambda \vee \mu \in \mathcal{P}(m+n)$ for the partition obtained by adding corresponding parts. For example, if $\lambda=[4^2,3,1^3]$ and $\mu = [5,1]$, then $\lambda \cup \mu = [5,4^2,3,1^4]$ and $\lambda \vee \mu = [9,5,3,1^3]$. The transpose of a partition $\lambda$ is denoted by $\lambda^t$. 

A partition of $2n+1$ is of \emph{type} $B$ if every even part occurs with even multiplicity. A partition of $2n$ is of \emph{type} $C$ if every odd part occurs with even multiplicity. A partition of $2n$ is of \emph{type} $D$ if every even part occurs with even multiplicity. Write $\mathcal{P}(n)$ for the set of partitions of $n$, and write $\mathcal{P}_B(2n+1) \subset \mathcal{P}(2n+1)$, $\mathcal{P}_C(2n) \subset \mathcal{P}(2n)$ and $\mathcal{P}_D(2n) \subset \mathcal{P}(2n)$ for the subsets of partitions of the corresponding types. We will sometimes write $\cP_1(m)$ for $\cP_C(m)$ and $\cP_0(m)$ for either $\cP_B(m)$ (when $m$ is odd) or $\cP_D(m)$ (when $m$ is even) in order to make concise statement about $\cP_\epsilon(m)$ for $\epsilon \in \{0, 1\}$. A partition of type $D$ is \emph{very even} if all parts are even, occuring with even multiplicity.

There is a partial order on $\mathcal{P}(n)$ defined in the following way:
$$\lambda \geq \mu \iff \sum_{i \leq j} \lambda_i \geq \sum_{i \leq j} \mu_i, \qquad \forall i$$
The \emph{B-collapse} of $\lambda \in \mathcal{P}(2n+1)$ (resp. \emph{C-collapse} of $\lambda \in \mathcal{P}(2n)$, \emph{D-collapse} of $\lambda \in \mathcal{P}(2n)$) is the (unique) largest partition $\lambda_B \in \mathcal{P}_B(2n+1)$ (resp. $\lambda_C \in \mathcal{P}_C(2n)$, $\lambda_D \in \mathcal{P}_D(2n)$) which is dominated by $\lambda$.

For a partition $\lambda = [\lambda_1, \lambda_2, \ldots , \lambda_k]$, we will sometimes write $\lambda=(c_l \geq c_{l-1} \geq \ldots \geq c_1)$ where $c_1, \ldots, c_l$ are the column lengths of the corresponding Young diagram, i.e. the members of the transpose partition $\lambda^t$ of $\lambda$.

It is possible to describe the sets $\cP_\epsilon(m)$ in terms of columns (see \cite[Prop. 2.3]{WongLVB}):
\begin{itemize}
    \item 
      $\cP_0(m)$ consists of partitions $\lambda = (a_{2k+1},  a_{2k}, \cdots, a_0)$ of size $m$ such that $a_{2i} + a_{2i-1}$ is even for all $i$ (we insist that there are {\bf even} number of columns, by taking $a_0=0$ if necessary).
    \item 
      $\cP_1(m)$ consists of partitions $\lambda = (a_{2k},  a_{2k-1}, \cdots, a_0)$ of size $m$ such that $a_{2i} + a_{2i-1}$ is even for all $i$ (we insist that there are {\bf odd} number of columns, by taking $a_0=0$ if necessary).
\end{itemize}

\subsection{Nilpotent orbits}\label{subsec:classical_orbit}

\begin{prop}[Section 5.1, \cite{CM}]\label{prop:orbitstopartitions}
Suppose $G$ is a simple group of classical type. Then the following are true:
\begin{enumerate}[label=(\alph*)]
    \item If $\fg = \mathfrak{sl}(n)$, then there is a bijection
    $$\Orb(G) \xrightarrow{\sim} \mathcal{P}(n) $$
    \item If $\fg = \mathfrak{so}(2n+1)$, then there is a bijection
    $$\Orb(G)\xrightarrow{\sim} \mathcal{P}_{B}(2n+1) $$
    \item If $\fg = \mathfrak{sp}(2n)$, then there is a bijection
    $$ \Orb(G)\xrightarrow{\sim} \mathcal{P}_C(2n)$$
    \item If $\fg = \mathfrak{so}(2n)$, then there is a surjection
    $$\Orb(G) \twoheadrightarrow \mathcal{P}_{D}(2n) $$
    Over very even partitions, this map is two-to-one, and over all other partitions, it is a bijection.
\end{enumerate}
\end{prop}

\begin{convention}
In the setting of the above proposition, we write $\OO_{\lambda} \in \Orb(G)$ for the orbit corresponding to a partition $\lambda$, except in the case when $\lambda$ is very even and $\fg$ is of type $D$. In such cases, there are two orbits corresponding to $\lambda$, which we denote by $\OO_{\lambda}^I$ and $\OO_{\lambda}^{II}$. For convenience, we sometimes allow very even partitions to be decorated by Roman numerals $I$ or $II$ to redefine $\mathcal{P}_{D}(2n)$, so that the map $\Orb(G) \twoheadrightarrow \mathcal{P}_{D}(2n)$ in Proposition \ref{prop:orbitstopartitions}(d) becomes a bijection.  If there is a property $P$ which $\OO_{\lambda}^I$ and $\OO_{\lambda}^{II}$ have in common, we will often say, to simplify the exposition, that $\OO_{\lambda}$ has property $P$. 
\end{convention}

\subsection{The group $A(\OO)$}\label{subsec:A_classical}

In this section, we will describe the component groups $A(\OO)$ for nilpotent orbits of classical groups. We will write $\fg_0(m) = \mathfrak{so}(m)$ (for arbitrary $m$) and $\fg_1(m)=\mathfrak{sp}(m)$ (for even $m$) in order to make concise statements about $\fg_\epsilon(m)$ for $\epsilon \in \{0,1\}$. Let $G_{\epsilon}(m)$ denote the simple classical group corresponding to $\fg_{\epsilon}(m)$ and let $G_{\epsilon}^{ad}(m)$ denote the adjoint quotient of $G_{\epsilon}(m)$. For $\OO_{\lambda} \in \Orb(G_{\epsilon}(m))$, we write $A(\OO_{\lambda})$ (resp. $A^{ad}(\OO_{\lambda})$) for the component group of $\OO_{\lambda}$ with respect to $G_{\epsilon}(m)$ (resp. $G_{\epsilon}^{ad}(m)$).

For $\lambda = [\lambda_1,...,\lambda_k] \in \cP_\epsilon(m)$, let $\lambda^\epsilon = [\lambda^\epsilon_1 \ge \lambda^\epsilon_2 \ge \cdots \ge \lambda^\epsilon_r > 0]$ be the subpartition consisting of the parts $\lambda^\epsilon_j$ of $\lambda$ such that $\lambda^\epsilon_j \not\equiv \epsilon \mod 2$. We introduce the natural convention that $\lambda^\epsilon_k = 0$ for any $k > r$. To each member $\lambda^\epsilon_i$ of $\lambda^\epsilon$, we assign a symbol $\upsilon_{\lambda^\epsilon_i}$ and write $A \simeq (\ZZ_2)^r$ for the elementary abelian 2-group with basis $\{\upsilon_{\lambda^\epsilon_i}\}$ (if $\lambda^\epsilon_i = \lambda^\epsilon_j$, then $\upsilon_{\lambda^\epsilon_i} = \upsilon_{\lambda^\epsilon_j}$). Let $A^0=A^0(\OO_\lambda)$ be the subgroup of $A$ consisting of elements which are products of an even number of $\upsilon_{\lambda^\epsilon_i}$'s, and let $A^1=A^1(\OO_\lambda) := A$. 

As in \cite{Sommers1998}, for $\lambda \in \cP_\epsilon(m)$ and $\delta \in \{0,1\}$, we set
$$ \mathcal{S}_\delta(\lambda) = \{ x \,|\, x \not\equiv \epsilon \textrm{ mod } 2 \text{ and } m_\lambda(x) \equiv \delta \textrm{ mod } 2 \}. $$
Define the element $\tilde{\upsilon} := \upsilon_{\lambda^\epsilon_1} \upsilon_{\lambda^\epsilon_2} \cdots \upsilon_{\lambda^\epsilon_r}$ in $A$. Then $\tilde{\upsilon}$ is of order $2$ if $\mathcal{S}_1 \neq \emptyset$, and is the identity element otherwise.
For $\lambda = [\lambda_1,...,\lambda_k] \in \mathcal{P}(m)$, write $a$ (resp. $b$) for the number of distinct odd (resp. even) parts of $\lambda$. 

The following can be deduced from the proof of Theorem 6.1.3 and Corollary 6.1.6 of \cite{CM}.

\begin{prop}\label{prop:Aclassical}
Let $G=G_{\epsilon}(m)$ and let $\OO_{\lambda} \in \Orb(G)$. Then the following are true:

\begin{itemize}
    \item[(i)] If $\fg=\mathfrak{sl}(n)$, then $A(\OO_{\lambda}) \simeq \ZZ_d$, where $d = \gcd{\{\lambda_i\}}$, and $A^{ad}(\OO_{\lambda})=1$.
    \item[(ii)] If $\fg = \mathfrak{so}(2n+1)$, then $A(\OO_{\lambda}) \simeq A^{ad}(\OO_{\lambda}) \simeq A^{0} \simeq (\ZZ_2)^{a-1}$.
    \item[(iii)] If $\fg = \mathfrak{sp}(2n)$, then $A(\OO_{\lambda}) \simeq A^1$ and 
    $$A^{ad}(\OO_{\lambda}) \simeq A^1/\{1, \tilde{v} \} = 
    \begin{cases*}
        A^1 = A \simeq (\ZZ_2)^b & if  $\mathcal{S}_1 = \emptyset$\\
        A^1/\{1, \tilde{v} \} = A/\{1, \tilde{v} \} \simeq (\ZZ_2)^{b-1} & if $\mathcal{S}_1 \neq \emptyset$
    \end{cases*}$$
    \item[(iv)] If $\fg = \mathfrak{so}(2n)$, then $A(\OO_{\lambda}) \simeq A^0$, and 
     $$A^{ad}(\OO_{\lambda}) \simeq A^0/\{1, \tilde{v} \} =
     \begin{cases*}
        A^0 \simeq (\ZZ_2)^{\mathrm{max}(a-1,0)} & if $\mathcal{S}_1 = \emptyset$\\
        A^0/\{1, \tilde{v} \} \simeq (\ZZ_2)^{\mathrm{max}(a-2,0)} & if $\mathcal{S}_1 \neq \emptyset$
    \end{cases*}$$
\end{itemize}
Parts (ii)-(iv) can be summarized as follows: $A(\OO_\lambda) \simeq A^{\epsilon}$ and
$A^{ad}(\OO_{\lambda}) \simeq A^\epsilon / (A^\epsilon \cap \{1, \tilde{v} \})$.
\end{prop}

For any $C \in A(\OO_\lambda)$, we can choose a (not necessarily unique) lift of $C$ in $A^\epsilon \subset A$. By Proposition \ref{prop:Aclassical}, this lift admits a (unique) decomposition as a product of $\upsilon_{\lambda^\epsilon_i}$'s . Let $\nu$ be the subpartition of $\lambda^\epsilon \subset \lambda$ consists of those $\lambda^\epsilon_i$ such that $\upsilon_{\lambda^\epsilon_i}$ appears in this decomposition. Then one can think of $\nu$ as a `marking' of $\lambda$. This motivates the following definition introduced in \cite{Sommers1998}.
\begin{definition}
Let $X \in \{B,C,D\}$. Define $\tilde{\cP}_X(m)$ to be the set of pairs of partitions $(\nu, \eta)$ such that:
   \begin{enumerate}[label=(\roman*)]
       \item 
       $\nu \cup \eta \in \cP_X(m)$.
       \item
       Each part of $\nu$ is odd (resp. even) if $X \in \{B,D\}$ (resp. $X=C$) and has multiplicity $1$.
       \item
       If $X\in \{B,D\}$, we require $\#\nu$ to be even. If $X=C$, we can always assume $\#\nu$ is even, by adding a zero if necessary.
   \end{enumerate}
   We will write any element $(\nu, \eta) \in \tilde{\cP}_X(m)$ as $\prescript{\langle \nu \rangle}{}{\lambda}$, where $\lambda = \nu \cup \eta$, and we think of $\prescript{\langle \nu \rangle}{}{\lambda}$ as a partition $\lambda \in \cP_X(m)$ with parts in the subpartition $\nu$ ``marked". When $\lambda \in \cP_D(2n)$ is a very even partition, $\nu$ is automatically empty. We then follow the convention in Section \ref{subsec:classical_orbit} and allow decoration by a Roman numeral $I$ or $II$ on the marked partition $\prescript{\langle \emptyset \rangle}{}{\lambda} \in \tilde{\cP}_D(2n)$.

   Similar to $\cP_\epsilon(m)$, we also sometimes write $\tilde{\cP}_\epsilon(m)$ for the sets defined above, where $\epsilon \in \{ 0, 1\}$.
\end{definition}

As discussed above, the marking $\nu = [\nu_1, \ldots, \nu_p]$ of a marked partition $\prescript{\langle \nu \rangle}{}{\lambda} \in \tilde{\cP}_\epsilon(m)$ gives rise to an element $\tilde{C}_\nu = \upsilon_{\nu_1} \upsilon_{\nu_2} \cdots \upsilon_{\nu_p}$ in $A^\epsilon$ and so determines a conjugacy class $C_\nu$ in the quotient group $A(\OO_\lambda)$. Let $C'_\nu$ denote the image of $C_\nu$ in $A^{ad}(\OO_\lambda)$. The following corollaries are just reformulations of Proposition \ref{prop:Aclassical}.

\begin{cor} \label{cor:conj_classical}
    We have a bijection $\tilde{\cP}_\epsilon(m) \xrightarrow{\sim} \Conj(G_\epsilon(m))$ sending a marked partition $\prescript{\langle \nu \rangle}{}{\lambda} \in \tilde{\cP}_\epsilon(m)$ to the $(\OO_\lambda, C_\nu) \in \Conj(G_\epsilon(m))$. When $G=SO(2n)$ and $\OO_\lambda$ is very even, this bijection preserves the Roman numerals.
\end{cor}

\begin{cor}\label{cor:conj_classical_adjoint}
The surjective composite map $\tilde{\cP}_\epsilon(m) \xrightarrow{\sim} \Conj(G_\epsilon(m)) \twoheadrightarrow \Conj(G_{\epsilon}^{ad}(m))$, $\prescript{\langle \nu \rangle}{}{\lambda} \mapsto (\OO_\lambda, C'_\nu)$, is one-to-one over $(\OO_\lambda, C'_\nu)$ if $\fg$ is of type $B$, or $\fg$ is of type $C$ or $D$ and $\mathcal{S}_1(\lambda) = \emptyset$, and two-to-one otherwise. In the latter case, the preimage of $(\OO_\lambda, C'_\nu)$ is of the form $\{ \prescript{\langle \nu \rangle}{}{\lambda}, \prescript{\langle \lambda^\epsilon \backslash \nu \rangle}{}{\lambda} \}$. When the group is $SO(2n)/\{ \pm I_{2n} \}$, the map $\tilde{\cP}_\epsilon(2n) \twoheadrightarrow \Conj(SO(2n)/\{ \pm I_{2n} \} )$ sends any even partition $\lambda = \prescript{\langle \emptyset \rangle}{}{\lambda}$ decorated by a Roman numeral to $(\OO_\lambda, 1)$ with the same numeral.
\end{cor}

\begin{rmk}\label{rmk:A_column}
  Following \cite{WongLVB}, we can also describe the group $A(\OO_\lambda)$ in terms of the columns of the partition $\lambda$ (see the paragraph above \cref{subsec:classical_orbit}). Namely, for each $\lambda^\epsilon_i \in \lambda^\epsilon$, let $c^\epsilon_i := \height_\lambda(\lambda^\epsilon_i)$ be the height of $\lambda^\epsilon_i$ in $\lambda$. The integers $c^\epsilon_i$ are the lengths of certain columns of the Young diagram of $\lambda$, namely the members of $\lambda^t$ whose heights in $\lambda^t$ are of different parity than $\epsilon$. Note that we have a symmetry $\lambda_i^\epsilon = \height_{\lambda^t}(c^\epsilon_i)$. We write the basis element $\upsilon_{\lambda^\epsilon_i}$ of the group $A$ associated to $\lambda$ alternatively as $\vartheta_{c^\epsilon_i}$ and again follow \cref{prop:Aclassical}.
\end{rmk}

\subsection{Levi subalgebras}\label{subsection:Levisclassical}

Let $\fg$ be a simple Lie algebra of classical type, and let $\fh \subset \fg$ be a Cartan subalgebra. Choose standard coordinates on $\fh$ as in \cite[VI, \S\,4]{Bourbaki46} and denote the coordinate functions by $\{e_i\}$. Write $\Delta \subset \fh^*$ for the roots of $\fg$. Then we can choose simple roots $\Pi \subset \Delta$ as follows:
\begin{itemize}
\item 
If $\fg = \mathfrak{sl}(n)$, then
$$\Delta = \{ \pm (e_i - e_j) \,|\, 1 \leq i, j \leq n, i \neq j \}, \quad \Pi = \{ e_1 - e_2,   e_2 - e_3, \ldots, e_{n-1} - e_n \}.$$
\item 
If $\fg = \mathfrak{so}(2n+1)$, then 
$$\Delta = \{ \pm e_i \pm e_j, \pm e_i \,|\, 1 \leq i, j \leq n, i \neq j \}, \quad \Pi = \{ e_1 - e_2,   e_2 - e_3, \ldots, e_{n-1} - e_n, e_n \}$$
and the lowest root is $-e_1 - e_2$.
\item
If $\fg = \mathfrak{sp}(2n)$, then 
$$\Delta = \{ \pm e_i \pm e_j, \pm 2 e_i \,|\, 1 \leq i, j \leq n, i \neq j \}, \quad \Pi = \{  e_1 - e_2,   e_2 - e_3, \ldots, e_{n-1} - e_n, 2e_n \}$$
and the lowest root is $-2e_1$.
\item
If $\fg = \mathfrak{so}(2n)$, then 
$$\Delta = \{ \pm e_i \pm e_j \,|\, 1 \leq i, j \leq n, i \neq j \}, \quad \Pi = \{  e_1 - e_2,   e_2 - e_3, \ldots, e_{n-1} - e_n, e_{n-1} + e_n \}$$
and the lowest root is $-e_1-e_2$.
\end{itemize}

If $\fg = \mathfrak{sl}(n)$ and $a = (a_1,...,a_t)$ is a partition of $n$, there is a Levi subalgebra
$$\mathfrak{s}(\mathfrak{gl}(a_1) \times ... \times \mathfrak{gl}(a_t)) := \mathfrak{sl}(n) \cap (\mathfrak{gl}(a_1) \times ... \times \mathfrak{gl}(a_t)) \subset \fg$$
corresponding to the roots
$$\{\pm(e_i -e_j)\}_{1 \leq i < j \leq a_1} \cup ... \cup \{\pm (e_i-e_j)\}_{n-a_t+1 \leq i < j \leq n} \subset \Delta$$
Every Levi subalgebra in $\fg$ is conjugate to one of this form, and no two such are conjugate.

If $\fg=\mathfrak{so}(2n+1)$, $\mathfrak{sp}(2n)$, or $\mathfrak{so}(2n)$, $0 \leq m \leq n$, and $a$ is a partition of $n-m$, there is a Levi subgroup
\begin{equation}\label{eq:Levi1}\mathfrak{gl}(a_1) \times ... \times \mathfrak{gl}(a_t) \times \fg(m) \subset \fg\end{equation}
corresponding to the roots
$$\{\pm (e_i-e_j)\}_{1\leq i < j \leq a_1} \cup ... \cup \{\pm (e_i - e_j)\}_{n-m-a_t+1 \leq i <j \leq n-m} \cup \Delta(m) \subset \Delta$$
where $\Delta(m) \subset \Delta$ has the obvious meaning. If $\fg=\mathfrak{so}(2n+1)$ or $\mathfrak{sp}(2n)$, then every Levi subalgebra in $\fg$ is conjugate to one of this form, and no two such are conjugate.

If $\fg=\mathfrak{so}(2n)$, and $a$ is a partition of $n$ with only even parts, there is a Levi subalgebra
\begin{equation}\label{eq:Levi2}\mathfrak{gl}(a_1) \times ...\times \mathfrak{gl}(a_t)' \subset \fg\end{equation}
corresponding to the roots
$$\{\alpha \in \Delta \mid \alpha(\underbrace{1,...,1}_{a_1},\underbrace{2,...,2}_{a_2},...,\underbrace{t,...,t,-t}_{a_t}) = 0\} \subset \Delta.$$
The prime is included to distinguish this subalgebra from the subgroup $\mathfrak{gl}(a_1) \times... \times \mathfrak{gl}(a_t) \subset \fg$ defined above (to which it is $O(2n)$-, but not $SO(2n)$-conjugate). Every Levi subalgebra in $\fg$ is conjugate to one of the form \eqref{eq:Levi1} or \eqref{eq:Levi2}, and no two such are conjugate. 

The Levi subalgebras listed above are called \emph{standard} with respect to $(\fh, \Delta, \Pi)$.

\subsection{Maximal pseudo-Levi subalgebras} \label{subsec:Maximalpseudo-Leviclassical}

Let $\fg$ be a simple Lie algebra of classical type. Fix $\fh$, $\Delta$, $\Pi$ and $\{e_i\}$ as in Section \ref{subsection:Levisclassical}.

If $\fg=\mathfrak{sl}(n)$, then there is a unique maximal pseudo-Levi, namely $\fg$ itself. Now assume $\fg$ is a simple Lie algebra of type $B_n$, $C_n$ or $D_n$, and $\fl$ is a maximal pseudo-Levi of $\fg$. Conjugating if necessary, we can assume that $\fl = \fg_I$ is a standard pseudo-Levi subalgebra associated to a proper subset $I$ of the set $\tilde{\Pi} = \Pi \cup \{\alpha_0\}$ of vertices of the extended Dynkin diagram of $\fg$, where $\alpha_0$ is the lowest root, i.e., the negative of the highest root. Moreover, there is a decomposition $I = I_1 \sqcup I_2$ and hence a Lie algebra decomposition $\fl = \fl_1 \oplus \fl_2$, where $\fl_i$ is the simple Lie subalgebra generated the root subspaces of the simple roots in $I_i$, $i=1,2$. In addition, we assume that the lowest root $\alpha_0$ (the extra node in the extended Dynkin diagram) belongs to $I_1$. In particular, $\fl_2$ is a simple subalgebra of the same type as $\fg$ and $\fl_1$ is a simple Lie algebra (or possibly zero). Then $\fl_1$, if non-zero, is of type $D$ (resp. $C$, $D$) when $\fg$ is of type $B$ (resp. $C$, $D$). More explicitly, write $k$ for the semisimple rank of $\fl_1$ ($0 \leq k \leq n-1$), so that $\fl_2$ is of semisimple rank $n-k$ (when $k=0$, $\fl=\fl_2$ is nothing else but $\fg$). Then the Cartan subalgebra $\fh_1 \simeq \CC^k$ of $\fl_1$ has coordinates $\{ e_1, e_2, \ldots e_k \}$, and the Cartan subalgebra $\fh_2 \simeq \CC^{n-k}$ of $\fl_2$ has coordinates $\{ e_{k+1}, e_{k+2}, \ldots e_n \}$. The Lie subalgebra $\fl_i$ is spanned by the Cartan subalgebra $\fh_i = \fh \cap \fl_i$ and root subspaces of $\fg$ of the root subsystems $\Delta_i \subset \Delta$, $i=1,2$, which we describe below. 

\begin{itemize}
    \item 
    If $\fg=\mathfrak{so}(2n+1)$, then
    $$\Delta_1 = \{ \pm e_i \pm e_j \,|\, 1 \leq i, j \leq k, i \neq j \}, \quad \Delta_2 = \{ \pm e_i \pm e_j, \pm e_i \,|\, k+1 \leq i, j \leq n, i \neq j \}, $$
    so that $\fl_1 \simeq \mathfrak{so}(2k)$ and $\fl_2 \simeq \mathfrak{so}(2n-2k+1)$.
    \item 
    If $\fg=\mathfrak{sp}(2n)$, then
    $$\Delta_1 = \{ \pm e_i \pm e_j, \pm 2e_i \,|\, 1 \leq i, j \leq k, i \neq j \}, \quad \Delta_2 = \{ \pm e_i \pm e_j, \pm 2e_i \,|\, k+1 \leq i, j \leq n, i \neq j \}, $$
    so that $\fl_1 \simeq \mathfrak{sp}(2k)$ and $\fl_2 \simeq \mathfrak{sp}(2n-2k)$.
    \item 
    If $\fg=\mathfrak{so}(2n)$, then
    $$\Delta_1 = \{ \pm e_i \pm e_j \,|\, 1 \leq i, j \leq k, i \neq j \}, \quad \Delta_2 = \{ \pm e_i \pm e_j \,|\, k+1 \leq i, j \leq n, i \neq j \}, $$
    so that $\fl_1 \simeq \mathfrak{so}(2k)$ and $\fl_2 \simeq \mathfrak{so}(2n-2k)$.
\end{itemize}

\subsection{Saturation of nilpotent orbits}

The following proposition is standard.

\begin{prop}\label{prop:inclusionclassical}
The following are true:
\begin{enumerate}
    \item[(i)] Suppose $\fg = \mathfrak{sl}(n)$. Let  $\mathfrak{m}=\mathfrak{s}(\mathfrak{gl}(a_1) \times ... \times \mathfrak{gl}(a_t))$ and let
    $$\mathbb{O}_M = \mathbb{O}_{\lambda_1} \times ... \times \mathbb{O}_{\lambda_t} \in \Orb(M)$$
    Then 
    $$\Sat^G_M \OO_M = \OO_{\lambda}$$
    where $\lambda = \bigcup_{j=1}^t \lambda_j$. 
    \item[(ii)] Suppose $\fg = \mathfrak{so}(2n+1)$, $\mathfrak{sp}(2n)$ or $\mathfrak{so}(2n)$. Let $\mathfrak{m}=  \mathfrak{gl}(a_1) \times \cdots \times \mathfrak{gl}(a_t) \times \mathfrak{g}(m)$ or possibly (if $\fg = \mathfrak{so}(2n)$) $\fm = \mathfrak{gl}(a_1) \times ... \times \mathfrak{gl}(a_t)'$ (with all $a_i$ even) and let
    $$\mathbb{O}_M =  \OO_{\lambda_1} \times ... \times \OO_{\lambda_t} \times \OO_{\lambda^0} \in \Orb(M)$$
    Then
    $$\Sat^G_M \OO_M = \OO_{\lambda}$$
    where
    $$\lambda =\lambda^0 \cup \bigcup_{j=1}^t (\lambda_j \cup \lambda_j).$$
    If $\fg=\mathfrak{so}(2n)$ and $\lambda$ is very even, then $\lambda^0$ is very even. In this case, $\lambda$ and $\lambda^0$ have the same decoration.
\end{enumerate}
\end{prop}

\begin{cor}[\cite{CM}, Thm 8.2.14]\label{cor: distinguished}
If $\fg = \mathfrak{sl}(n)$, then the only distinguished nilpotent orbit is the principal one. If $\fg = \mathfrak{so}(2n+1)$, $\mathfrak{sp}(2n)$, or $\mathfrak{so}(2n)$, then $\mathbb{O}$ is distinguished if and only if the corresponding partition has no repeated parts.
\end{cor}

\subsection{(Birational) induction of nilpotent orbits}

\begin{prop}[Theorem 7.3.3, \cite{CM}\label{prop:inductionclassical}]
The following are true:
\begin{itemize}
    \item[(i)] Suppose $\fg = \mathfrak{sl}(n)$. Let  $\mathfrak{m}=\mathfrak{s}(\mathfrak{gl}(a_1) \times ... \times \mathfrak{gl}(a_t))$ and let $\lambda^i\in \cP(a_i)$. Let
    $$\mathbb{O}_M = \mathbb{O}_{\lambda^1} \times \cdots \times \mathbb{O}_{\lambda^t} \in \Orb(M)$$
    Then 
    $$\Ind^G_M \OO_M = \OO_{\lambda}$$
    where $\lambda = \bigvee_{j=1}^t \lambda_j$. 
    \item[(ii)] Suppose $\fg = \mathfrak{so}(2n+1)$, $\mathfrak{so}(2n)$, or $\mathfrak{sp}(2n)$. Let $\mathfrak{m}=  \mathfrak{gl}(a_1) \times ... \times \mathfrak{gl}(a_t) \times \mathfrak{g}(m)$ or possibly (if $\fg = \mathfrak{so}(2n)$) $\fm = \mathfrak{gl}(a_1) \times ... \times \mathfrak{gl}(a_t)'$ (with all $a_i$ even). Let $\lambda^0\in \cP_B(2m+1)$, $\cP_D(2m)$, or $\cP_C(2m)$, and $\lambda^i\in \cP(a_i)$ for $1 \le i \le t$. Consider the orbit
    $$\mathbb{O}_M =  \OO_{\lambda^1} \times ... \times \OO_{\lambda^t} \times \OO_{\lambda^0} \in \Orb(M)$$
    Then
    $$\Ind^G_M \OO_M = \OO_{\lambda}$$
    where
    $$\lambda =(\lambda^0 \lor \bigvee_{j=1}^t (\lambda^j \lor \lambda^j))_{B/C/D}.$$ 

    If $\fg=\mathfrak{so}(2n)$ and $\lambda$ is very even, then its decoration can be deduced from \cite[Corollary 7.3.4]{CM}.
\end{itemize}
\end{prop}

\begin{prop}[\cite{Namikawa2022}, Sections 2 and 3]\label{prop:birationalpartitions}
In the setting of Proposition \ref{prop:inductionclassical}, $\Bind^G_M \OO_M = \OO_{\lambda}$ if and only if one of the following holds:

\begin{itemize}
    \item[(i)] $\fg=\fs\fl(n).$
    \item[(ii)] $\lambda = (\lambda^0 \lor \bigvee_{j=1}^t (\lambda^j \lor \lambda^j)).$
    \item[(iii)] $\fg=\fs\fo(2n)$, the partition $\beta=(\lambda^0 \lor \bigvee_{j=1}^t (\lambda^j \lor \lambda^j))$ has only even members, and there is exactly one distinct member $\beta_k$ with odd multiplicity. Then $\beta_k$ is the smallest member of $\beta$, and $\lambda$ is obtained from $\beta$ by replacing $(\beta_k)$ in $\beta$ with $(\beta_k-1, 1)$.
   
\end{itemize}
\end{prop}

\begin{rmk}\label{rmk:birationalpartitions}
We note that condition (iii) of Proposition \ref{prop:birationalpartitions} can be reformulated in terms of columns. The reformulated condition is as follows: $\fg=\fs\fo(2n)$, the partition $\beta=(\lambda^0 \lor \bigvee_{j=1}^t (\lambda^j \lor \lambda^j))$ is the join of pairs of columns $(c_i,c_i)$ of equal length, for $1 \le i \le l$, and there is exactly one distinct $c_j$ that is odd. In this case $\lambda = \beta_D$.
\end{rmk}

\subsection{Birationally rigid covers}

Let $\lambda \in \cP_\epsilon(n)$ and $\OO=\OO_\lambda$ be the corresponding nilpotent orbit in $\fg_\epsilon(n)$. Following \cite{Mitya2020}, we say that an integer $m$ is {\it $\lambda$-singular} if $\lambda_m - \lambda_{m+1} \geq 2$. Let $\cS(\lambda)$ denote the set of $\lambda$-singular numbers. Then $\cS(\lambda)$ is in bijection with the set of all codimension $2$ orbits in $\overline{\OO}$, see \cite{Kraft-Procesi}. For each $m \in \cS(\lambda)$, let $\OO_m$ denote the preimage of the corresponding codimension $2$ orbit in $\operatorname{Spec}(\CC[\OO])$. The singularity of $\operatorname{Spec}(\CC[\OO])$ along $\OO_m$ is equivalent to a Kleinian singularity $\Sigma_m = \CC^2 / \Gamma_m$ (see \cref{sec:Namikawa}). 

Given $m \in \cS(\lambda)$, let $\lambda^0 \in \cP_\epsilon(n-2m)$ be the partition obtained from $\lambda$ by removing two columns of length $m$ and let $\OO_0 = \OO_{\lambda^0}$ be the corresponding nilpotent orbit in $\fg_\epsilon(n-2m)$. Then $\OO=\OO_\lambda$ is birationally induced from the nilpotent orbit $\{0\} \times \OO_0$ in the Levi subalgebra $\fm=\mathfrak{gl}(m) \times \fg_\epsilon(n-2m) \subset \fg_\epsilon(n)$ by \cref{prop:birationalpartitions}. Recall from Section \ref{subsec:pi1_induction} there is a surjective homomorphism $\phi: A(\OO) = A(\OO_\lambda) \to A(\OO_0)$. We will compute this homomorphism in \cref{lem:phi_classical} below. Recall that by \cref{prop:Aclassical} and \cref{rmk:A_column}, both $A(\OO)$ and $A(\OO_0)$ can be regarded as subgroups of the elementary abelian 2-group $A$ associated to $\lambda$ with basis vectors $\vartheta_{c^\epsilon_i}$ associated to certain columns of length $c^\epsilon_i$ of $\lambda$.

\begin{lemma}\label{lem:phi_classical}
With the notations above, the map $\phi: A(\OO) = A(\OO_\lambda) \to A(\OO_0)$ can be computed in the following cases:
  \begin{enumerate}[label=(\roman*)]
    \item Suppose $\lambda_m \in \lambda^\epsilon$ and $\lambda_m = \lambda_{m+1} + 2$.
    \begin{itemize}
      \item
      {\bf Type $C_n$:} Let $\lambda^0=(a_{2q} \geqslant a_{2q-1} \geqslant \cdots \geqslant a_0 )$. Then $\phi: A(\OO) \to A(\OO_0)$ is given by 
	\begin{enumerate}
	   \item
		  $\phi(\vartheta_{a_{2k-1}}) = \vartheta_{a_{2k-1}}$, $\forall \, 1 \leqslant k \leqslant q$; 
	   \item
            $\phi(\vartheta_{m}) = \vartheta_{a_{2i+1}}$,  if $a_{2i+1} > m > a_{2i}$;
	   \item
            $\phi(\vartheta_{m}) = 1$,  if $ m > a_{2q}$.
	\end{enumerate}
        In particular the kernel of $\phi$ is generated by all $\vartheta_{a_{2i+1}} \vartheta_{m}$ with $a_{2i+1} > m > a_{2i}$ and by all $\vartheta_{m}$ with $m > a_{2q}$.
        \item
          {\bf Type $B_n$ (or $D_n$):} 
          Let $\lambda^0 =(a_{2q+1} \geqslant a_{2q} \geqslant \cdots \geqslant a_0)$.
          Then the kernel of $\phi$ is generated by all $ \vartheta_{a_{2i+1}} \vartheta_{m}$ with $a_{2i+1} > m > a_{2i}$.
	\end{itemize}
   \item 
     Otherwise, both $A(\OO)$ and $A(\OO_0)$ are  identified with $A^\epsilon$ by \cref{prop:Aclassical} and $\phi$ is the identity map.
   \end{enumerate}
\end{lemma}

\begin{proof} 
   We only prove the lemma for $\fg = \mathfrak{sp}(2n)$. The arguments for other types are similar, and left to the reader.
   As in \cite[\S 5.2]{CM}, we may realize $\fg = \mathfrak{sp}(2n)$ as the space of matrices 
   \begin{equation}\label{eq:sp}
      \mathfrak{sp}(2n) = \left\{
      \begin{pmatrix}
         Z_1  &  Z_2   \\
         Z_3  &  -Z_1^T  \\   
      \end{pmatrix}
      \,\middle|\,   Z_i \in M_n(\CC) \text{ and } Z_2, Z_3 \, \text{are symmetric}.
    \right\},    
   \end{equation}
   where $Z_1^T$ stands for the transpose matrix of $Z_1$. Let $\fh \subset \fg$ be the Cartan subalgebra of diagonal matrices of the form
   \[
     H = 
     \begin{pmatrix}
       D & 0 \\
       0 & -D \\
     \end{pmatrix}, \qquad D = \operatorname{diag}(h_1, h_2, \ldots, h_n)
   \]
Define linear functionals $e_i \in \fh^*$ by $e_i(H)=h_i$ for $1 \leq i \leq n$. As in \cref{subsection:Levisclassical}, the root system of $\fg$ is then $\Delta=\{ \pm e_i \pm e_j, \pm 2 e_i \,|\, 1 \leq i, j \leq n, i \neq j\}$, and we choose $\Delta^+=\{ e_i \pm e_j, 2 e_k \,|\, 1 \leq i < j \leq n, 1 \leq k \leq n\} \subset \Delta$ to be the set of positive roots and
   $\Pi = \{  e_1 - e_2,   e_2 - e_3, \ldots, e_{n-1} - e_n, 2e_n \}$ to be the set of simple roots. The $\alpha$-root space of $\fg$ is spanned by the root vector $X_\alpha$ defined below, where $E_{i,j}$ is the elementary matrix having a $1$ as its $(i,j)$-entry and zeros elsewhere.
   \begin{align*}
       X_{e_i-e_j} &= E_{i,j}-E_{j+n,i+n},  \\
       X_{e_i+e_j} &= E_{i,j+n} + E_{j, i+n}, \\
       X_{-e_i-e_j} &= E_{i+n,j} + E_{j+n,i}, \\
       X_{2e_i} &= E_{i,i+n}, \\
       X_{-2e_i} &= E_{i+n,i}.
   \end{align*}  
   We now construct an explicit nilpotent elment $\tilde{x} \in \OO=\OO_\lambda$ and an explicit $\mathfrak{sl}_2$-triple $\tilde{\psi} = (\tilde{x}, \tilde{y}, \tilde{h})$ in $\mathfrak{sp}(2n)$, following \cite[Section 5.2.2]{CM}. We break (the rows of) $\lambda$ up into chunks of the following two types: pairs $\{2p+1,2p+1\}$ of equal odd parts, and single even parts $\{2q\}$. Next we attach a set of consecutive indices in $\{1, 2, \ldots, n\}$ and the associated set of positive (but not necessarily simple) roots to each chunk $\cC$ as follows. If $\cC = \{2q\}$, choose a block $\cB = \cB(2q) = \{j+1, \ldots, j+q\}$ of consecutive indices and let $\cC^+=\cC^+(2q) = \{ e_{j+1} - e_{j+2}, \ldots, e_{j+q-1} - e_{j+q}, 2 e_{j+q}\}$. We set $\cB'=\{ j+1 \}$ if $2q \geq m$ and $\cB'=\emptyset$ otherwise. If $\cC=\{2p+1,2p+1\}$, choose a block $\cB = \cB(2p+1,2p+1) = \{l+1, \ldots, l+2p+1\}$ of consecutive indices and let $\cC^+=\cC^+(2p+1,2p+1) = \{e_{l+1} - e_{l+2}, \ldots, e_{l+2p} - e_{l+2p+1}\}$ (note that $\cC^+$ is empty if $\cC=\{1,1\}$). We set $\cB'=\{l+1, l+2p+1\}$ if $2p+1 \geq m$ and $\cB'=\emptyset$ otherwise. We require the blocks of indices attached to distinct chunks to be disjoint. Let $\tilde{x}$ ($\tilde{y}$ resp.) be the sum of all $\alpha$-root vectors $X_\alpha$ ($X_{-\alpha}$ resp.) for $\alpha$ appearing in some $\cC^+$. Let $\tilde{h} = \sum_{\cC} \tilde{h}_\cC$, where $\tilde{h}_\cC$ is defined by
     \[ \tilde{h}_\cC = \sum_{k=1}^q (2q-2k+1) (E_{k+l,j+k} - E_{j+k+n, j+k+n} ) \]
   if $\cC^+ = \{ e_{j+1} - e_{j+2}, \ldots, e_{j+q-1} - e_{j+q}, 2 e_{j+q}\}$, and
      \[ \tilde{h}_\cC = \sum_{k=0}^{2p} (2p-2k) (E_{l+1+k,l+1+k} - E_{l+1+k+n, l+1+k+n} ) \]
   if $\cC^+ = \{e_{l+1} - e_{l+2}, \ldots, e_{l+2p} - e_{l+2p+1}\}$. Then $\tilde{\psi} = (\tilde{x}, \tilde{y}, \tilde{h})$ defines an $\mathfrak{sl}_2$-triple in $\mathfrak{sp}(2n)$ associated to $\OO$.

   Now take $\cG$ to be the union of all $\cB'$ and $\tilde{\cG} := \cG \cup \{ i+n \,|\, i \in \cG \}$. Let $\fr$ be the Lie subalgebra of $\fg=\mathfrak{sp}(2n)$ consisting of block diagonal matrices $(Z, -Z^T)$, where $Z \in M_n(\CC)$ runs over all matrices with all entries equal to zero except for the $(i,j)$-th entries with $i,j \in \cG$. Then $\fr$ can be naturally identified with $\mathfrak{gl}(m)$. Let $\fk$ be the Lie subalgebra of $\fg$ consisting of matrices of the form \eqref{eq:sp} whose $i$-th row and $j$-th column contain only zeros whenever $i,j \in \tilde{\cG}$. Then $\fk$ is naturally identified with $\mathfrak{sp}(2n-2m)$ by deleting rows and columns with indices in $\cG$. Let $\fm$ be the (direct) sum of $\fr$ and $\fk$, so that $\fm \simeq \mathfrak{gl}(m) \times \mathfrak{sp}(2n-2m)$ is a Levi subalgebra of $\mathfrak{sp}(2n)$ associated to the roots
    \[ \Delta_\fm = \{ e_i-e_j \}_{i,j \in \cG, i \neq j} \cup \{ \pm e_i \pm e_j, \pm 2 e_i \,|\, i, j \notin \cG, i \neq j\}. \] 
   Note that $\fm$ is in general not one of the standard Levi subalgebras considered in \cref{subsection:Levisclassical}.

   Now we reset rows and columns with indices in $\tilde{\cG}$ of the matrices $\tilde{x}, \tilde{y}, \tilde{h}$ to be zero and obtain matrices $x, y, z$ lying in $\fk \simeq \mathfrak{sp}(2n-2m)$. It is easy to verify that $\psi=(x,y,h)$ forms an $\mathfrak{sl}_2$-triple in $\fk$ associated to the orbit $\OO_0$ corresponding to the partition $\lambda^0$. Let $\fq$ be the parabolic subalgebra spanned by $\fm$ and all the positive root vectors $X_\alpha$, $\alpha \in \Delta^+$, and let $\fu \subset \fq$ be the nilpotent radical corresponding to the roots in $\Delta^+ \backslash \Delta_{\fm}$. By construction we have $\tilde{x} = x + x'$, where $x' \in \fu$, hence $\tilde{x} \in \OO \cap (\OO_0 + \fu)$.
   
   We regard $x \in \OO_0$ as an element in $0 \times \mathfrak{sp}(2n-2m) \subset \mathfrak{gl}(m) \times \mathfrak{sp}(2n-2m) = \fm$. Let $M = R \times K \simeq GL(m) \times Sp(2n-2m)$ be the Levi subgroup of $G=Sp(2n)$ corresponding to $\fm$. Let $M^\psi$ and $M^{\tilde{\psi}}$ denote the $M$-centralizers of $\psi$ and $\tilde{\psi}$ respectively. Then one can check that $M^{\tilde{\psi}} \subset M^{\psi}$. We will describe explicitly this inclusion below. By Lemma \ref{lem:phi}, the map $\phi: A(\OO) \to A(\OO_0)$ can be computed as the induced natural map $\pi_0 (M^{\tilde{\psi}}) \to \pi_0 (M^\psi)$. First consider case (i) in the statement of the lemma. From the construction of $\tilde{\psi}$ and $\psi$, we have the following group isomorphisms (cf. \cite[Theoerem 6.1.3]{CM}):
   \begin{equation}\label{eq:M_psi_tilde}
    M^{\tilde{\psi}}  \simeq \underbrace{Sp(a_{2q} - a_{2q-1})  \times \cdots \times \textcolor{red}{O (a_{2i+1}-m)} }_{S} 
     \textcolor{red}{\times}
	   \underbrace{ \textcolor{red}{ O(m - a_{2i})} \times  \cdots \times O(a_1-a_0) \times Sp(a_0) }_{T} 
    \end{equation}
	and
    \begin{equation}\label{eq:M_psi}
     M^{\psi} \simeq \textcolor{blue}{GL(m)} \times \underbrace{Sp(a_{2q} - a_{2q-1}) \times \cdots  \times \textcolor{red}{O (a_{2i+1} - a_{2i})} \times  \cdots O(a_1-a_0) \times Sp(a_0)}_{W}.    
    \end{equation}
  We can describe the inclusion $M^{\tilde{\psi}} \subset M \simeq GL(m) \times Sp(2n-2m)$ in terms of the isomorphism \eqref{eq:M_psi_tilde} as follows: both subgroups $S$ and $T$ are included in $Sp(2n-2m)$ in the standard way, but $T$ is also included in $GL(m)$. The inclusion $S \hookrightarrow Sp(2n-2m)$ induces an injective homomorphism $\iota_S: S \hookrightarrow \{1\} \times Sp(2n-2m) \subset M$. We also have the product map $\iota_T: T \hookrightarrow GL(m) \times Sp(2n-2m)$ of the two inclusions of $T$ into $GL(m)$ and $ Sp(2n-2m)$. It is clear that the images of $\iota_S$ and $\iota_T$ commute with each other, so we have an induced injective homomorphism $(\iota_S, \iota_T): M^{\tilde{\psi}} = S \times T \hookrightarrow M$, which is exactly the inclusion map $M^{\tilde{\psi}} \hookrightarrow M$. We see that the image lies in $M^\psi$.

  Note that the $GL(m)$-factor is in the identity component of $M^\psi$, so 
  \[\pi_0(M^\psi) \simeq \pi_0(M^\psi/GL(m)) \simeq \pi_0(W).\] 
  Therefore we need only to understand the composite map $f_W: M^{\tilde{\psi}} \hookrightarrow M^\psi \twoheadrightarrow W$. Note that the composite map $M^{\tilde{\psi}} \hookrightarrow M \twoheadrightarrow M/GL(m) \simeq Sp(2n-2m)$ is also the composition of $f_W$ and the inclusion $W \hookrightarrow K \simeq Sp(2n-2m)$, we see that $f_W$ is the block-diagonal embedding $O (a_{2i+1}-m) \times O(m - a_{2i}) \hookrightarrow O (a_{2i+1} - a_{2i})$ of the two red simple factors in \eqref{eq:M_psi_tilde} into the red factor in \eqref{eq:M_psi} and is the identity map on the remaining simple factors which appear on the right hand sides of both \eqref{eq:M_psi} and \eqref{eq:M_psi_tilde}. Passing to the level of component groups $\pi_0$, we get the conclusion in (i) easily.

  The proof of case (ii) is similar and easier, so we omit it.
\end{proof}

For any $m \in \cS(\lambda)$, let $d_m = \lfloor (\lambda_m - \lambda_{m+1})/2 \rfloor$ and $\lambda^0 \in \cP_\epsilon(n-2m d_m)$ be the partition obtained from $\lambda$ by removing $2 d_m$ columns of length $m$ and let $\OO_0 = \OO_{\lambda^0}$ be the corresponding nilpotent orbit in $\fg_\epsilon(n-2m d_m)$. Then $\OO$ is birationally induced from the orbit $\{0\} \times \OO_0$ in the Levi subalgebra $\fm=\mathfrak{gl}(m)^{d_m} \times \fg_\epsilon(n-2md_m) \subset \fg_\epsilon(n)$ by \cref{prop:birationalpartitions}. Let $H_m$ be the kernel of the map $\phi: A(\OO) \to A(\OO_0)$. 

\begin{prop}\label{prop:codim2sing_classical_cover}
   Let $\widetilde{\OO}$ be a cover of $\OO=\OO_\lambda$ corresponding to a subgroup $H \subset A(\OO)$. For $\lambda_m \in \lambda$ such that $\lambda_m \geq \lambda_{m+1}+2$, let $\widetilde{\OO}_m \subset \widetilde{X} = \operatorname{Spec}(\CC[\widetilde{\OO}])$ be the corresponding cover of $\OO_m$. Then $\widetilde{\OO}_m$ is connected and the singularity of $\widetilde{X}$ along $\widetilde{\OO}_m$ is equivalent to $\CC^2 / \Gamma'_m$, where $\Gamma'_m$ is a subgroup of $\Gamma_m$ of index $2$ if $\lambda_m \in \lambda^{\epsilon}$ and  $H_m \not\subset H$, and $\Gamma'_m = \Gamma_m$ otherwise.
\end{prop}
\begin{proof}
    This is essentially Theorem 2.6 of \cite{Mitya2020}. Note that \cite{Mitya2020} deals with the adjoint groups, but the arguments in \cite[Section 5]{Mitya2020} work for classical groups as well.
\end{proof}

Thanks to \cref{lem:phi} (iii), we can apply \cref{lem:phi_classical} inductively and deduce that $H_m = \{1,  \upsilon_{\lambda_m}\upsilon_{\lambda_{m+1}} \}$. In terms of columns of $\lambda$, $m \in  \cS(\lambda)$ means that $m = c_j = c^{\epsilon}_i$ is a column of $\lambda = (c_p \ge c_{p-1} \ge \cdots \ge c_1)$ with multiplicity $m_{\lambda^t}(m) = \lambda_m - \lambda_{m+1} \ge 2$, and $H_m = \{1, \vartheta_{c_j} \vartheta_{c_{j+1}} \}$. This together with \cref{prop:codim2sing_classical_cover} immediately implies the following.

\begin{prop}\label{prop:2leafless_classical_cover}
   Suppose $G = G_{\epsilon}(n)$. Let $\OO_{\lambda} \in \Orb(G)$ and let $\widetilde{\OO}_{\lambda}$ be a $G$-equivariant cover of $\OO_{\lambda}$ corresponding to a subgroup $H \subset A(\OO_{\lambda}) = A^\epsilon$. Then $\Spec(\CC[\widetilde{\OO}_{\lambda}])$ has no codimension $2$ leaves if and only if $\lambda$ and $H$ satisfy the following conditions: 
   \begin{enumerate}[label=(\roman*)]
     \item 
      each column $c_j$ of $\lambda$ has multiplicity less or equal than $2$, and it can be $2$ only when $c_j=c^\epsilon_i$ for some $i$, i.e., $\height_{\lambda^t}(c_j) \equiv 1-\epsilon \mod 2$;
    \item 
      for each $c_j = c^\epsilon_i$ such that $m_{\lambda}(c^\epsilon_i)=2$, the subgroup $H$ of $A(\OO)$ does not contain the element $\vartheta_{c_j} \vartheta_{c_{j+1}} =\vartheta_{c^\epsilon_i}\vartheta_{c^\epsilon_{i+1}}$.
   \end{enumerate}
\end{prop}

\begin{lemma}\label{lem:H2_classical_cover}
    Suppose $G = G_{\epsilon}(n)$. Let $\OO_{\lambda} \in \Orb(G)$ and let $\widetilde{\OO}_{\lambda}$ be a $G$-equivariant cover of $\OO_{\lambda}$ corresponding to a subgroup $H \subset A(\OO_{\lambda}) = A^\epsilon$. Then $H^2(\widetilde{\OO}_\lambda, \CC) = 0$ if and only if, for each column $c^{\epsilon}_k$ appearing in $\lambda$ with multiplicity $2$, $H$ contains an element $\theta$ whose decomposition as product of $\vartheta_{c_i^\epsilon}$'s contains $\vartheta_{c^{\epsilon}_k}$.
\end{lemma}
\begin{proof}
    This follows immediately from \cref{lem:computeH2} and explicit description (in terms of rows) of the reductive centralizer of $\OO_\lambda$ for classical groups similar to \cite[Theorem 6.1.3]{CM}.
\end{proof}

\subsection{BVLS duality}

For a partition $\lambda = [\lambda_1,...,\lambda_k] \in \mathcal{P}(n)$, write
\begin{align*}
l(\lambda) &= [\lambda_1,...,\lambda_{k-1},\lambda_k-1] \in \mathcal{P}(n-1)\\
e(\lambda) &= [\lambda_1,...,\lambda_k,1] \in \mathcal{P}(n+1)\\
\lambda^+ &= [\lambda_1+1,\lambda_2,...,\lambda_k] \in \mathcal{P}(n+1)
\end{align*}

\begin{prop}[\cite{McGovern1994}, Thm 5.2]\label{prop:BVduality}
Suppose $G$ is a simple group of classical type. Let $\OO_{\lambda} \in \Orb(G^{\vee})$ and let let $d(\OO_{\lambda}) = \OO_{d(\lambda)}$. Then $d(\lambda)$ is given by the following formulas:

\begin{itemize}
    \item[(i)] If $\fg = \mathfrak{sl}(n)$, then $\fg^{\vee} = \mathfrak{sl}(n)$ and 
    $$d(\lambda) = \lambda^t.$$
    \item[(ii)] If $\fg = \mathfrak{so}(2n+1)$, then $\fg^{\vee} =  \mathfrak{sp}(2n)$ and 
    $$d(\lambda) = (e(\lambda)^t)_{B}.$$
    \item[(iii)] If $\fg = \mathfrak{sp}(2n)$, then $\fg^{\vee} = \mathfrak{so}(2n+1)$ and 
    $$d(\lambda) = (l(\lambda^t))_C.$$
    \item[(iv)] If $\fg = \mathfrak{so}(2n)$, then $\fg^{\vee} = \mathfrak{so}(2n)$ and
    $$d(\lambda) =(\lambda^t)_{D}.$$
    If $\lambda$ is very even, then $d(\lambda)=\lambda^t$ is very even. If $n$ is divisible by $4$, then $\lambda$ and $d(\lambda)$ have the same decoration. Otherwise, they have opposite decorations. 
\end{itemize}
\end{prop}

\subsection{Lusztig-Achar data}\label{subsection:Acharclassical}

Following \cite{Achar2003}, we introduce the notion of a \emph{reduced marked partition}.

\begin{definition}\label{defn:markable}
    Let $X \in \{B,C,D\}$ and let $\lambda \in \mathcal{P}_X(m)$ (where $m$ is either even or odd, depending on $X$). If $X = B$ (resp. $C$, $D$), a \emph{markable part (or row)} for $\lambda$ is an odd (resp. even, odd) positive integer $x$ such that $m_{\lambda}(x) \geq 1$ and $\mathrm{ht}_{\lambda}(x)$ is odd (resp. even, even).

    A \emph{reduced marked partition} of type $X$ is a marked partition $\prescript{\langle \nu \rangle}{}{\lambda}$ of type $X$ where $\nu$ consists of only markable parts of $\lambda$. We denote the subset of all reduced marked partitions in $\tilde{\cP}_X(m)$ by $\overline{\cP}_X(m)$. Again we allow decorations by Roman numerals $I$ and $II$ for the very even partitions of type $D$.
    
    Similar to $\cP_\epsilon(m)$, we also sometimes write $\overline{\cP}_\epsilon(m)$ for the sets defined above, where $\epsilon \in \{ 0, 1\}$.
\end{definition}

Recall that in Section \ref{subsec:A_classical}, we have defined a subpartition $\lambda^\epsilon = [\lambda^\epsilon_1, \lambda^\epsilon_2, \cdots, \lambda^\epsilon_r]$ and an elementary abelian 2-group $A \simeq (\ZZ_2)^r$ with basis $\{\upsilon_{\lambda^\epsilon_i}\}$ and the subgroups $A^\epsilon$ for $\epsilon = 0, 1$. Let $\lambda^m = [ \lambda^m_1 > \lambda^m_2 > \cdots > \lambda^m_l > 0]$ be the subpartition of $\lambda^\epsilon$ consisting of all the markable parts of $\lambda$ in the sense of Definition \ref{defn:markable}. We also make the convention that $\lambda^m_0 = \infty$ and $\lambda^m_{l+1} = 0$. Let $N$ denote the kernel of the composition $A^\epsilon \simeq A(\OO_\lambda) \twoheadrightarrow \bar{A}(\OO_\lambda)$ (see Proposition \ref{prop:Aclassical}). We describe the group $N$ below, which follows easily from the discussions in \cite[Section 5]{Sommers2001}. 

\begin{prop}\label{prop: A to barA}
The subgroup $N$ of $A^\epsilon$ is generated by the elements of form $\upsilon_{\lambda^\epsilon_{i}} \upsilon_{\lambda^m_{j}}$ such that $\lambda^m_{j} \le \lambda^\epsilon_{i} < \lambda^m_{j-1}$ for all $1 \le j \le l+1$ (which implies that $\upsilon_{\lambda^\epsilon_{i}} = \upsilon_{\lambda^\epsilon_{i}} \upsilon_0$ always lies in $N$ for $\lambda^\epsilon_{i} < \lambda^m_{l}$).
\end{prop}

\begin{cor}\label{cor:Abarclassical}
Suppose $G$ is a simple classical group. Let $\OO_{\lambda} \in \Orb(G)$. If $\fg = \mathfrak{sl}(n)$, then $\bar{A}(\OO_{\lambda})\simeq 1$. Otherwise, $\bar{A}(\OO_{\lambda}) \simeq (\ZZ_2)^d$, where
$$d = \begin{cases*}
            \text{$\#$ of markable parts} & if  $\fg = \mathfrak{sp}(2n)$\\
            \text{$\#$ of markable parts} -1  & if $\fg=\mathfrak{so}(2n+1)$ or $\mathfrak{so}(2n)$
                 \end{cases*} $$
\end{cor}

\begin{rmk}\label{rmk:Abar_column}
    We can also describe $\bar{A}(\OO)$ in terms of columns as we did for $A(\OO)$ in \cref{rmk:A_column} as follows. We can introduce the notion of a {\it markable column} of $\lambda$ as the dual to a markable part/row, namely those $c^\epsilon_i$ whose corresponding $\lambda^\epsilon_i$ are markable in the sense of \cref{defn:markable}. More concretely, If $X = B$ (resp. $C$, $D$), a markable column of $\lambda$ is an odd (resp. even, even)  positive integer $y$ such that $m_{\lambda^t}(y) \geq 1$ and $\mathrm{ht}_{\lambda^t}(y)$ is odd (resp. even, odd). Let $c^m = (c^m_l > c^m_{l-1} > \cdots > c^m_{1} > 0)$ denote the markable columns of $\lambda$. We also make the convention that $c^m_0 = 0$, $c^m_{l+1} = \infty$ and $\vartheta_{\infty}=0$. Then the subgroup $N$ of $A^\epsilon$ is generated by the elements of form $\vartheta_{c^\epsilon_{i}} \vartheta_{c^m_{j}}$ such that $c^m_{j-1} < c^\epsilon_{i} \le c^m_{j}$ for all $1 \le j \le l+1$ (which implies that $\vartheta_{c^\epsilon_{i}} = \vartheta_{c^\epsilon_{i}} \vartheta_{\infty}$ always lies in $N$ for $c^\epsilon_{i} > c^m_{l}$).
\end{rmk}

For $G$ simple classical, we have the following parametrizations of $\LA(G)$ (see \cite[Section 3.4]{Achar2003}). Here we adopt the notations of Section \ref{subsec:A_classical}.

\begin{prop}\label{prop:markedpartitionsAchardata}
Suppose $\fg$ is a simple Lie algebra of classical type, then the following are true:
\begin{enumerate}[label=(\roman*)]
    \item Suppose $\fg=\mathfrak{sl}(n)$. Then $\bar{A}(\OO)\simeq 1$ for every $\OO \in \Orb(G)$. In particular, there is a bijection $\mathcal{P}(n) \xrightarrow{\sim} \LA(G)$.    
    \item Suppose $\fg=\fg_\epsilon(m)$, $\epsilon \in \{0,1\}$, is of type $B$, $C$ or $D$. Then the composite map
      $$ \overline{\cP}_\epsilon(m) \hookrightarrow \tilde{\cP}_\epsilon(m) \twoheadrightarrow \Conj(G_\epsilon(m)) \xrightarrow{\sim} \LA(G_\epsilon(m))  $$
    is a bijection, where $\overline{\cP}_\epsilon(m) \hookrightarrow \tilde{\cP}_\epsilon(m)$ is the natural inclusion map, $\tilde{\cP}_\epsilon(m) \xrightarrow{\sim} \Conj(G_\epsilon(m))$ is the bijection in Corollary \ref{cor:conj_classical} and $\Conj(G_\epsilon(m)) \twoheadrightarrow \LA(G_\epsilon(m))$ is the natural projection.
\end{enumerate}
\end{prop}

\subsection{Special Lusztig-Achar data}

Let $G$ be a simple group of type $X \in \{B,C,D\}$. By Proposition \ref{prop:markedpartitionsAchardata}, $\LA(G)$ is parameterized by $\overline{\cP}_X(m)$. Recall the subset $\LA^*(G) \subset \LA(G)$ of special Lusztig-Achar data, cf. Section \ref{subsec:specialAchar}. Let $\overline{\cP}_X^*(m)$ denote the subset of $\overline{\cP}_X(m)$ corresponding to $\LA^*(G)$. The elements of $\overline{\cP}_X^*(m)$ are called \emph{special marked partitions}.

\begin{prop}[Propositions 5.8 and 5.16, \cite{Achar2003}]\label{prop:special_LA_classical}
A reduced marked partition $^{\langle \nu\rangle} \lambda$ of type $B$ (resp. $C$, $D$) is special if and only if there are no even (resp. odd, even) parts $x$ of $\lambda$ such that $\mathrm{ht}_{\nu}(x)$ is odd and $\height_\lambda(x)$ is odd (resp. even, even).
\end{prop}

\subsection{Saturation of conjugacy data}

The following is elementary. We leave the straightforward verification to the reader. 

\begin{prop}
Let $G=G_\epsilon(m)$. Let $M = GL(a_1) \times ... \times GL(a_t) \times G_{\epsilon}(n)$ or possibly (if $G$ is even orthogonal) $M = GL(a_1) \times ... \times GL(a_t)'$ (with all $a_i$ even) and suppose $(\OO_M,C_M) \in \Conj(M) \simeq \Conj(GL(a_1)) \times ... \times \Conj(GL(a_t)) \times \Conj(G_\epsilon(n))$ corresponds, under the map of Corollary \ref{cor:conj_classical} to a tuple
$$(\lambda^1,...,\lambda^t, ^{\langle \nu^0 \rangle}\lambda^0) \in \mathcal{P}(a_1) \times ... \times \mathcal{P}(a_t) \times \tilde{\cP}_{\epsilon}(n)$$
Then $\Sat^G_M (\OO_M,C_M)$ corresponds to the marked partition $^{\langle \nu\rangle}\lambda$, where
$$\lambda = \lambda^0 \cup \bigcup_{j=1}^t (\lambda^j \cup \lambda^j), \qquad \nu = \nu^0$$
\end{prop}

\begin{cor}\label{cor:dist_conj_Som}
Let $G=G_{\epsilon}^{ad}(m)$. Suppose $(\OO,C) \in \Conj(G)$ corresponds under the map of Corollary \ref{cor:conj_classical_adjoint} to a marked partition $\prescript{\langle \nu \rangle}{}{\lambda} \in \tilde{\cP}_\epsilon(m)$. Then $(\OO, C)$ is distinguished if and only if the following equivalent conditions are satisfied:
\begin{enumerate}[label=(\roman*)]
    \item
       $\nu$ and $\eta = \lambda \backslash \nu$ are distinguished partitions, i.e., all their members have multiplicity $1$. 
    \item 
      For all i, $\lambda_j \not\equiv \epsilon \mod 2$, i.e., $\lambda=\lambda^\epsilon$. Moreover, $m_\lambda(x) \leq 2$ for any $x$, and if $m_\lambda(x)=2$, then $x \in \nu$.
    \item 
      Under the isomorphism $\Som(G) \xrightarrow{\sim} \Conj(G)$ of Lemma \ref{Lem:Sommers_data}, $(\OO, C) \in \Conj(G)$ corresponds to the pair $(\fl_1 \times \fl_2, \OO_\nu \times \OO_{\eta})$, where $\fl_1 \times \fl_2$ is a maximal pseudo-Levi subalgebra of $\fg$ (see Section \ref{subsec:Maximalpseudo-Leviclassical}) and $\OO_\nu$ and $\OO_{\eta}$ are distinguised nilpotent orbits in the simple classical Lie algebras $\fl_1$ and $\fl_2$ corresponding to partitions $\nu$ and $\eta$ respectively. 
\end{enumerate}
\end{cor}

\begin{rmk}\label{rmk:Som_classical}
    Note that when $\fg$ is of type $C$ or $D$, the simple factors $\fl_1$ and $\fl_2$ are of the same type, and  $(\fl_1 \times \fl_2, \OO_\nu \times \OO_{\eta})$ is $G^{ad}_\epsilon(m)$-conjugate to $(\fl_2 \times \fl_1, \OO_\eta \times \OO_{\nu})$ in $\Som(G^{ad}_\epsilon(m))$. This corresponds to the fact that both $\prescript{\langle \nu \rangle}{}{\lambda}$ and $\prescript{\langle \eta \rangle}{}{\lambda}$ correspond to the same pair $(\OO, C) \in \Conj(G^{ad}_\epsilon(m))$ as in Corollary \ref{cor:conj_classical_adjoint}.
\end{rmk}

\subsection{Saturation of Lusztig-Achar data}\label{subsec:Sat_LA}

The following is elementary. We leave the straightforward verification to the reader. 

\begin{prop}
Let $G=G_\epsilon(m)$. Let $M = GL(a_1) \times ... \times GL(a_t) \times G_{\epsilon}(n)$ or possibly (if $G$ is even orthogonal) $M = GL(a_1) \times ... \times GL(a_t)'$ (with all $a_i$ even) and suppose $(\OO_M,\bar{C}_M) \in \LA(M) \simeq \LA(GL(a_1)) \times ... \times \LA(GL(a_t)) \times \LA(G)$ corresponds, under the bijections of Proposition \ref{prop:markedpartitionsAchardata} to a tuple
$$(\lambda^1,...,\lambda^t, ^{\langle \nu^0 \rangle}\lambda^0) \in \mathcal{P}(a_1) \times ... \times \mathcal{P}(a_t) \times \overline{\cP}_{\epsilon}(n)$$
Then $\Sat^G_M (\OO_M,\bar{C}_M)$ corresponds to the reduced marked partition $^{\langle \nu\rangle}\lambda$, where
$$\lambda = \lambda^0 \cup \bigcup_{j=1}^t (\lambda^j \cup \lambda^j), \qquad \nu = \nu^0$$
\end{prop}

\begin{cor}\label{cor:dist_achar}
Suppose $\prescript{\langle \nu \rangle}{}{\lambda}$ is the reduced marked partition associated to a Lusztig-Achar datum $(\OO, \bar{C}) \in \LA(G)$. Then $(\OO, \bar{C})$ is distinguished if and only if the following equivalent conditions are satisfied:

\begin{enumerate}[label=(\roman*)]
  \item
     $\nu$ and $\eta = \lambda \backslash \nu$ are distinguished partitions.
  \item
     For all i, $\lambda_j \not\equiv \epsilon \mod 2$, i.e., $\lambda=\lambda^\epsilon$. Moreover, $m_\lambda(x) \leq 2$ for any $x$, and if $m_\lambda(x)=2$, then $x \in \nu$ (which implies that $x$ is markable in the sense of \cref{subsection:Acharclassical}).
\end{enumerate}
\end{cor}     

Now we give an alternative description of $\bar{A}(\OO)$ of $\OO = \OO_\lambda$ in the special case when there is a distinguished Lusztig-Achar datum $(\OO, \bar{C}) \in \LA(G)$. Recall that in Section \ref{subsec:A_classical}, we have defined a subpartition $\lambda^\epsilon = [\lambda^\epsilon_1, \lambda^\epsilon_2, \cdots, \lambda^\epsilon_r]$ and an elementary abelian 2-group $A \simeq (\ZZ_2)^r$ with basis $\{\upsilon_{\lambda^\epsilon_i}\}$ and the subgroups $A^\epsilon$ for $\epsilon = 0, 1$. By \cref{cor:dist_achar}, we have $\lambda = \lambda^\epsilon$. Let $N$ denote the kernel of the map $A^\epsilon = A(\OO_\lambda) \twoheadrightarrow \bar{A}(\OO_\lambda)$ (see Proposition \ref{prop:Aclassical}). The next proposition follows easily from the discussion in \cite[Section 5]{Sommers2001}. 

\begin{prop}\label{prop: A to barA distinguished}
The following are true:
\begin{enumerate}[label=(\roman*)]
    \item
    If $\fg$ is of type $B$, then $N$ is the subgroup of $A^\epsilon$ generated by the elements $\upsilon_{\lambda_{2i}} \upsilon_{\lambda_{2i+1}}$ for all $i \geqslant 1$. 
    \item
    If $\fg$ is of type $C$ or $D$, then $N$ is the subgroup of $A^\epsilon$ generated by the elements $\upsilon_{\lambda_{2i-1}} \upsilon_{\lambda_{2i}}$ for all $i \geqslant 1$ (note that when $\fg$ is of type $C$ and $r$ is odd, this implies that $\upsilon_{\lambda_{r}} = \upsilon_{\lambda_{r}} \upsilon_0$ lies in $N$).
\end{enumerate}
\end{prop}

Now we define a special basis of $\bar{A}(\OO_\lambda)$. 

\begin{definition}\label{defn:Abar_basis}
With notations above, define the following elements in $A^\epsilon$:
\begin{itemize}
    \item 
    $\theta_i := \upsilon_{\lambda_{2i-1}} \upsilon_{\lambda_{2i+1}}$ for $1 \leqslant i \leqslant  \frac{r-1}{2}$, if $\fg$ is of type $B$;
    \item 
    $\theta_i := \upsilon_{\lambda_{2i}} \upsilon_{\lambda_{2i+2}}$ for $1 \leqslant i \leqslant  \lfloor \frac{r}{2} \rfloor$, if $\fg$ is of type $C$;
    \item 
    $\theta_i := \upsilon_{\lambda_{2i}} \upsilon_{\lambda_{2i+2}}$ for $1 \leqslant i \leqslant \frac{r}{2} - 1$, if $\fg$ is of type $D$.
\end{itemize}
Note that when $\fg$ is of type $C$, by our convention $\theta_{k} = \upsilon_{\lambda_{2k}} \upsilon_0 = \upsilon_{\lambda_{2k}}$ if $r = 2k $ or $2k+1$. 	
\end{definition}

The various $\theta_i$'s generate a subgroup $K(\OO_\lambda)$ of $A(\OO_\lambda)=A^\epsilon$ which maps isomorphically onto its image under the quotient map $A(\OO_\lambda) \twoheadrightarrow \bar{A}(\OO_\lambda)$. In other words, we get a splitting $s:  \bar{A}(\OO_\lambda) \to A(\OO_\lambda)$ whose image is the subgroup $K(\OO_\lambda)$. We write $\bar{\theta}_i$ for the image of $\theta_i$ in $\bar{A}(\OO_\lambda)$ and so $\{ \bar{\theta}_i \}$ forms a basis of $\bar{A}(\OO_\lambda)$. Note that the group $K(\OO_\lambda)$ is mapped isomorphically to its image under the quotient map $A(\OO_\lambda) \twoheadrightarrow A^{ad}(\OO_\lambda)$, so we can also regard $K(\OO_\lambda)$ as a subgroup of $A^{ad}(\OO_\lambda)$. Then $K(\OO_\lambda) \subset A^{ad}(\OO_\lambda)$ is exactly the group in \cref{lem:K(O)}. Elements in $A^0$, i.e., those that can be written as products of even number of $\upsilon_{\lambda_i}$, e.g., $\theta_i$, will also be identified with their images in $A^{ad}(\OO_\lambda)$ by abuse of notation.

\subsection{Sommers duality}\label{sec: Sommers}

\begin{prop}\cite[Theorem 12]{Sommers2001}\label{prop:Sommersdualityclassical}
Let $G^\vee = G_\epsilon(m)$ and $(\OO^{\vee},\bar{C}) \in \LA(G^\vee)$ corresponds to a reduced marked partition $^{\langle \nu\rangle} \lambda \in \overline{\cP}_X(m)$(see Proposition \ref{prop:markedpartitionsAchardata}). Let $\OO_{d_S(^{\langle \nu\rangle} \lambda )} = d_S(\OO^{\vee},\bar{C})$. Then, writing $\eta = \lambda \setminus \nu$, the partition $d_S(^{\langle \nu\rangle} \lambda )$ is given by the following formulas:
\begin{itemize}
    \item[(i)] If $\fg^{\vee}=\mathfrak{sp}(2n)$, then $\fg=\mathfrak{so}(2n+1)$ and
    $$d_S(^{\langle \nu\rangle} \lambda ) = ((\nu \cup ((\eta)^+)_B)^t)_B.$$ 
    \item[(ii)] If $\fg^{\vee}=\mathfrak{so}(2n+1)$, then $\fg=\mathfrak{sp}(2n)$ and
    $$d_S(^{\langle \nu\rangle} \lambda ) = ((\nu \cup l(\eta)_C)^t)_C.$$
    \item[(iii)] If $\fg^{\vee}=\mathfrak{so}(2n)$, then $\fg=\mathfrak{so}(2n)$ and
    $$d_S(^{\langle \nu\rangle} \lambda ) = ((\nu \cup ((\eta^t)_D)^t)^t)_D.$$
\end{itemize}
\end{prop}

Now assume that $(\OO^\vee, \bar{C})$ is a distinguished pair in $\LA(G^{\vee})$. We will choose lifts $\tilde{\theta}_i \in A(\OO^\vee)$ of $\bar{\theta}_i \in \bar{A}(\OO^\vee)$ different from those of Definition \ref{defn:Abar_basis} as follows.

\begin{definition}\label{defn:Abar_basis_tilde}
Let $G^\vee=G_\epsilon(m)$. Define the following elements in $A(\OO^\vee)=A^\epsilon$:
	\begin{itemize}
		\item 
		  $\tilde{\theta}_i := \upsilon_{\lambda_{2i-1}} \upsilon_{\lambda_{2i}}$ for $1 \leqslant i \leqslant  \frac{r-1}{2}$, if $\fg^\vee$ is of type $B$;
		\item 
		  $\tilde{\theta}_i := \upsilon_{\lambda_{2i}} \upsilon_{\lambda_{2i+1}}$ for $1 \leqslant i \leqslant  \lfloor \frac{r}{2} \rfloor$, if $\fg^\vee$ is of type $C$;
		\item 
		  $\tilde{\theta}_i := \upsilon_{\lambda_{2i}} \upsilon_{\lambda_{2i+1}}$ for $1 \leqslant i \leqslant \frac{r}{2} - 1$, if $\fg^\vee$ is of type $D$.
	\end{itemize}
	Note that when $\fg^\vee$ is of type $C$, by our convention $\tilde{\theta}_{k} = \upsilon_{\lambda_{2k}} \upsilon_0 = \upsilon_{\lambda_{2k}}$ if $r = 2k $ or $2k+1$. This induces a splitting $\tilde{s}: \bar{A}(\OO^\vee) \hookrightarrow A(\OO^\vee)$ defined by $\tilde{s}(\bar{\theta}_i) = \tilde{\theta}_i$. 
\end{definition}

\begin{definition}\label{defn:C0}	
    Suppose $(\OO^\vee, \bar{C})$ is a distinguished pair in $\LA(G^{\vee})$. Set $C_0 := \tilde{s}(\bar{C}) \in A(\OO^\vee)=A^\epsilon$. The marked partition associated to $C_0$ (via Corollary \ref{cor:conj_classical}) is denoted by $\prescript{\langle \nu_0 \rangle}{}{\lambda}$ with $\lambda = \nu_0 \cup \eta_0$.
\end{definition}

We now simplify the formulas in Proposition \ref{prop:Sommersdualityclassical} in the case when $(\OO^\vee, \bar{C})$ is distinguished. We first introduce a new operation on partitions.

\begin{definition}\label{uparrow}
For any distinguished partition $p = [p_1, p_2, \ldots , p_l]$ such that all $p_i$ have the same parity, define $p^{\uparrow} := [p^{\uparrow}_1, p^{\uparrow}_2, \ldots, p^{\uparrow}_l]$ to be the partition such that $p^{\uparrow}_i = p_i +1$ for odd $i$ and $p^{\uparrow}_i = p_i -1$ for even $i$. (In terms of Young diagrams, what we do here is to move one box from the second row to the first one, from the fourth row to the third one, and so on.)
\end{definition}

\begin{lemma}\label{lem:d_S_dist}
  Suppose $(\OO^\vee, \bar{C})$ is a distinguished pair in $\LA(G^{\vee})$ and $\prescript{\langle \nu \rangle}{}{\lambda}$ is the corresponding reduced marked partition. Then we have 
	\begin{itemize}
		\item 
		  $d_S (\prescript{\langle \nu \rangle}{}{\lambda}) = (\nu_0 \cup {l(\eta_0)}_C)^t$ if $\fg^\vee$ is of type $B$;
		\item 
		  $d_S (\prescript{\langle \nu \rangle}{}{\lambda}) = (\nu_0 \cup {\eta_0^+}_B)^t$ if $\fg^\vee$ is of type $C$;  
		\item 
		  $d_S (\prescript{\langle \nu \rangle}{}{\lambda}) = (\nu_0 \cup \eta_0^\uparrow)^t$ if $\fg^\vee$ is of type $D$.
	\end{itemize}
\end{lemma}

\begin{proof}
   One can argue by induction using block decomposition in \cref{sec:basic block} and \cref{prop:block division of dual}, then apply Proposition \ref{prop:Sommersdualityclassical} and Corollary \ref{cor:dist_achar}. We leave the details to the reader.
\end{proof}

\subsection{Block decompositions}\label{sec:basic block}

If $^{\langle \nu\rangle} \lambda$, $^{\langle \nu^1\rangle} \lambda^1$ and $^{\langle \nu^2\rangle} \lambda^2$ are marked partitions, we write $^{\langle \nu\rangle} \lambda=^{\langle \nu^1\rangle} \lambda^1\cup ^{\langle \nu^2\rangle} \lambda^2$ if $\lambda=\lambda^1\cup \lambda^2$ and $\nu=\nu^1\cup \nu^2$. We say that $\lambda^1$ is \emph{evenly} (resp. \emph{oddly}) \emph{superior} to $\lambda^2$ if there is an even (resp. odd) integer $m$ such that $\lambda^1_{\#\lambda^1} \geq m \geq \lambda^2_1$.

\begin{definition}\label{defn:block}
Let $^{\langle \nu\rangle} \lambda$ be a marked partition of type $B$ (resp. $C$, $D$). A \emph{block decomposition} of $^{\langle \nu\rangle} \lambda$ is a decomposition of $^{\langle \nu\rangle} \lambda$ into marked partitions
$$^{\langle \nu\rangle} \lambda = ^{\langle \nu^1\rangle} \lambda^1 \cup ... \cup ^{\langle \nu^k\rangle} \lambda^k$$
such that
\begin{itemize}
    \item[(i)] $^{\langle \nu^i\rangle} \lambda^i$ is a marked partition of type $D$ (resp. $C$, $D$) for $i > 1$ and $^{\langle \nu^1 \rangle} \lambda^1$ is a marked partition of type $B$ (resp. $C$, $D$).
    \item[(ii)] If $\fg^\vee$ is of type $C$, $\#\lambda^i$ is even for $1 \leq i \leq k-1$ and $\#\nu^i$ is even for all $i$, where the last part of $\nu^k$ is allowed to be $0$.
    \item[(iii)] $\lambda^i$ is evenly (resp. oddly, evenly) superior to $\lambda^{i+1}$ for all $i$.
\end{itemize}
We say that a marked partition $^{\langle \nu\rangle} \lambda$ of type $B$ (resp. $C$, $D$) is a \emph{basic block} if $\nu$ has two elements, namely the smallest part of $\lambda$ and the largest part of $\lambda$ of odd (resp. even, even) height, or if $X=C$ and $\nu$ is a singleton consisting of the largest odd part of $\lambda$ of even height. 
    \end{definition}

    \begin{prop}({\cite[Proposition 4.11]{Achar2003}})\label{division into basic blocks}
       Every reduced marked partition admits a block decomposition
       $$^{\langle \nu\rangle} \lambda = ^{\langle \nu^1\rangle} \lambda^1 \cup ... \cup ^{\langle \nu^k\rangle} \lambda^k$$
       such that each for each $i$, either $\nu^i = \emptyset$ or $^{\langle \nu^i\rangle} \lambda^i$ is a basic block.
    \end{prop}

    We call the block decomposition in \cref{division into basic blocks} a {\it block decomposition into basic blocks}.
    
    \begin{rmk}\label{rmk:block_special}
    It is clear from the definitions that each block in a block decomposition of a special marked partition is special.\end{rmk}

    For a partition $\lambda$, let $\lambda_- = l(\lambda^t)^t$. The following is immediate from \cite[Proposition 4.9]{Achar2003}. 
    
    \begin{prop}\label{prop:block division of dual}
    Let $X \in \{B,C,D\}$ and let $^{\langle \nu\rangle} \lambda \in \overline{\cP}_X(m)$. Suppose that $^{\langle \nu\rangle} \lambda=^{\langle \nu^1\rangle} \lambda^1\cup ^{\langle \nu^2\rangle} \lambda^2\cup \ldots \cup ^{\langle \nu^k\rangle} \lambda^k$ is a block decomposition. Then
\begin{itemize}
    \item[(i)] If $X=B$, then 
    $$d_S(^{\langle \nu\rangle}\lambda)=d_S(^{\langle \nu^1\rangle}\lambda
^1) \vee d_S(^{\langle \nu^2\rangle}\lambda^2) \vee \ldots \vee d_S(^{\langle \nu^k\rangle}\lambda^k).$$
\item[(ii)] If $X=C$, then
$$d_S(^{\langle \nu\rangle}\lambda)=d_S(^{\langle \nu^1\rangle}\lambda^1)_{\mathunderscore} \vee d_S(^{\langle \nu^2\rangle}\lambda^2)_{\mathunderscore} \vee \ldots \vee d_S(^{\langle \nu^{k-1}\rangle}\lambda^{k-1})_{\mathunderscore} \vee  d_S(^{\langle \nu^k\rangle}\lambda^k).$$
\item[(iii)] If $X=D$, then
$$d_S(^{\langle \nu\rangle}\lambda)=d_S(^{\langle \nu^1\rangle}\lambda^1) \vee d_S(^{\langle \nu^2\rangle}\lambda^2) \vee \ldots \vee d_S(^{\langle \nu^k\rangle}\lambda^k)$$
\end{itemize}
    \end{prop}

\subsection{Unipotent infinitesimal characters}

Let $G$ be a simple classical group of type $B$, $C$, or $D$, and let $\widetilde{\OO} \in \Cov(G)$ be a birationally rigid nilpotent cover. In this section, we will recall formulas from \cite[Section 8.2]{LMBM} computing the unipotent infinitesimal character $\gamma(\widetilde{\OO})$ attached to $\widetilde{\OO}$ in terms of the partition for $\OO$.

\begin{definition}[Definition 8.2.1, \cite{LMBM}]
Suppose $q = [q_1, q_2, \ldots, q_l]$ is a partition of $n$. Define $\rho^+(q) \in \left(\frac{1}{2} \ZZ \right)^{\lfloor \frac{n}{2} \rfloor}$ by appending the \emph{positive} elements of the sequence
 \[ \left(  \frac{q_i - 1}{2}, \frac{q_i - 3}{2}, \ldots, \frac{3 - q_i}{2}, \frac{1 - q_i}{2}  \right) \]
for each $i \geqslant 1$, and then adding $0$'s if necessary so that the length of the sequence $\rho^+(q)$ equals $\lfloor \frac{n}{2} \rfloor$.
\end{definition}

\begin{rmk}\label{rem:rho^+}
If $p$ is a partition with all even members and $q$ is a partition with all odd members, we have $\rho^+(p \cup q) = \rho^+(p) \cup \rho^+(q)$.
\end{rmk}

\begin{definition}[Definition 8.2.2, \cite{LMBM}]
Let $q \in \mathcal{P}(n)$ be a partition. Define $f_0(q) \in \mathcal{P}(n)$ as follows: for every odd $i$ with $q_i \geqslant q_{i+1} + 2$, replace $[q_i, q_{i+1}]$ in $q$ by $[q_i - 1, q_{i+1} + 1]$. Define $f_1(q) \in \mathcal{P}(n+1)$ as follows: for every even $i$ with $q_i \geqslant q_{i+1} + 2$, replace $[q_i, q_{i+1}]$ in $q$ by $[q_i - 1, q_{i+1} + 1]$ and finally replace $q_1$ by $q_1 + 1$. If $q=\emptyset$ is the empty partition, define $f_1(q)=(1)$.
\end{definition}

\begin{definition}[Definition 8.2.7, \cite{LMBM}]
Let $q$ be a partition. Let $x(q)$ be the subpartition of $q$ consisting of all multiplicity $1$ parts and let $y(q)$ be the subpartition of $q$ consisting of all multiplicty $2$ parts.

Suppose y is a partition such that every part of it has multiplicity 2. Define a partition $g(y)$ (of the same size as $y$) by replacing every pair $[y_i, y_i]$ with $[y_i+1, y_i-1]$.    
\end{definition}

\begin{prop}[Proposition 8.2.8, \cite{LMBM}]\label{prop:gamma_cover}
Let $p$ be the partition corresponding to $\OO$. Form the partitions $x=x(p^t)$ and $y=y(p^t)$. Then for $G=G_{\epsilon}(m)$ the infinitesimal character $\gamma(\widetilde{\OO})$ is given by the following formula

$$\gamma(\widetilde{\OO}) = \rho^+(g(y) \cup f_\epsilon(x)).$$

\end{prop}

\begin{lemma}\label{lem:gamma_0}
	Suppose $\fg$ is of classical type and $(\OO^\vee, \bar{C})$ is a distinguished pair in $\LA(G^{\vee})$ with the corresponding reduced marked partition $\prescript{\langle \nu \rangle}{}{\lambda}$. Then for the cover $\widetilde{\OO} = D(\OO^\vee, \bar{C})$, we have $\gamma(\widetilde{\OO}) = \rho^+(\nu^\uparrow_0 \cup \eta_0)$.
\end{lemma}

\begin{proof}
Suppose $\fg^\vee$ is of type $B$, so that $\fg$ is of type $C$. Let $p$ be the transpose of the partition of $d_S (\prescript{\langle \nu \rangle}{}{\lambda})$. Then by Lemma \ref{lem:d_S_dist}, $p = \nu_0 \cup {\eta_0^-}_C$.  Partition of this form has the property that any member of it has multiplicity at most 2, and furthermore, any multiplicity-2 part (if exists) is of the form $p_i = p_{i+1}$, where $i$ is odd. Therefore the formula for $\gamma(\widetilde{\OO})$ in \cref{prop:gamma_cover} (ii) simplifies to 
  \[ \rho^+(f_1(p^t)) = \rho^+ (f_1 (\nu_0 \cup {\eta_0^-}_C)) = \rho^+(\nu^\uparrow_0 \cup \eta_0), \]
where the last equality follows from Definition \ref{defn:Abar_basis_tilde}.
The arguments for type $B$ and $D$ are similar.
\end{proof}

\section{Proofs of main results in classical types}\label{sec:proofsclassical}

\subsection{Proof of Proposition \ref{prop:distinguishedbirigid}}\label{subsec:proofsclassical1}

\begin{prop}\label{prop:distinguishedbirigidclassical}
Suppose $G$ is a simple adjoint group of classical type. Then the following are true:
\begin{itemize}

    \item[(i)] Suppose $(\OO^{\vee},\bar{C}) \in \LA^*(G^{\vee})$ is a special distinguished Lusztig-Achar datum. Let $\OO=d_S(\OO^{\vee},\bar{C})$ and let $\widetilde{\OO}$ denote the 
    Lusztig cover of $\OO$ (see Definition \ref{def:Lusztigcover}). Then $\widetilde{\OO}$ is birationally rigid.
    \item[(ii)] The map $d_S: \LA^*(G^{\vee}) \to \Orb(G)$ is injective when restricted to the set of special distinguished Lusztig-Achar data.
\end{itemize}
\end{prop}

\begin{proof}
First suppose $\fg^{\vee}=\mathfrak{sl}(n)$. By Proposition \ref{prop:markedpartitionsAchardata}, there is a unique distinguished element in $\LA^*(G^{\vee})$, namely $(\OO_{prin},1)$, and $d_S(\OO_{prin},1) =d(\OO_{prin}) = \{0\}$. Now (i) and  (ii) are immediate.

Next, suppose $\fg^{\vee} = \mathfrak{so}(2n+1)$ (the other cases are completely analogous and are left to the reader). We will first prove (ii). First, we observe that if $(\OO^\vee, \Bar{C})$ is special and distinguished, then $\OO^\vee$ must be special. For this, it is sufficient to show that the partition of $\OO^\vee$ contains no even parts (any such orbit is even and thus special). Suppose, to the contrary, that the partition of $\OO^{\vee}$ contains an even part $k$. Then by \cref{prop:inclusionclassical} $\OO^\vee$ is saturated from an orbit $\OO_L^\vee$ in the Levi subalgebra $\fl^\vee=\fg\fl(k)\times \fs\fo(2n+1-2k)$, and by \cref{cor:Abarclassical}$, \Bar{A}(\OO^\vee)\simeq \Bar{A}(\OO_{L^\vee}^\vee)$. Thus, $(\OO^\vee, \Bar{C})$ is not distinguished, a contradiction. We conclude that $\OO^{\vee}$ is special, as asserted.

Now suppose $(\OO^\vee_1, \Bar{C}_1)$ and $(\OO^\vee_2, \Bar{C}_2)$ are special distinguished Lusztig-Achar data such that $d_S(\OO^\vee_1, \Bar{C}_1)=d_S(\OO^\vee_2, \Bar{C}_2)=\OO$. By \cite[Remark 14]{Sommers2001}, $\OO$ is in the special piece of both $d(\OO^\vee_1)$ and $d(\OO^\vee_2)$. Since both $\OO_1^{\vee}$ and $\OO_2^{\vee}$ are special (in the sense of Lusztig), it follows that $\OO^\vee_1=\OO^\vee_2=d(\OO)$ and hence by Lemma \ref{lem:AtoAbar} that $\Bar{C}_1=\Bar{C}_2$. Therefore, $(\OO^\vee_1, \Bar{C}_1)=(\OO^\vee_2, \Bar{C}_2)$. This completes the proof of (ii).

We now proceed to proving (i). Let $^{\langle \nu \rangle}\lambda \in \overline{\cP}_B(2n+1)$ be the reduced marked partition corresponding to $(\OO^{\vee},\bar{C}) \in \LA(G^{\vee})$, cf. Proposition \ref{prop:markedpartitionsAchardata} and let $\pi = d_S(^{\langle \nu \rangle}\lambda) \in \mathcal{P}_C(2n)$, the partition corresponding to $\OO$.
We need to show that $\widetilde{\OO}$ is birationally rigid. By \cref{prop:criterionbirigidcover}, this is equivalent to showing that $\Spec(\CC[\widetilde{\OO}])$ has no codimension 2 leaves and $H^2(\widetilde{\OO},\CC)=0$. For the former, we need to check the two conditions in \cref{prop:2leafless_classical_cover}.
By \cref{lem:d_S_dist}, $\pi = (\nu_0 \cup {l(\eta_0)}_C)^t$. Therefore columns $c_j$ of $\pi$ are just rows of the partition $\pi^t = \nu_0 \cup {l(\eta_0)}_C$. Since $\nu_0$ and $\eta_0$ are distinguished, it is easy to verify that $m_{\pi^t}(c_j) \leq 2$ for all $c_j$. Now if $m_{\pi^t}(c_j) = 2$, then $c_j \in {l(\eta_0)}_C$. Since $\eta_0$ has odd number of members and all members of $\eta_0$ are odd, we deduce that all members of ${l(\eta_0)}_C$ are even and so is $c_j$. Moreover, the height of $c_j$ in ${l(\eta_0)}_C$ 
 is even. But the definition of $\nu_0$ and $\eta_0$ (\cref{defn:C0}), the height of any $x \in {l(\eta_0)}_C$ in ${l(\eta_0)}_C$ and the height of $x$ in $\pi^t$ has the same parity. This means that $c_j$ is a markable column of $\pi$ by \cref{rmk:Abar_column} and condition (i) of \cref{prop:criterionbirigidcover} is satisfied. Condition (ii) also follows immediately from \cref{rmk:Abar_column}. 

Now we check that $H^2(\widetilde{\OO},\CC)=0$ using \cref{lem:H2_classical_cover}. Assume that $c_j = c^\epsilon_k$ is a column of $\pi$ with even height in $\pi^t= \nu_0 \cup {l(\eta_0)}_C$, such that $c_j = c_{j-1} + 2$. Then any such part  $c_j = c^\epsilon_k$ must belong to $\nu_0$, and hence is odd. This is because $\eta_0$ is distinguished, so if $c_j \in {l(\eta_0)}_C$, $c_j = c_{j-1}+2$ would imply $c_j$ is of odd height in $\pi^t$, which yields contradiction. Therefore $c^\epsilon_k$ is not a markable column for $\pi$. As in \cref{rmk:Abar_column}, assume that the markable columns of $\pi$ are $(c^m_l > c^m_{l-1} > \cdots > c^m_{1} > 0)$ and $c^m_{q-1} < c_j=c^\epsilon_{k} < c^m_{q}$ for some $1 \le q < l+1$. Then we can take $\theta=\vartheta_{c^\epsilon_{k}} \vartheta_{c^m_{q}} \in N$, which fullfills the condition in \cref{lem:H2_classical_cover}. This finishes the proof that $H^2(\widetilde{\OO},\CC)=0$.

\end{proof}

\subsection{Proof of Theorem \ref{thm:inflchars}}\label{subsec:proofsclassical2}

Since $(\OO^\vee, \bar{C})$ is distinguished, the set $S(\OO^\vee,  \bar{C}) = \LA^{-1}(\OO^\vee, \bar{C}) \subset \fh_{\RR}^*$ is already $W$-invariant and therefore can be regarded as a subset of $\fh^* / W$. Now Theorem \ref{thm:inflchars} for classical types will follow from Lemma \ref{lem:gamma_0} and the following result (for the definition of $\nu_0$ and $\eta_0$, see \cref{defn:C0}). 

\begin{theorem}\label{thm:R_cbar_min} 
Suppose $G$ is a simple group of classical type. Let $(\OO^\vee, \bar{C}) \in \LA(G^{\vee})$ be a distinguished Lusztig-Achar datum corresponding to a reduced marked partition $\prescript{\langle \nu \rangle}{}{\lambda}$. Then $\rho^+(\nu^\uparrow_0 \cup \eta_0)$ is the unique minimal-length $W$-orbit in $S(\OO^\vee, \bar{C})$.
\end{theorem}

The remainder of this section is dedicated to the proof of Theorem \ref{thm:R_cbar_min}. 
First of all, recall that a Lusztig-Achar datum $(\OO^{\vee},\bar{C})$ is specified by indicating the corresponding equivalence class $\{(M^{\vee}_1,\OO_{M^{\vee}_1}),...,(M^{\vee}_k,\OO_{M^{\vee}_k})\}$ of Sommers data,
see Lemma \ref{Lem:Sommers_data} and the discussion preceding it. By Lemma \ref{Lem:Sommers_data}(iii), $(\OO^{\vee},\bar{C})$ is distinguished if and only if one (equivalently, all) of the pseudo-Levi subgroups $M^\vee_1,...,M^\vee_k$ is of maximal semisimple rank. In this case, $(M^{\vee}_i,\OO_{M^{\vee}_i})$ correspond via the bijection $\Som(G^\vee) \xrightarrow{\sim} \Conj(G^\vee_{ad})$ to the distinguished pairs $(\OO^\vee, C_i)$ in the preimage of $(\OO^\vee,\bar{C})$ under the projection $\Conj(G^\vee_{ad}) \twoheadrightarrow \LA(G^\vee)$. For any distinguished pair $(\OO^\vee, C) \in \Conj(G^\vee_{ad})$, set $S(\OO^\vee, C) \subset \fh^*$ to be the preimage of the composition 
$$ \fh^* \xrightarrow{\Som}  \Som(G^\vee_{ad}) \xrightarrow{\sim} \Conj(G^\vee_{ad}), $$
where the map $\Som$ is given by 
\[ \Som: \fh^* \to  \Som(G^\vee_{ad}), \quad \gamma \mapsto (L^{\vee}_{\gamma}, \Ind^{L^{\vee}_{\gamma}}_{L^{\vee}_{\gamma,0}}\{0\}), \] 
where $L^{\vee}_{\gamma,0} = Z_{G^{\vee}}(\gamma)$ and $L^{\vee}_{\gamma} = Z_{G^{\vee}}(\exp(2\pi i \gamma))^{\circ}$. Then $S(\OO^\vee,  \bar{C})$ is the union of all $S(\OO^\vee, C_i)$ where $(\OO^\vee, C_i)$ are all the lifts of $(\OO^\vee,\bar{C}) \in \LA(G^\vee)$ in $\Conj(G^\vee_{ad})$.

Now fix a Cartan subalgebra $\fh^\vee \simeq \fh^*$ of $\fg^\vee$, the root system $\Delta^\vee$, and the coordinates of $\{e_i\}$ as in Section \ref{subsection:Levisclassical}. For a distinguished pair $(\OO^\vee, C) \in \Conj(G^\vee_{ad})$, let $(L^\vee, \OO^\vee_{\fl^\vee})$, denote the corresponding element in $\Som(G^\vee)$, where $L^\vee$ has Lie algebra $\fl^\vee$. Then by Corollary \ref{cor:dist_conj_Som}, $L^\vee$ is of maximal semisimple rank and so we can assume that $\fl^\vee$ is of the form $\fl^\vee_1 \times \fl^\vee_2$ as in Section \ref{subsec:Maximalpseudo-Leviclassical}, where $\fl^\vee_i$ are simple classical Lie algebras, $i=1, 2$, so that the Cartan subalgebra $\fh_1 \simeq \CC^k$ of $\fl_1$ has coordinates $\{ e_1, e_2, \ldots e_k \}$, and the Cartan subalgebra $\fh_2 \simeq \CC^{n-k}$ of $\fl_2$ has coordinates $\{ e_{k+1}, e_{k+2}, \ldots e_n \}$. Moreover, we can write $\OO^\vee_{\fl^\vee} = \OO^\vee_{\lambda^1} \times \OO^\vee_{\lambda^2}$, where $\OO^\vee_{\lambda^i}$ is a distinguished orbit in $\fl^\vee_i$ whose corresponding partition is denoted as $\lambda^i$, $i=1, 2$. We also write $\nu = \lambda^1$, $\eta = \lambda^2$, so that $(\OO^\vee, C)$ corresponds to the marked partition $\prescript{\langle \nu \rangle}{}{\lambda}$. 

By the discussions above, for an element $\gamma \in \fh^*$ to lie in $S(\OO^\vee,  C)$, we need $\fl^\vee_\gamma = \fl^\vee$ possibly after some $W$-conjugation, and additionally $\Ind^{L_{\gamma}^{\vee}}_{L_{\gamma,0}^{\vee}} \{0\}  = \OO^\vee_{\fl^\vee}$. Set
\[ \gamma_1 := (e_1(\gamma), e_2(\gamma), \ldots, e_k(\gamma)) \in \fh^\vee_1 \quad \text{and} \quad \gamma_2 := ( e_{k+1}(\gamma), e_{k+2}(\gamma), \ldots, e_n (\gamma) ) \in \fh^\vee_2.  \]
Additionally, $\gamma$ satisfies $\Ind^{L_{\gamma}^{\vee}}_{L_{\gamma,0}^{\vee}} \{0\}  = \OO^\vee_{\fl^\vee}$. Let $\fl^\vee_{\gamma_i, 0} \subset \fl^\vee_i$ be the Levi subalgebra of $\fl^\vee_i$ determined by the singular datum of $\gamma_i$, for $i=1, 2$.  We have a decomposition $\fl^\vee_{\gamma,0} = \fl^\vee_{\gamma_1, 0} \times \fl^\vee_{\gamma_2, 0}$. Now the condition $\Ind^{L_{\gamma}^{\vee}}_{L_{\gamma,0}^{\vee}} \{0\}  = \OO^\vee_{\fl^\vee}$ amounts to $\Ind^{L_i^{\vee}}_{L_{\gamma_i,0}^{\vee}} \{0\}  = \OO^\vee_{\lambda^i}$ for $i=1, 2$. 

We now examine separately the cases of type $B$, $C$, and $D$. 
\vskip 0.5em
\noindent 1. Type $B_n$: $\fg^\vee = \mathfrak{so}(2n+1)$ ($n \geqslant 3$). It is easy to see that the condition $\fl^\vee_\gamma = \fl^\vee$ is equivalent to the conditions that $e_i (\gamma) \in \frac{1}{2} + \ZZ$ for $1 \leqslant i \leqslant k$ and $e_i (\gamma) \in \ZZ$ for $k+1 \leqslant i \leqslant n$. Therefore, $S(\OO^\vee,  C)$ coincides with the image of the subset
  \begin{equation} \label{eq:R_BD}
  	 S_0(\prescript{\langle \nu \rangle}{}{\lambda}) := W \cdot \left\{ \gamma = (\gamma_1, \gamma_2) \in \fh^\vee \,|\,  \Ind^{L_i^{\vee}}_{L_{\gamma_i,0}^{\vee}} \{0\}  = \OO^\vee_{\lambda^i}, \gamma_1 \in  \left( \tfrac{1}{2} + \ZZ \right)^k, \gamma_2 \in \ZZ^{n-k}  \right\}, 
\end{equation}
under $\fh^* \to \fh^*/W$.
\vskip 0.5em
\noindent 2. Type $D_n$: $\fg^\vee = \mathfrak{so}(2n)$ ($n \geqslant 4$). 
The condition that $\fl^\vee_\gamma = \fl^\vee$ is equivalent to requiring either $\gamma_1 \in (\frac{1}{2} + \ZZ)^k$ and $\gamma_2 \in \ZZ^{n-k}$, or $\gamma_1 \in  \ZZ^{k}$ and $\gamma_2 \in (\frac{1}{2} + \ZZ)^{n-k}$. Note that by switching the roles of $\nu$ and $\eta$, we can reduce the latter case to the former. If we still define $S_0(\prescript{\langle \nu \rangle}{}{\lambda})$ as in \eqref{eq:R_BD}, then $S(\OO^\vee,  C)$ is the union of $S_0(\prescript{\langle \nu \rangle}{}{\lambda})$ and $S_0(\prescript{\langle \eta \rangle}{}{\lambda})$.
\vskip 0.5em
\noindent 3. Type $C_n$: $\fg^\vee = \mathfrak{sp}(2n)$ ($n \geqslant 2$). 
This is similar to the case of type $D$, but this time we only consider the condition that $\gamma_1 \in  \ZZ^{k}$ and $\gamma_2 \in (\frac{1}{2} + \ZZ)^{n-k}$, for the convenience of writing. 
\vskip 0.5em

To treat all types uniformly, write $k_1 = k$, $k_2 = n-k$, and define the set
  \begin{equation} \label{eq:R_BCD}
    S_\epsilon(\prescript{\langle \nu \rangle}{}{\lambda}) := W \cdot \left\{ \gamma = (\gamma_1, \gamma_2) \in \fh^\vee \,|\,  \Ind^{L_i^{\vee}}_{L_{\gamma_i,0}^{\vee}} \{0\}  = \OO^\vee_{\lambda^i}, \gamma_i \in  \left( \tfrac{\epsilon+i}{2} + \ZZ \right)^{k_i}, i=1, 2 \right\} 
\end{equation}
for $\epsilon \in \{0, 1\}$, and for any marked partition $\prescript{\langle \nu \rangle}{}{\lambda} \in \tilde{\cP}_\epsilon(m)$.  The discussion above amounts to a proof of the following lemma.

\begin{lemma} \label{lem:R_e_min}
    The set $S(\OO^\vee,  \bar{C})$ is the union of $S_\epsilon(\prescript{\langle \nu \rangle}{}{\lambda})$, where $\prescript{\langle \nu \rangle}{}{\lambda}$ runs over all marked partitions in $\tilde{\cP}_\epsilon(m)$ which are mapped to $(\OO^\vee,  \bar{C})$.
\end{lemma}

Therefore we should first find elements of minimal length in each set $S_\epsilon(\prescript{\langle \nu \rangle}{}{\lambda})$. Before stating the result, we first introduce some notations.

\begin{definition}\label{defn:cut_partition}
	For any partition $p = [p_1, p_2, \ldots, p_l]$ and any $0 \leqslant k \leqslant l$, set $p_{\leqslant k} := [p_1, \ldots, p_k]$ to be the subpartition of $p$ consisting of the first $k$ rows and set $p_{>k} := [p_{k+1}, \ldots, p_l]$ to be the subpartition of $p$ consisting of the last $l-k$ rows. Define $p^{-1} := [p_1 -1, p_2-1, \ldots, p_l - 1]$, i.e., the partition obtained from $p$ by subtracting $1$ from each part.
\end{definition}

\begin{prop}\label{prop:R_c_min}
    For any marked partition $\prescript{\langle \nu \rangle}{}{\lambda}$ such that $\lambda \in \mathcal{P}_{\epsilon}(m)$, the sequence $\rho^+(\nu^\uparrow \cup \eta)$ gives an element of minimal length in the set $S_{\epsilon}(\prescript{\langle \nu \rangle}{}{\lambda})$. When $\fg^\vee$ is of type $B$ or $C$, any element of minimal length in $S_{\epsilon}(\prescript{\langle \nu \rangle}{}{\lambda})$ is $W$-conjugate to $\rho^+(\nu^\uparrow \cup \eta)$. When $\fg^\vee$ is of type $D$, this is also true except for the case when $\eta = \emptyset$. In that case, there are exactly two $W$-orbits with minimal length.
\end{prop}

\begin{proof}
    We treat each type separately.
	\vskip 0.5 em
    \noindent {\bf Types $B$ and $D$.} Suppose $\fg^\vee = \mathfrak{so}(2n + \delta)$, where $2n+ \delta \geqslant 7$ and $\delta = 0$ or $1$ depending on whether $\fg^\vee$ is of type $D$ or $B$. We adopt the same notations introduced before Lemma \ref{lem:R_e_min}. We can choose $\gamma_1$ and $\gamma_2$ independently to minimize the length of $\gamma= (\gamma_1, \gamma_2)$. By $L^\vee_i$-conjugation we can assume that $\fl^\vee_{\gamma_i,0}$ is a standard Levi sublagebra in $\fl^\vee_i$ with respect to choice of simple roots mentioned above, for $i=1, 2$.
    The subalgebra $\fl^\vee_1$ is of type $D$. Assume $\fl^\vee_{\gamma_1, 0} = (\fl^\vee_1)_{J_1}$, where $J_1 \subset I_1$. If $J_1$ contains both the positive roots $e_1 - e_2$ and the lowest root $-e_1 - e_2$, then $e_1(\gamma) = e_2(\gamma) = 0$, which contradicts the condition $e_1(\gamma) \in \frac{1}{2} + \ZZ$.  Therefore $J_1$ can only contain at most one of the two roots $e_1 - e_2$ and $-e_1 - e_2$ and hence $\fl^\vee_{\gamma_1,0}$ is isomorphic to a product of factors of type $A$.
    
    First consider the case when the lowest root $-e_1 - e_2$ is not contained in $J_1$.  In this case, $\fl^\vee_{\gamma_1,0}$ is uniquely determined by a partition $q_1,  q_2, \ldots, q_l$ of $k$, such that $J_1 = I_1 \setminus \bigcup_{i=1}^{l-1} \{ e_{b_i} - e_{b_i + 1} \}$, where $b_i = \sum_{j=1}^{i} q_j$ (note $b_0 = 0$). Then $\fl^\vee_{\gamma_1, 0} \simeq \mathfrak{gl}(q_1) \times \mathfrak{gl}(q_2) \times \cdots \times \mathfrak{gl}(q_l)$. By permuting the coordinates $e_1, \dots, e_k$ (this corresponds to conjugating $\fl^\vee_{\gamma_1, 0}$ by $L^\vee_1$), we may assume that the sequence $q_1, \ldots, q_l$ is non-increasing. Define the partition $q = (q_1, q_1, q_2, q_2, \ldots, q_l, q_l)$ in terms of columns. Then the condition $\OO^\vee_{\lambda^1} = \OO^\vee_{\nu} = \Ind_{L^\vee_{\gamma_1,0}}^{L^\vee_1} \{0\}$ just means that $\nu = q_D$, the $D$-collapse of $q$. Since $\nu$ is a distinguished partition with all odd members and $\#\nu$ is even, it is not hard to see that $q = \nu^\uparrow$. 
    
    The constraint on $\gamma_1$ determined by $\fl^\vee_{\gamma_1,0}$ consists of a family of equalities 
    \[ e_{b_i +1} (\gamma) = e_{b_i+2} (\gamma) = \cdots = e_{b_{i+1}}(\gamma), \quad 0 \leqslant i \leqslant l-1. \] 
    It is not hard to see that all $\gamma_1$ of minimal length that satisfy these equalities are of the form       
    \begin{equation} \label{eq:gamma1_signed_frac}
    	\bigg( \underbrace{ \frac{1}{2} \epsilon_1 , \,\ldots\,,  \frac{1}{2}}_{q_1} \epsilon_1, \underbrace{ \frac{3}{2} \epsilon_2 , \,\ldots\,,  \frac{3}{2}}_{q_2} \epsilon_2, \,\ldots\,, \underbrace{ \frac{2l-1}{2} \epsilon_l, \,\ldots\,, \frac{2l-1}{2}}_{q_l} \epsilon_l \bigg). 
    \end{equation}
    where $\epsilon_i = \pm 1$ for $1 \leqslant i \leqslant l$. After signed permutations, they become
    \begin{equation} \label{eq:gamma1_pos_frac}
    	\gamma_1^+ = \bigg( \underbrace{\frac{1}{2}, \,\ldots\,, \frac{1}{2}}_{q_1} , \underbrace{\frac{3}{2}, \,\ldots\,, \frac{3}{2}}_{q_2}, \,\ldots\,, \underbrace{\frac{2l-1}{2}, \,\ldots\,, \frac{2l-1}{2}}_{q_l} \bigg). 
    \end{equation}
    It is straightforward to check that \eqref{eq:gamma1_pos_frac} coincides with $\rho^+(q) = \rho^+(\nu^\uparrow)$.  
	
	When the lowest root $-e_1 - e_2$ is contained in $J_1$, $\gamma_1$ can be obtained from \eqref{eq:gamma1_signed_frac} by multiplying the first coordinate by $-1$. 
	The remaining discussions are the same as above and again any allowed $\gamma_1$ with minimal length can be changed to $\rho^+(\nu^\uparrow)$ by a signed permuation. 
	 
	 Next we determine $\gamma_2$. Since $\fl^\vee_2 \simeq \mathfrak{so}(2n - 2k + \delta)$, we may assume $\fl^\vee_{\gamma_2, 0} \simeq \mathfrak{so}(2t+ \delta) \times \mathfrak{gl}(a_1) \times \mathfrak{gl}(a_2) \times \cdots \times \mathfrak{gl}(a_d)$, where $0 \leqslant t \leqslant n-k$ and $a_1, \ldots, a_d$ is a non-increasing sequence of positive integers. Define the partition $q := (a_1, a_1, a_2, a_2, \ldots, a_d, a_d)$ in terms of columns. Add one column of length $2t+ \delta$ gives the partition $\tilde{q} = q \lor (2t+ \delta)$. Then the condition $\OO^\vee_{\lambda^2} = \OO^\vee_\eta = \Ind_{L^\vee_{\gamma_2,0}}^{L^\vee_2} \{0\}$ just means that $\eta$ is the $B$- (or $D$-)collapse of $\tilde{q}$ and $2t+ \delta \leqslant \#\eta$, when $\fg^\vee$ is of type $B$ (or $D$). Set $\eta_{(2t+ \delta)} := (\eta_{\leqslant 2t+ \delta})^{-1} \cup \eta_{>2t+ \delta}^\uparrow$ (see Definition \ref{defn:cut_partition}). Since $\eta$ is distinguished with all odd members, $\eta_{(2t+ \delta)}$ has only even members and it is easy to see that $\tilde{q}_B = \eta$ (or $\tilde{q}_D = \eta$) implies that $\eta_{(2t+ \delta)} = q$ and hence $a_i$'s are uniquely determined by $\eta$ and $t$. Analogous to the case of $\gamma_1$, for fixed $t$, the $\gamma_2$ that minimizes the length can by taken as
        \[ \gamma^+_2 = \big( \underbrace{d, \ldots, d}_{a_d},  \ldots, \underbrace{2, \ldots, 2}_{a_2}, \underbrace{1, \ldots, 1}_{a_1}, \underbrace{0, \ldots, 0}_{t}\big)\]
     and any other choices are can be obtained from $\gamma_2^+$ by signed permutations. Apparently taking $t$ to be the maximal value $\frac{1}{2} (\#\eta- \delta)$ gives the minimal length, so that $\gamma^+_2 = \rho^+(\eta)$. Any other choice of $t$ would give strictly greater length. In this case $q = \eta_{(\#\eta)} = \eta^{-1}$ and $\tilde{q} = \eta$, hence the induction is birational, i.e., $\OO^\vee_{\lambda^2} = \OO^\vee_\eta = \operatorname{Bind}_{L^\vee_{\gamma_2,0}}^{L^\vee_2} \{0\}$ by Proposition \ref{prop:inductionclassical}. 
     
     We conclude that any $\gamma$ of minimal length can be changed to $ (\gamma^+_1, \gamma^+_2)$ by signed permuations, which in turn is equivalent to $\rho^+(\nu^\uparrow) \cup \rho^+(\eta)$ up to permutations. The latter equals $\rho^+(\nu^\uparrow \cup \eta)$ by Remark \ref{rem:rho^+}. When $\fg^\vee$ is of type $B$, this implies that all minimal elements are $W$-conjugate to each other.  When $\fg^\vee$ is of type $D$, the Weyl group consists of only signed permuations that have even number of sign changes. When $\eta \neq \emptyset$, $\gamma_2$ always have at least one coordinate equal to $0$ and hence we have the same conclusion. When $\eta = \emptyset$ so that $\nu = \lambda$, there are exactly two $W$-conjugacy classes of elements with minimal length, one contains $\rho^+(\lambda^\uparrow)$ and the other one contains $\rho^+(\lambda^\uparrow)$ with the first (or any) coordinate changed to its negative\footnote{They become the same orbit if we consider the disconnected Weyl group of $O(2n)$}.
    \vskip 0.5em 
    \noindent {\bf Type $C$.} Let $\fg^\vee = \mathfrak{sp}(2n)$. The condition that $e_n(\gamma) \in \frac{1}{2} + \ZZ$ forces that $\fl^\vee_{\gamma_2, 0} \simeq \mathfrak{gl}(q_1) \times \mathfrak{gl}(q_2) \times \cdots \times \mathfrak{gl}(q_l)$ without any factor of type $C$, where the sequence $q_1,  q_2, \ldots, q_l$ is a partition of $n-k$.  As in the discussions about $\gamma_1$ in type $B$ and $D$, we can assume $q_1,  q_2, \ldots, q_l$ is non-increasing after permutation. Define the partition $q := (q_1, q_1, q_2, q_2, \ldots, q_l, q_l)$ in terms of columns, then the condition $\OO^\vee_{\lambda^2} = \OO^\vee_\eta = \Ind_{L^\vee_{\gamma_2,0}}^{L^\vee_2} \{0\}$ means that $\eta = q$ and hence the induction is birational, i.e., $\OO^\vee_{\lambda^2} = \OO^\vee_\eta = \operatorname{Bind}_{L^\vee_{\gamma_2,0}}^{L^\vee_2} \{0\}$ by Proposition \ref{prop:inductionclassical}. Therefore
    \[ \gamma^+_2 = \rho^+(\eta) = \bigg( \underbrace{\frac{1}{2}, \ldots, \frac{1}{2}}_{q_1} , \underbrace{\frac{3}{2}, \ldots, \frac{3}{2}}_{q_2}, \ldots, \underbrace{\frac{2l-1}{2}, \ldots, \frac{2l-1}{2}}_{q_l} \bigg) \]
	achieve minimal length and all other choices differ from this one by a signed permutation.
	
	Next we determine $\gamma_1$. Since $\fl^\vee_1 \simeq \mathfrak{sp}(2k)$, we may assume $\fl^\vee_{\gamma_2, 0} \simeq \mathfrak{sp}(2t) \times \mathfrak{gl}(a_1) \times \mathfrak{gl}(a_2) \times \cdots \times \mathfrak{gl}(a_d)$, where $0 \leqslant t \leqslant n-k$ and $a_1, \ldots, a_d$ is a non-increasing sequence of positive integers. Define the partition $q:=(a_1, a_1, a_2, a_2, \ldots, a_d, a_d)$ in terms of columns. Add one column of length $2t$ gives the partition $\tilde{q} = q \lor (2t)$. Then the condition $\OO^\vee_{\lambda^1} = \OO^\vee_\nu = \Ind_{L^\vee_{\gamma_1,0}}^{L^\vee_1} \{0\}$ just means that $\nu = \tilde{q}_C$ and $2t \leqslant \#\nu$. Set $\nu^{(2t)} := [(\nu_{\leqslant 2t})^{-1}]^\uparrow \cup \nu_{>2t}$. Since $\nu$ is distinguished with all even members, $\nu^{(2t)}$ has only even members and it is easy to see that $\tilde{q}_C = \nu$ implies that $\nu^{(2t)} = q$ and hence $a_i$'s are uniquely determined by $\nu$ and $t$. With fixed $t$, the choice of $\gamma_1$ which minimizes length is
	\[ \gamma^+_1 = \big( \underbrace{d, \ldots, d}_{a_d},  \ldots, \underbrace{2, \ldots, 2}_{a_2}, \underbrace{1, \ldots, 1}_{a_1}, \underbrace{0, \ldots, 0}_{t}\big)\]
	Clearly taking $t$ to be the maximal value $\lfloor \frac{1}{2} \#\nu \rfloor$ gives the minimal length and $\gamma^+_1  = \rho^+(\nu^\uparrow)$. Any other choice of $t$ would give strictly greater length. Therefore $\gamma = (\gamma^+_1, \gamma^+_2)$ achieves the minimal length, which is equivalent to $\rho^+(\nu^\uparrow \cup \eta)$ up to permutations. Like in the case of type $B$, any other choice with minimal length is $W$-conjugate to $\rho^+(\nu^\uparrow \cup \eta)$. 
\end{proof}

We introduce more notation before continuing with the proof of Theorem \ref{thm:R_cbar_min}.

\begin{definition}\label{defn:2rows}
	Let $q = [q_1, q_2]$ be a partition where $q_1 \geqslant q_2 \geqslant 0$ and $q_1 > 0$ (but $q_2$ can be zero). Define $q^\uparrow := [q_1 + 1, \operatorname{max}(q_2 - 1, 0)]$.
\end{definition}

\begin{proof}[Proof of Theorem \ref{thm:R_cbar_min}]
    Recall that by Proposition \ref{prop:Aclassical}, there is a surjective composite map $A^\epsilon \simeq A(\OO^\vee) \twoheadrightarrow \bar{A}(\OO^\vee)$ of maps. 	Let $\tilde{C}$ be any lift of $\bar{C}$ in $A^\epsilon \subset A$. As usual, by decomposing $\tilde{C}$ uniquely as a product of $\upsilon_{\lambda_i}$, we can associate to $\tilde{C}$ a marked partition $\prescript{\langle \nu \rangle}{}{\lambda} \in \tilde{\cP}_\epsilon (m)$ with $\lambda = \nu \cup \eta$. By Lemma \ref{lem:R_e_min} and Proposition \ref{prop:R_c_min}, it suffices to show that, $\lVert \rho^+(\nu^\uparrow \cup \eta) \rVert > \lVert \rho^+(\nu^\uparrow_0 \cup \eta_0) \rVert$ for any $\tilde{C}$ different from $C_0$ (see \cref{defn:C0} for the definition of $C_0$, $\nu_0$ and $\eta_0$).
	
    By the description of the kernel $N$ of the quotient map $A^\epsilon \twoheadrightarrow \bar{A}(\OO^\vee)$ in \cref{prop: A to barA distinguished}, we see that $\tilde{C}$ can be written as the product of $C_0$ and the element $\prod_{i = 1}^k \zeta_i^{\epsilon_i}$, where $\epsilon_i = 0$ or $1$, and $\zeta_i  := \upsilon_{\lambda_{2i}} \upsilon_{\lambda_{2i+1}}$ (resp. $\zeta_i := \upsilon_{\lambda_{2i-1}} \upsilon_{\lambda_{2i}}$) depending on whether $\fg^\vee$ is of type $B$ (resp. $C/D$), and $k = \frac{1}{2}(\#\lambda - 1 )$ (resp. $\frac{1}{2} \#\lambda$, $\lceil \frac{1}{2} \#\lambda\rceil$) depending on whether $\fg^\vee$ is of type $B$ (resp. $C/D$). Set 
	  \[ C_j := C_0 \prod_{1 \leqslant i \leqslant j} \zeta_i^{\epsilon_i}, \quad0 \leqslant j \leqslant k, \] 
	then $C_{k} = \tilde{C}$ and $C_{j} = C_{j-1} \zeta_j^{\epsilon_i}$. Let $\prescript{\langle \nu_j \rangle}{}{\lambda}$ with $\lambda = \nu_j \cup \eta_j$ denote the marked partition associated to $C_{j}$. 
	 We can divide the change from $C_0$ to $C_k$ into several steps, each time going from $C_{j-1}$ to $C_{j}$ by multiplying $C_{j-1}$ by $\zeta_j^{\epsilon_j}$. If either $\epsilon_j = 0$ or $\zeta_j = 1$, which happens exactly when $\lambda_{2j} = \lambda_{2j+1}$ (resp. $\lambda_{2j-1} = \lambda_{2j}$) if $\fg^\vee$ is of type $B$ (resp. $C/D$), then $C_{j} = C_{j-1}$ and nothing is changed. Otherwise, $C_{j} = C_{j-1} \zeta_j$ and $\zeta_j \neq 1$, so we need to look at the effect on  the partition $\tau_j := \nu_j^\uparrow \cup \eta_j$ when $j$ goes to $j+1$. There are two cases:
	 \vskip 0.5em
	 \noindent {\it Case 1.} If $\fg^\vee$ is of type $C$, $\#\lambda = 2k + 1$ is odd and $\tilde{C} = C_k= C_{k-1} \zeta_k = C_{k-1} \upsilon_{\#\lambda}$, then $\tau_{j+1}$ can be obtained from $\tau_j$ by adding $1$ to its last row, hence $\lVert \tau_{j+1} \rVert  > \lVert \tau_j \rVert$.
	 \vskip 0.5em
	 \noindent {\it Case 2.} Otherwise, by the definitions of $C_0$ and $\zeta_j$,  $\tau_{j+1}$ can be obtained from $\tau_j$ by replacing the adjacent two rows $q = (q_1, q_2)$ in $\tau_j$, which correspond to the adjacent rows of $\lambda$ in the definition of $\zeta_j$, with the two rows of $q^\uparrow$. Now by the elementary Lemma \ref{lem:comparison} below, again we have $\lVert \tau_{j+1} \rVert  > \lVert \tau_j \rVert$. 
	 \vskip 0.5em
	 When $\tilde{C} \neq C_0$, at least one multiplication by some nontrivial $\zeta_j$ occurs, therefore $\rho^+(\nu^\uparrow_0 \cup \eta_0)$ achieves the unique minimal length among all $\rho^+(\nu_j^\uparrow \cup \eta_j)$. Note that in the type $D$ case, $\eta_0$ is never empty by definition. Therefore the statement about uniqueness up to $W$-conjugation also follows from Proposition \ref{prop:R_c_min}.
\end{proof}

\begin{lemma}\label{lem:comparison}
	Let $q = (q_1, q_2)$ be a partition with $q_1 \geqslant q_2 > 0$. Then $\lVert \rho^+(q) \rVert < \lVert \rho^+ (q^\uparrow) \rVert$.
\end{lemma}

\begin{proof}
	First consider the case when $q_1$ and $q_2$ are both even. Write squares of the norms as
	\[  \lVert \rho^+(q) \rVert^2 = \sum_{k=1}^{q_2 / 2} \left[ \left( \frac{2k-1}{2} \right)^2 + \left( \frac{2k-1}{2} \right)^2 \right] + \sum_{k=1}^{\frac{q_1 - q_2}{2}} \left( \frac{q_2 + 2k -1 }{2} \right)^2  \]
	and
	\[  \lVert \rho^+(q^\uparrow) \rVert^2 = \sum_{k=1}^{q_2 / 2} \left[ (k-1)^2 + k^2 \right] + \sum_{k=1}^{\frac{q_1 - q_2}{2}} \left( \frac{q_2 + 2k}{2} \right)^2.  \]
	We have $(k-1)^2 + k^2 > \left( \frac{2k-1}{2} \right)^2 + \left( \frac{2k-1}{2} \right)^2$ by the Cauchy-Schwartz inequality. Hence $ \lVert \rho^+(q) \rVert < \lVert \rho^+(q^\uparrow) \rVert$.
	
	Next suppose $q_1$ and $q_2$ are both odd.  Write squares of the norms as
	\[  \lVert \rho^+(q) \rVert^2 =  \sum_{k=1}^{\frac{q_2-1}{2}} \left[ k^2 + k^2 \right] + \sum_{k=1}^{\frac{q_1 - q_2}{2}} \left( \frac{q_2 + 2k -1 }{2} \right)^2  \]
	and
	\[  \lVert \rho^+(q^\uparrow) \rVert^2 = \frac{1}{2} + \sum_{k=1}^{\frac{q_2-1}{2}} \left[ \left( \frac{2k-1}{2} \right)^2 + \left( \frac{2k+1}{2} \right)^2 \right] + \sum_{k=1}^{\frac{q_1 - q_2}{2}} \left( \frac{q_2 + 2k}{2} \right)^2.  \]
	The same argument as before gives $ \lVert \rho^+(q) \rVert < \lVert \rho^+(q^\uparrow) \rVert$.
	
	Finally consider the case when $q_1 \not\equiv q_2 \mod 2$. Write squares of the norms as
	\[  \lVert \rho^+(q) \rVert^2 =  \sum_{\substack{1 \leqslant j \leqslant q_2 - 2 \\ j \equiv q_2  \mod 2}} \left( \frac{j}{2} \right)^2 +  \sum_{\substack{1 \leqslant j \leqslant q_2 - 1 \\ j \equiv q_2 - 1 \mod 2}} \left( \frac{j}{2} \right)^2   + \sum_{k=1}^{\frac{q_1 - q_2-1}{2}} \left( \frac{q_2 + 2k}{2} \right)^2  \]
	and
	\[  \lVert \rho^+(q^\uparrow) \rVert^2 =  \sum_{\substack{1 \leqslant j \leqslant q_2 - 2 \\ j \equiv q_2 \mod 2}} \left( \frac{j}{2} \right)^2 +  \sum_{\substack{1 \leqslant j \leqslant q_2 - 1 \\ j \equiv q_2 - 1 \mod 2}} \left( \frac{j}{2} \right)^2 + \sum_{k=1}^{\frac{q_1 - q_2-1}{2}} \left( \frac{q_2 + 2k+1}{2} \right)^2.  \]
	It is clear that $ \lVert \rho^+(q) \rVert < \lVert \rho^+(q^\uparrow) \rVert$.
\end{proof}

\subsection{Proof of Proposition \ref{prop:twoORvs}}\label{subsec:twoORvsclassical}

First, let $(\OO^\vee, \Bar{C})$ be a distinguished pair in $\LA(G^\vee)$, and let $(\eta_0, \nu_0)$ be as in \cref{defn:C0}. By \cref{thm:R_cbar_min}, we have $\gamma(\OO^\vee, \Bar{C})=\rho^+(\nu^\uparrow_0 \cup \eta_0)$. We note that since $\OO^\vee$ is distinguished, all members of $\eta_0$ and $\nu_0$ are odd (resp. even, odd) if $\fg^\vee$ is of type $B$ (resp. $C$, $D$). Thus, all members of $\nu^\uparrow_0$ are even (resp. odd, even), and $\fr^\vee = \fs\fo(|\nu_0|)\times \fs\fo(|\eta_0|)$ (resp. $\fr^\vee = \fs\fp(|\nu_0|)\times \fs\fp(|\eta_0|)$, $\fr^\vee = \fs\fo(|\nu_0|)\times \fs\fo(|\eta_0|)$). A direct computation shows that the partitions corresponding to the orbits in the two factors are $(\nu^\uparrow_0)_D=\nu_0$ (resp. $(\nu^\uparrow_0)_C=\nu_0$, $(\nu^\uparrow_0)_D=\nu_0$) and $\eta_0$ respectively. 

Now let $(\OO^\vee, \Bar{C})\in \LA^*(G^\vee)$, and let $(L^\vee, \OO_{L^\vee}, \Bar{C}_{L^\vee})\in \LA_0(G^\vee)$ be the triple corresponding to $(\OO^\vee, \Bar{C})$ under the bijection of \cref{prop:uniquedistinguishedAbar}. Note that $\fl^\vee=\prod_{i\in I} \fg\fl(a_i)\times \fs\fo(2k+1)$ for a set of integers $\{a_i \in i \in I\}$, and $\OO_{L^\vee}=\prod_i \OO_{[a_i]}\times \OO^\vee_0$. Let $\Bar{C}_0$ be the image of $\Bar{C}_{L^\vee}$ under the isomorphism $\Bar{A}(\OO_{L^\vee})\simeq \Bar{A}(\OO^\vee_0)$, and let $(\eta_0^0, \nu_0^0)$ correspond to the distinguished pair $(\OO_0^\vee, \Bar{C}_0)$ as in Definition \ref{defn:C0}. Let $I^0$ be the set of $i \in I$ such that $a_i$ is odd (resp. even, odd), and let $I^1=I\setminus I^0$. Let $\eta_0=\eta_0^0\cup \bigcup_{i\in I^0} [a_i, a_i]$, and $\nu_0=\nu_0^0\cup \bigcup_{i\in I^1} [a_i, a_i]$. The statement below follows from the discussion above and the definition of $R^\vee$.

\begin{lemma}\label{Rcheck saturate}
In the setting above, the following are true:
    \begin{itemize}
        \item If $\fg^{\vee}$ is of type $B$ or $D$ (resp. $C$), then $\mathfrak{r}^{\vee} = \fs\fo(|\nu_0|)\times \fs\fo(|\eta_0|)$ (resp. $\fr^\vee = \fs\fp(|\nu_0|)\times \fs\fp(|\eta_0|)$).
        \item $\Sat_{R_0^\vee}^{R^\vee} \OO_{R_0^\vee} = \OO_{\nu_0}\times \OO_{\eta_0}$.
    \end{itemize}
\end{lemma}

Recall from \cref{uparrow} the partition $(\nu_0^0)^\uparrow$, and let $\nu_0^\uparrow=(\nu_0^0)^\uparrow\cup \bigcup_{i\in I^1} [a_i, a_i]$. Note that $\gamma:=\gamma(\OO^\vee, \Bar{C})=\gamma(\OO_0^\vee, \Bar{C}_0) \cup \rho(a) = \rho^+(\nu_0^\uparrow\cup \eta_0)$, and $Z_{\fg}(\gamma)= Z_{\fs\fo(|\nu_0|)}(\rho^+(\nu_0^\uparrow)) \times Z_{\fs\fo(|\eta_0|)}(\rho^+(\eta_0))$. It remains to show that $\OO_{\nu_0}$ and $\OO_{\eta_0}$ are the Richardson orbits corresponding to the Levi subalgebras $Z_{\fs\fo(|\nu_0|)}(\rho^+(\nu_0^\uparrow))$ and $Z_{\fs\fo(|\eta_0|)}(\rho^+(\eta_0))$ respectively. 

Let $^{\langle \nu \rangle} \lambda$ be the reduced marked partition corresponding to $(\OO^\vee, \Bar{C})$. We first make the following basic observations: assume $\fg^\vee$ is of type $B$ (resp. $C$, $D$), then
    \begin{itemize}
        \item[(a)]
            If $a$ is not a member of $\nu$, then $\height_{\eta_0}(a)$ is odd (resp. even, even) if and only if $\height_\nu(a)$ is even and $\height_\lambda(a)$ is odd (resp. even, even);
        \item[(b)] 
            $\height_{\nu_0}(a)$ is odd if and only if $\height_\nu(a)$ is odd and $\height_\lambda(a)$ is odd (resp. even, even). 
        \end{itemize}
These observations can be deduced by inductive arguments using block decomposition (cf. \cref{sec:basic block}).  

Assume that $\fg^\vee$ is of type $B$, the other types are completely analogous and are left to the reader. Since $(\OO^\vee, \Bar{C})$ is special Lusztig-Achar datum, the observation above implies that $\height_{\nu_0} (a_i)$ is even for all $i\in I^1$, see Proposition \ref{prop:special_LA_classical}. It follows that $\nu_0=(\nu_0^\uparrow)_D$. Note that $\nu_0^\uparrow$ has only even members, and therefore $Z_{\fs\fo(|\nu_0|)}(\rho^+(\nu_0^\uparrow))= \prod \fg\fl(i)^{k_i}$, and $k_i=\frac{1}{2}((\nu_0^{\uparrow})_i-(\nu_0^{\uparrow})_{i+1})$. It follows immediately that the partition of the corresponding Richardson orbit is $\nu_0=(\nu_0^\uparrow)_D$.

For the second factor, all members of $\eta_0$ are odd, and therefore $Z_{\fs\fo(|\eta_0|)}(\rho^+(\eta_0))= \prod \fg\fl(i)^{k_i}\times \fs\fo(\#\eta_0)$, where $k_i=\frac{1}{2}((\eta_0)_i-(\eta_0)_{i+1})$. The partition of the corresponding Richardson orbit is $\eta_0$.

\subsection{Proof of Theorem \ref{thm:Gamma}}\label{sec:proofsclassical3}

Suppose $G^\vee=G_{\epsilon}(m)$ is a simple classical group of type $B$, $C$, or $D$. Let $(\OO^{\vee},\bar{C}) \in \LA^*(G^{\vee})$ and form the McNinch-Sommers datum $(R^{\vee}, sZ^\circ, \OO_{R^{\vee}}) := \mathbb{L}(\OO^{\vee},\bar{C}) \in \MS(G^{\vee})$ as in the paragraph preceding \cref{thm:Gamma}. Write $^{\langle \nu \rangle}\lambda$ for the reduced marked partition corresponding to $(\OO^{\vee},\bar{C})$ and $\pi$ for the partition corresponding to $\OO=d_S(\OO^{\vee},\bar{C})$.

\begin{lemma}\label{lem:Gamma_distinguished_classical}
Suppose $G^\vee=G_{\epsilon}(m)$ is a simple classical group of type $B$, $C$, or $D$ and let $(\OO^{\vee},\bar{C}) \in \LA(G^{\vee})$. Assume $(\OO^{\vee},\bar{C})$ is distinguished. Then there is a group isomorphism
$$\bar{A}(\OO_{R^{\vee}}) \simeq A^{ad}(d_S(\OO^{\vee},\bar{C})).$$
\end{lemma}

\begin{proof}
  Recall from \cref{defn:C0} the subpartitions $\nu_0$ and $\eta_0$ of $\lambda$ attached to the distinguished pair $(\OO^\vee, \Bar{C})\in \LA^*(G^\vee)$ corresponding to the marked partition $^{\langle \nu\rangle} \lambda$, such that $\nu_0$ is a type $D$ (resp. $C$, $D$) partition and $\eta_0$ is a type $B$ (resp. $C$, $D$) partition when $\fg^\vee$ is of type $B$ (resp. $C$, $D$). Both $\nu_0$ and $\eta_0$ are distinguished. By \cref{lem:gamma_0}, $\OO_{R^\vee}=\OO_{\nu_0}\times \OO_{\eta_0}$. Then the claim follows from a straightforward computation using \cref{lem:d_S_dist} and \cref{cor:Abarclassical} (see also the paragraph following \cref{defn:Abar_basis}).
\end{proof}

To prove \cref{thm:Gamma} for a general $(\OO^{\vee},\bar{C}) \in \LA^*(G^{\vee})$, assume that it is saturated from a distinguished pair $(\OO^\vee_{L^\vee}, \bar{C}_{L^\vee}) \in \LA^*(L^\vee)$ for a Levi subgroup $L^\vee \subset G^\vee$ with Lie algebra $\fl^\vee = \mathfrak{gl}(a_1) \times \cdots \times \mathfrak{gl}(a_t) \times \fg^\vee_\epsilon(n)$, such that $\OO^\vee_{L
^\vee} = \OO_{[a_1]} \times \cdots \times \OO_{[a_t]} \times \OO_{0}^\vee$. Here if $\lambda$ and $\lambda_0$ are the partitions corresponding to $\OO^\vee$ and $\OO_0^\vee$ respectively, $\lambda = \lambda_0 \cup \bigcup_{i=1}^t [a_i, a_i]$. Again form the McNinch-Sommers datum $(R^{\vee}_0, sZ^\circ_{R^\vee_0}, \OO_{R^{\vee}_0}) := \mathbb{L}(\OO_{\lambda_0}, \bar{C}_0) \in \MS(G_\epsilon(n))$ as in the paragraph preceding \cref{thm:Gamma}. We have $\OO_{R^\vee_0}=\OO_{\nu_0^0}\times \OO_{\eta_0^0}$ for (distinguished) subpartitions $\nu_0^0$ and $\eta_0^0$ of $\lambda_0$ as in \cref{defn:C0}. 

Let $(\nu_0, \eta_0)$ be the subpartitions of $\lambda$ defined as in \cref{subsec:twoORvsclassical}, so that $\OO_{R^\vee}\simeq \OO_{\nu_0}\times \OO_{\eta_0}$. Let $L_{R^\vee}$ be the Levi subgroup of $R^\vee$ with the Lie algebra $\fl_{R^\vee}=\mathfrak{gl}(a_1) \times \cdots \times \mathfrak{gl}(a_t) \times \fr^\vee_0$. By \cref{prop:twoORvs} $\OO_{R^\vee}=\Sat_{L_{R^\vee}}^{R^\vee} \OO_{[a_1]} \times \cdots \times \OO_{[a_t]} \times \OO_{R_0^\vee}$. We will apply an inductive argument to reduce the general case to the distinguished case. For this purpose, it suffices to analyze what happens when we saturate from a maximal Levi subalgebra. Suppose that 
$(\OO^\vee, \bar{C}) = \Sat^{G^{\vee}}_{M^{\vee}} (\OO_{M^{\vee}},\bar{C}_{M^{\vee}})$
for a maximal Levi subalgebra $\fm^{\vee}= \mathfrak{gl}(a)\times \fg^\vee_\epsilon(n)\subset \fg^\vee$. We can always assume that $\OO_{M^{\vee}} = \OO_{prin} \times \underline{\OO}^{\vee}$. Let $\underline{\bar{C}}\in \bar{A}(\underline{\OO}^{\vee})\simeq \bar{A}(\OO_{M^{\vee}})$ be the element corresponding to $\bar{C}_{M^{\vee}}\in \bar{A}(\OO_{M^{\vee}})$. Let $^{\langle \nu\rangle} \underline{\lambda}$ be the reduced marked partition corresponding to the pair $(\underline{\OO}^\vee, \underline{\bar{C}})$. Define subpartitions $\underline{\nu}_0$ and $\underline{\eta}_0$ of $\underline{\lambda}$ as above. We note that $\lambda=\underline{\lambda}\cup [a,a]$. Let $\underline{\OO}=d_S(\underline{\OO}^\vee, \underline{\bar{C}})$, and ${\OO}=d_S({\OO}^\vee, {\bar{C}})$. 
%
The following lemma is immediate from \cref{cor:Abarclassical}.

\begin{lemma}\label{lem:Abar_saturation}
      Suppose that $(\OO^\vee, \bar{C})\in \LA(G^\vee)$. Then $\Bar{A}(\OO_{R^\vee})$ is not isomorphic to $\Bar{A}(\OO_{\underline{R}^\vee})$ if and only if the following conditions are satisfied
    \begin{enumerate}
        \item $a$ is not a member of $\underline{\lambda}$;
        \item $a$ is odd (resp. even, odd) if $\fg^\vee$ is of type $B$ (resp. $C$, $D$); 
        \item $\height_{\eta_0}(a)$ is odd (resp. even, even) if $\fg^\vee$ is of type $B$ (resp. $C$, $D$);
        \item If $\fg^\vee$ is of type $D$, then $\underline{\eta}_0 \neq \emptyset$.
    \end{enumerate}
    In this case, $\Bar{A}(\OO_{R^\vee})\simeq \Bar{A}(\OO_{\underline{R}^\vee})\times \ZZ_2$.
\end{lemma}

\begin{lemma}\label{lem:bind_saturation}
      Suppose that $(\OO^\vee, \bar{C})\in \LA(G^\vee)$. Then $\OO$ is not birationally induced from $\{0\}\times \underline{\OO}$ if and only if the following conditions are satisfied
    \begin{enumerate}
        \item $a$ is not a member of $\underline{\lambda}$.
        \item  When $\fg^\vee$ is of type $B$ (resp. $C$, $D$), 
            \begin{itemize}
                \item[(i)]
                    If $a$ is odd (resp. even, odd), then $\height_{\eta_0}(a)$ is odd (resp. even, even); moreover, if $\fg^\vee$ is of type $D$, then $\underline{\lambda}$ is not empty or very even;
                \item[(ii)]
                    If $a$ is even (resp. odd, even), then $\height_{\nu_0}(a)$ is odd.
            \end{itemize}
    \end{enumerate}
    In this case, $\Bind_M^G (\{0\} \times \underline{\OO})$ is a $2$-fold cover of $\OO$.
\end{lemma}

\begin{proof}
    This follows from an inductive argument using block decompositions (cf. \cref{sec:basic block}) and \cref{prop:block division of dual}. We only point out that the condition that $\underline{\lambda}$ is not empty or very even when $\fg^\vee$ is of type $D$ in (2.i) is due to \cref{prop:birationalpartitions}(iii) and \cref{rmk:birationalpartitions}.
\end{proof}

\begin{claim}\label{claim}
    Assume that $\fg^\vee$ is of type $D$. Then $\underline{\eta}_0 = \emptyset$ if and only if  $\underline{\lambda}$ is empty or very even (cf. \cref{lem:Abar_saturation} (4)).
\end{claim} 

\begin{proof}
    The ``if" part of the claim is clear. For the ``only if" part, note that $\underline{\eta}_0 = \emptyset$ implies that $\eta_0^0 = \emptyset$, which forces $\nu = \emptyset$, since otherwise the largest part of $\nu$ is of even height in $\lambda^0$ and hence $\eta_0^0$ contains at least the largest part of $\lambda^0$. Now $\nu = \emptyset$ implies that $\nu_0^0 = \emptyset$, hence $\underline{\nu}_0$ is either empty or consists of pairs of equal parts $[a_j,a_j]$ with $a_j$ even. Therefore $\underline{\lambda} = \underline{\nu}_0 \cup \underline{\eta}_0 = \underline{\nu}_0$ is either empty or very even.
\end{proof}

\begin{proof}[Proof of \cref{thm:Gamma} for classical groups]
    By observation (b) above and \cref{prop:special_LA_classical}, condition (2.ii) in \cref{lem:bind_saturation} is equivalent to that $(\OO^\vee, \bar{C})$ is not special. Therefore by \cref{lem:Abar_saturation}, \cref{lem:bind_saturation} and \cref{claim}, we can conclude that, for special $(\OO^\vee, \bar{C})$, $\Bar{A}(\OO_{R^\vee})$ is not isomorphic to $\Bar{A}(\OO_{\underline{R}^\vee})$ if and only if $\OO$ is not birationally induced from $\{0\}\times \underline{\OO}$.   
    Moreover, in this case $\Bind_M^G (\{0\} \times \underline{\OO})$ is a $2$-fold cover of $\OO$, and $\Bar{A}(\OO_{R^\vee})\simeq \Bar{A}(\OO_{\underline{R}^\vee})\times \ZZ_2$. By induction, we can reduce \cref{thm:Gamma} to the case when $(\OO^\vee, \bar{C})$ is distinguished. This case was handled in \cref{lem:Gamma_distinguished_classical}.
\end{proof}

\begin{rmk}
    We note that \cref{thm:Gamma} does not hold if we drop the assumption that $(\OO^\vee, \Bar{C})$ is special. For example, let $(\OO^\vee, \Bar{C})=^{\langle [5,1] \rangle}[5,4,4,3,1]$ in $\fs\fo(17)$. Then $(\OO^{\vee},\bar{C})$ is saturated from the Lusztig-Achar datum corresponding to $^{\langle [5,1] \rangle}[5,3,1]$. In this case, $\gamma(\OO^\vee, \Bar{C})=(\frac{5}{2}, \frac{3}{2},\frac{3}{2},\frac{3}{2}, \frac{1}{2},\frac{1}{2},\frac{1}{2},\frac{1}{2})$ and $\mathbb{L}(\OO^\vee, \Bar{C})=(SO(16), sZ^{\circ}, \OO_{[7,4^2, 1]})$. Hence $\Bar{A}(\OO_{R^\vee})\simeq 1$. On the other hand, $D(\OO^\vee, \Bar{C})=\Bind_{GL(4)\times Sp(8)}^{Sp(16)} \OO_{[2^3,1^2]}$, which is the $2$-fold cover of $\OO_{[4^3,2^2]}$. So $\Gamma(D(\OO^\vee, \Bar{C}))\simeq \ZZ_2$.
\end{rmk}
        
\section{Proofs of main results in exceptional types}\label{sec:proofsexceptional}

In this section, we prove Proposition \ref{prop:distinguishedbirigid}, Theorem \ref{thm:inflchars}, Proposition \ref{prop:twoORvs}, and Theorem \ref{thm:Gamma} for $\fg$ a simple exceptional Lie algebra. 

We begin by establishing some notational conventions which will remain in place for the remainder of this section:
\begin{itemize}
    \item Nilpotent orbits are denoted using Bala-Carter notation, see \cite{Carter1993}.
    \item Infinitesimal characters are represented as dominant elements of $\fh^*$, written in \emph{fundamental weight coordinates}. For example, $\rho$ is denoted by $(1,...,1)$. 
    \item Levi and pseudo-Levi subgroups are denoted by indicating their Lie types. The Lie type of a (pseudo)Levi subgroup $L \subset G$ does not determine $L$ uniquely up to conjugacy (or even up to isomorphism). This ambiguity will turn out not to matter.
    \item  Lusztig-Achar data are labeled according to the conventions of \cite[Section 4]{Sommers1998} and \cite[Section 3]{Achar2003}. Namely, a Lusztig-Achar datum $(\OO^{\vee},\bar{C})$ is specified by indicating the corresponding equivalence class $\{(M^{\vee}_1,\OO_{M^{\vee}_1}),...,(M^{\vee}_k,\OO_{M^{\vee}_k})\}$ of Sommers data,
    see Lemma \ref{Lem:Sommers_data} and the discussion preceding it. By Lemma \ref{Lem:Sommers_data}(iii), $(\OO^{\vee},\bar{C})$ is distinguished if and  only if one (equivalently, all) of the pseudo-Levi subgroups $M^\vee_1,...,M^\vee_k$ is of maximal semisimple rank. A list of Achar data and special Achar data in simple exceptional types can be found in \cite[Section 6]{Achar2003}.
\end{itemize}
    
\subsection{Proof of Proposition \ref{prop:distinguishedbirigid} and Theorem \ref{thm:inflchars}}\label{subsec:proofsexceptional1}

In Tables \ref{table:G2}-\ref{table:E8} below, we list all distinguished special 
Lusztig-Achar data $(\OO^{\vee},\bar{C}) \in \LA(G^{\vee})$ in simple exceptional types. As explained above, each such Lusztig-Achar datum is denoted by specifying the corresponding equivalence class of Sommers data. For each such Sommers datum $(M^{\vee},\OO_{M^{\vee}})$ in this equivalence class, we record below
\begin{itemize}
    \item The nilpotent orbit $\OO=d_S(\OO^{\vee},\bar{C}) = j^G_M d(\OO_{M^{\vee}})\in \Orb(G)$. This orbit is computed using the atlas software.
    \item The minimal-length $W$-orbit $\gamma(M^{\vee},\OO_{M^{\vee}})$ in the set
    $$\{\gamma \in \fh_{\RR}^*/W \mid \MS(\gamma) = (M^{\vee},tZ^{\circ},\OO_{M^{\vee}})\}$$ 
    This is computed (and shown to exist) using the method described in Remark \ref{rmk:algorithmgamma} below. 
    \item The minimal-length $W$-orbit $\gamma(\OO^{\vee},\bar{C})$ in
    $$S(\OO^{\vee},\bar{C}) = \{\gamma \in \fh_{\RR}^*/W \mid \LA(\gamma) = (\OO^{\vee},\bar{C})\}$$
    By Lemma \ref{Lem:Sommers_data}(iv), this is simply the minimal-length element among the various $\gamma(M^{\vee},\OO_{M^{\vee}})$ corresponding to $(\OO^{\vee},\bar{C})$. 
    \item The (Lie type of the) reductive part $\mathfrak{r}(\OO)$ of the centralizer of an element in $\OO$. This can be read off of the tables in \cite[Section 13.1]{Carter1993}.
    \item The unipotent infinitesimal character $\gamma(D(\OO^{\vee},\bar{C}))$ attached to the Lusztig cover of $d_S(\OO^{\vee},\bar{C})$. This can be read off of the tables in \cite[Section 4.3]{MBMat}.
\end{itemize}

By inspection of Tables \ref{table:G2}-\ref{table:E8}, we see that the following are true:

\begin{itemize}
    \item[(i)] If $(\OO^{\vee},\bar{C}) \in \LA(G)$ is special and distinguished, then $d_S(\OO^{\vee},\bar{C})$ admits a birationally rigid cover. This follows by comparison with \cite[Proposition 3.9.5]{MBMat}.
    \item[(ii)] If $(\OO^{\vee},\bar{C}) \in \LA(G)$ is special and distinguished, $\mathfrak{r}(\OO)$ is semisimple.
    \item[(iii)] If $(\OO^{\vee},\bar{C}) \in \LA(G)$ is special and distinguished and $\bar{C} \neq 1$,
    $$\gamma(\OO^{\vee},\bar{C}) = \gamma(D(\OO^{\vee},\bar{C})).$$
    \item[(iv)] $d_S$ is injective when restricted to the set of special and distinguished Lusztig-Achar data.
\end{itemize}

\begin{proof}[Proof of Proposition \ref{prop:distinguishedbirigid} in exceptional types]
Proposition \ref{prop:distinguishedbirigid}(ii) follows at once from observation (iv). We proceed to proving Proposition \ref{prop:distinguishedbirigid}(i). Suppose $(\OO^{\vee},\bar{C}) \in \LA(G^{\vee})$ is special and distinguished. Let $\OO=d_S(\OO^{\vee},\bar{C})$, and let $\widetilde{\OO}_{univ}$ be the universal $G$-equivariant cover of $\OO$. In all cases, we have that $A(\OO) \simeq \bar{A}(\OO)$, and therefore $\widetilde{\OO}_{Lus} = \widetilde{\OO}_{univ}$. By observation (i), $\OO$ admits a birationally rigid cover, and by observation (ii) we have $H^2(\widetilde{\OO}_{univ},\CC)=0$, see Remark \ref{rmk:H2}.  These two facts imply that $\widetilde{\OO}_{univ}$ is birationally rigid, see \cite[Proposition 3.7.1]{MBMat}. 
\end{proof}

\begin{proof}[Proof of Theorem \ref{thm:inflchars} in exceptional types]
This is an immediate consequence of observation (iii). 
\end{proof}

\begin{rmk}\label{rmk:algorithmgamma}
Let $(M^{\vee},\OO_{M^{\vee}})$ be a Sommers datum such that $M^{\vee}$ is of maximal semisimple rank. Here, we describe a finite algorithm for computing the minimal-length $W$-orbit $\gamma(M^{\vee},\OO_{M^{\vee}})$ in the set
\begin{equation}\label{eq:candsmin}\{\gamma \in \fh_{\RR}^*/W \mid \MS(\gamma) = (M^{\vee},tZ^{\circ},\OO_{M^{\vee}})\}.\end{equation}
(the termination of this algorithm shows that a unique minimum exists). Choose a system of positive roots $\Delta^+(\fh^{\vee},\fm^{\vee})$ and let $F=\{\varpi^{\vee}_1,...,\varpi^{\vee}_k\} \subset \fh^{\vee} \simeq \fh^*$ denote the fundamental coweights. Note that (\ref{eq:candsmin}) is a subset of the lattice $\ZZ F$ generated by $F$. Consider the finite set
\begin{equation}\label{eq:candsmin2}\{\gamma \in \ZZ F \mid \MS(\gamma) = (M^{\vee},tZ^{\circ},\OO_{M^{\vee}}), \lVert \gamma \rVert \leq \lVert \rho \rVert\}.
\end{equation}
This set can be computed using the atlas software. For any $(M^{\vee},\OO_{M^{\vee}})$, we can directly verify that it contains a unique minimal-length $W$-orbit. This $W$-orbit must be of minimal-length in (\ref{eq:candsmin}).
\end{rmk}

\subsection{Proof of Proposition \ref{prop:twoORvs} and Theorem \ref{thm:Gamma}}\label{subsec:proofsexceptional2}

In Tables \ref{table:G2Gamma}-\ref{table:E8Gamma} below, we list all special Lusztig-Achar data $(\OO^{\vee},\bar{C})\in \LA^*(G^{\vee})$, except those of the form $(\OO^{\vee},1)$ for even $\OO^{\vee}$, in exceptional types. Each such Lusztig-Achar datum is indicated by specifying an equivalence class of Sommers data as explained in Section \ref{subsec:proofsexceptional1}. Recall, cf. Proposition \ref{prop:uniquedistinguishedAbar}, that $(\OO^{\vee},\bar{C})$ corresponds to a unique pair $(L^{\vee},(\OO_{L^{\vee}},\bar{C}_{L^{\vee}})) \in \LA_0(G^{\vee})$. For each Lusztig-Achar datum, we compute the following:
\begin{itemize}
    \item The nilpotent orbit $\OO = d_S(\OO^{\vee},\bar{C}) \in \Orb(G)$.
    \item The Levi subgroup $L \subset G$.
    \item The nilpotent orbit $\OO_L = d_S(\OO_{L^{\vee}},\bar{C}_{L^{\vee}}) \in \Orb(L)$.
    \item $R^{\vee}$ and $\OO_{R^{\vee}}$ in $(R^{\vee},sZ^{\circ},\OO_{R^{\vee}})=\mathbb{L}(\OO^{\vee},\bar{C}) \in \MS(G^{\vee})$, see (\ref{eq:defofL}). 
    \item The groups $\bar{A}_{R^{\vee}} = \bar{A}(\OO_{R^{\vee}})$, $A_L = A(\OO_L)$, and $A=A(\OO)$ (taking $G^{\vee}$ to be adjoint).
    \item The group $\Gamma(D(\OO^{\vee},\bar{C}))$. This requires some case-by-case analysis, see subsection \ref{subsubsec:Gamma}. 
\end{itemize}

By inspection of Tables \ref{table:G2Gamma}-\ref{table:E8Gamma}, we see that the following is true
\begin{itemize}
    \item[(v)] For all $(\OO^{\vee},\bar{C}) \in \LA^*(G^{\vee})$ not of the form $(\OO^{\vee},1)$ for even $\OO^{\vee}$, there is a group isomorphism
    $$\bar{A}(\OO_{R^{\vee}}) \simeq \Gamma(D(\OO^{\vee},\bar{C}))$$
    \item[(vi)] For all $(\OO^{\vee},\bar{C}) \in \LA^*(G^{\vee})$ not of the form $(\OO^{\vee},1)$ for even $\OO^{\vee}$, the statement of Proposition \ref{prop:twoORvs} is true.
\end{itemize}

\begin{proof}[Proof of Proposition \ref{prop:twoORvs} in exceptional types]
Let $(\OO^{\vee},\bar{C}) \in \LA^*(G^{\vee})$. First suppose $\OO^{\vee}$ is even and $\bar{C}=1$. Choose $(L^{\vee},\OO_{L^{\vee}}) \in \Orb_0(G^{\vee})$ such that $\OO^{\vee}=\Sat^{G^{\vee}}_{L^{\vee}} \OO_{L^{\vee}}$. Then $(\OO^{\vee},1) = \Sat^{G^{\vee}}_{L^{\vee}} (\OO_{L^{\vee}},1)$. Since $\OO^{\vee}$ and $\OO_{L^{\vee}}$ are even, we have, in the notation of Proposition \ref{prop:twoORvs}, $R^{\vee}=G^{\vee}$  and $R_0^{\vee}=L^{\vee}$. Moreover, by Proposition \ref{prop:even}, we have $\OO_{R^{\vee}} = \OO^{\vee}$ and $\OO_{R_0^{\vee}}=\OO_{L^{\vee}}$. Now Proposition \ref{prop:twoORvs} is immediate. For all other cases, we use observation (vi). 
\end{proof}

\begin{proof}[Proof of Theorem \ref{thm:Gamma} in exceptional types]
 Let $(\OO^{\vee},\bar{C}) \in \LA^*(G^{\vee})$. If $\OO^{\vee}$ is even and $\bar{C}=1$, then $\bar{A}(\OO_{R^{\vee}}) \simeq \Gamma(D(\OO^{\vee},\bar{C}))$ by Remark \ref{rmk:even}. For all other cases, we use observation (v).
\end{proof}

\subsubsection{Computations of $\Gamma$}\label{subsubsec:Gamma}

Let $\widetilde{\OO} = D(\OO^{\vee},\bar{C}) = \Bind^G_L \widetilde{\OO}_{L,Lus}$. In this subsection, we explain how 
the group $\Gamma=\Gamma(\widetilde{\OO})=\Aut(\widetilde{\OO}_{max},\OO)$ is computed in Tables \ref{table:G2Gamma}-\ref{table:E8Gamma}. We begin with the following observation. Since $\widetilde{\OO}_{L,univ}$ is birationally rigid, we have that $\widetilde{\OO}_{L,univ} = \widetilde{\OO}_{L,Lus}$. So by Theorem \ref{thm:maximaltomaximal}, $\widetilde{\OO}_{max} = \Bind^G_L \widetilde{\OO}_{L,univ}$.

In all but six cases, the automorphism group $\Gamma=\Aut(\widetilde{\OO}_{max},\OO)$ can be computed using one of the following methods (in all such cases, we indicate the method used in the column labeled `$\#$'):

\begin{enumerate}[leftmargin=*]
    \item Suppose $A(\OO)\simeq 1$. Then $\widetilde{\OO} = \OO$, and therefore $\Gamma \simeq 1$. 
    \item Suppose $A(\OO) \simeq A(\OO_L)$. Since $\widetilde{\OO}_{max}$ is birationally induced from $\widetilde{\OO}_{L,univ}$, the degree of $\widetilde{\OO} \to \OO$ must be divisible by the order of $A(\OO_L)$, see \cite[Proposition 2.4.1(iv)]{LMBM}. So our assumption implies that the $\widetilde{\OO}_{max} = \widetilde{\OO}_{univ}$. It follows that $\Gamma \simeq A(\OO)$. 
    \item Suppose $\OO$ is even and $L$ contains (a $G$-conjugate of) the Jacobson-Morozov Levi $L_{\OO}$ associated to $\OO$ (cf. \ref{eq:JMlevi}). Then $\OO=\mathrm{Bind}^G_L \OO_L$ by Proposition \ref{prop:even}. Hence, $\Gamma \simeq A(\OO_L)$ by Proposition \ref{prop:propsofbind}(ii). 
    \item Suppose $\mathcal{P}_{rig}(\OO)$ contains a unique element $(K,\OO_K)$ such that $\dim(\fz(\mathfrak{k})) = m(\OO)$ (cf. \ref{eq:m(O)}) and $L$ contains a $G$-conjugate of $K$. Then by Lemma \ref{lem:bindcriterion}, $\OO = \Bind^G_L \OO_L$. So by Proposition \ref{prop:propsofbind}(ii), we have $\Gamma \simeq A(\OO_L)$.
    \item Suppose there exists a Levi subgroup $K \subset G$ containing $L$ and that $A(\OO) \simeq A(\OO_K)$, where $\widetilde{\OO}_K = \Bind^K_L \widetilde{\OO}_{L,univ}$. Then by Lemma \ref{prop:propsofbind}, $\Gamma(\widetilde{\OO}_K) \simeq \Gamma$. So the computation of $\Gamma$ can be reduced to the case of a smaller rank group.
    \item 
    In some cases, we can compute $\widetilde{\OO}_{max}$ using Springer theory. The general argument is as follows. Let $L \subset G$ be a Levi subgroup, let $\widetilde{\OO}_L \in \Cov(L)$, and let $\widetilde{\OO}=\Bind^G_L \widetilde{\OO}_L$. Recall that $\widetilde{\OO}$ is determined by the subgroup $A(\widetilde{\OO}) \subset A(\OO)$ (up to conjugation). In all cases, we will impose the following additional assumption:
    \begin{enumerate}[label=$(\ast)$]
	  \item \label{cond:relevant}
	  Every irreducible constituent of the induced representation $\Ind_{A(\widetilde{\OO}_L)}^{A(\OO_L)} \mathbbm{1}$ is of Springer type, cf. Section \ref{subsec:truncated}.
    \end{enumerate}
    Then we can define the character of $W_L$,
      \begin{equation}\label{eq:EOL}
        E_{\widetilde{\OO}_L} := \sum_{\psi} [\Ind_{A(\widetilde{\OO}_L)}^{A(\OO_L)} \mathbbm{1} : \psi ] E_{(\OO_L, \psi)},
      \end{equation}
    where the sum runs over all irreducible $A(\OO_L)$-representations.
    
    \begin{prop}
      Assume condition \ref{cond:relevant}. Then there is an equality in the representation ring of $A(\OO)$
        \begin{equation}
          \Ind_{A(\widetilde{\OO})}^{A(\OO)} \mathbbm{1} = \sum_\psi [\Ind_{W_L}^W E_{\widetilde{\OO}_L} : E_{(\OO, \psi)}] \psi  
        \end{equation} 
    where the sum runs over all irreducible $A(\OO)$-representations $\psi$ of Springer type. 
    \end{prop}

    \begin{proof}
        The special case when $\widetilde{\OO}_L = \OO_L$ was proven in \cite[Corollary 5.6]{fuetall2015}. The proof for general $\widetilde{\OO}_L$ satisfying the condition \ref{cond:relevant} is essentially the same and is based on \cite{BorhoMacPherson}. See especially Proposition 1.10, Theorem 3.3 and Corollary 3.9 in \cite{BorhoMacPherson}.
    \end{proof}

        For any irreducible representation $\rho$ of a Weyl group $W$ of exceptional type and irreducible representation $\rho_0$ of a maximal parabolic subgroup $W_0\subset W$, the multiplicity $[\Ind_{W_0}^W \rho_0 : \rho)]$ can be found in \cite{Alvis}. Using the transitivity of induction and the Littlewood-Richardson rule, we can compute  $[\Ind_{W_L}^W \rho_0 : \rho]$ in general.
    
    It remains to recover the subgroup $A(\widetilde{\OO})$ of $A(\OO)$ (up to conjugation) from the representation $\Ind_{A(\widetilde{\OO})}^{A(\OO)} \mathbbm{1}$. 

    \begin{definition}
    Let $H$ be a finite group. We say that $H$ has \emph{linearly distinguishable subgroups} if for any pair of subgroups $H_1,H_2 \subset H$,
    $$\Ind^H_{H_1} \mathbbm{1} \simeq \Ind^H_{H_2} \mathbbm{1} \implies H_1 \text{ and $H_2$ are $H$-conjugate}$$
    \end{definition}
    
    A standard case-by-case check shows that all symmetric groups $S_n$ with $n \leq 5$ have linearly distinguishable subgroups (cf. \cite[Lemma 1.9]{Beaulieu1991}). \footnote{Notably $S_n$ with $n \geqslant 6$ does not have this property due to an example of Gassmann (\cite{Gassmann1926}), see also (\cite[Chapter 4]{Beaulieu1991}).} If $G$ is a simple exceptional group, then $A(\OO)$ is a product of symmetric groups $S_n$ with $n\leq 5$. So in the cases of interest, it is possible to deduce $A(\widetilde{\OO})$ (up to conjugation) from the $A(\OO)$-representation $\Ind_{A(\widetilde{\OO})}^{A(\OO)} \mathbbm{1}$.

    \item Consider the set $\mathrm{Unip}(\widetilde{\OO})$ of irreducible objects in the category $\HC^G(U(\fg)/I(\widetilde{\OO}))$. This set can be enumerated using the atlas software. By Theorem \ref{thm:classificationHC}, the cardinality of $\mathrm{Unip}(\widetilde{\OO})$ is equal to the number of conjugacy classes in $\Gamma$. In several cases, this information will determine $\Gamma$ uniquely. 
    \item Covers of $\OO$ (up to isomorphism) are in one-to-one correspondence with conjugacy classes of subgroups of $A(\OO)$. The birationally rigid covers are listed in \cite[Propositions 3.8.3, 3.9.5]{MBMat}. Suppose that all covers of $\OO$, except for a single cover $\breve{\OO}$, are either birationally rigid or birationally induced from a birationally rigid cover for a Levi subgroup not conjugate to $L$. Then $\breve{\OO} \simeq \widetilde{\OO}$. 
\end{enumerate}

Sample computations:

\begin{itemize}[leftmargin=*]
    \item[(3)] Let $(\OO^{\vee},M^{\vee}) = (C_3,C_3)$ in type $F_4$. Then $(L,\OO_L) = (C_3,\{0\})$ and $\OO=A_2$. The weighted Dynkin diagram for $\OO$ consists of $0$'s and $2$'s, with $0$'s on the subdiagram of type $C_3$. In particular, $\OO$ is even and $L_{\OO}=L$. Hence, $\Gamma \simeq 1$. 
    \item[(4)] Let $(\OO^{\vee},M^{\vee}) = ((A_5)',A_5)$ in type $E_7$. Then $(L,\OO_L) = (A_5,\{0\})$ and $\OO=D_4(a_1)+A_1$. By \cite[Table 7]{deGraafElashvili}, $\mathcal{P}_{rig}(\OO) = \{(L,\OO_L)\}$. So $m(\OO)=2$. From the diagrams in \cite[Section 13]{fuetall2015} it is clear that $\Spec(\CC[\OO])$ has two dimension 2 singularities, both of type $A_1$. And by \cite[Theorems 5.11,5.12]{BISWAS}, $H^2(\OO,\CC)=0$. Hence, $\dim \fP(\OO)=2$. So Lemma \ref{lem:bindcriterion} implies that $\OO=\Bind^G_L \OO_L$, and therefore $\Gamma \simeq A(\OO_L) \simeq 1$.
    \item[(5)] Let $(\OO^{\vee},M^{\vee})=((A_3+A_1)',A_3+A_1)$ in type $E_7$. Then $(L,\OO_L)=(A_3+A_1,\{0\})$ and $\OO=E_7(a_5)$. Let $K$ be a Levi subgroup of type $E_6$ containing $L$. Then $\OO_K = D_4(a_1)$. Note that $A(\OO) \simeq S_3 \simeq A(\OO_K)$. So $\Gamma \simeq \Gamma(\widetilde{\OO}_K)$, where $\widetilde{\OO}_K= \Bind^K_L \OO_{L, univ}$. But $\widetilde{\OO}_K$ corresponds under $D^K$ to the Lusztig-Achar datum $(A_3+A_1,1) \in \LA(K^{\vee})$. Using argument $(8)$, we find that  $\Gamma \simeq \Gamma(\widetilde{\OO}_K) = \Gamma(D^K(A_3+A_1,1)) \simeq 1$.
    \item[(6)] Let $(\OO^{\vee},M^{\vee})= (E_7(a_5), E_7)$ in type $E_8$, which corresponds to $(\OO^\vee, \bar{C}) = (E_7(a_5), 1)$. In this case $L^\vee = M^\vee = E_7$, $(\OO_{L^{\vee}},\bar{C}_{L^{\vee}}) = (E_7(a_5), 1)$ and the Sommers dual of $(E_7(a_5), 1)$ is $\OO_L = d_S(\OO_{L^{\vee}},\bar{C}_{L^{\vee}}) = D_4(a_1)$ in $L^\vee = E_7$. Then $\pi_1(\OO_L)=A(\OO_L) \simeq S_3$ and $\widetilde{\OO}_L = D^{L}(\OO_{L^{\vee}},\bar{C}_{L^{\vee}})$ is the $6$-fold universal cover of the orbit $\OO_L$. There are three irreducible representations of $S_3$, labelled as $\psi_3$ (the trivial representation), $\psi_{21}$ and $\psi_{1^3}$ (the sign representation), corresponding to the partitions of $3$ in the subscripts. Then the Springer representations associated to $D_4(a_1)$ in $L = E_7$  are 
      \[ E_{(\OO_L, \psi_3)} = \phi_{315,16} = 315_a, \, E_{(\OO_L,  \psi_{21})} = \phi_{280,18} = 280_a, \, E_{(\OO_L, \psi_{1^3})} = \phi_{35,22} = 35_a. \] 
    Here we use the tables of \cite[Section 13.3]{Carter1993} for the Springer correspondence of exceptional groups. Note that \cite{Carter1993} labels irreducible characters of exceptional Weyl groups by the notation $\phi_{d,e}$, where $d$ is the degree of the character and $e$ is its fake degree (\cite[Section 11.3]{Carter1993}), while \cite{Alvis} follows the notations of \cite{Frame1951} (for $E_6$ and $E_7$) and \cite{Frame1970} (for $E_8$). For instance, $315_a$ is an irreducible character of $W_{E_8}$ of dimension $315$. The subscripts are used to distinguish characters with the same dimension. The two conventions of notations can be translated to each other, for instance, by the tables in \cite[Appendix]{GeckPfeiffer2000}.
    
    Now consider the Sommers dual $\OO = E_8(a_7)$ of $(E_7(a_5), 1)$ in $E_8$. We have $\pi_1(\OO) = A(\OO) \simeq S_5$. Again we label the irreducible representations of $S_5$ by $\psi_\lambda$ were $\lambda$ is a partition of $5$, such that $\psi_{5}$ is the trivial representation and $\psi_{1^5}$ is the sign representation. The Springer representations attached to $E_8(a_7)$ in $E_8$ are given in the Table \ref{table:Springer_E8(a7)} together with their multiplicities in $\Ind_{W_L}^W E_{(\OO_L, \psi_\eta)}$, where $W$ (resp. $W_L$) is the Weyl group of $E_8$ (resp. $E_7$), $\OO_L = D_4(a_1)$ in $E_7$ and $\eta$ runs over all partitions of $3$. Note that the column for $\psi_{1^5}$ is left blank since it does not appear in the Springer correspondence for $E_8(a_7)$.

    The induction $\Ind_{\{1\}}^{S_3} \mathbbm{1}$ of the trivial representation of the trivial subgroup to $S_3$ decomposes as $\psi_3 + 2 \psi_{21} + \psi_{1^3}$. Replacing each $\psi_\eta$ in this decomposition by the corresponding Springer representation $E_{(\OO_L, \psi_\eta)}$ of $W_{E_7}$, we get the $W_{E_7}$-representation $E_{\widetilde{\OO}_L} = 315_a + 2 \cdot 280_a + 35_a$ as defined in \eqref{eq:EOL}. According to Table \ref{table:Springer_E8(a7)}, we have
    \[ 
      \begin{split}
        \Ind_{A(\widetilde{\OO})}^{S_5} \mathbbm{1} & = \sum_{\substack{\lambda \in \mathcal{P}(5), \\ \lambda \neq 1^5}} [\Ind_{W_{E_7}}^{W_{E_8}} E_{\widetilde{\OO}_L} : E_{(\OO, \psi_\lambda)}] \psi_\lambda \\
        & = 1 \cdot \psi_5  + 3 \cdot \psi_{41} + 3 \cdot \psi_{32} + 3 \cdot \psi_{31^2} + 2 \cdot \psi_{2^21} + 1 \cdot \psi_{21^3} + 0 \cdot \psi_{1^5},
      \end{split}
    \]
    which is isomorphic to $\Ind_{\langle (12) \rangle}^{S_5} \mathbbm{1}$ by the Littlewood-Richardson rule, where $\langle (12) \rangle \simeq S_2$ is the subgroup of $S_5$ generated by the permutation $(12)$. Therefore $A(\widetilde{\OO})$ is conjugate to $\langle (12) \rangle$ and its normalizer subgroup in $A(\OO)\simeq S_5$ is conjugate to the subgroup generated by $(12)$ and $(345)$. So $\Gamma \simeq \langle (12), (345) \rangle / \langle (12) \rangle \simeq S_3$.

    \begin{table}[H] \label{table:Springer_E8(a7)}
        \begin{tabular}{|c|c|c|c|c|c|c|c|} \hline
          $\lambda$ & $[5]$ & $[4,1]$ & $[3,2]$ & $[3,1^2]$ & $[2^2,1]$ & $[2,1^3]$ & $[1^5]$ \\ \hline
          $E_{(\OO, \psi_\lambda)}$ & \makecell{$\phi_{4480,16}$ \\ $4480_y$} & \makecell{$\phi_{5670,18}$ \\ $5670_y$} & \makecell{$\phi_{4536, 18}$ \\ $4536_y$} & \makecell{$\phi_{1680,22}$ \\ $1680_y$} & \makecell{$\phi_{1400,20}$ \\ $1400_y$} & \makecell{$\phi_{70,32}$ \\ $70_y$} &  \\ \hline
          $[\Ind_{W_L}^W E_{(\OO_L, \psi_{3})} : E_{(\OO, \psi_\lambda)}]$ & 1 & 1 & 1 & 0 & 0 & 0 & \\ \hline
          $[\Ind_{W_L}^W E_{(\OO_L, \psi_{21})} : E_{(\OO, \psi_\lambda)}]$ & 0 & 1 & 1 & 1 & 1 & 0 & \\ \hline
          $[\Ind_{W_L}^W E_{(\OO_L, \psi_{1^3})} : E_{(\OO, \psi_\lambda)}]$ & 0 & 0 & 0 & 1 & 0 & 1 & \\
          \hline
        \end{tabular}
        \caption{Springer representations attached to $\OO=E_8(a_7)$}
    \end{table}

    \item[(7)] There are three special Lusztig-Achar data in $F_4$ which map under $d_S$ to $\OO = F_4(a_3)$, namely $(A_2+\widetilde{A}_1,1)$, $(\widetilde{A}_2+A_1,1)$, and $(C_3(a_1),1)$. The corresponding infinitesimal characters are $(0,1,0,0)/2$, $(1,0,1,0)/2$, and $(0,1,0,1)/2$. Using atlas, we determine that there are $1$, $1$, and $2$ unipotent Harish-Chandra bimodules annihilated by the respective maximal ideals. Since $A(\OO)=S_4$, this means that the groups $\Gamma$ associated to the corresponding covers must be $1$, $1$, and $\ZZ_2$, respectively. 
    \item[(8)] There are three special Lusztig-Achar data in $E_6$ which map under $d_S$ to $\OO=D_4(a_1)$, namely $(2A_2+A_1,1)$, $(A_3+A_1,1)$, and $(D_4(a_1),1)$ (the final one is even, and therefore does not appear in Table \ref{table:E6Gamma}). Using argument (3), we see that $D(2A_2+A_1,1) = \OO$. 

    By inspection of the incidence diagrams in \cite[Section 13]{fuetall2015}, we see that $\Spec(\CC[\OO])$ contains a unique dimension 2 singularity, of type $A_1$, corresponding to codimension 2 orbit $\OO'=A_3+A_1 \subset \overline{\OO}$. Furthermore, $A(\OO)=S_3$ and $A(\OO')=1$. In particular, there are three nontrivial covers of $\OO$ (up to isomorphism), of degrees $2$, $3$, and $6$. Denote them by $\widetilde{\OO}_2$, $\widetilde{\OO}_3$, and $\widetilde{\OO}_6$. We first note that $\Spec(\CC[\widetilde{\OO}_6]) \to \Spec(\CC[\OO])$ smoothens the dimension 2 singularity. Otherwise, $\dim \fP(\widetilde{\OO}_6) \geq |A(\OO)|/|A(\OO')| = 6$ by \cite[Lemma 7.6.13]{LMBM}, a contradiction. Since an $A_1$ singularity cannot be smoothened by a 3-fold cover, it follows that $\Spec(\CC[\widetilde{\OO}_2])$ has no dimension 2 singularities, and hence $\widetilde{\OO}_2 \sim \widetilde{\OO}_6$.

    By the transitivity of birational induction, we have
    $$D(A_3+A_1,1) = \Bind^{E_6}_{A_3+A_1} \{0\} = \Bind^{E_6}_{D_5} \Bind^{D_5}_{A_3+A_1} \{0\}$$
    Note that $\Bind^{D_5}_{A_3+A_1}\{0\}$ is the orbit $\OO_{[3^2,1^4]}$ in $D_5$ corresponding to the partition $[3^2,1^4]$. Hence, $D(A_3+A_1,1) = \Bind^{E_6}_{D_5} \OO_{[3^2,1^4]}$. On the other hand, $\Spec(\CC[\OO_{[3^2,1^4]}])$ contains a dimension 2 singularity, of type $A_1$, see \cite[Section 3.4]{Kraft-Procesi}. Thus, $\Spec(\CC[D(A_3+A_1,1)])$ contains a dimension 2 singularity. So by the preceding paragraph, we must have $D(A_3+A_1,1) = \widetilde{\OO}_3$. 

    This leaves the case of $D(D_4(a_1),1)$. Note that $D_4(a_1)$ is saturated from the orbit $\OO_{[5,3]}$ in the Levi of type $D_4$. So $D(D_4(a_1),1)$ is birationally induced from the orbit $d(\OO_{[5,3]})=\OO_{[2^2,1^4]}$ in the Levi of type $D_4$. By the transitivity of birational induction
    $$D(D_4(a_1),1) = \Bind^{E_6}_{D_4} \OO_{[2^2,1^4]} = \Bind^{E_6}_{D_5} \Bind^{D_5}_{D_4} \OO_{[2^2,1^4]}$$
    Note that $\Bind^{D_5}_{D_4}\OO_{[2^2,1^4]}$ is the double cover $\widetilde{\OO}_{[3^2,1^4]}$. So $D(D_4(a_1),1) = \Bind^{E_6}_{D_4}\widetilde{\OO}_{[3^2,1^4]}$. By the preceding paragraph, $\Bind^{E_6}_{D_4}\widetilde{\OO}_{[3^2,1^4]} = \widetilde{\OO}_3$, so $D(D_4(a_1),1)$ must have degree divisible by $3$. The only possibility is $D(D_4(a_1),1)= \widetilde{\OO}_6$.

\end{itemize}

There are six cases in type $E_8$ for which special arguments are required (these cases are marked with asterisks in Tables \ref{table:G2Gamma}-\ref{table:E8Gamma}):

\begin{itemize}[leftmargin=*]
    \item \underline{$(A_3+2A_1,1)$} and \underline{$(D_4(a_1)+A_1,1)$}: In both cases, the Sommers dual is $\OO= E_8(a_6)$. By \cite{fuetall2015}, there is a unique codimension $2$ symplectic leaf $\fL\subset X=\Spec(\CC[\OO])$, corresponding to the nilpotent orbit $\OO'=D_7(a_1)$. The corresponding singularity is of type $A_3$, with non-trivial monodromy. Since $A(\OO) = S_3$, there are (up to isomorphism) three nontrivial covers of $\OO$, of degrees $2$, $3$, and $6$. Denote them by $\widetilde{\OO}_i$ for $i \in \{2,3,6\}$. For each $i$, let $X_i = \Spec(\CC[\widetilde{\OO}_i])$ and $\fL_i$ be the preimage of $\fL$ under the map $X_i \to X$.
    
    We claim that $X_2$ has a unique codimension $2$ leaf, and that the corresponding singularity is of type $A_1$. Indeed, if the singularity of (a connected component of) $\fL_2$ is of type $A_3$, then the same is true for $\fL_6$. In this case, \cite[Lemma 7.6.13]{LMBM} implies $\dim \fP(\widetilde{\OO}_6) \ge 6$. But $m(\OO)=3$, a contradiction. Similarly, if $\fL_2$ has $2$ connected components, then $\fL_6$ has $6$ connected components since $A(\OO')\simeq \ZZ_2$ and therefore $\dim \fP(\widetilde{\OO}_6)\ge 6$. 
    
    It follows from the claim in the previous paragraph that $X_6$ has $3$ codimension $2$ symplectic leaves, the connected components of $\fL_6$, and the corresponding singularities are of type $A_1$. The reductive part of the centralizer of an element in $\OO$ is trivial. So $\dim \fP(\widetilde{\OO}_2)=1$, and $\dim \fP(\widetilde{\OO}_6)=3$. It follows that $\{D(A_3+2A_1,1), D(D_4(a_1)+A_1,1)\} = \{\widetilde{\OO}_3,\widetilde{\OO}_6\}$. Let $L$, $M_1$, and $M_2$ denote the Levi subgroups of type $D_7$, $A_3+2A_1$, and $D_4+A_1$, respectively. Then by the transitivity of birational induction 
    \begin{align*}D(A_3+2A_1,1) &= \Bind^G_{M_1} \{0\} = \Bind^G_L (\Bind^L_{M_1}\{0\})\\
    D(D_4(a_1)+A_1,1) &= \Bind^G_{M_2} \OO_{[2^2,1^4]} \times \{0\} = \Bind^G_L (\Bind^L_{M_2} \OO_{[2^2,1^4]} \times \{0\})\end{align*}
    By Proposition \ref{prop:birationalpartitions}, $\Bind^L_{M_1} \{0\}$ is the orbit corresponding to the partition $[5^2,1^4]$ and $\Bind^L_{M_2} \OO_{[2^2,1^4]} \times \{0\}$ is a nontrivial cover of the same orbit. Thus, $D(D_4(a_1)+A_1,1)$ is a nontrivial cover of $D(A_3+2A_1,1)$. Hence $D(A_3+2A_1,1)=\widetilde{\OO}_3$ and $ D(D_4(a_1)+A_1,1) = \widetilde{\OO}_6$. From the analysis of codimension 2 leaves, it is clear that both $\widetilde{\OO}_3$ and $\widetilde{\OO}_6$ are maximal in their equivalence classes. So the corresponding $\Gamma$s are $1$ and $S_3$ respectively.
    
    \item \underline{$(A_4+A_1,1)$}: The Sommers dual is $\OO=E_6(a_1)+A_1$. By \cite{fuetall2015}, there is a unique codimension $2$ leaf $\fL \subset X=\Spec(\CC[\OO])$, corresponding to the orbit $D_7(a_2)$. The corresponding singularity is of type $A_2$ with non-trivial monodromy. The reductive part of the centralizer of an element in $\OO$ is one-dimensional. Hence $\dim \fP(\OO) \le 2$. Hence $D(A_4+A_1,1) \neq \OO$. But $A(\OO) \simeq S_2$. So $D(A_4+A_1,1)$ is the unique nontrivial (2-fold) cover of $\OO$ and $\Gamma \simeq S_2$.

    \item \underline{$(2A_3,1)$} and \underline{$(A_4+2A_1,1)$}: In both cases, the Sommers dual is $\OO=D_7(a_2)$. Let $L$ be the Levi subgroup of type $D_7$ and let $\OO_L$ be the nilpotent orbit corresponding to the partition $[3^4, 1^2]$ of $14$. Then $A(\OO_L)\simeq \ZZ_2$ and $\OO_L$ is birationally (resp. non-birationally) induced from the $0$ orbit in the Levi subgroup of $L$ of type $2A_3$ (resp. $A_4+2A_1$). Hence, $D(A_4+2A_1,1)$ is a nontrivial cover of $D(2A_3,1)$. Since $A(\OO)\simeq S_2$, we must have that $D(2A_3,1)=\OO$ and $D(A_4+2A_1,1)$ is the unique nontrivial (2-fold) cover of $\OO$. The corresponding $\Gamma$s are $1$ and $S_2$ respectively.

    \item \underline{$(D_5(a_1)+A_1,1)$}: The Sommers dual is $\OO=D_7(a_4)$. Let $L \subset G$ be the Levi subgroup of type $E_7$ and let $\OO_L=A_3+A_2$. Note that $A(\OO_L)\simeq \ZZ_2$. We claim that $D(D_5(a_1)+A_1,1)=\OO$. For this it is enough to show that $\OO_L$ is birationally induced from the orbit $\OO_{[2^2,1^6]} \times \{0\}$ in the Levi subgroup of $L$ of type $D_5+A_1$. This was shown above.
    
\end{itemize}

\section{Tables}\label{sec:tables}

    \begin{table}[H]
        \begin{tabular}{|c|c|c|c|c|c|c|} \hline
        $\OO^{\vee}$ & $\OO_{M^{\vee}}$ & $d_S(\OO^{\vee},\bar{C})$ & $\gamma(M^{\vee},\OO_{M^{\vee}})$ & $\gamma(\OO^{\vee},\bar{C})$ & $\gamma(D(\OO^{\vee},\bar{C}))$ & $\fr(\OO)$  \\ \hline
        $G_2(a_1)$ & $G_2(a_1)$ & $G_2(a_1)$ & $(1,0)$ & $(1,0)$ & $(1,0)$ & $1$\\ \hline
        $G_2(a_1)$ & $A_1+\widetilde{A}_1$ & $\tilde{A}_1$ & $(1,1)/2$ & $(1,1)/2$ & $(1,1)/2$ & $A_1$\\ \hline
        $G_2(a_1)$ & $A_2$ & $A_1$ & $(3,1)/3$ & $(3,1)/3$ & $(3,1)/3$ & $A_1$\\ \hline
        $G_2$ & $G_2$ & $\{0\}$ & $(1,1)$ & $(1,1)$ & $(1,1)$ & $G_2$ \\ \hline
        \end{tabular}
    \caption{Special distinguished Lusztig-Achar data in type $G_2$}
    \label{table:G2}
    \end{table}
    
        \begin{table}[H]
        \begin{tabular}{|c|c|c|c|c|c|c|} \hline
        $\OO^{\vee}$ & $\OO_{M^{\vee}}$ & $d_S(\OO^{\vee},\bar{C})$ & $\gamma(M^{\vee},\OO_{M^{\vee}})$ & $\gamma(\OO^{\vee},\bar{C})$ & $\gamma(D(\OO^{\vee},\bar{C}))$ & $\fr(\OO)$ \\ \hline
        $F_4(a_3)$ & $F_4(a_3)$ & $F_4(a_3)$ & $(0,0,1,0)$ & $(0,0,1,0)$ & $(0,0,1,0)$  & $1$ \\ \hline
        $F_4(a_3)$ & $B_4$ & $B_2$ & $(0,1,0,2)/2$ & $(0,1,0,2)/2$ & $(0,1,0,2)/2$  & $A_1+A_1$\\ \hline
        $F_4(a_3)$ & $C_3(a_1)+A_1$ & $C_3(a_1)$ & $(1,0,1,1)/2$ & $(1,0,1,1)/2$ & $(1,0,1,1)/2$  & $A_1$\\ \hline
        $F_4(a_3)$ & $2A_2$ & $\tilde{A}_2+A_1$ & $(1,1,1,1)/3$ & $(1,1,1,1)/3$ & $(1,1,1,1)/3$  & $A_1$ \\ \hline
        $F_4(a_3)$ & $A_3+A_1$ & $A_2+\tilde{A}_1$ & $(1,1,2,2)/4$ & $(1,1,2,2)/4$ & $(1,1,2,2)/4$  & $A_1$ \\ \hline
        $F_4(a_2)$ & $F_4(a_2)$ & $A_1+\tilde{A}_1$ & $(1,0,1,0)$ & $(1,0,1,0)$ & $(1,0,1,0)$  & $A_1$\\ 
        & $C_3+A_1$ & & $(1,1,1,1)/2$ & & &\\ \hline
        $F_4(a_1)$ & $F_4(a_1)$ & $\tilde{A}_1$ & $(1,0,1,1)$ & $(1,0,1,1)$ & $(1,0,1,1)$  & $A_3$\\ \hline
        $F_4(a_1)$ & $B_4$ & $A_1$ & $(1,1,2,2)/2$ & $(1,1,2,2)/2$ & $(1,1,2,2)/2$  & $C_3$\\ \hline
        $F_4$ & $F_4$ & $\{0\}$ & $(1,1,1,1)$ & $(1,1,1,1)$ & $(1,1,1,1)$  & $F_4$\\ \hline
        \end{tabular}
    \caption{Special distinguished Lusztig-Achar data in type $F_4$}
    \label{table:F4}
    \end{table}

            \begin{table}[H]
        \begin{tabular}{|c|c|c|c|c|c|c|} \hline
        $\OO^{\vee}$ & $\OO_{M^{\vee}}$ & $d_S(\OO^{\vee},\bar{C})$ & $\gamma(M^{\vee},\OO_{M^{\vee}})$ & $\gamma(\OO^{\vee},\bar{C})$ & $\gamma(D(\OO^{\vee},\bar{C}))$ & $\fr(\OO)$   \\ \hline
        $D_4(a_1)$ & $3A_2$ & $2A_2+A_1$ & $(1,1,1,1,1,1)/3$ & $(1,1,1,1,1,1)/3$ & $(1,1,1,1,1,1)/3$  & $A_1$ \\ \hline
        $E_6(a_3)$ & $E_6(a_3)$ & $A_2$ & $(1,0,0,1,0,1)$ & $(1,0,0,1,0,1)$ & $(1,0,0,1,0,1)$  & $A_2+A_2$\\ \hline
        $E_6(a_3)$ & $A_5+A_1$ & $3A_1$ & $(1,1,1,1,1,1)/2$ & $(1,1,1,1,1,1)/2$ & $(1,1,1,1,1,1)/2$ & $A_2+A_1$ \\ \hline
        $E_6(a_1)$ & $E_6(a_1)$ & $A_1$ & $(1,1,1,0,1,1)$ & $(1,1,1,0,1,1)$ & $(1,1,1,0,1,1)$  & $A_5$\\ \hline
        $E_6$ & $E_6$ & $\{0\}$ & $(1,1,1,1,1,1)$ & $(1,1,1,1,1,1)$ & $(1,1,1,1,1,1)$  & $E_6$\\ \hline
        \end{tabular}
    \caption{Special distinguished Lusztig-Achar data in type $E_6$}
    \label{table:E6}
    \end{table}
    
                \begin{table}[H]
                \small
        \begin{tabular}{|c|c|c|c|c|c|c|} \hline
        $\OO^{\vee}$ & $\OO_{M^{\vee}}$ & $d_S(\OO^{\vee},\bar{C})$ & $\gamma(M^{\vee},\OO_{M^{\vee}})$ & $\gamma(\OO^{\vee},\bar{C})$ & $\gamma(D(\OO^{\vee},\bar{C}))$ & $\fr(\OO)$  \\ \hline
        $E_7(a_5)$ & $E_7(a_5)$ & $D_4(a_1)$ & $(0,0,0,1,0,0,1)$ & $(0,0,0,1,0,0,1)$ & $(0,0,0,1,0,0,1)$  & $3A_1$ \\ \hline
        $E_7(a_5)$ & $D_6(a_2)+A_1$ & $(A_3+A_1)'$ & $(1,1,0,1,0,1,1)/2$ & $(1,1,0,1,0,1,1)/2$ & $(1,1,0,1,0,1,1)/2$  & $3A_1$\\ \hline
        $E_7(a_5)$ & $A_5+A_2$ & $2A_2+A_1$ & $(1,1,1,1,1,1,1)/3$ & $(1,1,1,1,1,1,1)/3$ & $(1,1,1,1,1,1,1)/3$  & $2A_1$\\ \hline
        $E_7(a_4)$ & $E_7(a_4)$ & $A_2+2A_1$ & $(1,0,0,1,0,0,1)$ & $(1,0,0,1,0,0,1)$ & $(1,0,0,1,0,0,1)$  & $3A_1$\\ 
        & $D_6+A_1$ & & $(1,1,1,1,0,1,1)/2$ & & &\\ \hline
        $E_6(a_1)$ & $A_7$ & $4A_1$ & $(1,1,1,1,1,1,1)/2$ & $(1,1,1,1,1,1,1)/2$ & $(1,1,1,1,1,1,1)/2$  & $C_3$ \\ \hline
        $E_7(a_3)$ & $E_7(a_3)$ & $A_2$ & $(1,0,0,1,0,1,1)$ & $(1,0,0,1,0,1,1)$ & $(1,0,0,1,0,1,1)$  & $A_5$ \\ \hline
        $E_7(a_2)$ & $E_7(a_2)$ & $2A_1$ & $(1,1,1,0,1,0,1)$ & $(1,1,1,0,1,0,1)$ & $(1,1,1,0,1,0,1)$  & $B_4+A_1$\\ \hline
        $E_7(a_1)$ & $E_7(a_1)$ & $A_1$ & $(1,1,1,0,1,1,1)$ & $(1,1,1,0,1,1,1)$ & $(1,1,1,0,1,1,1)$  & $D_6$\\ \hline
        $E_7$ & $E_7$ & $\{0\}$ & $(1,1,1,1,1,1,1)$ & $(1,1,1,1,1,1,1)$ & $(1,1,1,1,1,1,1)$  & $E_7$\\ \hline
        \end{tabular}
    \caption{Special distinguished Lusztig-Achar data in type $E_7$}
    \label{table:E7}
    \end{table}

                \begin{table}[H]
                \tiny
        \begin{tabular}{|c|c|c|c|c|c|c|} \hline
        $\OO^{\vee}$ & $\OO_{M^{\vee}}$ & $d_S(\OO^{\vee},\bar{C})$ & $\gamma(M^{\vee},\OO_{M^{\vee}})$ & $\gamma(\OO^{\vee},\bar{C})$ & $\gamma(D(\OO^{\vee},\bar{C}))$ & $\fr(\OO)$  \\ \hline
        $E_8(a_7)$ & $E_8(a_7)$ & $E_8(a_7)$ & $(0,0,0,0,1,0,0,0)$ & $(0,0,0,0,1,0,0,0)$ & $(0,0,0,0,1,0,0,0)$  & $1$ \\ \hline
        $E_8(a_7)$ & $E_7(a_5)+A_1$ & $E_7(a_5)$ & $(0,0,1,0,1,0,0,1)/2$ & $(0,0,1,0,1,0,0,1)/2$ & $(0,0,1,0,1,0,0,1)/2$  & $A_1$\\ \hline
        $E_8(a_7)$ & $D_8(a_5)$ & $D_6(a_2)$ & $(1,0,0,1,0,0,1,0)/2$ & $(1,0,0,1,0,0,1,0)/2$ & $(1,0,0,1,0,0,1,0)/2$ & $2A_1$\\ \hline
        $E_8(a_7)$ & $E_6(a_3)+A_2$ & $E_6(a_3)+A_1$ & $(0,1,1,0,1,0,1,1)/3$ & $(0,1,1,0,1,0,1,1)/3$ & $(0,1,1,0,1,0,1,1)/3$ & $A_1$\\ \hline
        $E_8(a_7)$ & $D_5(a_1)+A_3$ & $D_5(a_1)+A_2$ & $(1,1,1,0,1,1,1,1)/4$ & $(1,1,1,0,1,1,1,1)/4$ & $(1,1,1,0,1,1,1,1)/4$ & $A_1$\\ \hline
        $E_8(a_7)$ & $2A_4$ & $A_4+A_3$ & $(1,1,1,1,1,1,1,1)/5$ & $(1,1,1,1,1,1,1,1)/5$ & $(1,1,1,1,1,1,1,1)/5$ & $A_1$\\ \hline
        $E_8(a_7)$ & $A_5+A_2+A_1$ & $A_5+A_1$ & $(2,2,1,1,1,1,1,1)/6$ & $(2,2,1,1,1,1,1,1)/6$ & $(2,2,1,1,1,1,1,1)/6$ & $2A_1$\\ \hline
        $D_7(a_2)$ & $E_7+A_1$ & $2A_3$ & $(1,1,1,1,1,1,1,1)/4$ & $(1,1,1,1,1,1,1,1)/4$ & $(1,1,1,1,1,1,1,1)/4$ & $B_2$\\ \hline
        $E_8(b_6)$ & $E_8(b_6)$ & $D_4(a_1)+A_2$ & $(0,0,0,1,0,0,0,1)$& $(0,0,0,1,0,0,0,1)$ & $(0,0,0,1,0,0,0,1)$ & $A_2$\\
        & $E_6+A_2$ & & $(1,1,1,0,1,1,1,3)/3$ & & &\\ \hline
        $E_8(b_6)$ & $D_8(a_3)$ & $A_3+A_2+A_1$ & $(1,0,0,1,0,1,1,1)/2$ & $(1,0,0,1,0,1,1,1)/2$ & $(1,0,0,1,0,1,1,1)/2$ & $2A_1$\\ \hline
        $E_8(a_6)$ & $E_8(a_6)$ & $D_4(a_1)+A_1$ & $(0,0,0,1,0,0,1,0)$ & $(0,0,0,1,0,0,1,0)$ & $(0,0,0,1,0,0,1,0)$ & $3A_1$\\ \hline
        $E_8(a_6)$ & $D_8(a_2)$ & $A_3+2A_1$ & $(1,1,1,0,1,0,1,1)/2$ & $(1,1,1,0,1,0,1,1)/2$ & $(1,1,1,0,1,0,1,1)/2$ & $B_2+A_1$\\ \hline
        $E_8(a_6)$ & $A_8$ & $2A_2+2A_1$ & $(1,1,1,1,1,1,1,1)/3$ & $(1,1,1,1,1,1,1,1)/3$ & $(1,1,1,1,1,1,1,1)/3$ & $B_2$\\ \hline
        $E_8(b_5)$ & $E_8(b_5)$ & $D_4(a_1)$ & $(0,0,0,1,0,0,1,1)$ & $(0,0,0,1,0,0,1,1)$ & $(0,0,0,1,0,0,1,1)$ & $D_4$\\ \hline
        $E_8(b_5)$ & $E_7(a_2)+A_1$ & $A_3+A_1$ & $(1,1,0,1,0,1,1,2)/2$ & $(1,1,0,1,0,1,1,2)/2$ & $(1,1,0,1,0,1,1,2)/2$ & $B_3+A_1$\\ \hline
        $E_8(b_5)$ & $E_6+A_2$ & $2A_2+A_1$ & $(1,1,1,1,1,1,1,3)/3$ & $(1,1,1,1,1,1,1,3)/3$ & $(1,1,1,1,1,1,1,3)/3$ & $G_2+A_1$ \\ \hline
        $E_8(a_5)$ & $E_8(a_5)$ & $2A_2$ & $(1,0,0,1,0,0,1,0)$ & $(1,0,0,1,0,0,1,0)$ & $(1,0,0,1,0,0,1,0)$ & $2G_2$\\ \hline
        $E_8(a_5)$ & $D_8(a_1)$ & $A_2+3A_1$ & $(1,1,1,0,1,1,1,1)/2$ & $(1,1,1,0,1,1,1,1)/2$ & $(1,1,1,0,1,1,1,1)/2$ & $G_2+A_1$\\ \hline
        $E_8(b_4)$ & $E_8(b_4)$ & $A_2+2A_1$ & $(1,0,0,1,0,0,1,1)$ & $(1,0,0,1,0,0,1,1)$ & $(1,0,0,1,0,0,1,1)$ & $B_3+A_1$\\ 
        & $E_7+A_1$ & & $(1,1,1,1,0,1,1,2)/2$ & & &\\ \hline
        $E_8(a_4)$ & $E_8(a_4)$ & $A_2+A_1$ & $(1,0,0,1,0,1,0,1)$ & $(1,0,0,1,0,1,0,1)$ & $(1,0,0,1,0,1,0,1)$ & $A_5$\\ \hline
        $E_8(a_4)$ & $D_8$ & $4A_1$ & $(1,1,1,1,1,1,1,1)/2$ & $(1,1,1,1,1,1,1,1)/2$ & $(1,1,1,1,1,1,1,1)/2$ & $C_4$\\ \hline
        $E_8(a_3)$ & $E_8(a_3)$ & $A_2$ & $(1,0,0,1,0,1,1,1)$ & $(1,0,0,1,0,1,1,1)$ & $(1,0,0,1,0,1,1,1)$ & $E_6$\\ \hline
        $E_8(a_3)$ & $E_7+A_1$ & $3A_1$ & $(1,1,1,1,1,1,2,2)/2$ & $(1,1,1,1,1,1,2,2)/2$ & $(1,1,1,1,1,1,2,2)/2$ & $F_4+A_1$\\ \hline
        $E_8(a_2)$ & $E_8(a_2)$ & $2A_1$ & $(1,1,1,0,1,0,1,1)$ & $(1,1,1,0,1,0,1,1)$ & $(1,1,1,0,1,0,1,1)$ & $B_6$\\ \hline
        $E_8(a_1)$ & $E_8(a_1)$ & $A_1$ & $(1,1,1,0,1,1,1,1)$ & $(1,1,1,0,1,1,1,1)$ & $(1,1,1,0,1,1,1,1)$ & $E_7$ \\ \hline
        $E_8$ & $E_8$ & $\{0\}$ & $(1,1,1,1,1,1,1,1)$ & $(1,1,1,1,1,1,1,1)$ & $(1,1,1,1,1,1,1,1)$ & $E_8$\\ \hline
        \end{tabular}
    \caption{Special distinguished Lusztig-Achar data in type $E_8$}
    \label{table:E8}
    \end{table}

    \begin{table}[H]
    \tiny
        \begin{tabular}{|c|c|c|c|c|c|c|} \hline
        $(\OO^{\vee},\OO_{M^{\vee}})$ & $\OO$ & $(L,\OO_L)$ & $(R^{\vee},\OO_{R^{\vee}})$ & $\bar{A}_{R^{\vee}},A_L,A$ & $\Gamma$ & $\#$ \\ \hline 
        
        $(G_2(a_1),2A_1)$ & $\widetilde{A}_1$ & $(G_2,\widetilde{A}_1)$ & $(2A_1, [2] \times [2])$ & $1,1,1$ & $1$ & $(1)$\\ \hline 
        $(G_2(a_1),A_2)$ & $A_1$ & $(G_2,A_1)$ & $(A_2,[3])$ & $1,1,1$ & $1$ & $(1)$ \\ \hline
        
        
        $(A_1,A_1)$ & $G_2(a_1)$ & $(A_1,\{0\})$ & $(2A_1,[2] \times \{0\})$ & $1,1,S_3$ & $1$ & (5) \\ \hline
        
        $(\widetilde{A}_1,A_1)$ & $G_2(a_1)$ & $(A_1,\{0\})$ & $(2A_1,\{0\} \times [2])$ & $1,1,S_3$ & $1$ & $(3)$\\ \hline

        \end{tabular}
    \caption{Special Lusztig-Achar data (not of the form $(\OO^{\vee},1)$ with even $\OO^{\vee}$) in type $G_2$}
    \label{table:G2Gamma}
    \end{table}

\begin{table}[H]
    \tiny
        \begin{tabular}{|c|c|c|c|c|c|c|} \hline
        $(\OO^{\vee},\OO_{M^{\vee}})$ & $\OO$ & $(L,\OO_L)$ & $(R^{\vee},\OO_{R^{\vee}})$ & $\bar{A}_{R^{\vee}},A_L,A$ & $\Gamma$ & $\#$ \\ \hline 
        
        $(F_4(a_3),B_4(a_1))$ & $B_2$ & $(F_4,B_2)$ & $(B_4,[5,3,1])$  & $\ZZ_2,S_2,S_2$ & $S_2$ & $(2)$\\ \hline 
        
        $(F_4(a_3),C_3(a_1)+A_1)$ & $C_3(a_1)$ & $(F_4,C_3(a_1))$ & $(C_3+A_1,[4,2] \times [2])$ & $\ZZ_2,S_2,S_2$ & $S_2$ & $(2)$\\ \hline
        
        $(F_4(a_3),2A_2)$ & $\widetilde{A}_2+A_1$ & $(F_4,\widetilde{A}_2+A_1)$  & $(2A_2,[3]^2)$ & $1,1,1$ & $1$ & $(1)$ \\ \hline
        
        $(F_4(a_3),A_3+A_1)$ & $A_2+\widetilde{A}_1$ & $(F_4,A_2+\widetilde{A}_1)$  & $(A_3+A_1,[4] \times [2])$ & $1,1,1$ & $1$ & $(1)$ \\ \hline
        
        $(F_4(a_1),B_4)$ & $A_1$ & $(F_4,A_1)$ & $(B_4,[9])$ & $1,1,1$ & $1$ & $(1)$\\ \hline
        
        $(C_3(a_1),A_1+B_2)$ & $C_3(a_1)$ & $(B_3,(2^2,1^3))$ & $(B_4,[5,2^2])$ & $1,1,S_2$ & $1$ & $(4)$\\ \hline
        
        $(B_2,A_3)$ & $B_2$ & $(C_3,(2,1^4))$ & $(A_3+A_1,[4]\times \{0\})$  & $1,1,S_2$ & $1$ & $(4)$\\ \hline
    
        $(A_1,A_1)$ & $F_4(a_1)$ & $(A_1,\{0\})$ & $(C_3+A_1,\{0\} \times [2])$ &  $1,1,S_2$ & $1$ & $(3)$\\ \hline
        
        $(\widetilde{A}_1,A_1)$ & $F_4(a_1)$ & $(A_1,\{0\})$ & $(B_4,[3,1^6])$ & $\ZZ_2,1,S_2$ & $S_2$ & $(8)$\\ \hline

        $(A_1+\widetilde{A}_1,2A_1)$ & $F_4(a_2)$ & $(2A_1,\{0\})$ & $(C_3+A_1,[2^3] \times \{0\})$ & $1,1,S_2$ & $1$ & $(3)$\\ \hline

        $(A_2+\widetilde{A}_1,A_2+A_1)$ & $F_4(a_3)$ & $(A_2+A_1,\{0\})$ & $(B_4,[3^3])$ & $1,1,S_4$ & $1$ & $(7)$ \\ \hline

        $(\widetilde{A}_2+A_1,A_2+A_1)$ & $F_4(a_3)$ & $(A_2+A_1,\{0\})$  & $(C_3+A_1,[3^2] \times [2])$ & $1,1,S_4$ & $1$ & $(7)$\\ \hline
        
        $(C_3(a_1),C_3(a_1))$ & $F_4(a_3)$ & $(C_3,(4,2))$ & $(C_3+A_1,[4,2] \times \{0\})$ & $\ZZ_2,S_2,S_4$ & $S_2$ & $(7)$ \\ \hline

        $(C_3,C_3)$ & $A_2$ & $(C_3,\{0\})$ & $(C_3+A_1,[6] \times \{0\})$  & $1,1,S_2$ & $1$ & $(3)$\\ \hline
        \end{tabular}
    \caption{Special Lusztig-Achar data (not of the form $(\OO^{\vee},1)$ with even $\OO^{\vee}$) in type $F_4$}
    \label{table:F4Gamma}
    \end{table}

    \begin{table}[H]
    \tiny
        \begin{tabular}{|c|c|c|c|c|c|c|} \hline
        $(\OO^{\vee},\OO_{M^{\vee}})$ & $\OO$ & $(L,\OO_L)$ & $(R^{\vee},\OO_{R^{\vee}})$ & $\bar{A}_{R^{\vee}},A_L,A$ & $\Gamma$ & $\#$\\ \hline 
        
        $(D_4(a_1),3A_2)$ & $2A_2+A_1$ & $(E_6,2A_2+A_1)$ & $(3A_2, [3] \times [3] \times [3])$ & $1,1,1$ & $1$ & $(1)$\\ \hline 
        
        $(E_6(a_3),A_5+A_1)$ & $3A_1$ & $(E_6,3A_1)$ & $(A_5+A_1,[6] \times [2])$ & $1,1,1$ & $1$ & $(1)$\\ \hline

        $(D_4(a_1),A_3+2A_1)$ & $A_3+A_1$ & $(D_5,(3,2^2,1^3))$ & $(A_5+A_1,[4,2] \times [2])$ & $1,1,1$ & $1$ & $(1)$\\ \hline
        
        $(A_2,4A_1)$ & $A_5$ & $(D_4,(3,2^2,1))$ & $(A_5+A_1,[3,1^3] \times \{0\})$ & $1,1,1$ & $1$ & $(1)$\\ \hline
        
        $(A_1,A_1)$ & $E_6(a_1)$ & $(A_1,\{0\})$  & $(A_5+A_1,\{0\} \times [2])$  & $1,1,1$ & $1$ & $(1)$\\ \hline

        $(2A_1,2A_1)$ & $D_5$ & $(2A_1,\{0\})$ & $(D_5,[3,1^7])$ &  $1,1,1$ & $1$ & $(1)$\\ \hline

        $(3A_1,3A_1)$ & $E_6(a_3)$ & $(3A_1,\{0\})$ & $(A_5+A_1,[2^3] \times \{0\})$ & $1,1,S_2$ & $1$ & $(3)$\\ \hline

        $(A_2+A_1,A_2+A_1)$ & $D_5(a_1)$ & $(A_2+A_1,\{0\})$ & $(A_5+A_1,[3,1^4] \times [2])$ & $1,1,1$ & $1$ & $(1)$\\ \hline
        
        $(A_2+2A_1,A_2+2A_1)$ & $A_4+A_1$ & $(A_2+2A_1,\{0\})$ & $(D_5,[3^3,1])$ & $1,1,1$ & $1$ & $(1)$\\ \hline
        
        $(A_3,A_3)$  & $A_4$ & $(A_3,\{0\})$ & $(D_5,[5,1^5])$ & $1,1,1$ & $1$ & $(1)$\\ \hline
        
        $(2A_2+A_1,2A_2+A_1)$ & $D_4(a_1)$ & $(2A_2+A_1,\{0\})$ & $(A_5+A_1,[3^2] \times [2])$ & $1,1,S_3$ & $1$ & $(3)$\\ \hline
        
        $(A_3+A_1,A_3+A_1)$ & $D_4(a_1)$ & $(A_3+A_1,\{0\})$ & $(A_5+A_1,[4,2] \times \{0\})$ & $1,1,S_3$ & $1$ & $(8)$ \\ \hline
        
        $(A_4+A_1,A_4+A_1)$ & $A_2+2A_1$ & $(A_4+A_1,\{0\})$ & $(A_5+A_1,[5,1] \times [2])$ & $1,1,1$ & $1$ & $(1)$\\ \hline
        
        $(A_5,A_5)$ & $A_2$ & $(A_5,\{0\})$ & $(A_5+A_1,[6] \times \{0\})$ & $1,1,S_2$ & $1$ & $(3)$\\ \hline
        
        $(D_5(a_1),D_5(a_1))$ & $A_2+A_1$ & $(D_5,(2^2,1^6))$ & $(D_5,[7,3])$  & $1,1,1$ & $1$ & $(1)$\\ \hline
        
        \end{tabular}
    \caption{Special Lusztig-Achar data (not of the form $(\OO^{\vee},1)$ with even $\OO^{\vee}$) in type $E_6$}
    \label{table:E6Gamma}
    \end{table}

    \begin{table}[H]
    \tiny
        \begin{tabular}{|c|c|c|c|c|c|c|c|c|c|c|c|} \hline
        $(\OO^{\vee},\OO_{M^{\vee}})$ & $\OO$ & $(L,\OO_L)$ & $(R^{\vee},\OO_{R^{\vee}})$ & $\bar{A}_{R^{\vee}},A_L, A$ & $\Gamma$ & $\#$ \\ \hline 
        
        $(E_7(a_5),D_6(a_2)+A_1)$ & $(A_3+A_1)'$ & $(E_7,(A_3+A_1)')$ & $(D_6+A_1,[7,5] \times [2])$ & $1,1,1$ & $1$ & $(1)$\\ \hline 
    
        $(E_7(a_5),A_5+A_2)$ & $2A_2+A_1$ & $(E_7,2A_2+A_1)$ & $(A_5+A_2,[6] \times [3])$ & $1,1,1$ & $1$ & $(1)$\\ \hline

         $(E_6(a_1),A_7)$ & $4A_1$ & $(E_7,4A_1)$ & $(A_7,[8])$ & $1,1,1$ & $1$ & $(1)$\\ \hline
        
        $(E_6(a_3),A_5+A_1)$ & $D_5+A_1$ & $(E_6,3A_1)$ & $(D_6+A_1,[6^2]^I \times [2])$ & $1,1,1$ & $1$ & $(1)$\\ \hline
        
        $(A_4,2A_3)$  & $D_4+A_1$ & $(D_6,(3,2^4,1))$ & $(2A_3+A_1,[4]^2 \times \{0\})$ & $1,1,1$ & $1$ & $(1)$\\ \hline
        
        $(D_4(a_1)+A_1,A_3+3A_1)$ & $(A_5)'$ & $(D_5+A_1,[3,2^2] \times \{0\})$  & $(D_6+A_1,[5,3,2^2] \times \{0\})$ & $1,1,1$ & $1$ & $(1)$\\ \hline
        
        $(D_4(a_1),A_3+2A_1)$ & $D_6(a_2)$ & $(D_5,(3,2^2,1^3))$ & $(D_6+A_1,[4^2,2^2]^I \times \{0\})$ & $1,1,1$ & $1$ & $(1)$\\ \hline
        
        $(A_2,4A_1)$ & $D_6$ & $(D_4,(3,2^2,1))$ & $(D_6+A_1,[2^6]^I \times [2])$  & $1,1,1$ & $1$ & $(1)$\\ \hline
        
        $(A_1,A_1)$  & $E_7(a_1)$ & $(A_1,\{0\})$  & $(D_6+A_1,\{0\} \times [2])$  & $1,1,1$ & $1$ & $(1)$\\ \hline
        
        $(2A_1,2A_1)$ & $E_7(a_2)$ & $(2A_1,\{0\})$ & $(D_6+A_1,[3,1^9] \times \{0\})$  & $1,1,1$ & $1$ & $(1)$\\ \hline
        
        $((3A_1)',(3A_1)')$ & $E_7(a_3)$ & $(3A_1,\{0\})$ & $(D_6+A_1,[2^6]^I \times \{0\})$ & $1,1,S_2$ & $1$ & $(3)$\\ \hline
        
        $(4A_1,4A_1)$  & $E_6(a_1)$ & $(4A_1,\{0\})$ & $(D_6+A_1,[2^6]^{II} \times [2])$  & $1,1,S_2$ & $1$ & $(3)$\\ \hline
        
        $(A_2+A_1,A_2+A_1)$ & $E_6(a_1)$ & $(A_2+A_1,\{0\})$ & $(D_6+A_1, [3^2,1^6] \times [2])$  & $\ZZ_2,1,S_2$ & $S_2$ & $(8)$\\ \hline
        
        $(A_3,A_3)$ & $D_6(a_1)$ & $(A_3,\{0\})$ & $(D_6+A_1,[5,1^7] \times \{0\})$ & $1,1,1$ & $1$ & $(1)$\\ \hline
        
        $(2A_2+A_1,2A_2+A_1)$ & $E_7(a_5)$ & $(2A_2+A_1,\{0\})$ & $(D_6+A_1,[3^4] \times [2])$ & $1,1,S_3$ & $1$ & $(3)$\\ \hline
        
        $((A_3+A_1)',(A_3+A_1)')$  & $E_7(a_5)$ & $(A_3+A_1,\{0\})$  & $(D_6+A_1,[4^2,2^2]^I \times \{0\})$ & $1,1,S_3$ & $1$ & $(5)$\\ \hline
        
        $(A_3+2A_1,A_3+2A_1)$ & $E_6(a_3)$ & $(A_3+2A_1,\{0\})$ & $(D_6+A_1,[4^2,2^2]^{I} \times [2])$ & $1,1,S_2$ & $1$ & $(3)$\\ \hline
        
        $(D_4(a_1)+A_1,D_4(a_1)+A_1)$ & $(A_5)''$ & $(D_4+A_1,[2^2,1^4] \times \{0\})$  & $(D_6+A_1, [4^2,2^2]^{II} \times [2])$ & $1,1,1$ & $1$ & $(1)$\\ \hline
        
        $(A_3+A_2,A_3+A_2)$ & $D_5(a_1)+A_1$ & $(A_3+A_2,\{0\})$ & $(D_6+A_1,[5,3^2,1] \times \{0\})$ & $1,1,1$ & $1$ & $(1)$\\ \hline
        
        $(D_4+A_1,D_4+A_1)$ & $A_4$ & $(D_4+A_1,\{0\})$ & $(D_6+A_1,[7,1^5] \times [2])$ & $1,1,S_2$ & $1$ & $(3)$\\ \hline
        
        $(A_4+A_1,A_4+A_1)$  & $A_4+A_1$ & $(A_4+A_1,\{0\})$ & $(D_6+A_1,[5^2,1^2] \times [2])$ & $\ZZ_2,1,S_2$ & $S_2$ & $(8)$\\ \hline
        
        $(D_5(a_1),D_5(a_1))$ & $A_4$ & $(D_5,(2^2,1^6))$ & $(D_6+A_1,[7,3,1^2] \times \{0\})$ & $\ZZ_2,1,S_2$ & $S_2$ & $(8)$\\ \hline
        
        $((A_5)',(A_5)')$ & $D_4(a_1)+A_1$ & $(A_5,\{0\})$ & $(D_6+A_1,[6^2]^I \times \{0\})$ & $1,1,S_2$ & $1$ & $(4)$\\ \hline
        
        $(A_5+A_1,A_5+A_1)$ & $D_4(a_1)$ & $(A_5+A_1,\{0\})$ & $(D_6+A_1,[6^2]^{II} \times [2])$ & $1,1,S_3$ & $1$ & $(3)$\\ \hline
        
        $(D_6(a_2),D_6(a_2))$ & $D_4(a_1)$ & $(D_6,(2^4,1^4))$ & $(D_6+A_1, [7,5] \times \{0\})$ & $1,1,S_3$ & $1$ & $(8)$ \\ \hline

        $(D_5+A_1,D_5+A_1)$ & $2A_2$ & $(D_5+A_1,\{0\})$ & $(D_6+A_1,[9,1^3] \times [2])$ & $1,1,1$ & $1$ & $(1)$ \\ \hline

        $(D_6(a_1),D_6(a_1))$ & $A_3$ & $(D_6,(2^2,1^8))$ & $(D_6+A_1,[9,3] \times \{0\})$ & $1,1,1$ & $1$ & $(1)$ \\ \hline
        
        $(D_6,D_6)$ & $A_2$ & $(D_6,\{0\})$ & $(D_6+A_1,[11,1] \times \{0\})$ &  $1,1,S_2$ & $1$ & $(3)$\\ \hline
        \end{tabular}
    \caption{Special Lusztig-Achar data (not of the form $(\OO^{\vee},1)$ with even $\OO^{\vee}$) in type $E_7$}
    \label{table:E7Gamma}
    \end{table}

{\tiny
\begin{longtable}[H]{|c|c|c|c|c|c|c|}
\hline
        $(\OO^{\vee},\OO_{M^{\vee}})$ & $\OO$ & $(L,\OO_L)$ & $(R^{\vee},\OO_{R^{\vee}})$ & $\bar{A}_{R^{\vee}},A_L,A$ & $\Gamma$ & $\#$ \\ \hline 
        
        $(E_8(a_7),E_7(a_5)+A_1)$ & $E_7(a_5)$ & $(E_8,E_7(a_5))$ & $(E_7+A_1,E_7(a_5) \times [2])$ & $S_3,S_3,S_3$ & $S_3$ & $(2)$ \\\hline 
        
        $(E_8(a_7),D_8(a_5))$ & $D_6(a_2)$ & $(E_8,D_6(a_2))$ & $(D_8,[7,5,3,1])$ & $\ZZ_2,S_2,S_2$ & $S_2$ & $(2)$\\ \hline
        
        $(E_8(a_7),E_6(a_3)+A_2)$ & $E_6(a_3)+A_1$ & $(E_8,E_6(a_3)+A_1)$ & $(E_6+A_2,E_6(a_3) \times [3])$ & $S_2,S_2,S_2$ & $S_2$ & $(2)$\\ \hline

        $(E_8(a_7),D_5(a_1)+A_3)$ & $D_5(a_1)+A_2$ & $(E_8,D_5(a_1)+A_2)$ & $(D_5+A_3,[7,3] \times [4])$ & $1,1,1$ & $1$ & $(1)$\\ \hline
        
        $(E_8(a_7),2A_4)$ & $A_4+A_3$ & $(E_8,2A_4)$ & $(2A_4,[5]^2)$ & $1,1,1$ & $1$ & $(1)$\\ \hline
        
        $(E_8(a_7), A_5+A_2+A_1)$ & $A_5+A_1$ & $(E_8,A_5+A_1)$  & $(A_5+A_2+A_1,[6] \times [3] \times [2])$ & $1,1,1$ & $1$ & $(1)$\\ \hline
        
        $(D_7(a_2),E_7+A_1)$ & $2A_3$ & $(E_8,2A_3)$ & $(D_5+A_3,[7,3] \times [4])$ & $1,1,1$ & $1$ & $(1)$\\ \hline

        $(E_8(b_6),D_8(a_3))$ & $A_3+A_2+A_1$ & $(E_8,A_3+A_2+A_1)$  & $(D_8,[9,7])$ & $1,1,1$ & $1$ & $(1)$\\ \hline

        $(E_8(a_6),D_8(a_2))$ & $A_3+2A_1$ & $(E_8,A_3+2A_1)$  & $(D_8,[11,5])$ & $1,1,1$ & $1$ & $(1)$\\ \hline

        $(E_8(a_6),A_8)$ & $2A_2+2A_1$ & $(E_8,2A_2+2A_1)$  & $(A_8,[9])$ & $1,1,1$ & $1$ & $(1)$ \\ \hline
        
        $(E_8(b_5),E_7(a_2)+A_1)$ & $A_3+A_1$ & $(E_8,A_3+A_1)$ & $(E_7+A_1,E_7(a_2) \times [2])$ & $1,1,1$ & $1$ & $(1)$\\ \hline
        
        $(E_8(b_5),E_6+A_2)$ & $2A_2+A_1$ & $(E_8,2A_2+A_1)$ & $(E_6+A_2, E_6 \times [2])$ & $1,1,1$ & $1$ & $(1)$\\ \hline
        
        $(E_8(a_5),D_8(a_1))$ & $A_2+3A_1$ & $(E_8,A_2+3A_1)$ & $(D_8,[13,3])$ & $1,1,1$ & $1$ & $(1)$\\ \hline
        
        $(E_8(a_4),D_8)$ & $4A_1$ & $(E_8,4A_1)$ & $(D_8,[15,1])$ & $1,1,1$ & $1$ & $(1)$\\ \hline
        
        $(E_8(a_3),E_7+A_1)$ & $3A_1$ & $(E_8,3A_1)$ & $(E_7+A_1,E_7 \times [2])$ & $1,1,1$ & $1$ & $(1)$\\ \hline

        $(E_6(a_1),A_7)$ & $D_4+A_1$ & $(E_7,4A_1)$ & $(A_7,[7,1])$ & $1,1,1$ & $1$ & $(1)$\\ \hline

        $(D_6(a_1),D_5+2A_1)$ & $A_5$ & $(D_7,[3,2^2,1^7])$ & $(E_7+A_1, D_6(a_1) \times [2])$ & $1,1,1$ & $1$ & $(1)$\\ \hline
        
        $(E_7(a_5),A_5+A_2)$ & $E_6(a_3)+A_1$ & $(E_7,2A_2+A_1)$ & $(A_5+A_2,[6] \times [2,1])$ & $1,1,S_2$ & $1$ & $(4)$ \\ \hline

        $(E_7(a_5),A_1+D_6(a_2))$ & $E_7(a_5)$ & $(E_7,(A_3+A_1)')$ & $(D_8,[7,5,2,2])$ & $1,1,S_3$ & $1$ & $(8)$ \\ \hline

        $(D_6(a_2),D_4+A_3)$ & $D_6(a_2)$ & $(D_7,(3,2^4,1^3))$ & $(D_5+A_3,[5^2] \times [4])$ & $1,1,S_2$ & $1$ & $(4)$\\ \hline

        $(E_6(a_3)+A_1,A_5+2A_1)$ & $E_7(a_5)$ & $(E_6+A_1,3A_1 \times \{0\})$ & $(D_8,[6,6,3,1])$ & $1,1,S_3$ & $1$ & $(8)$ \\ \hline

        $(E_6(a_3),A_5+A_1)$ & $D_5+A_1$ & $(E_6,3A_1)$ & $(E_7+A_1,(A_5)'' \times [2])$ & $1,1,1$ & $1$ & $(1)$\\ \hline

        $(A_4,2A_3)$ & $D_6$ & $(D_6,(3,2^4,1))$ & $(D_5+A_3, [5^2] \times \{0\})$ & $1,1,1$ & $1$ & $(1)$\\ \hline

        $(D_4(a_1)+A_2,A_3+A_2+2A_1)$ & $A_7$ & $(D_5+A_2,[3,2^2,1^3] \times \{0\})$ & $(E_7+A_1,(A_5)'' \times \{0\})$ & $1,1,1$ & $1$ & $(1)$\\ \hline
        
        $(D_4(a_1),3A_2)$ & $E_6+A_1$ & $(E_6,2A_2+A_1)$ & $(E_6+A_2,2A_2 \times [3])$ & $1,1,1$ & $1$ & $(1)$\\ \hline

        $(D_4(a_1),A_3+2A_1)$ & $E_7(a_2)$ & $(D_5,(3,2^2,1^3))$ & $(E_7+A_1,(A_3+A_1)'' \times [2])$ & $1,1,1$ & $1$ & $(1)$\\ \hline

        $(2A_2,A_2+4A_1)$ & $D_7$ & $(D_4+A_2,(3,2^2,1) \times \{0\})$ & $(E_7+A_1,2A_2 \times \{0\})$ & $1,1,1$ & $1$ & $(1)$\\ \hline

        $(A_2,4A_1)$ & $E_7$ & $(D_4,(3,2^2,1))$ & $(E_7+A_1,(3A_1)'' \times [2])$  & $1,1,1$ & $1$ & $(1)$ \\ \hline
         
        $(A_1,A_1)$ & $E_8(a_1)$ & $(A_1,\{0\})$ & $(E_7+A_1,\{0\} \times [2])$ & $1,1,1$ & $1$ & $(1)$\\ \hline

        $(2A_1,2A_1)$ & $E_8(a_2)$ & $(2A_1,\{0\})$ & $(D_8,[3,1^{13}])$ & $1,1,S_2$ & $1$ & $(3)$ \\ \hline

        $(3A_1,3A_1)$ & $E_8(a_3)$ & $(3A_1,\{0\})$ & $(E_7+A_1,(3A1)'' \times \{0\})$ & $1,1,S_2$ & $1$ & $(3)$\\ \hline
        
        $(4A_1,4A_1)$ & $E_8(a_4)$ & $(4A_1,\{0\})$ & $(D_8,[2^8]^I)$ & $1,1,S_2$ & $1$ & $(3)$\\ \hline

        $(A_2+A_1,A_2+A_1)$ & $E_8(a_4)$ & $(A_2+A_1,\{0\})$ & $(E_7+A_1,A_2 \times [2])$ & $S_2,1,S_2$ & $S_2$ & $(8)$\\ \hline

        $(A_2+2A_1,A_2+2A_1)$ & $E_8(b_4)$ & $(A_2+2A_1,\{0\})$ & $(D_8,[3^3,1^7])$ & $1,1,S_2$ & $1$ & $(3)$\\ \hline
        
        $(A_3,A_3)$ & $E_7(a_1)$ & $(A_3,\{0\})$  & $(D_8,[5,1^{11}])$ & $1,1,1$ & $1$ & $(1)$\\ \hline
        
        $(A_2+3A_1,A_2+3A_1)$ & $E_8(a_5)$ & $(A_2+3A_1,\{0\})$ & $(E_7+A_1,(A_2+3A_1) \times \{0\})$ & $1,1,S_2$ & $1$ & $(3)$\\ \hline
        
        $(2A_2+A_1,2A_2+A_1)$ & $E_8(b_5)$ & $(2A_2+A_1,\{0\})$ & $(E_7+A_1,2A_2 \times [2])$ & $1,1,S_3$ & $1$ & $(3)$\\ \hline
        
        $(A_3+A_1,A_3+A_1)$ & $E_8(b_5)$ & $(A_3+A_1,\{0\})$ & $(E_7+A_1,(A_3+A_1)'' \times \{0\})$ & $1,1,S_3$ & $1$ & $(5)$\\ \hline
        
        $(2A_2+2A_1,2A_2+2A_1)$ & $E_8(a_6)$ & $(2A_2+2A_1,\{0\})$  & $(D_8,[3^5,1])$ & $1,1,S_3$ & $1$ & $(3)$ \\ \hline

        $(A_3+2A_1,A_3+2A_1)$ & $E_8(a_6)$ & $(A_3+2A_1,\{0\})$  & $(D_8,[6^2,2^2]^I)$ & $1,1,S_3$ & $1$ & $(\ast)$ \\ \hline
        
        $(D_4(a_1)+A_1,D_4(a_1)+A_1)$ & $E_8(a_6)$ & $(D_4+A_1,[2^2,1^4] \times \{0\})$ & $(E_7+A_1,D_4(a_1) \times [2])$ & $S_3,1,S_3$ & $S_3$ & $(\ast)$\\ \hline
        
        $(A_3+A_2,A_3+A_2)$ & $D_7(a_1)$ & $(A_3+A_2,\{0\})$  & $(D_8,[5,3^2,1^5])$ & $1,1,S_2$ & $1$ & $(3)$\\ \hline
        
        $(A_3+A_2+A_1,A_3+A_2,A_1)$ & $E_8(b_6)$ & $(A_3+A_2+A_1,\{0\})$ & $(E_7+A_1,(A_3+A_2+A_1) \times \{0\})$ & $1,1,S_3$ & $1$ & $(3)$ \\ \hline
        
        $(D_4+A_1,D_4+A_1)$ & $E_6(a_1)$ & $(D_4+A_1,\{0\})$ & $(E_7+A_1,D_4 \times [2])$ & $1,1,S_2$ & $1$ & $(3)$\\ \hline

        $(A_4+A_1,A_4+A_1)$ & $E_6(a_1)+A_1$ & $(A_4+A_1,\{0\})$ & $(E_7+A_1,A_4 \times [2])$ & $S_2,1,S_2$ & $S_2$ & $(\ast)$\\ \hline

        $(2A_3,2A_3)$  & $D_7(a_2)$ & $(2A_3,\{0\})$ & $(D_8,[4^4]^I)$ & $1,1,S_2$  & $1$ & $(\ast)$\\ \hline
        
        $(D_5(a_1), D_5(a_1))$ & $E_6(a_1)$ & $(D_5,(2^2,1^6))$ & $(D_8,[7,3,1^6])$ & $\ZZ_2,1,S_2$  & $S_2$ & $(8)$\\ \hline

        $(A_4+2A_1,A_4+2A_1)$ & $D_7(a_2)$ & $(A_4+2A_1,\{0\})$ & $(D_8,[5^2,3,1^3])$ & $\ZZ_2,1,S_2$ & $S_2$ & $(\ast)$\\ \hline
        
        $(A_5,A_5)$ & $D_6(a_1)$ & $(A_5,\{0\})$  & $(E_7+A_1,(A_5)'' \times \{0\})$ & $1,1,S_2$ & $1$ & $(4)$ \\ \hline
        
        $(D_5(a_1)+A_1,D_5(a_1)+A_1)$ & $E_7(a_4)$ & $(D_5+A_1,[2^2,1^6] \times \{0\})$ & $(E_7+A_1,(D_5(a_1)+A_1) \times \{0\})$ & $1,1,S_2$  & $1$ & $(\ast)$\\ \hline
        
        $(A_4+A_2+A_1,A_4+A_2+A_1)$ & $A_6+A_1$ & $(A_4+A_2+A_1,\{0\})$ & $(E_7+A_1,(A_4+A_2) \times [2])$ & $1,1,1$ & $1$ & $(1)$\\ \hline
        
        $(A_4+A_3,A_4+A_3)$ & $E_8(a_7)$ & $(A_4+A_3,\{0\})$  & $(D_8,[5^3,1])$ & $1,1,S_5$ & $1$ & $(3)$\\ \hline

        $(A_5+A_1,A_5+A_1)$ & $E_8(a_7)$ & $(A_5+A_1,\{0\})$ & $(D_8,[6^2,2^2]^{I})$ & $1,1,S_5$ & $1$ & $(6)$\\ \hline

        $(D_5(a_1)+A_2,D_5(a_1)+A_2)$  & $E_8(a_7)$ & $(D_5+A_2,[2^2,1^6] \times \{0\})$ & $(D_8,[6^2,2^2]^{II})$ & $1,1,S_5$ & $1$ & $(6)$\\ \hline

        $(E_6(a_3)+A_1,E_6(a_3)+A_1)$ & $E_8(a_7)$ & $(E_6+A_1,A_2 \times \{0\})$ & $(E_7+A_1,E_6(a_3) \times (2))$ & $S_2,S_2,S_5$ & $S_2$ & $(6)$\\ \hline

        $(D_6(a_2), D_6(a_2))$ & $E_8(a_7)$ & $(D_6,[2^4,1^4])$ & $(D_8,[7,5,1^4])$ &  $S_2,1,S_5$ & $S_2$ & $(6)$\\ \hline
    
        $(E_7(a_5),E_7(a_5))$ & $E_8(a_7)$ & $(E_7,D_4(a_1))$ & $(E_7+A_1,E_7(a_5) \times \{0\})$ & $S_3,S_3,S_5$ & $S_3$ & $(6)$\\ \hline

        $(D_5+A_1,D_5+A_1)$  & $E_6(a_3)$ & $(D_5+A_1,\{0\})$ & $(E_7+A_1,D_5 \times (2))$ & $1,1,S_2$ & $1$ & $(3)$\\ \hline

        $(D_6(a_1),D_6(a_1))$ & $E_6(a_3)$ & $(D_6,[2^2,1^8])$ & $(D_8,[9,3,1^4])$ & $S_2,1,S_2$ & $S_2$ & $(8)$\\ \hline

        $(A_6+A_1,A_6+A_1)$ & $A_4+A_2+A_1$ & $(A_4+A_2+A_1,\{0\})$ & $(E_7+A_1,A_6 \times [2])$  &$1,1,1$ & $1$ & $(1)$\\ \hline

        $(E_7(a_4),E_7(a_4))$ & $D_5(a_1)+A_1$ & $(E_7,A_2+2A_1)$ & $(E_7+A_1,E_7(a_4) \times \{0\})$ & $1,1,1$ & $1$ & $(1)$\\ \hline

        $(D_6,D_6)$ & $A_4$ & $(D_6,\{0\})$ & $(D_8,[11,1^5])$ & $1,1,S_2$ & $1$ & $(3)$\\ \hline

        $(D_7(a_2),D_7(a_2))$ & $A_4+2A_1$ & $(D_7,[2^4,1^6])$ & $(D_8,[9,5,1^2])$ & $\ZZ_2,1,S_2$ & $S_2$ & $(8)$\\ \hline

        $(A_7,A_7)$ & $D_4(a_1)+A_2$ & $(A_7,\{0\})$ & $(D_8,[8^2]^I)$ & $1,1,S_2$ & $1$ & $(3)$\\ \hline

        $(E_6(a_1)+A_1,E_6(a_1)+A_1)$ & $A_4+A_1$ & $(E_6+A_1,A_1 \times \{0\})$ & $(E_7+A_1,E_6(a_1) \times [2])$ & $S_2,1,S_2$ & $S_2$ & $(8)$\\ \hline

        $(E_7(a_3),E_7(a_3))$ & $A_4$ & $(E_7,A_2)$  & $(E_7+A_1,D_6 \times [2])$  & $S_2,1,S_2$ & $S_2$ & $(8)$ \\ \hline

        $(E_6+A_1,E_6+A_1)$ & $D_4(a_1)$ & $(E_6+A_1,\{0\})$  & $(E_7+A_1,E_6 \times [2])$ & $1,1,S_3$ & $1$ & $(3)$\\ \hline

        $(E_7(a_2),E_7(a_2))$ & $D_4(a_1)$ & $(E_7,2A_1)$ & $(E_7+A_1,E_7(a_2) \times \{0\})$ & $1,1,S_3$ & $1$ & $(8)$ \\ \hline

        $(D_7,D_7)$ & $2A_2$ & $(D_7,\{0\})$  & $(D_8,[11,5])$ & $1,1,S_2$ & $1$ & $(3)$\\ \hline
        
        $(E_7(a_1),E_7(a_1))$ & $A_3$ & $(E_7,A_1)$  & $(E_7+A_1,E_7(a_1) \times \{0\})$ & $1,1,1$ & $1$ & $(1)$\\ \hline

        $(E_7,E_7)$ & $A_2$ & $(E_7,\{0\})$  & $(E_7+A_1,E_7 \times \{0\})$ & $1,1,S_2$ & $1$ & $(3)$\\ \hline
        
        \caption{Special Lusztig-Achar data (not of the form $(\OO^{\vee},1)$ with even $\OO^{\vee}$) in type $E_8$}
        \label{table:E8Gamma}
    \end{longtable}
}

\begin{sloppypar} \printbibliography[title={References}] \end{sloppypar}

@misc{Alvis,
    eprint={0506377},
    archivePrefix={arXiv},
    primaryClass={math.RT},
	year  = {2005},
	author = {Alvis, D.},
	title = {Induce/restrict matrices for exceptional Weyl groups}
}

@article {Frame1951,
    AUTHOR = {Frame, J. S.},
     TITLE = {The classes and representations of the groups of {$27$} lines
              and {$28$} bitangents},
   JOURNAL = {Ann. Mat. Pura Appl. (4)},
  FJOURNAL = {Annali di Matematica Pura ed Applicata. Serie Quarta},
    VOLUME = {32},
      YEAR = {1951},
     PAGES = {83--119},
 %     ISSN = {0003-4622},
   MRCLASS = {20.0X},
  MRNUMBER = {47038},
MRREVIEWER = {H. S. M. Coxeter},
%       DOI = {10.1007/BF02417955},
 %      URL = {https://doi.org/10.1007/BF02417955},
}

@incollection{KP_special,
     author = {Kraft, Hanspeter and Procesi, Claudio},
     title = {A special decomposition of the nilpotent cone of a classical {Lie} algebra},
     series = {Ast\'erisque},
     number = {173-174},
     year = {1989}
}

@article{fuetall2015,
 % doi = {10.1016/j.aim.2016.09.010},
  %url = {https://doi.org/10.1016/j.aim.2016.09.010},
  year = {2017},
  month = jan,
  publisher = {Elsevier {BV}},
  volume = {305},
  pages = {1--77},
  author = {B. Fu, D. Juteau, P. Levy and E. Sommers},
  title = {Generic singularities of nilpotent orbit closures},
  journal = {Advances in Mathematics}
}

@inproceedings {Frame1970,
    AUTHOR = {Frame, J. S.},
     TITLE = {The characters of the {W}eyl group {$E_{8}$}},
 BOOKTITLE = {Computational {P}roblems in {A}bstract {A}lgebra ({P}roc.
              {C}onf., {O}xford, 1967)},
     PAGES = {111--130},
 PUBLISHER = {Pergamon, Oxford},
      YEAR = {1970},
   MRCLASS = {20.80},
  MRNUMBER = {0269751},
MRREVIEWER = {H. S. M. Coxeter},
}

@article {SommersGunnells,
    AUTHOR = {Sommers. E. and Gunnells, P.},
    TITLE = {A characterization of Dynkin elements},
   JOURNAL = {Math. Res. Let.},
    year = {2003}
}

@book {GeckPfeiffer2000,
    AUTHOR = {Geck, M. and Pfeiffer, G.},
     TITLE = {Characters of finite {C}oxeter groups and {I}wahori-{H}ecke
              algebras},
    SERIES = {London Mathematical Society Monographs. New Series},
    VOLUME = {21},
 PUBLISHER = {The Clarendon Press, Oxford University Press, New York},
      YEAR = {2000},
  %   PAGES = {xvi+446},
%      ISBN = {0-19-850250-8},
   MRCLASS = {20C15 (20C08 20F55)},
  MRNUMBER = {1778802},
MRREVIEWER = {Jian-yi Shi},
}

@book {Beaulieu1991,
    AUTHOR = {Beaulieu, P. W.},
     TITLE = {A new construction of subgroups inducing isomorphic
              representations},
   %   NOTE = {Thesis (Ph.D.)--Louisiana State University and Agricultural \&
    %          Mechanical College},
% PUBLISHER = {ProQuest LLC, Ann Arbor, MI},
      YEAR = {1991},
  %   PAGES = {27},
   MRCLASS = {Thesis},
  MRNUMBER = {2686575},
 %      URL =    {http://gateway.proquest.com/openurl?url_ver=Z39.88-2004&rft_val_fmt=info:ofi/fmt:kev:mtx:dissertation&res_dat=xri:pqdiss&rft_dat=xri:pqdiss:9200050},
}

@article {Gassmann1926,
    AUTHOR = {Gassmann, F.},
     TITLE = {Bemerkungen zu der vorstehenden Arbeit von Hurwitz (comments on \"{U}ber {B}eziehungen zwischen den {P}rimidealen eines
              algebraischen {K}\"{o}rpers und den {S}ubstitutionen seiner
              {G}ruppe, by Hurwitz)},
   JOURNAL = {Math. Z.},
  FJOURNAL = {Mathematische Zeitschrift},
    VOLUME = {25},
      YEAR = {1926},
    NUMBER = {1},
     PAGES = {665--675}
}

@article {BrylinskiKostant1994,
    AUTHOR = {Brylinski, R. and Kostant, B.},
     TITLE = {Nilpotent orbits, normality and {H}amiltonian group actions},
   JOURNAL = {J. Amer. Math. Soc.},
    VOLUME = {7},
      YEAR = {1994},
      pages={269--298},
    NUMBER = {2},
}

@article {BarbaschVogan1985,
    AUTHOR = {Barbasch, D. and Vogan, D.},
     TITLE = {Unipotent representations of complex semisimple groups},
   JOURNAL = {Ann. of Math. (2)},
    VOLUME = {121},
    pages={41--110},
      YEAR = {1985},
    NUMBER = {1},
}

@article{Namikawa2011,
	year  = {2011},
	publisher = {Duke University Press},
	volume = {156},
	number = {1},
	author = {Namikawa, Y.},
	title = {Poisson deformations of affine symplectic varieties},
	pages={51--85},
	journal = {Duke Math. J.}
}

@misc{IvanNotes,
  author        = {Losev, I.},
  title         = {Automorphisms and isomorphisms, continued.},
  month         = {November},
  year          = {2022},
  howpublished = "\url{https://gauss.math.yale.edu/~il282/Math720_Lec24.pdf}",
}

@article{Beauville2000,
	year  = {2000},
	publisher = {Springer Nature},
	volume = {139},
	number = {3},
	author = {Beauville, A.},
	pages = {541--549},
	title = {Symplectic singularities},
	journal = {Invent.\ Math.}
}

@book{CM,
	Author = {Collingwood, D. and McGovern, W.},
	Title = {Nilpotent Orbits In Semisimple Lie Algebra: An Introduction},
	Publisher = {Chapman and Hall/CRC},
	Year = {1993},
}

@book{Bourbaki46,
	year  = {2002},
	publisher = {Springer},
	author = {Bourbaki, N.},
	pages={1--300},
	title = {Lie Groups and Lie Algebras, Chapters 4-6}
}

@article{Kraft-Procesi,
	year  = {1982},
	publisher = {European Mathematical Publishing House},
	volume = {57},
	number = {1},
	pages = {539--602},
	author = {Kraft, H. and Procesi, C.},
	title = {On the geometry of conjugacy classes in classical groups},
	journal = {Comment. Math. Helv.}
}

@article{SommersMcNinch,
	year  = {2003},
	author = {McNinch, G. and Sommers, E.},
	title = {Component groups of unipotent centralizers in good characteristic},
	journal = {J. Algebra}
}

@article{Losev_isofquant,
	year = {2012},
	volume = {231},
	number = {3-4},
	author = {Losev, I.},
	pages = {1216--1270},
	title = {Isomorphisms of quantizations via quantization of resolutions},
	journal = {Adv. Math.}
}

@article{Namikawa2022,
  year = {2022},
  volume = {28},
  number = {4},
  author = {Yoshinori Namikawa},
  title = {Birational geometry for the covering of a nilpotent orbit closure},
  journal = {Selecta Mathematica}
}

@article{LosevHC,
	year  = {2021},
	journal={Transform. Groups},
	volume={26},
	number={2},
	pages={565--600},
	author = {Losev, I.},
	title = {Harish-Chandra bimodules over quantized symplectic singularities
	}
}

@article {Losev4,
	AUTHOR = {Losev, Ivan},
	TITLE = {Deformations of symplectic singularities and orbit method for
	semisimple {L}ie algebras},
	JOURNAL = {Selecta Math. (N.S.)},
	VOLUME = {28},
	YEAR = {2022},
	NUMBER = {2},
}

@article{Namikawa_bir,
	year  = {2015},
	pages={339--359},
	author = {Namikawa, Y.},
	journal={J. Math. Sci. Univ. Tokyo},
	volume={22},
	number={1},
	title = {Poisson deformations and birational geometry}
}

@article{LosevSRA,
	year = {2019},
	publisher = {Oxford University Press ({OUP})},
	author = {Losev, I.},
	pages={442--472},
	title = {Derived Equivalences for Symplectic Reflection Algebras},
	journal = {Int. Math. Res. Not.}
}

@article{LS,
	year  = {1979},
	publisher = {Oxford University Press ({OUP})},
	volume = {s2-19},
	number = {1},
	author = {Lusztig, G. and Spaltenstein, N.},
	title = {Induced Unipotent Classes},
	journal = {Journal of the London Mathematical Society}
}

@article{Springerconstruction,
    year={1978},
    journal={Invent. Math.},
    volume={44},
    author={Springer, T.A.},
    title={A construction of representations of Weyl groups}}

@article{BK,
	year  = {2004},
	volume = {4},
	number = {3},
	pages = {559-592},
	author = {Bezrukavnikov, R. and Kaledin, D.},
	title = {Fedosov quantization in algebraic context},
	pages = {559--592},
	journal = {Mosc. Math. J.}
}

@book{Lusztig1984,
 author = {George Lusztig},
 publisher = {Princeton University Press},
 title = {Characters of Reductive Groups over a Finite Field. (AM-107)},
 year = {1984}
}

@article{BISWAS,
	year  = {2012},
	volume = {23},
	number = {08},
	pages = {1--21},
	author = {Biswas, I. and Chatterjee, P.},
	title = {{On} {the} {exactness} {of} {Kostant}{\textendash}{Kirillov} {form} {and} {the} {second} {cohomology} {of} {nilpotent} {orbits}},
	journal = {Int. J. Math.}
}

@article {Lusztig1997,
    AUTHOR = {Lusztig, G.},
     TITLE = {Notes on unipotent classes},
   JOURNAL = {Asian J. Math.},
    VOLUME = {1},
      YEAR = {1997},
    NUMBER = {1},
     PAGES = {194--207},
}

@article{Namikawa,
	year  = {2011},
	publisher = {Duke University Press},
	volume = {156},
	number = {1},
	author = {Namikawa, Y.},
	title = {Poisson deformations of affine symplectic varieties},
	journal = {Duke Math. J.}
}

@article{Namikawa2,
	year  = {2010},
	publisher = {Duke University Press},
	volume = {50},
	number = {4},
	pages = {727--752},
	author = {Namikawa, Y.},
	title = {Poisson deformations of affine symplectic varieties,  {II}},
	journal = {Kyoto Journal of Mathematics}
}

@ARTICLE{Namikawa3,
	author = {Namikawa, Y.},
	title = {Flops and Poisson deformations of symplectic varieties},
	journal = {Publ. RIMS, Kyoto Univ},
	year = {2008},
	pages = {259--314}
}

@incollection{Arthur1983,
author={Arthur, J.},
title={On some problems suggested by the trace formula},
booktitle={Lie Group Representations II},
volume={1041},
publisher={Springer-Verlag},
year={1983},
pages={1--49},
}

@book{McGovern1994,
author={McGovern, W.},
title={Completely Prime Maximal Ideals and Quantization},
series={Mem. Amer. Math. Soc.},
publisher={AMS},
date={1994},
}

@article{Kostant1959,
author={Kostant, B.},
title={The principal three-dimensional subgroup and the Betti numbers of a complex simple Lie group},
journal={Amer.\ J.~Math.},
volume={81},
date={1959},
pages={973--1032},
}

@article{Barbasch1989,
author={Barbasch, D.},
title={The unitary dual for complex classical Lie groups},
journal={Invent.\ Math.},
volume={96},
date={1989},
pages={103--176},
}

@article{Joseph1985,
author={Joseph, A.},
title={On the associated variety of a primitive ideal},
journal={J.~Algebra},
volume={93},
pages={509--523},
date={1985},
}

@misc{Mitya2020,
    title={On the affinization of a nilpotent orbit cover},
    author={Matvieievskyi, D.},
    year={2020},
    eprint={2003.09356},
    archivePrefix={arXiv},
    primaryClass={math.RT}
}

@incollection {BorhoMacPherson,
    AUTHOR = {Borho, Walter and MacPherson, Robert},
     TITLE = {Partial resolutions of nilpotent varieties},
 BOOKTITLE = {Analysis and topology on singular spaces, {II}, {III}
              ({L}uminy, 1981)},
    SERIES = {Ast\'{e}risque},
    VOLUME = {101},
     PAGES = {23--74},
 PUBLISHER = {Soc. Math. France, Paris},
      YEAR = {1983},
   MRCLASS = {14L30 (20C30 32C40)},
  MRNUMBER = {737927},
MRREVIEWER = {Naohisa Shimomura},
}

@misc{deGraafElashvili,
    title={Induced nilpotent orbits of the simple Lie algebras of exceptional type},
    author={De Graaf, W. and Elashvili, A.},
    year={2009},
    eprint={0905.2743},
    archivePrefix={arXiv},
    primaryClass={math.RT}
}

@book{Carter1993,
   author={Carter, R.},
   title={Finite groups of Lie type},
   series={Wiley Classics Library},
   publisher={John Wiley \& Sons, Ltd.},
   date={1993},
}

@misc{Wong2023,
      title={Unipotent Representations of Exceptional Richardson Orbits}, 
      author={Wong, K.D.},
      year={2023},
      eprint={1706.02510},
      archivePrefix={arXiv},
      primaryClass={math.RT}
}

@incollection{Barbasch_dualpairs,
  year = {2017},
  publisher = {Birkhäuser},
  pages = {47--85},
  author = {Barbasch, D.},
  volume={323},
  SERIES = {Progr. Math.},
  title = {Unipotent Representations and the Dual Pair Correspondence},
  booktitle = {Representation Theory, Number Theory, and Invariant Theory}
}

@misc{MBMat,
%	author = {Mason-Brown, L. and Matvieievskyi, D.},
	%title = {Unipotent ideals in exceptional types},
%	note={In preparation}
%}

@article{Kaledin2006,
    title={Symplectic singularities from the Poisson point of view.},
    author={Kaledin, D.},
    pages={135--156},
    journal={J. Reine Angew Math.},
    year={2006}
}

@article{WongLVB,
author={Wong, K.D.},
title={Some Calculations of the Lusztig-Vogan Bijection for Classical
Nilpotent Orbits},
journal={J. Algebra},
volume={487},
pages={317--339},
date={2018},
}

@misc{LMBM,
      title={Unipotent Ideals and Harish-Chandra Bimodules}, 
      author={Losev, I. and Mason-Brown, L. and Matvieievskyi, D.},
      year={2021},
      eprint={2108.03453},
      archivePrefix={arXiv},
      primaryClass={math.RT}
}

@article{Sommers2001,
  year = {2001},
  month = sep,
  publisher = {Elsevier {BV}},
  volume = {243},
  number = {2},
  pages = {790--812},
  author = {Eric Sommers},
  title = {Lusztig{\textquotesingle}s Canonical Quotient and Generalized Duality},
  journal = {Journal of Algebra}
}

@article{Sommers1998,
  author={Sommers, Eric},
  journal={International Mathematics Research Notices}, 
  title={A generalization of the Bala-Carter theorem for nilpotent orbits}, 
  year={1998},
  volume={1998},
  number={11},
  pages={539-562}}

@article{Achar2003,
  author={Achar, P.},
  journal={Transf. Groups}, 
  title={An order-reversing duality map for conjugacy classes in Lusztig's canonical quotient}, 
  year={2003},
  volume={8}}

@article{NairPrasad,
  author={Prasad, D. and Nair, A.},
  journal={J.~London Math.}, 
  title={Cohomological representations for real reductive groups}, 
  year={2021},
  volume={104}}
   
\end{document}